\documentclass[11pt]{article}
\newcommand{\rev}[1]{\ifediting\textcolor{red}{#1}\else#1\fi}   
    \newcommand{\revc}{\ifediting\color{red}\fi}
     \newenvironment{revblock}{\begingroup\ifediting\color{red}\fi}{\endgroup}
\usepackage[letterpaper,margin=1in]{geometry}
\usepackage{graphicx}
\usepackage{todonotes}
\usepackage{amsmath,amssymb,amsthm,amsfonts,mathtools,bm}
\usepackage{enumitem}
\usepackage{algorithm,algpseudocode}
\usepackage{booktabs}
\usepackage{microtype}
\usepackage[normalem]{ulem}
\usepackage{xcolor}
\usepackage{xr-hyper}
\usepackage{lineno}
\usepackage{tikz}
\usepackage{subcaption}
\usepackage[colorlinks=true,linkcolor=blue,citecolor=blue,urlcolor=blue]{hyperref} 

\usetikzlibrary{positioning,calc,fit,backgrounds,arrows.meta}

\makeatletter
\newcommand*{\addFileDependency}[1]{%
  \typeout{(#1)}%
  \@addtofilelist{#1}%
  \IfFileExists{#1}{}{\typeout{No file #1.}}%
}
\makeatother

\addFileDependency{supp_v9.tex}
\addFileDependency{supp_v9.aux}

\definecolor{editorange}{RGB}{204,85,0}
\definecolor{editblue}{RGB}{0,82,180}
\definecolor{burntorange}{rgb}{0.8, 0.33, 0.0}
\newif\ifediting
\definecolor{editblue}{rgb}{0.11, 0.22, 0.73}
\newcommand{\mg}[1]{\ifediting\textcolor{editorange}{#1}\else#1\fi}
\newcommand{\mgb}[1]{\ifediting\textcolor{magenta}{#1}\else#1\fi}
\newcommand{\mgc}[1]{\ifediting\textcolor{editblue}{#1}\else#1\fi}
\renewcommand{\varepsilon}{\epsilon}
\newcommand{\mgd}[1]{\textcolor{black}{#1}}
\def\sa#1{\textcolor{black}{#1}}
\def\saa#1{\textcolor{black}{#1}}

\newcommand{\C}{\mathbb{C}}
\newcommand{\I}{I}
\newcommand{\Rea}{\operatorname{Re}}
\newcommand{\D}{\mathbb{D}}

\graphicspath{{Delayed_Acceleration_Arxiv/}{./}}

\definecolor{hamzapurple}{HTML}{6A0DAD}
\definecolor{hamzabackground}{HTML}{E8D9FF}

\newcommand{\hamzatext}[1]{%
    \textcolor{black}{#1}%
}
\newenvironment{hamzablock}
    {\begingroup\color{black}}
    {\endgroup}

\definecolor{wubteal}{HTML}{00695C}
\definecolor{wubbackground}{HTML}{E0F2F1}
\definecolor{wubdeleted}{HTML}{666666}
\definecolor{wubnotetext}{HTML}{C2185B}
\definecolor{wubnotebackground}{HTML}{FCE4EC}
\newcounter{wubcomment}
\newcommand{\wubid}[1]{%
    [WUB-\arabic{wubcomment}%
    \if\relax\detokenize{#1}\relax\else\ / #1\fi]%
}

\newcommand{\wubtext}[1]{{\color{black}#1}}

\newcommand{\wubtextb}[1]{{\color{black}#1}}

\definecolor{originaltextcolor}{HTML}{8A4B08}

\allowdisplaybreaks

\newtheorem{theorem}{Theorem}[section]
\newtheorem{proposition}[theorem]{Proposition}
\newtheorem{lemma}[theorem]{Lemma}
\newtheorem{corollary}[theorem]{Corollary}
\theoremstyle{definition}
\newtheorem{definition}[theorem]{Definition}

\newtheorem{remark}[theorem]{Remark}

\DeclareMathOperator{\Ree}{Re}

\newcommand{\Sml}{\mathcal S_{\mu,L}}
\newcommand{\Fl}{\mathcal F_{L}}
\newcommand{\R}{\mathbb R}

\newcommand{\dd}{\mathrm d}
\newcommand{\abs}[1]{\lvert #1\rvert}
\newcommand{\norm}[1]{\lVert #1\rVert}
\newcommand{\Htwo}{\mathcal H_2}
\newcommand{\Hinf}{\mathcal H_\infty}
\newcommand{\QmuL}{\mathcal Q_{\mu,L}}
\newcommand{\rhost}{\rho^\star}
\newcommand{\Pt}{\widetilde P}
\newcommand{\Pinf}{P_\infty}
\newcommand{\beq}{\begin{eqnarray}}
\newcommand{\eeq}{\end{eqnarray}}
\newcommand{\beqs}{\begin{eqnarray*}}
\newcommand{\eeqs}{\end{eqnarray*}}
\newcommand{\cO}{\mathcal{O}}
\newcommand{\nagP}{P^{\scriptscriptstyle \mathrm{NAG}}_{\tau,\delta}}
\newcommand{\nagPfast}{P^{\scriptscriptstyle \mathrm{NAG}}_{\tau,\Delta}}
\newcommand{\nagPslow}{P^{\scriptscriptstyle \mathrm{NAG}}_{\tau,\delta_\mu}}
\newcommand{\nagPslowone}{P^{\scriptscriptstyle \mathrm{NAG}}_{1,\delta_\mu}}
\newcommand{\hbP}{P^{\scriptscriptstyle \mathrm{HB}}_{\tau,\delta}}
\newcommand{\hbPfast}{P^{\scriptscriptstyle \mathrm{HB}}_{\tau,\Delta}}
\newcommand{\hbPslow}{P^{\scriptscriptstyle \mathrm{HB}}_{\tau,\delta_\mu}}
\newcommand{\hbK}{K^{\scriptscriptstyle \mathrm{HB}}_{\tau,\delta}}

\newcommand{\hbKslow}{K^{\scriptscriptstyle \mathrm{HB}}_{\tau,\delta_\mu}}
\newcommand{\hbPzero}{P^{\scriptscriptstyle \mathrm{HB}}_{0,\delta}}
\newcommand{\hbKone}{K^{\scriptscriptstyle \mathrm{HB}}_{1,\delta}}
\newcommand{\dist}{\operatorname{dist}}

\newcommand{\DREAM}{$\mathsf{DREAM}$}

\title{
\DREAM: Delay-Resistant Acceleration with Momentum\\
}
\author{%
Mert G\"urb\"uzbalaban\thanks{%
M. G\"urb\"uzbalaban (Departments of Management Science and Information Systems, Electrical and Computer Engineering, and Statistics), M. H. Rafi (Department of Electrical and Computer Engineering), and W. U. Bajwa (Departments of Electrical and Computer Engineering and Statistics) are at Rutgers University--New Brunswick, NJ 08854, USA (e-mails: \texttt{\{mg1366, hamza.rafi, waheed.bajwa\}@rutgers.edu}). Part of this work was carried out while Mert G\"urb\"uzbalaban was visiting the Department of Industrial Engineering and Operations Research at Columbia University, New York, NY 10027.}\hspace{0.3em}, M. Hamza Rafi\footnotemark[2]\hspace{0.3em}, Necdet Serhat Aybat\thanks{%
N. S. Aybat is with the Department of Industrial and Manufacturing Engineering, The Pennsylvania State University, University Park, PA 16802, USA (e-mail: \texttt{nsa10@psu.edu}).}\hspace{0.3em}\ and Waheed U. Bajwa\footnotemark[2]\hspace{0.3em}}
\date{June 2026}

\begin{document} 
\setcounter{footnote}{1}
\maketitle

\begingroup
\renewcommand{\thefootnote}{}
\footnotetext{This research was supported in part by the Office of Naval Research under award numbers N00014-24-1-2628 and N00014-24-1-2666.}
\endgroup


\begin{abstract}
When minimizing $L$-smooth, $\mu$-strongly convex objectives using exact gradients without delay, Nesterov's accelerated gradient (NAG) method achieves the optimal $\sqrt{\kappa}$ dependence on the condition number $\kappa=L/\mu$, while the heavy-ball (HB) method achieves the same scaling on strongly convex quadratics. In asynchronous and distributed optimization, however, gradients are typically inexact and may arrive with delays that vary across workers and iterations. To isolate the effect of gradient staleness \wubtext{on accelerated methods in such problems}, we study an idealized \wubtext{mathematical} model with a fixed delay $\tau\ge1$, in which the update at iteration $k$ uses a stale gradient $\nabla f(y_{k-\tau})$ instead of $\nabla f(y_k)$, where $y_k$ is the gradient-evaluation point \wubtext{at iteration $k$ and $y_{k-\tau}$ is the corresponding point from $\tau$ iterations earlier}. For this fixed-delay model, \wubtext{Arjevani, Shamir, and Srebro} %
establish a worst-case lower bound of $\Omega(\tau\sqrt{\kappa}\log(1/\varepsilon))$ iterations \wubtext{to reach objective-value error at most $\varepsilon$ as $\varepsilon\to0$}, raising the question of whether NAG and HB can attain this scaling by adapting their parameters to the delay\wubtext{, while using the available stale gradient at every iteration}. We answer negatively: under any fixed positive delay, neither NAG nor HB can retain $\sqrt{\kappa}$-acceleration on strongly convex quadratics \wubtext{when these methods take a step in each iteration}, regardless of how their constant stepsizes and positive momentum coefficients are chosen to ensure stability. We then introduce delay-resistant \mgd{acceleration with} momentum (\DREAM), a generalized momentum method with explicit delay-dependent parameters that \wubtext{decouples gradient-query extrapolation from momentum, updates at every iteration, and} achieves a linear convergence rate of $1-\Theta(1/(\tau\sqrt{\kappa}))$ on all smooth strongly convex objectives, including quadratics, \mgd{for each $\tau$ and all sufficiently large $\kappa$}. %
This rate yields an iteration complexity that matches the delayed first-order lower bound \wubtext{of Arjevani \emph{et al.}} up to logarithmic factors. Our proof combines characteristic-polynomial stability analysis \wubtext{of the delayed momentum recurrences on quadratics} with an analytic Zames--Falb integral quadratic constraint certificate. We further extend the approach to smooth convex objectives, obtaining near-optimal \wubtext{iteration} complexity up to logarithmic factors in the regime when the target accuracy is sufficiently small. Numerical experiments on strongly convex quadratics and regularized logistic regression corroborate the theoretical results, demonstrating that \DREAM~preserves accelerated condition-number dependence under delays, unlike classical momentum methods. Additional experiments suggest that this accelerated scaling persists even when the algorithm parameters are chosen using a conservative upper bound on the actual delay.
\end{abstract}



\section{Introduction}\label{sec:intro}

We consider unconstrained minimization of a smooth convex function,
\begin{equation}\label{eq:intro-problem}
    \sa{f_*:=\inf_{x\in\R^d}} f(x), \qquad f\in\wubtext{\mathcal S_{0,L}},
\end{equation}
where \wubtext{$\mathcal S_{0,L}$} denotes the class of continuously differentiable \wubtext{convex} functions $f:\R^d\to\R$ \wubtext{with $L$-Lipschitz gradients \sa{for some $L>0$}. We study both this class and its strongly convex subclass $\Sml\subseteq\mathcal S_{0,L}$, with most of our analysis devoted to the subclass $\mathcal{S}_{\mu,L}$.} More precisely, \wubtext{for} $0<\mu\le L$, $\Sml$ \wubtext{consists of functions} satisfying
\begin{equation}\label{eq:intro-function-class}
    \frac{\mu}{2}\norm{x-y}^2
    \le f(x)-f(y)-\langle\nabla f(y),x-y\rangle
    \le \frac{L}{2}\norm{x-y}^2,
    \qquad x,y\in\R^d,
\end{equation}
where we use \wubtext{$\norm{\cdot}$ for} the Euclidean norm. \wubtext{Setting $\mu=0$ in~\eqref{eq:intro-function-class} gives the smooth convex class $\mathcal S_{0,L}$ introduced above. \mgd{In the convex case},} we assume that the minimum is attained, %
\sa{i.e., %
there exists $x_*\in\R^d$ such that $f_*=f(x_*)$.} 
\wubtext{For functions in $\Sml$, the minimizer is unique, and we write} $\kappa\mathrel{\wubtext{:=}}L/\mu$ \wubtext{for the condition number} \sa{of the function $f$}.

In the classical first-order oracle model \wubtext{for iteratively solving~\eqref{eq:intro-problem}}, where an exact gradient \wubtext{$\nabla f(y_k)$ at the current query point $y_k$} is available \wubtext{without \sa{any} delay} at each iteration \wubtext{$k$}, the optimal worst-case iteration complexity for the class $\Sml$ of $L$-smooth, $\mu$-strongly convex functions scales as $\Theta(\sqrt{\kappa}\log(1/\varepsilon))$~\wubtext{\cite{NesterovBook}}, where $\varepsilon>0$ is the target objective-value error. For the class $\mathcal{S}_{0,L}$ of $L$-smooth convex functions, the corresponding complexity is $\Theta(\sqrt{LR^2/\varepsilon})$~\wubtext{\cite{NesterovBook}}, where $R$ bounds the initial distance to a minimizer. \mgc{These lower bounds are established through
hard instances in sufficiently high dimensions.
The underlying finite-iteration lower bounds
hold for every $\kappa>1$ in the strongly convex
case; the commonly stated
$\Theta(\sqrt{\kappa}\log(1/\varepsilon))$
complexity scaling is understood for condition
numbers bounded away from %
$1$ and sufficiently
small target accuracy $\epsilon$.} %
{ \wubtext{In terms of achievability, gradient descent (GD) does not attain these bounds in general: its worst-case %
\sa{iteration complexity} is \sa{$\cO(\kappa\log(1/\varepsilon))$} in the strongly convex case and $\cO(1/\varepsilon)$ in the convex case.} \wubtext{On the other hand,} Nesterov's accelerated gradient method (NAG)~\cite{Nesterov1983} achieves these accelerated rates with constant step-size and momentum parameters for functions in $\Sml$, and with appropriately varying momentum parameters for functions in $\mathcal{S}_{0,L}$~\cite{NesterovBook}. Polyak's heavy-ball method (HB)~\cite{Polyak1964} also achieves the accelerated $\sqrt{\kappa}$ dependence on strongly convex quadratics, although the same guarantee does not extend to arbitrary smooth strongly convex objectives under the classical HB parameter choices~\wubtext{\cite{LessardRechtPackard2016}}. These optimal rates of NAG and HB provide a natural benchmark for \wubtext{characterizing the performance of} first-order methods under settings that deviate from the classical first-order oracle model.

In this paper, we study a \emph{fixed-delay} \wubtext{variation of the classical first-order oracle} model \wubtext{for solving~\eqref{eq:intro-problem}, with a focus on \emph{non-idling} methods}. Specifically, for a fixed integer $\tau\ge1$, the oracle returns the exact gradient $\nabla f(y_{k-\tau})$ \wubtext{at iteration $k$}, where \sa{$y_{k-\tau}$} is the point queried at iteration \sa{$k-\tau$}~\cite{ArjevaniShamirSrebro2020, stich_delayed_sgd, choi2024non}. \mgd{We call a method \emph{non-idling} if, after initialization,
it uses the arriving delayed gradient to update
a single coupled iterate sequence at every iteration.\footnote{\mgd{
Throughout this paper, we focus on single-trajectory
momentum methods, rather than schemes that maintain
multiple independent optimization trajectories and
interleave their updates. Such schemes can achieve
non-idling operation by cyclically updating one
trajectory per iteration while the remaining
trajectories await their delayed gradients.
However, they do not belong to the single-recurrence
GMM framework studied here.
}}}%
An \sa{\textit{idling}} (\emph{waiting}) method instead postpones its next update until the gradient at its current query point is available. This non-idling, fixed-delay setting is motivated by asynchronous and distributed optimization. For example, a worker \wubtext{in such a system} may evaluate a gradient at one iterate while updates from other workers change the shared iterate. When the gradient arrives, it \wubtext{therefore corresponds to an earlier iterate and} is \emph{stale}~\cite{bertsekas2015parallel,soori2020dave,async_saeed,can2021dqn}. %
Although delays in such systems can vary across workers and iterations, we consider exact gradients and a common fixed delay to isolate the effect of staleness. \wubtext{Analyzing this setting can help develop tools for studying more general asynchronous distributed methods, including those with time-varying delays and stochastic gradient errors. \emph{We are particularly interested in whether non-idling methods can attain the optimal iteration-complexity scaling in this fixed-delay model.}}

We use the worst-case lower bounds of Arjevani, Shamir, and Srebro~\cite[Theorem~4]{ArjevaniShamirSrebro2020} as the target iteration-complexity scale for non-idling methods. These bounds hold for a deterministic fixed-delay first-order \emph{span model}, which permits both waiting and non-idling methods. \mgd{In the update from iteration $k$ to $k+1$, the span model requires the new iterate} %
\sa{to be a linear combination of gradients from queries made no later than iteration $k-\tau$}.
\wubtext{The precise span condition is stated later in Theorem~\ref{thm:asr-lower}.}
For sufficiently large dimension \wubtext{and under the same accuracy and conditioning qualifications as above}, their results imply the following iteration-complexity lower bounds for this span model%
\begin{equation}\label{eq:intro-delayed-lower-bounds}
    \begin{aligned}
      \Omega\!\left((\tau+1)\sqrt\kappa\log\frac1{\wubtext{\varepsilon}}\right)
      &\quad\text{\sa{for $f\in \mathcal S_{\mu,L},$}}
      \\
      \Omega\!\left((\tau+1)\sqrt{\frac{LR^2}{\varepsilon}}\right)
      &\quad\text{for \sa{$f\in\mathcal S_{0,L}$}.}
    \end{aligned}
\end{equation}
The key point is that, compared with the undelayed accelerated rates, these lower bounds introduce an additional factor of $\tau+1$. This delay penalty occurs even for quadratic objectives.

\wubtext{In terms of achievability in this fixed-delay, non-idling setting,} Arjevani, Shamir, and Srebro\wubtext{~\cite{ArjevaniShamirSrebro2020}} analyze \emph{delayed} gradient descent (GD),
\begin{equation}\label{eq:intro-dgd}
    x_{k+1}=x_k-\alpha\nabla f(x_{k-\tau}),
\end{equation}
which processes a $\tau$-stale gradient at every iteration. On smooth strongly convex quadratics, their guarantees \wubtext{for delayed GD} have complexity of order $(\tau+1)\kappa\log(1/\epsilon)$, and on smooth convex quadratics the iteration complexity is of order $(\tau+1)LR^2/\varepsilon$~\cite{ArjevaniShamirSrebro2020}. Hence delayed GD matches the linear dependence on delay in the lower bounds, but, as in the undelayed setting, it does not achieve the accelerated dependence on $\kappa$ in the strongly convex case or $\varepsilon$ in the convex case. In particular, they leave open whether an accelerated gradient method can attain these rates while updating with a stale gradient at every iteration. \emph{This raises two related questions:} can the classical HB and NAG recurrences retain acceleration under delay with suitably chosen parameters while remaining non-idling, and, if not, can a more general \wubtext{non-idling} momentum recurrence attain these %
lower complexity bounds? %

To answer these questions, we study a \emph{generalized momentum method} (GMM) with constant parameters under a fixed delay in the non-idling setting. \wubtext{With $x_k$ denoting the main iterate and $y_k$ the gradient query point,} \sa{the \wubtext{\emph{delayed} GMM} updates take the following form:}
\begin{equation}\label{eq:gmm-intro}
    \begin{aligned}
        y_k&=x_k+\nu(x_k-x_{k-1}),\\
        x_{k+1}&=x_k-\alpha\nabla f(y_{k-\tau})
                   +\beta(x_k-x_{k-1}),\qquad k\ge0.
    \end{aligned}
\end{equation}
Here $\alpha>0$ is the stepsize, $\beta\ge0$ is the momentum coefficient, and $\nu\ge0$ is the gradient-query extrapolation coefficient. The parameters \wubtext{$(\alpha,\beta,\nu)$} remain fixed across iterations, although their values may depend %
on the smoothness and strong-convexity %
\sa{constants} and the delay. The initial values $x_{-\tau-1},\ldots,x_0$ determine the query points \wubtext{needed to start the iteration}. We use the terms \emph{delayed HB} and \emph{delayed NAG} for the specializations \wubtext{$\nu=0$ and} $\nu=\beta$\wubtext{, respectively, in~\eqref{eq:gmm-intro}}. \wubtext{These two specializations are the classical HB and NAG updates using the arriving $\tau$-delayed gradient at every iteration.} %
Delayed GMM allows $\nu$ to be chosen independently of $\beta$, rather than requiring $\nu=0$ or $\nu=\beta$. We use this freedom to adapt the extrapolation to the delay parameter $\tau$.

The main results \wubtext{in this paper} identify a sharp qualitative difference between the delayed HB and NAG specializations and the general \wubtext{delayed GMM} family. Under every fixed positive delay \wubtext{$\tau\ge1$}, \wubtext{delayed} HB \wubtext{with constant parameters} fails to retain $\sqrt\kappa$-acceleration on strongly convex quadratics; the same is true of delayed NAG, even when gradients are delayed by only one iteration ($\tau=1$), for every choice of constant parameters in the standard positive-momentum regime \wubtext{that ensures convergence on strongly convex quadratics}. In contrast, we give an explicit choice of parameters, with $\nu$ proportional to the delay, for which delayed GMM achieves $\sqrt{\kappa}$-acceleration on the entire class $\Sml$ for every fixed delay $\tau\geq 1$ and for all sufficiently large $\kappa$. We call this method \emph{delay-resistant acceleration with momentum} (\DREAM). \mgd{Throughout, \emph{delayed GMM} denotes the three-parameter family \eqref{eq:gmm-intro}, whereas \DREAM{} denotes its specific delay-dependent tuning in Theorem~\ref{thm:main-nonexplicit-constants}. The sharper one-delay quadratic tuning of Proposition~\ref{prop:q-gmmaccel} is identified separately.} We also extend this approach to \emph{the smooth convex class $\mathcal S_{0,L}$} \wubtext{by using \DREAM\ as an inner solver in an accelerated proximal-point construction}. For sufficiently small target error, the resulting iteration complexity matches the delay and accuracy dependence of the convex lower bound in~\eqref{eq:intro-delayed-lower-bounds} \wubtext{up to logarithmic factors}.

\subsection{Main results and contributions}
We now state the \wubtext{main} contributions \wubtext{of this paper} more precisely. \wubtext{But first, we clarify their relation to the lower bounds in~\eqref{eq:intro-delayed-lower-bounds}.
These bounds apply to both waiting and non-idling methods, whereas our focus is on non-idling methods.} Indeed, a waiting method can match the order of these lower bounds by
running NAG, making an update only after %
{the gradient at its current query point} has been returned, and remaining idle for the intervening $\tau$ iterations~\cite{ArjevaniShamirSrebro2020}. Thus one accelerated update is performed every $\tau+1$ iterations. This construction establishes attainability of the lower-bound \emph{scaling}, but it does not exploit gradients that arrive while the algorithm is waiting. In an asynchronous distributed architecture, forcing every update to wait for a fresh gradient can sacrifice the benefits of pipelining or require waiting for the slowest worker. 
\mgd{The question we address is therefore whether acceleration is achievable within a single, non-idling momentum recurrence that uses the arriving $\tau$-stale gradient at every iteration.}

\wubtext{We now turn to our main results, presented in Sections~\ref{sec:quadratic}--\ref{sec:convex-restart}.} In what follows, \wubtext{we also use the class} $\mathcal Q_{\mu,L}\subset\Sml$ of quadratic objectives whose symmetric Hessians have spectra contained in $[\mu,L]$. Throughout these results, $\tau$ is \wubtext{a} fixed positive integer.

\paragraph{Limitations of delayed HB and NAG.} \wubtext{In Section~\ref{sec:quadratic},} we first characterize the convergence of \wubtext{delayed} GMM on $\mathcal Q_{\mu,L}$. For fixed parameters $(\alpha,\beta,\nu)$, the \wubtext{delayed} GMM iterates
converge to the minimizer $x_*$ for every $f\in\mathcal Q_{\mu,L}$ and every initial history if and only if
\begin{equation*}
    \rho(\alpha,\beta,\nu)
    :=\sup_{\lambda\in[\mu,L]}
    \max\bigl\{|z|:\wubtext{z\in\mathbb C,\,}P_{\tau,\alpha\lambda}(z)=0\bigr\}<1,
\end{equation*}
where \wubtext{the characteristic polynomial of the scalar delayed GMM error recurrence along a Hessian eigendirection with eigenvalue $\lambda$ is given by}
\begin{equation}\label{eq:intro-characteristic}
    P_{\tau,\alpha\lambda}(z)
    \mathrel{\wubtext{:}=}z^\tau(z-1)(z-\beta)
      +\alpha\lambda\bigl((1+\nu)z-\nu\bigr).
\end{equation}
When $\rho(\alpha,\beta,\nu)\in(0,1)$, the iterates satisfy
\begin{equation*}
    \norm{x_k-x_*}
    \le c_k\,\rho(\alpha,\beta,\nu)^k,
    \qquad k\ge 0,
\end{equation*}
for some non-negative sequence $\{c_k\}_{k\ge 0}$
that grows at most polynomially in $k$. 
We call such a parameter choice \emph{Schur stable}, as %
\sa{it implies} %
that the polynomial $P_{\tau,\alpha\lambda}$ is \emph{Schur stable}, i.e., \sa{all the roots lie} strictly inside the unit circle for all $\lambda \in [\mu,L]$. Thus, $\rho(\alpha,\beta,\nu)$ characterizes
the worst-case asymptotic convergence factor
over $\mathcal Q_{\mu,L}$. In particular,
for every $\widehat\rho$ satisfying
$\rho(\alpha,\beta,\nu)<\widehat\rho<1$,
the iterates admit a geometric bound of the form
$\|x_k-x_*\|\le C_{\widehat\rho}
\widehat\rho^k$ for some constant $C_{\widehat\rho}\wubtext{>0}$ \wubtext{independent of $k$}.%

For the HB and NAG specializations of delayed GMM, let $\rho^*_{\rm HB}(\tau,\kappa)$ and $\rho^*_{\rm NAG}(\tau,\kappa)$ denote the infima of $\rho(\alpha,\beta,0)$ and $\rho(\alpha,\beta,\beta)$, respectively, over Schur-stable choices with $\alpha>0$ and $0<\beta<1$.\footnote{\wubtext{For these two specializations, Schur stability with $\tau\ge1$ requires $\beta<1$, as established in \mgc{Lemma \ref{lem:q-nag-beta-upper} and \mgd{Proposition} \ref{prop:q-hbfast} of Section~\ref{sec:quad-setup}}.}} 
\mgd{For a Schur-stable tuning, we call
$1-\rho(\alpha,\beta,\nu)$ its \emph{spectral gap}.
Accordingly, $1-\rho^*_{\mathrm{fam}}(\tau,\kappa)$
is the supremum of the spectral gaps over
Schur-stable tunings within the family
$\mathrm{fam}\in\{\mathrm{HB},\mathrm{NAG}\}$.}
 Theorems~\ref{thm:q-hblower} and~\ref{thm:q-naglower} establish that, for every fixed $\tau\ge1$, as $\kappa\to\infty$,
\begin{equation}\label{eq:intro-classical-barriers}
\begin{aligned}
1-\rho^*_{\rm HB}(\tau,\kappa)
&=\cO_\tau(\kappa^{-1}), &&\tau\ge1,\\
1-\rho^*_{\rm NAG}(\tau,\kappa)
&\le(2+o(1))\kappa^{-2/3}, &&\tau=1,\\
1-\rho^*_{\rm NAG}(\tau,\kappa)
&=\cO_\tau(\kappa^{-1}), &&\tau\ge2.
\end{aligned}
\end{equation}
{Here and throughout, \wubtext{$o(1)$ is with respect to $\kappa\to\infty$, while} $\cO_\tau(\cdot)$ denotes
big-O notation with constants that may depend
on $\tau$ but are independent of $\kappa$.}
Thus, even with parameters adapted to the delay, neither delayed HB nor delayed NAG retains the accelerated $\sqrt{\kappa}$ dependence under any fixed positive delay. {For the single-delay case \(\tau=1\), Lemma~\ref{lem:q-nag-tau1-achievability} also establishes that, for every \(\kappa\ge64\),
\[
1-\rho^*_{\rm NAG}(1,\kappa)
\ge\frac{7}{22}\kappa^{-2/3}.
\]
Together with the upper bound above, this shows that the optimal spectral gap for delayed NAG with \(\tau=1\) has order \(\Theta(\kappa^{-2/3})\). Thus, for %
\sa{$\tau=1$}, \wubtext{delayed} NAG improves on the \(\kappa^{-1}\) gap of delayed HB, but neither method attains the accelerated \(\kappa^{-1/2}\) gap.}
\wubtext{In summary, our first contribution is a characterization of delayed GMM convergence on quadratics through its characteristic polynomial. We then use this characterization to derive stability and spectral-gap bounds showing that neither delayed HB nor delayed NAG retains acceleration under any fixed positive delay.}}

\paragraph{Acceleration and optimal parameters for delayed GMM on quadratics.} \mgc{Our next contribution, established in Section~\ref{sec:quadratic}, shows that allowing the extrapolation parameter $\nu$ to be
chosen independently of $\beta$ enables delayed GMM to achieve acceleration on $\mathcal Q_{\mu,L}$. Specifically, for the one-step delay $\tau=1$,
Proposition~\ref{prop:q-gmmaccel} gives an explicit
\wubtext{delayed} GMM tuning with $\rho(\alpha,\beta,\nu)=1-1/(2\sqrt{\kappa})$ for every $\kappa\ge1$.
Moreover, Proposition~\ref{prop:tau-one-optimality}
shows that the coefficient $1/2$ \wubtext{in the spectral gap $1-\rho(\alpha,\beta,\nu)=1/(2\sqrt{\kappa})$}
is asymptotically sharp \wubtext{as $\kappa\to\infty$} among Schur-stable,
constant-parameter \wubtext{delayed} GMM tunings with $\beta,\nu\ge0$.}

\mgc{To obtain accelerated tunings for every fixed delay $\tau\ge1$,\footnote{\wubtext{For $\tau=1$, the characteristic polynomial in \eqref{eq:intro-characteristic} is cubic, allowing an explicit tuning with an exact worst-case spectral radius. The general-delay result in Theorem~\ref{thm:q-gate} also applies when $\tau=1$, but gives a rate bound for sufficiently large $\kappa$.}} we \wubtext{next} consider the
parameterization
\beq
    \alpha=\frac{\delta_L}{L},
    \qquad
    \beta=1-\frac{2\sqrt{\delta_L}}{\sqrt{\kappa}}, \label{eq-alpha-beta-choice-delayed-gmm-intro}
\eeq
with $\delta_L>0$ and $\nu$ to be chosen.
Theorem~\ref{thm:q-gate} \wubtext{then} shows that, on $\mathcal Q_{\mu,L}$,
this parameterization satisfies $\rho(\alpha,\beta,\nu)\le1-\gamma/\sqrt{\kappa}$ for every
$0<\gamma<\sqrt{\delta_L}$ and all sufficiently
large $\kappa$, provided that \wubtext{the limiting characteristic polynomial as $\kappa\to\infty$}
\[
    P_\infty(z;\delta,\nu)
    :=z^\tau(z-1)^2+\delta[(1+\nu)z-\nu]
\]
is Schur stable for every
$0<\delta\le\delta_L$.
The theorem also shows that $\nu>\tau$ is necessary
for this stability condition.}\looseness=-1

Since Theorem~\ref{thm:q-gate} allows every $0<\gamma<\sqrt{\delta_L}$, increasing the admissible value of $\delta_L$ improves the guaranteed convergence rate. We therefore choose $\delta_L$ and $\nu$ to make $\delta_L$ as large as the stated stability condition permits.
Let $\delta_\tau^{\mathrm{crit}}$ denote the supremum
of admissible values of $\delta_L$ over
$\nu>\tau$.
Corollary~\ref{cor:q-designoptimal} characterizes
this threshold exactly and identifies the
corresponding optimal extrapolation parameter
$\nu_\tau^{\mathrm{crit}}$.
In particular, every $0<\gamma<\sqrt{\delta_\tau^{\mathrm{crit}}}$
is achievable for sufficiently large $\kappa$
by choosing $\nu=\nu_\tau^{\mathrm{crit}}$ and
$\gamma^2<\delta_L<\delta_\tau^{\mathrm{crit}}$.
Thus, $\sqrt{\delta_\tau^{\mathrm{crit}}}$ is the
supremum of the rate coefficients $\gamma$ guaranteed
by Theorem~\ref{thm:q-gate} when $\delta_L$ and $\nu$ are varied. These results provide explicit accelerated tunings of delayed GMM on $\mathcal Q_{\mu,L}$, with an asymptotically sharp spectral-gap coefficient for $\tau=1$ and, for each fixed $\tau\ge1$, an exact characterization of the supremum of rate coefficients guaranteed by Theorem~\ref{thm:q-gate} within the displayed parameterization. We also consider a simple explicit tuning. Writing $M:=2\tau+1$, we take
\beq\mgd{\delta_L=1/M^2 \quad \mbox{and} \quad \nu=M-\tfrac12
\label{eq-param-gmm-intro}}.\eeq Theorem~\ref{thm:q-alldelay} then guarantees an accelerated rate constant $\gamma=1/(2M)$ \wubtext{on $\mathcal Q_{\mu,L}$}.
Moreover, Proposition \ref{prop:q-scaling} in \mgd{Appendix~\ref{sec:design}} shows that, as $\tau\to\infty$,
\[
\frac{\sqrt{
\wubtext{\delta_\tau^{\mathrm{crit}}}}}{1/(2M)}
\longrightarrow
4a_\star\sqrt{\cos a_\star}<3,
\]
where $a_\star\in(0,\pi/2)$ is the unique solution
of $a_\star\tan a_\star=2$. Thus, for sufficiently large delays, this simple explicit tuning
retains more than one third of the best rate coefficient guaranteed
by Theorem~\ref{thm:q-gate}. We adopt this tuning for \DREAM{} because the
same parameters also guarantee accelerated convergence on $\Sml$,
as discussed next.

\paragraph{Accelerated convergence of \DREAM\ on smooth strongly convex objectives.}
Our next contribution, presented in Sections~\ref{sec-strongly-convex}--\ref{subsec:main-proof}, is an accelerated convergence guarantee for \DREAM\ on the full class $\Sml$. Recall \mgd{from \eqref{eq-alpha-beta-choice-delayed-gmm-intro} and \eqref{eq-param-gmm-intro}} that \DREAM{} uses the explicit parameter choice:
\begin{equation}\label{eq:intro-dream-parameters}
    \alpha=\frac{1}{M^2L},\qquad
    \beta=1-\frac{2}{M\sqrt\kappa},\qquad
    \nu=M-\frac12=2\tau+\frac12.
\end{equation}

Proving convergence over the broader function class $\Sml$ for \DREAM\ requires a different proof approach from that used for quadratics. Theorem~\ref{thm:main-nonexplicit-constants},
\wubtext{stated in Section~\ref{sec-strongly-convex} and proved in Section~\ref{sec:master}},
shows that, for every fixed integer
$\tau\ge1$ and every $c\in(0,2/3)$, there exists
$\kappa_0(\tau,c)$ such that, for
$\kappa\ge\kappa_0(\tau,c)$, \wubtext{the \DREAM\ iterates for} every $f\in\Sml$
and every choice of initial iterates $x_{-\tau-1},\ldots,x_0\in\mathbb R^d$ satisfy
\begin{equation}\label{eq:intro-main-rate}
    \norm{x_k-x_*}\le \mgd{C_{\mathrm{rate}}}\rho^k,
    \qquad
    \rho=1-\frac{c}{M\sqrt\kappa},
    \qquad k\ge0,
\end{equation}
where $\mgd{C_{\mathrm{rate}}}$ may depend on the initial history and
problem parameters but is independent of $k$. Note that this bound does not imply that the distance to the minimizer decreases at every iteration. \mgd{Next, for the constant-history initialization
$x_{-\tau-1}=\cdots=x_0=x_{\rm init}$,
\mgd{the case $x_{\rm init}=x_*$ is trivial,
since all iterates remain at the minimizer.
For $x_{\rm init}\neq x_*$,}
Theorem~\ref{thm:main}
\wubtext{in Section~\ref{subsec:main-proof}}
gives polynomial bounds
\wubtext{on the threshold $\kappa_0(\tau,c)$
and the prefactor $\mgd{C_{\mathrm{rate}}}$
in~\eqref{eq:intro-main-rate}}:}
\begin{equation}\label{eq:intro-polynomial-bounds}
    \kappa_0(\tau,c)=\cO_c(M^8),
    \qquad
    \mgd{C_{\mathrm{rate}}}=\cO_c(M^5\kappa^2)
       \norm{x_{\rm init}-x_*},
\end{equation}
where $\cO_c(\cdot)$ %
{hides} constants depending
\wubtext{only} on $c$. \mgd{The threshold estimates are proved in Appendices~\ref{sec:finiteq}
and~\ref{sec:midpoint}; the prefactor estimate is proved in
Appendix~\ref{sec:explicit}, using the auxiliary bounds in
Appendix~\ref{app:gains}.} Consequently, for
$R\mathrel{\wubtext{:}=}\norm{x_{\rm init}-x_*}>0$ and
$0<\epsilon<\mu R^2/4$, smoothness implies
that the accuracy $f(x_k)-f_*\le\epsilon$
is attained within
\begin{equation}\label{eq:intro-sc-complexity}
    \cO_c\!\left(
    M\sqrt\kappa
    \left[
      \log\frac{\mu R^2}{\epsilon}
      +\log M+\log\kappa
    \right]\right)
\end{equation}
updates. %
\wubtext{Thus, \DREAM\ achieves accelerated convergence on $\Sml$ without idle iterations. In particular, its iteration complexity matches the strongly convex lower bound in~\eqref{eq:intro-delayed-lower-bounds} up to the displayed logarithmic factors}.

\paragraph{\wubtext{Extending \DREAM\ to smooth convex optimization}.}
\wubtext{Our next contribution, presented in Section~\ref{sec:convex-restart}, is an accelerated method for smooth convex optimization that uses \DREAM\ to approximately solve a sequence of strongly convex regularized problems, with the outer iterations following the Catalyst accelerated proximal-point framework~\cite{lin2018catalyst}.
Specifically, let} $f\in\mathcal S_{0,L}$ \wubtext{have} a nonempty minimizer set. Fix $x_*\in\arg\min f$ and \wubtext{write} $R_\star\mathrel{\wubtext{:}=}\norm{x_0-x_*}$.
\wubtext{At each outer iteration $s\ge1$, we run \DREAM\ for a prescribed number of steps to obtain an approximate minimizer $x_s$ of}
\[
    H_s(x)=f(x)+\frac{\lambda}{2}\norm{x-v_{s-1}}^2
\]
\wubtext{for a regularization parameter $\lambda>0$. Here $v_{s-1}$ is obtained by extrapolating the preceding outer iterates, as specified later in Section~\ref{sec:convex-restart}. Before each run of \DREAM, its starting point and all past iterates required by the delayed recurrence are set to $v_{s-1}$.}
Since $H_s$ is $\lambda$-strongly convex and $(L+\lambda)$-smooth, \wubtext{its} condition number is $\chi_\lambda=1+L/\lambda$. \wubtext{Set} $\bar\kappa_0:=\max\{2,\kappa_0(\tau,c)\}$ and choose $\lambda=L/(\bar\kappa_0-1)$\wubtext{. Then} $\chi_\lambda=\bar\kappa_0\mathrel{\wubtext{\ge}}\wubtext{\kappa_0(\tau,c)}$\wubtext{, so \DREAM's strongly convex convergence guarantee in Theorem~\ref{thm:main} applies to minimizing each regularized objective $H_s$}.
\wubtext{Consequently,} Theorem~\ref{thm:convex-regularized} \wubtext{and Corollary~\ref{cor:main-complexity-convex} show that, for every $\varepsilon>0$, the method produces} an outer iterate \wubtext{$x_s$} satisfying $f(x_{\wubtext{s}})-f_*\le\varepsilon$ \wubtext{using a total of}
\begin{equation}\label{eq:intro-convex-complexity}
    \widetilde{\cO}_c\!\left(
        M\left[\sqrt{\bar\kappa_0}
             +\sqrt{\frac{LR_\star^2}{\varepsilon}}\right]\right)
\end{equation}
inner delayed-gradient updates\wubtext{. Here} $\widetilde{\cO}_c$ \wubtext{denotes big-O notation with logarithmic factors omitted and a constant that may depend on $c$}. The \wubtext{$M\sqrt{\bar\kappa_0}$} term \wubtext{in this expression} accounts for \wubtext{the cost of the first run of \DREAM, and the prescribed inner iteration counts do not require prior knowledge of $R_\star$. When} $\varepsilon\le\lambda R_\star^2$, \wubtext{the first term} is \wubtext{bounded} by \wubtext{a constant multiple of} the second \sa{term\wubtext{.}} \wubtext{Recalling that $M=2\tau+1$, the bound therefore simplifies to}
\[
    \widetilde{\cO}_c\!\left((\tau+1)
        \sqrt{\frac{LR_\star^2}{\varepsilon}}\right).
\]
\wubtext{Thus, when $\varepsilon\le\lambda R_\star^2$, the accelerated proximal-point method with \DREAM\ as its inner solver matches the} delay and accuracy dependence \sa{of} the convex lower bound in~\eqref{eq:intro-delayed-lower-bounds}\wubtext{, up to logarithmic factors}. 
\mgd{The complexity bound counts delayed-gradient
updates performed by the inner solvers and does
not include any additional wall-clock latency
associated with initializing a new gradient-computation
pipeline at the start of each regularized subproblem.}

\paragraph{An analytic proof technique valid for every fixed delay.}
Our convergence guarantees for general $f\in\Sml$ do not follow
from quadratic stability analysis alone: the gradient mapping is
nonlinear, and stability on quadratics does not necessarily imply
convergence at the same rate over the full function class.
To overcome this difficulty, we represent~\eqref{eq:gmm-intro}
as a feedback interconnection between a linear delayed system
and the nonlinear gradient mapping.
Our analysis relies on finite-horizon, rate-weighted Zames--Falb
integral quadratic constraints (IQCs), which exploit relations
among gradients evaluated at different iterates.
Unlike pointwise inequalities derived from smoothness and strong
convexity, these constraints capture additional relations between
gradients along the trajectory.\looseness=-1

\mgd{To certify a linear convergence rate $\rho\in(0,1)$,
we consider the centered, rate-weighted iterates
$\widehat{x}_k:=\rho^{-k}(x_k-x_*)$.
Uniform boundedness of these weighted iterates implies
$\|x_k-x_*\|\le C_{\mathrm{rate}}\rho^k$
for some finite constant $C_{\mathrm{rate}}$.} \mgd{We then decompose the gradient mapping into a
linear component with slope $(\mu+L)/2$, the midpoint of the interval $[\mu,L]$, 
and a nonlinear remainder. This choice centers the range of possible
curvatures around zero after subtracting
the linear component. Absorbing this linear component into the
delayed GMM recurrence yields a linear
system driven by the nonlinear remainder.
We establish Schur stability of this system
in rate-weighted coordinates and analytically
verify a strict frequency-domain inequality. Through the Kalman--Yakubovich--Popov lemma,
this inequality yields a quadratic energy estimate which,
combined with the hard IQC, establishes boundedness of
the rate-weighted iterates and hence the desired
linear convergence rate.}
For our delay-dependent parameterization, we show that,
as $\kappa\to\infty$, verification of the frequency-domain
condition reduces to proving positivity of an explicit
trigonometric polynomial.
To extend the result to sufficiently large finite $\kappa$,
we establish quantitative estimates separately for frequencies
near zero, where the limiting approximation requires special
care, and for frequencies away from zero.
This yields explicit convergence guarantees and the polynomial
bounds in~\eqref{eq:intro-polynomial-bounds}.\looseness=-1

\mgd{To the best of our knowledge, existing IQC-based
approaches to the synthesis and convergence-rate
certification of accelerated momentum methods with
delayed gradients have relied on numerically solving
linear matrix inequalities or semidefinite programs \cite{michalowsky2021robust,
SchererEbenbauer2021,
SchererEbenbauerHolicki2023,
abdelsamed2026synthesistransient}.
For prescribed values of the delay $\tau$ and
condition number $\kappa$, these approaches can
synthesize algorithm parameters and numerically
certify convergence rates. However, their results
for delayed gradients do not provide closed-form
expressions for the certified rates in terms of
$\tau$ and $\kappa$, nor do they establish
analytically how the spectral gap scales with
these parameters as $\kappa\to\infty$.
Furthermore, although feasible matrix-inequality
certificates imply finite transient prefactors,
these numerical approaches do not provide explicit
bounds on their dependence on $\tau$ and $\kappa$.
Consequently, they do not establish iteration-complexity
guarantees with provably near-optimal dependence on
the delay and condition number.} \mgd{In contrast, our approach provides explicit
delay-dependent algorithm parameters and analytically
certifies the convergence factor
$\rho=1-c/((2\tau+1)\sqrt{\kappa})$
for every fixed integer $\tau\ge1$, every
$c\in(0,2/3)$, and all sufficiently large $\kappa$,
without requiring numerical optimization,
matrix-inequality feasibility tests, or numerical
frequency searches.}
\mgc{The resulting method is a constant-parameter,
non-idling momentum algorithm that achieves accelerated
convergence using stale gradients, supported by a
fully analytic convergence certificate valid for
every fixed positive delay. Moreover, under
constant-history initialization, our analysis
provides explicit polynomial bounds on the
condition-number threshold and transient prefactor.
Together, these results yield iteration-complexity
guarantees that match the delayed first-order
lower bounds up to logarithmic factors.}

\subsection{Related work}

\paragraph{Delayed gradients and asynchronous first-order methods.}
Choi, Choi, and Seok~\cite{choi2024non} establish non-ergodic linear-convergence estimates for delayed gradient descent under smooth strong convexity and extend their analysis to objectives satisfying the Polyak--\L{}ojasiewicz condition, including stochastic, time-varying delays. In the stochastic setting, Stich and Karimireddy~\cite{stich_delayed_sgd} establish nonasymptotic bounds for delayed stochastic gradient descent on smooth quasi-convex and nonconvex objectives; delays affect the deterministic optimization term but not the leading stochastic term, and \sa{\textit{accelerated} methods are \textit{not}} considered. Fang, Huang, and Lin~\cite{FangHuangLin2018} introduce momentum compensation for asynchronous first-order methods with bounded, potentially time-varying delays. By modifying Nesterov's extrapolation point to account for momentum accumulated during the delay, their accelerated method attains accelerated dependence on the condition number~\cite[Corollary~1]{FangHuangLin2018}, but its iteration complexity depends \textit{quadratically} on the delay bound. 
Ajanthan et al.~\cite{pmlr-v267-ajanthan25a} propose a modified Nesterov method for asynchronous pipeline-parallel optimization and establish sublinear convergence under fixed gradient delays. Pond, Zhao, and Hale~\cite{pond2026distributedasynchronousgeneralizedmomentum} study distributed generalized momentum methods allowing arbitrarily long computation and communication delays, obtaining convergence guarantees under additional structural assumptions. These settings and guarantees differ from the fixed-delay, constant-parameter recurrence studied here. Devolder, Glineur, and Nesterov~\cite{devolder2014first,devolder2013stronglyconvex} develop inexact first-order oracle frameworks encompassing gradients evaluated at shifted points. Although delayed gradients can be modeled in this way, the resulting errors depend on the iterate history, and standard inexact-oracle bounds do not directly establish accelerated convergence. \mgc{
Related asynchronous acceleration results have also
been established under different oracle and
computational models. Hannah, Feng, and Yin
\cite{hannah2018abcd} develop an accelerated
asynchronous randomized block coordinate descent
method with optimal complexity under suitable
restrictions on the delay. \mgd{Hendrikx, Bach, and Massouli\'e
\cite{HendrikxBachMassoulie2019} establish accelerated
convergence for decentralized optimization with
randomized local updates, where delays arise from
communication and computation across a network.
Their convergence guarantees depend on the
network structure and the randomized update
mechanism. In contrast, we consider a deterministic
fixed-delay model in which the full gradient
$\nabla f(y_{k-\tau})$ is used to update the iterates
at every iteration. Consequently, their analysis
does not directly apply to our setting, although
both works demonstrate that acceleration is possible
in the presence of delays.}\looseness=-1
}

\paragraph{Foundations of IQCs and exponential-rate certificates.}
Zames and Falb~\cite{ZamesFalb1968} introduce multipliers for monotone and slope-restricted nonlinearities, while Megretski and Rantzer~\cite{MegretskiRantzer1997} develop the integral-quadratic-constraint (IQC) framework for robust stability analysis. Related foundations include Willems' dissipativity theory~\cite{Willems1972} and Rantzer's frequency-domain formulation of the Kalman--Yakubovich--Popov lemma~\cite{Rantzer1996}. Pfifer and Seiler~\cite{PfiferSeiler2015} develop IQCs for delayed systems. Lessard, Recht, and Packard~\cite{LessardRechtPackard2016} introduce a systematic IQC approach to \sa{analyze} first-order optimization algorithms \sa{in a unified manner}, including gradient descent, heavy-ball, and Nesterov's method, \sa{while} Hu and Lessard~\cite{hu2017dissipativity} provide a complementary dissipativity analysis of acceleration. Boczar, Lessard, and Recht~\cite{BoczarLessardRecht2015} develop rate-weighted IQCs, including Zames--Falb multipliers for slope-restricted nonlinearities, to certify prescribed exponential convergence rates through frequency-domain inequalities that can be checked numerically using linear matrix inequalities.\vspace*{-2mm}

\paragraph{Zames--Falb analysis and undelayed momentum methods.}
Generalized momentum parameterizations separating extrapolation from inertia have also been studied without delays~\cite{VanScoyFreemanLynch2018,gurbuzbalaban2023robustly,can2025entropic}. Cyrus, Hu, Van Scoy, and Lessard~\cite{cyrus_robust_momentum} employ a rate-weighted Zames--Falb IQC to analyze robust momentum without delays and derive an analytical tuning that trades convergence speed against robustness to gradient errors. Freeman~\cite{Freeman2018} extends exponential-rate Zames--Falb analysis to noncausal multipliers for feedback interconnections of linear time-invariant systems and static nonlinearities, yielding less conservative convergence-rate certificates. The author also establishes the validity of these multipliers for vector-valued gradient mappings by exploiting cyclic monotonicity. Zhang, Seiler, and Carrasco~\cite{ZhangSeilerCarrasco2022} show that nonlinear feedback systems can converge more slowly than their linear counterparts and develop a convex search over noncausal Zames--Falb multipliers supported on finitely many lags to certify exponential convergence rates. Van Scoy and Lessard~\cite{scoy_lessard_15} develop an undelayed generalized momentum method for twice continuously differentiable smooth strongly convex objectives. Using an off-by-one Zames--Falb multiplier, they establish global convergence and characterize the asymptotic rate through linearization at the minimizer, obtaining a bound of the form $\|x_k-x_*\|\le c(k)\rho^k$ with \sa{a} subexponentially growing $c(k)$. \mgd{Their iteration-complexity comparison focuses on
the leading-order convergence rate as $k\to\infty$,
without explicitly accounting for the contribution
of $\log c(k)$.}\vspace*{-2mm}

\paragraph{IQC-based numerical certificates with delayed gradients.}
Michalowsky, Scherer, and Ebenbauer~\cite{michalowsky2021robust} develop an IQC-based framework for analyzing and synthesizing gradient algorithms with certified convergence rates and robustness guarantees, using rate-weighted Zames--Falb multipliers and exploiting structural properties of the objective. Scherer and Ebenbauer~\cite{SchererEbenbauer2021} develop a general convex synthesis framework based on rate-weighted Zames--Falb IQCs and the Youla parameterization, in which algorithm parameters and convergence certificates are obtained through semidefinite programming. Their framework accommodates arbitrary fixed gradient delays, as illustrated by numerical synthesis and convergence-rate certification for one- and two-step delays; Scherer, Ebenbauer, and Holicki~\cite{SchererEbenbauerHolicki2023} provide a tutorial treatment, including delayed gradient information. \mgd{In concurrent work, Abdelsamed
et al.~\cite{abdelsamed2026synthesistransient}
develop an IQC-based semidefinite programming framework
for certifying finite-horizon worst-case errors and
synthesizing optimization algorithms for strongly convex and smooth objectives, including methods
with delayed gradient measurements. Their SDPs involve
$T\times T$ matrix decision variables, and storing the
full gradient history for the synthesized algorithms
requires $\mathcal O(dT)$ memory, where $T$ is the
horizon length and $d$ is the problem dimension.
The authors identify these computational and memory
requirements as limitations for large horizons and
high-dimensional problems.} Together, these works establish rigorous and general
IQC-based frameworks for analyzing and synthesizing
optimization algorithms.
\mgd{The earlier delayed-gradient synthesis results
use numerical semidefinite optimization to select
algorithm parameters and certify asymptotic convergence
rates for prescribed delays. The more recent
finite-horizon framework certifies transient performance
and synthesizes algorithms for a prescribed horizon.
Although these certificates provide rigorous
performance guarantees for prescribed problem
parameters, the reported delayed-gradient results
do not establish closed-form, delay-dependent
constant-parameter momentum tunings with an
analytically certified convergence rate exhibiting
explicit dependence on both the delay and the
condition number. Furthermore, while asymptotic convergence-rate
certificates imply finite transient prefactors
and finite-horizon certificates directly bound
transient errors, these approaches do not establish
explicit polynomial bounds on the transient
prefactor and condition-number threshold that
yield near-optimal iteration complexity for a
single non-idling momentum recurrence.} In contrast, we derive closed-form, delay-dependent parameters for a generalized momentum method and analytically verify a delay- and rate-dependent Zames--Falb certificate for every fixed positive delay and sufficiently large condition number, without numerical optimization or semidefinite feasibility searches. Our convergence guarantees hold over the full class of smooth strongly convex objectives and include a quantitative condition-number threshold and explicit polynomial bounds on the transient prefactor. These results yield near-optimal iteration complexity in the delay and condition number, up to logarithmic factors. We further extend our guarantees to general smooth convex objectives through a regularized proximal-point construction, obtaining near-optimal complexity in the high-accuracy regime while explicitly accounting for initialization and transient costs.

\subsection{Organization}\label{subsec:roadmap}

\mgd{Section~\ref{sec:prelim} introduces the delayed
first-order model and notation. Section~\ref{sec:quadratic} analyzes the
linear error dynamics of delayed GMM on
strongly convex quadratics through their
associated characteristic polynomials,
whose roots determine stability and
convergence rates. It then establishes
the limitations of classical momentum
methods and derives accelerated parameter
choices for GMM. Section~\ref{sec-strongly-convex} states
our main accelerated convergence guarantee for
general smooth strongly convex objectives and
introduces the linear-system representation and
integral quadratic constraints underlying its proof. Section~\ref{sec:master} then combines these tools
to construct an analytic convergence-rate
certificate and proves the main theorem by
verifying the required stability and
frequency-domain conditions. Section~\ref{subsec:main-proof} outlines the
quantitative bounds on the condition-number
threshold and transient prefactor.
Section~\ref{sec:convex-restart} extends the
approach to smooth convex objectives, followed
by numerical experiments and concluding remarks.}\looseness=-1

\mgd{Appendix~\ref{subsec:lemma-proofs} provides the
supporting proofs for the convergence certificate,
while Appendix~\ref{sec:design} investigates
optimal parameter design on quadratics.
Appendices~\ref{sec:finiteq} and~\ref{sec:midpoint}
establish the quantitative frequency-domain and
stability estimates underlying the explicit
condition-number threshold.
Finally, Appendix~\ref{sec:explicit} establishes
the transient-prefactor and initialization bounds,
using the auxiliary estimates for the iterates,
extrapolated points, and IQC terms developed in
Appendix~\ref{app:gains}.}

\subsection{Notation}\label{subsec:notation}
For nonempty $S\subset\mathbb C$, set
$\operatorname{dist}(w,S)=\inf_{\zeta\in S}|w-\zeta|$;
$\mathbf 1_{\{E\}}$ denotes the indicator of a statement $E$. \mgd{The set of positive integers is denoted by $\mathbb{Z}_+$}. For a square matrix $A$, its \emph{spectrum} $\sigma(A)$ is the set of its eigenvalues, and its \emph{spectral radius} is
$\rho(A)=\max_{\lambda\in\sigma(A)}|\lambda|$.
For a nonconstant polynomial $P$, we similarly define
$\rho(P)=\max\{|\lambda|:P(\lambda)=0\}$; equivalently,
$\rho(P)=\rho(C_P)$, where $C_P$ is its companion matrix.
A square matrix or polynomial is \emph{Schur stable} if its spectral radius is strictly less than one. \mgd{We use the terms \emph{Schur stable} and \emph{Schur} interchangeably.} A polynomial is \emph{Hurwitz stable} if all its roots have strictly negative real parts.\looseness=-1

Consider the following discrete-time dynamics with state $s_k\in\mathbb C^n$, input $u_k\in\mathbb C^m$, and output $y_k\in\mathbb C^p$, described by
\begin{equation*}
    s_{k+1}=As_k+Bu_k,\qquad y_k=Cs_k+Du_k.
\end{equation*}
where $A, B, C, D$ are fixed matrices. For a two-sided sequence $(x_k)_{k\in\mathbb Z}$, its \emph{bilateral $z$-transform} is
$
    X(z)=\sum_{k\in\mathbb Z}x_kz^{-k},
$
wherever the series converges. The \emph{zero-state response} of the system is obtained by setting $s_0=0$ and applying an input with $u_k=0$ for $k<0$. Thus, for $k\ge1$, the state is determined entirely by the input:
$
    s_k=\sum_{j=0}^{k-1}A^{k-1-j}Bu_j.
$
Extending the resulting state and output sequences by zero to negative indices, their bilateral $z$-transforms satisfy
\begin{equation*}
    Y(z)=G(z)U(z),\qquad
    G(z)=C(zI-A)^{-1}B+D,
\end{equation*}
where the relevant transforms converge and $zI-A$ is invertible. \mgd{
The matrices $(A,B,C,D)$ constitute a
\emph{state-space realization} of $G$: they describe
a finite-dimensional linear recursion with the
specified zero-state input--output behavior.} The rational matrix $G$ is called the \emph{transfer matrix}, or the \emph{transfer function} in the scalar case when $m=p=1$. \mgd{When the meaning is clear from context, we use \emph{transfer} as shorthand for either \emph{transfer function} or \emph{transfer matrix}}. On the unit circle, the \emph{discrete-time Fourier transform} is
$
    X(e^{i\theta})
    =\sum_{k\in\mathbb Z}x_ke^{-ik\theta}$ for
    $\theta\in[0,2\pi)$, 
interpreted pointwise when the series converges and in the $L^2$ sense for square-summable sequences.

For a rational transfer matrix $F(z)$, a \emph{pole} is a finite complex number at which at least one entry of $F$ has a singularity after cancellation of common polynomial factors. The transfer matrix is \emph{stable} if all its poles lie in the open unit disk
$\mathbb D:=\{z\in\mathbb C:|z|<1\}$.
It is \emph{proper} if the entrywise limit
$
    F(\infty):=\lim_{z\to\infty}F(z)
$
exists and is finite. Equivalently, each nonzero entry of $F$, expressed as a ratio of coprime polynomials in $z$, has numerator degree no greater than denominator degree. For a rational transfer matrix $F$, its \emph{para-Hermitian adjoint} is
$F^{\sim}(z):=F(1/\overline z)^*$, \mgd{where $\overline z$ denotes the complex conjugate of $z\in\mathbb{C}$} and the $^*$ denotes conjugate
transpose; thus $F^{\sim}(e^{i\theta})=F(e^{i\theta})^*$. We use standard terminology from discrete-time linear systems theory; for further background on state-space representations, $z$-transforms, Fourier transforms, transfer matrices, and stability, 
see Hinrichsen and Pritchard~\cite[Secs.~2.2--2.3 and Appendix~A.3]{hinrichsen_pritchard_book}.
\section{Preliminaries}
\label{sec:prelim}
\subsection{The \sa{problem} class and the algorithm}

\paragraph{\sa{Smooth (strongly) convex functions.}}
\mgd{We begin by reviewing the function classes
introduced in Section~\ref{sec:intro}.}
For $0<\mu\le L$, let $\kappa=L/\mu$ \mgd{and set $q:=\kappa^{-1/2}$}. The class $\Sml$ consists of $\mu$-strongly
convex, continuously differentiable functions
$f:\R^d\to\R$ that are $L$-smooth,
\sa{i.e., $\nabla f$ is $L$-Lipschitz.}
\mgd{Equivalently, these functions satisfy}
\begin{equation}\label{eq:interp}
    \frac{\mu L}{\mu+L}\,\norm{x-y}^2
    +\frac{1}{\mu+L}
    \norm{\nabla f(x)-\nabla f(y)}^2
    \le
    \langle \nabla f(x)-\nabla f(y),x-y\rangle,
    \qquad x,y\in\R^d;
\end{equation}
see, e.g., \cite[Theorem~2.1.12]{NesterovBook}.
\sa{Since $\mu>0$,} $f$ admits a unique
minimizer $x_*$, i.e., the global minimum
$f_*:=\inf_{x\in\R^d}f(x)$
\sa{is achieved at $x_*$.}
\mgd{We denote the subclass of $\Sml$
consisting of quadratic functions by $\QmuL$.} \mgd{The smooth convex class
$\Fl=\mathcal S_{0,L}$ consists of
continuously differentiable convex functions
with $L$-Lipschitz gradients
\cite{NesterovBook}.}
Clearly,
$\QmuL\subset\Sml\sa{\subset\Fl}$.

\paragraph{Generalized momentum methods.} 
\mgd{We recall the delayed generalized momentum
method (GMM) introduced in
Section~\ref{sec:intro}. For $f\in\Sml$,
its update is given by}
\beq
x_{k+1}=x_{k}-\alpha \left( \nabla f( y_{k-\tau} ) \right) +\beta(x_{k}- x_{k-1}), \quad
y_{k}=x_k+\nu (x_k- x_{k-1}), \label{eq:gmm} 
\eeq
\mgd{starting from a prescribed finite history $x_{-\tau-1},\ldots,x_0\in\mathbb R^d$, where $\tau\geq 1$ is an integer delay. We use exact gradients throughout. The constant-history initialization $x_{-\tau-1}=\cdots=x_0$ is imposed when stated; the quadratic stability analysis below allows arbitrary finite histories.}
GMM admits three parameters $\alpha, \beta$ and $\nu$: The parameter $\alpha>0$ is the stepsize, %
\sa{and} $\beta,\nu\geq 0$ are the \emph{momentum parameters}. In the absence of delays (when $\tau=0$), these methods have been previously studied in the literature with and without gradient errors \cite{gurbuzbalaban2023robustly, VanScoyFreemanLynch2018}. As noted in Section \ref{sec:intro}, GMM recovers
Polyak's heavy-ball (HB) method when $\nu=0$
and Nesterov's accelerated gradient (NAG)
method when $\nu=\beta$ in the undelayed
case $\tau=0$.
In addition, triple momentum method (TMM) is obtained by the specific parameter choice $
				\alpha=\tfrac{1+\bar \rho}{L}$,
				$\beta=\tfrac{\bar\rho^2}{2-\bar\rho}$, 
				$\nu=\tfrac{\bar\rho^2}{(1+\bar\rho)(2-\bar\rho)}$,
				where $\bar\rho = 1 - 1/\sqrt{\kappa}$~\cite{VanScoyFreemanLynch2018}. In the absence of momentum (when $\beta=\nu=0$), GMM recovers \sa{the gradient descent method.}

\paragraph{GMM as a dynamical system.} We next write GMM as a dynamical system. Notice that if we introduce the vector
 $$\xi_k = \begin{bmatrix} x_k \\ x_{k-1} \\ 
  \cdots \\ x_{k-\tau-1}  \end{bmatrix} \in \mathbb{R}^{(\tau+2)d}, $$
which concatenates the iterates from time $k-\tau-1$ to time $k$, then the delayed GMM dynamics in \eqref{eq:gmm} can be rewritten as a discrete-time dynamical system
 \begin{subequations}
				\label{Sys:GMM}
				\begin{eqnarray} 
					\xi_{k+1}&=&A \xi_{k}+B\nabla f(C\xi_k),
				\end{eqnarray}
			\end{subequations}
where the matrices $A,B$ and $C$ are defined as
			\beq A = \tilde{A}\otimes I_d, \quad B = \tilde{B} \otimes I_d, \quad C= \tilde{C} \otimes I_d, 
			\label{def-A-B-C}
			\eeq
with 
$$\tilde{B} = \begin{bmatrix} -\alpha & 0 & \dots  & 0 \end{bmatrix}^\top \in \mathbb{R}^{\sa{(\tau+2)\times 1}}, \quad \tilde{C} = \begin{bmatrix}  0_{1\times \tau} & 1+\nu & -\nu \end{bmatrix}  \in \mathbb{R}^{1\times (\tau+2)},
$$
and 
$$  \tilde{A} = 
			\begin{bmatrix}
1+\beta & -\beta & 0      & \cdots & 0      & 0 \\
1       & 0      & 0      & \cdots & 0      & 0 \\
0       & 1      & 0      & \cdots & 0      & 0 \\
\vdots  & \ddots & \ddots & \ddots & \vdots & \vdots \\
0       & \cdots & 0      & 1      & 0      & 0 \\
0       & \cdots & 0      & 0      & 1      & 0
\end{bmatrix} 
 \in \mathbb{R}^{(\tau+2)\times(\tau+2)}   \quad \mbox{for} \quad \tau > 0,
$$
whereas 
$$
\tilde{A} = \begin{bmatrix}
1+\beta & -\beta   \\
1 & 0 
\end{bmatrix}  \in \mathbb{R}^{2 \times 2}   \quad \mbox{for} \quad \tau=0.
$$  
\paragraph{Linear convergence rate.}
\mgd{For $f\in\Sml$ with minimizer $x_*$, define the
centered initial-history vector and its norm
\beq
v_{\mathrm{init}}
:=
\begin{bmatrix}
x_{-\tau-1}-x_*\\
x_{-\tau}-x_*\\
\vdots\\
x_0-x_*
\end{bmatrix}
\in\mathbb R^{(\tau+2)d}, \quad 
\Psi_0
:=
\|v_{\mathrm{init}}\|
=
\left(
\sum_{j=-\tau-1}^{0}
\|x_j-x_*\|^2
\right)^{1/2}.\label{def-Psi-zero}
\eeq}We say that GMM converges (globally) linearly with rate $\rho \in (0,1)$ if for every initialization \mgd{$x_{-\tau-1}, x_{-\tau}, \dots, x_0 \in \mathbb{R}^d$}, it holds that $\|x_k - x_*\| \leq a_k \rho^{k}\mgd{\Psi_0}$, where scalar coefficients $a_k$ have at most polynomial growth, i.e., there exist real scalars $c_1, c_2\geq 0$ (that do not depend on the initialization $v_{\mathrm init}$) such that $0 \leq a_k \leq c_1 (\mg{1}+k)^{c_2}$ for every $k\geq 0$. \mgd{
By $L$-smoothness and $\nabla f(x_*)=0$,
we have
\[
f(x_k)-f(x_*)
\le \frac{L}{2}\|x_k-x_*\|^2
\le \frac{L}{2}a_k^2\rho^{2k}\Psi_0^2.
\]
Thus, linear convergence of GMM in the
sense defined above also yields a
corresponding bound on the objective-value
error.
}%
\subsection{Lower bounds}
\begin{theorem}[Lower bound under delay; {\cite[Theorem~4]{ArjevaniShamirSrebro2020}}]
\label{thm:asr-lower}
Fix a delay $\tau\ge 0$ and consider any first-order method whose iterates
$x_0,x_1,\dots$ are generated from delayed gradients, in the sense that
\begin{equation}\label{eq:span-delay}
  x_0=\dots=x_\tau=0,
  \qquad
  x_{k+1}\in\operatorname{span}\!\bigl\{\nabla F(x_0),\dots,\nabla F(x_{k-\tau})\bigr\}
  \quad\text{for all }k\ge\tau .
\end{equation}
Then for every \mg{$\kappa>1$}, $k\ge\tau+1$ and every sufficiently large dimension $d$:
\begin{enumerate}
  \item\emph{(Strongly convex case.)} There exists $F\in\Sml$ that is quadratic,
  with minimizer $x_*$, such that
  \begin{equation}\label{eq:asr-sc}
    F(x_k)-F(x_*)\;\ge\;
    \frac{\mu}{4}\,
    \exp\!\left(-\frac{5k}{(\sqrt{\kappa}-1)\,(\tau+1)}\right)
    \norm{x_0-x_*}^2 .
  \end{equation}
  \item\emph{(Convex case.)} There exists an $L$-smooth convex quadratic
  $F\colon\R^d\to\R$, with minimizer $x_*$, such that
  \begin{equation}\label{eq:asr-cvx}
    F(x_k)-F(x_*)\;\ge\;
    \frac{L\,(\tau+1)^2\,\norm{x_0-x_*}^2}{45\,k^2}.
  \end{equation}
\end{enumerate}
\end{theorem}
\mgd{The quadratic nature of the hard instances
in Theorem~\ref{thm:asr-lower} follows from
the explicit tridiagonal construction in
the proof of \cite[Theorem~4]{ArjevaniShamirSrebro2020}
(see equation~(24) in their appendix)}. \sa{One can interpret} \hamzatext{Theorem} \ref{thm:asr-lower} \mgd{as a restriction on convergence bounds that hold uniformly over the function class and dimension: the coefficient in an accelerated contraction factor must decrease as the delay grows.} %
We make this precise below.

\begin{corollary}[The contraction factor must degrade with the delay]
\label{cor:c-degrades}
\mgc{Fix $\kappa=L/\mu>1$ and a delay $\tau\ge1$,
and consider a first-order method satisfying
\eqref{eq:span-delay} in every dimension.
Suppose there exist constants $C_0>0$ and
$c\in(0,\sqrt{\kappa}]$, depending only on
$\mu,L,\tau$ and the chosen method, such that
for every dimension $d$, every
$F:\R^d\to\R$ in $\Sml$, and every $k\ge0$,
\begin{equation}\label{eq:linrate-ansatz}
    F(x_k)-F(x_*)
    \le C_0\rho^{2k}\norm{x_0-x_*}^2,
    \qquad
    \rho=1-\frac{c}{\sqrt{\kappa}}\in[0,1).
\end{equation}
Then, \sa{necessarily it holds that}
\begin{equation}\label{eq:c-bound}
    c\le
    \frac{5\sqrt{\kappa}}
    {2(\sqrt{\kappa}-1)(\tau+1)}.
\end{equation}}
\mgd{In particular, for $\kappa\geq2$, we have $c\leq c_0/(\tau+1)$, where
$c_0=5\sqrt{2}/(2(\sqrt{2}-1))$ is a universal constant.}
\mg{Hence,} $c$ cannot be chosen independently of $\tau$: any constant rate of
the form $\rho=1-c/\sqrt{\kappa}$ with $c$ bounded away from $0$ uniformly in $\tau$
is unattainable.%
\end{corollary}

\begin{proof}
Fix $\kappa>1$ and $\tau\ge1$.
\mgd{For each integer $k\ge\tau+1$,
Theorem~\ref{thm:asr-lower} provides a
strongly convex quadratic hard instance
in sufficiently large dimension.
The explicit construction in \cite[Appendix C]{ArjevaniShamirSrebro2020}
has $x_0=0$ and $x_*\neq0$ for
$\kappa>1$, so that
$\|x_0-x_*\|^2>0$.} Combining the lower bound~\eqref{eq:asr-sc}
with the assumed upper bound
\eqref{eq:linrate-ansatz}, we obtain
\[
\frac{\mu}{4}
\exp\!\left(
-\frac{5k}{(\sqrt{\kappa}-1)(\tau+1)}
\right)\|x_0-x_*\|^2
\le F(x_k)-F(x_*)
\le C_0\rho^{2k}\|x_0-x_*\|^2.
\]
\mgd{Although the hard instance and its
dimension may depend on $k$, the constant
$C_0$ is independent of both, by assumption.
Since $\|x_0-x_*\|^2>0$, we may cancel
this factor, obtaining}
\[
\frac{\mu}{4}
\exp\!\left(
-\frac{5k}{(\sqrt{\kappa}-1)(\tau+1)}
\right)
\le C_0\rho^{2k}.
\]
\mgd{The left-hand side is strictly positive,
so $\rho=0$ is impossible. Hence
$0<\rho<1$.
Using $\rho=1-c/\sqrt{\kappa}$ and
$1-u\le e^{-u}$ for $u\ge0$ and taking logarithms, we obtain}
\[
\left(
\frac{2c}{\sqrt{\kappa}}
-\frac{5}{(\sqrt{\kappa}-1)(\tau+1)}
\right)k
\le \log\frac{4C_0}{\mu}.
\]
\mgd{Since the right-hand side is independent
of $k$, dividing by $k$ and letting
$k\to\infty$ yields}
\[
c\le
\frac{5\sqrt{\kappa}}
{2(\sqrt{\kappa}-1)(\tau+1)}.
\]
For $\kappa\ge2$, this implies \mgd{$c\leq c_0/(\tau+1)$, where
$c_0=5\sqrt{2}/(2(\sqrt{2}-1))$ is a universal constant as claimed.}
\end{proof}

\mgd{Under zero initialization, it is straightforward to show that the lower bounds of
Theorem~\ref{thm:asr-lower} also apply to GMM,
and hence to \DREAM{}, as established below.}

\begin{lemma}[Applicability of delayed first-order
lower bounds to GMM]
\label{lem:dream-span} 
\mgd{Consider delayed GMM iterations \eqref{eq:gmm} with parameters $\alpha>0, \beta\geq 0, \nu\geq 0$, delay $\tau\geq 1$ and constant-history initialization
$
x_{-\tau-1}=\cdots=x_0=0.
$
For any quadratic objective $f$, define
$z_t:=x_{t-\tau}$ for $t\geq -1$.
Then $z_0=\cdots=z_\tau=0$, and
\begin{equation}
z_{t+1}\in
\operatorname{span}
\{\nabla f(z_0),\ldots,
\nabla f(z_{t-\tau})\},
\qquad t\geq\tau.
\label{eq:dream-delayed-span}
\end{equation}
Consequently, the delayed first-order lower
bounds of Theorem \ref{thm:asr-lower} apply to
the main iterates $x_k$ of GMM. Accounting for the initial delay introduces only an additive
shift of $\tau$ in the iteration count and does not
affect the asymptotic lower-bound scaling in $\epsilon$.}
\end{lemma}

\begin{proof}
\mgd{Since $f$ is quadratic, its gradient is affine,
and hence
$
\nabla f(y_j)
=(1+\nu)\nabla f(x_j)
-\nu\nabla f(x_{j-1}).
$
Thus, setting $t=k+\tau$, the GMM recurrence
becomes
\[
\begin{aligned}
z_{t+1}
={}(1+\beta)z_t-\beta z_{t-1}
-\alpha\big[(1+\nu)\nabla f(z_{t-\tau})
-\nu\nabla f(z_{t-\tau-1})\big].
\end{aligned}
\]
Since $z_{-1}=z_0=\cdots=z_\tau=0$,
\eqref{eq:dream-delayed-span} follows by
induction: each new iterate is a linear
combination of previous iterates and gradients
with indices at most $t-\tau$.}

\mgd{Therefore, the reindexed main iterates satisfy
the delayed span condition \eqref{eq:span-delay}.
Since the hard instances in
Theorem~\ref{thm:asr-lower} are quadratic,
its lower-bound argument applies to $f(x_k)-f_*$
with $t=k+\tau$. The additive shift does not
affect the asymptotic lower-bound scaling.}
\end{proof}
\section{The quadratic case: necessity, exact constants, and optimality}\label{sec:quadratic}
\mgd{Throughout this section, gradients are exact and the initial history is arbitrary unless stated otherwise.} Consider a strongly convex quadratic objective $f\in \QmuL$. Then, $f$ admits the representation \beq f(x) = \frac{1}{2} x^\top Q x + \sa{b^\top} x  + \mgd{c_f},
\label{eq-quad-bis}
\eeq
where $Q$ is a symmetric positive-definite $d \times d$ matrix, $b\in \mathbb{R}^d$ and $\mgd{c_f}\in \mathbb{R}$. %
Consider the eigenvalue decomposition $Q = U \Lambda U^\top$ where $U$ is orthogonal satisfying $UU^\top = U^\top U = I_d$ and $\Lambda$ is diagonal where $\lambda_i := \Lambda_{ii}$ are the eigenvalues of $Q$ in non-decreasing order, i.e., $$\mu\mathrel{\leq}\lambda_1 \leq 
\lambda_2 \leq \dots \leq \lambda_d\mathrel{\leq}L.$$
Note that for quadratic $f$, we have $ \nabla f(x) = Q(x-x_*)$ so that the GMM system \sa{in} \eqref{Sys:GMM} becomes 
\beq \xi_{k+1} -\xi_* =A_Q (\xi_k - \xi_*)   \quad \mbox{with} \quad A_Q:= A + B Q C,
\label{gmm-quadratic-recursion}
\eeq
where  we used $A \xi_* = \xi_*$ and $C\xi_* = x_*$ and the vector
				 $$\xi_* := \begin{bmatrix} x_* ^\top & x_{*}^\top & 
  \cdots & x_{*}^\top  \end{bmatrix}^\top \in \mathbb{R}^{(\tau+2)d}, $$
concatenates $\tau+2$ copies of the unique optimum $x_*$. By the recursion \eqref{gmm-quadratic-recursion}, $%
\sa{\xi_{k}} -\xi_* = A_Q^{k} (\xi_0 - \xi_*)$ for any $k\geq 0$; therefore, \mgd{by the standard theory of linear recursions \cite{varga1999matrix}, the worst-case linear convergence rate $\rho$ over initial histories is the spectral radius of $A_Q$.}
The following result characterizes the convergence rate in terms of the roots \mgd{of a family of polynomials, which coincide with the eigenvalues of the iteration matrix of the delayed GMM recursion}.
\begin{theorem}\label{thm-spec-rad-iter-matrix} Consider time-delayed GMM with delay parameter $\tau\geq 0$ when $f$ is a quadratic of the form \eqref{eq-quad-bis}. For any given $\lambda>0$, consider the following degree $(\tau+2)$-polynomial 
  \begin{equation}\label{eq:q-charpoly}
P_{\tau,\delta}(z)=z^{\tau}(z-1)(z-\beta)+\delta\bigl[(1+\nu)z-\nu\bigr],
\qquad \delta:=\alpha \lambda .
\end{equation}
Then, %
\sa{$z\in\sigma(A_Q)$ %
if and only if $P_{\tau,\delta_i}(z)=0$ with $\delta_i = \alpha \lambda_i$ and $\lambda_i\in\sigma(Q)$} for some $1\leq i \leq d$; \sa{therefore,} we have $$ \rho(A_Q) = \max_{1\leq i\leq d} \max \{ |z|:  P_{\tau, \delta_i}(z) = 0, \quad \delta_i = \alpha \lambda_i \}.$$ 
\mgd{In particular, if
$(\alpha,\beta,\nu)\in\mathcal{S}_{\mathrm{GMM}}(Q)
:=\{(\alpha,\beta,\nu):\rho(A_Q)<1\}$,
then GMM converges to the minimizer for every
initial history. If $0<\rho(A_Q)<1$, this
convergence is linear with rate
$\rho=\rho(A_Q)$ in the sense defined above.
If $\rho(A_Q)=0$, then $A_Q$ is nilpotent,
and the iteration reaches the minimizer
in finitely many steps.}%
 \end{theorem} 
\begin{proof} \mgd{
Since the dynamics are linear for quadratic
objectives, convergence to the minimizer
for every initial history is equivalent to
$\rho(A_Q)<1$; see
\cite{varga1999matrix}. When
$0<\rho(A_Q)<1$, standard bounds on powers
of a matrix give linear convergence with
rate $\rho(A_Q)$ and a prefactor growing
at most polynomially in $k$. When
$\rho(A_Q)=0$, the matrix $A_Q$ is nilpotent,
so convergence occurs in finitely many
iterations. We now derive the claimed
formula for $\rho(A_Q)$.} Using the definition of $B$ and $C$, we first note that the matrix $A_Q= A + BQC$ admits a block representation of the form $A_Q = \big[M_{ij} + N_{ij}(Q)\big]_{i,j=1}^{\tau+2}$ where each $M_{ij}$ and\footnote{\mgd{We slightly
overload the notation $M$: the block matrices
$M_{ij}$ used in this context are unrelated to the scalar $M:=2\tau+1$
used in the delay-dependent parameterization.}} $N_{ij}(Q)$ \mgd{are} a $d\times d$ block given by%
{\small
\begin{equation}
\label{eq:MN}
M_{ij} =
\begin{cases}
(1+\beta)I_d, & (i,j)=(1,1), \\[4pt]
\sa{-\beta I_d},%
& (i,j)=(1,2), \\[4pt]
I_d, & i=j+1, \\[4pt]
0_d, & \text{otherwise},
\end{cases}
\quad \mbox{and} \quad
N_{ij}(Q) =
\begin{cases}
-\alpha(1+\nu)Q, & (i,j)=(1,\tau+1), \\[4pt]
\alpha\nu Q, & (i,j)=(1,\tau+2), \\[4pt]
0_d, & \text{otherwise}.
\end{cases}
\end{equation}}
The matrix $N_{ij}(Q)$ is a function of the input matrix $Q$. For example, for $\tau\geq 3$, we have
\[
A_{Q} \;=\;
\begin{bmatrix}
(1+\beta)I_d & -\beta I_d & 0_d & \cdots & 0_d & -\alpha(1+\nu)Q  & \alpha\nu Q  \\[6pt]
I_d & 0_d & 0_d & \cdots & 0_d & 0_d & 0_d \\[2pt]
0_d & I_d & 0_d & \cdots & 0_d & 0_d & 0_d \\[2pt]
0_d & 0_d & I_d & \cdots & 0_d & 0_d & 0_d \\[2pt]
\vdots & \vdots & \vdots & \ddots & \vdots & \vdots & \vdots \\[2pt]
0_d & 0_d & 0_d & \cdots & I_d & 0_d & 0_d \\[2pt]
0_d & 0_d & 0_d & \cdots & 0_d & I_d & 0_d
\end{bmatrix} \in \mathbb{R}^{(\tau+2)d \times (\tau+2)d}.
\]
In the following, we will argue that $A_Q$ is similar to a block diagonal matrix with $(\tau+2)\times (\tau+2)$ diagonal blocks. %
\sa{Since} $Q$ admits the eigenvalue decomposition $Q = U\Lambda U^\top$, using the fact that $UU^\top = I_d$, if we define 
$$ V =\mbox{diag}(U, U, \dots, U)$$
then $V^{-1}=V^\top$ and
\beq A_Q = V A_\Lambda V^\top \quad \mbox{where} \quad
A_\Lambda = \big[ M_{ij} + N_{ij}(\Lambda)\big]_{i,j=1}^{\tau+2}
\label{def-A-lambda}
\eeq
is a block matrix with diagonal blocks\sa{---here, $N_{ij}(\Lambda)$ is defined as in \eqref{eq:MN}, replacing $Q$ with $\Lambda$.}
Let  
$\mgd{P_{\mathrm{perm}}}\in \mathbb{R}^{(\tau+2)d\times (\tau+2)d}$ be the permutation matrix that reorders the elements of $\xi_k = \begin{bmatrix} x_k^\top & x_{k-1}^\top & \dots & x_{k-\tau-1}^\top \end{bmatrix}^\top$ in increasing coordinates, i.e., $\mgd{P_{\mathrm{perm}}}$ is the permutation matrix such that given any $\xi_k$,  %
 \sa{$\mgd{P_{\mathrm{perm}}}\xi_k$ is formed according to}
 {\small
$$ \mgd{P_{\mathrm{perm}}}\xi_k %
= \begin{bmatrix} \xi_k^{[1]} \\ 
 \xi_k^{[2]} \\ 
 \vdots  \\
 \xi_k^{[d]} 
\end{bmatrix}  \quad \mbox{with} \quad \xi_k^{[i]} =   \begin{bmatrix} {(x_k)}_i  \\ {(x_{k-1})}_i \\ 
  \cdots \\ {(x_{k-\tau-1})}_i  \end{bmatrix},\quad\sa{\forall~i=1,\ldots,d,}$$}%
where ${(x_k)}_i$ denotes the $i$-th (element) coordinate of $x_k$. Note that such a matrix $\mgd{P_{\mathrm{perm}}}$ is independent of \sa{the iteration counter $k$}. \sa{Thus,} it is straightforward to check 
{\small
\beq \mgd{P_{\mathrm{perm}}} A_\Lambda \mgd{P^\top_{\mathrm{perm}}} = \mbox{blkdiag} (\tilde{A}^{(\lambda_1)}, \tilde{A}^{(\lambda_2)}, \dots, \tilde{A}^{(\lambda_d)}), 
\label{eq-blk-diag-ALambda-bis}
\eeq
}
with 
 $$ \tilde{A}^{(\lambda)}  = \big[\tilde{M}_{ij} + \tilde{N}_{ij}(\lambda)\big]_{i,j=1}^{\tau+2} \in \mathbb{R}^{(\tau+2) \times (\tau+2)},$$ where each $\tilde{M}_{ij}$ and $\tilde{N}_{ij}(\lambda)$ is a scalar given by
 {\small
\[
\tilde{M}_{ij} =
\begin{cases}
\sa{1+\beta}, & (i,j)=(1,1), \\[4pt]
-\beta, & (i,j)=(1,2), \\[4pt]
1, & i=j+1, \\[4pt]
0, & \text{otherwise},
\end{cases}
\quad \mbox{and} \quad
\tilde{N}_{ij}(\lambda) =
\begin{cases}
-\alpha(1+\nu)\lambda, & (i,j)=(1,\tau+1), \\[4pt]
\alpha\nu \lambda, & (i,j)=(1,\tau+2), \\[4pt]
0, & \text{otherwise}.
\end{cases}
\]}
Hence, from \eqref{eq-blk-diag-ALambda-bis}, \eqref{def-A-lambda} we conclude that \sa{for $S := V P^\top_{\mathrm{perm}}$, we have}
$$ A_Q = S  \ \mbox{blkdiag} (\mg{\tilde{A}}^{(\lambda_1)}, \mg{\tilde{A}}^{(\lambda_2)}, \dots, \mg{\tilde{A}}^{(\lambda_d)}) \
S^{-1},$$
\sa{which follows from $S^{-1}=S^\top$ due to} $\mgd{P^\top_{\mathrm{perm}}} = \mgd{P^{-1}_{\mathrm{perm}}}$ and $V^{-1} = V^\top$. Then, \sa{using a similarity transformation implies that} the characteristic polynomial of $A_Q$ satisfies
$$ p_\tau(z):= \mbox{det}(z I - A_Q) = \mbox{det} (z I - \mbox{blkdiag} (\mg{\tilde{A}}^{(\lambda_1)}, \mg{\tilde{A}}^{(\lambda_2)}, \dots, \mg{\tilde{A}}^{(\lambda_d)}))
=\prod_{i=1}^d  p_{\tau,\lambda_i}(z)$$
where
\beq p_{\tau,\lambda_i}(z) := \det ( zI - \mg{\tilde{A}}^{(\lambda_i)}),\quad\sa{\forall~i=1,\ldots,d.} 
\label{def-p-tau-lambda}
\eeq 
It is straightforward to see that the matrix $\mg{\tilde{A}}^{(\lambda)}$ is similar to a companion matrix and admits the characteristic polynomial
\mg{$$ p_{\tau,\lambda}(z) = z^{\tau+2}+\sum_{j=0}^{\tau+1}a_j(\lambda)z^j,
$$}
with $a_{\tau+2-j}(\lambda) := -(\tilde{M}_{1j} + \tilde{N}_{1j}(\lambda))$ for $1\leq j\leq \tau+2$. Equivalently, %
{\small
\begin{equation}
a_{\tau+1}(\lambda)=\begin{cases} 
-1-\beta  & \mbox{if} \quad \tau\geq 2, \\
-1-\beta  & \mbox{if} \quad \tau=1, \\
 -1-\beta + \alpha (1+\nu)\lambda & \mbox{if} \quad \tau =0,
 \end{cases}\qquad
 a_{\tau}(\lambda)=\begin{cases} 
 \beta & \mbox{if} \quad \tau\geq 2, \\
 \beta + \alpha \lambda (1+\nu)  & \mbox{if} \quad \tau= 1, \\
 \beta - \alpha\nu\lambda & \mbox{if} \quad \tau= 0,
 \end{cases} 
\end{equation} 
}
{\small
\begin{equation}
a_{1}(\lambda)=\begin{cases} 
\alpha(1+\nu) \lambda & \mbox{if} \quad \tau\geq 2, \\ 
\beta+\alpha(1+\nu)\lambda & \mbox{if} \quad \tau= 1, \\
-1-\beta+\alpha(1+\nu) \lambda  & \mbox{if} \quad \tau= 0,
 \end{cases}\qquad
 a_0 (\lambda)=\begin{cases} 
 -\alpha\nu\lambda & \mbox{if} \quad \tau\geq 2, \\
 -\alpha\nu\lambda & \mbox{if} \quad \tau = 1, \\
\beta- \alpha\nu\lambda  & \mbox{if} \quad \tau= 0,
 \end{cases} 
\end{equation}}and $a_n(\lambda)=0$ if $\tau \geq 3$ and $n \in \{2,\dots,\tau-1\}$. \mgd{When $\tau=1$, the expressions for $a_\tau(\lambda)$ and $a_1(\lambda)$ coincide, as they represent the same polynomial coefficient.} We also notice $p_{\tau,\lambda}(z)=  P_{\tau, \delta}(z)$ where the latter is a more compact representation. \sa{Thus, $z$ is an eigenvalue of $A_Q$ if and only if} $p_\tau(z) = \prod_{i=1}^d p_{\tau,\lambda_i}(z)=\prod_{i=1}^d P_{\tau,\delta_i}(z) =0$, \sa{which implies that}
$$\rho(A_Q) = \max_i \rho(\mg{\tilde{A}}^{(\lambda_i)}) = \max_{1\leq i\leq d} \max \{ |z|:  P_{\tau, \delta_i}(z) = 0\}.$$
 This concludes the proof.
\end{proof} 
{Let $Q=U\Lambda U^{\top}$ be the eigenvalue decomposition of $Q$, with orthonormal eigenvectors $u_1,\dots,u_d$ (the columns of $U$) and $Qu_i=\lambda_iu_i$. Let $e_k:=x_k-x_*$ denote the error and let
\[
\varepsilon^{(i)}_k:=u_i^{\top}e_k,\qquad i=1,\dots,d,
\]
denote its $i$-th coordinate in the eigenvector basis, so that $e_k=\sum_{i=1}^d\varepsilon^{(i)}_k u_i$. Two facts make these coordinates convenient. First, since $\nabla f(y)=Q(y-x_*)$ and $u_i^{\top}Q=\lambda_iu_i^{\top}$,
\[
u_i^{\top}\nabla f(y)=\lambda_i\,u_i^{\top}(y-x_*)\qquad\sa{\forall~y\in\R^d,}
\]
\mgd{In particular, the $i$-th coordinate of $\nabla f(x_k)$ in this basis equals $\lambda_i\varepsilon_k^{(i)}$}. Second, \sa{$y_{k}=x_k+\nu (x_k- x_{k-1})$ in \eqref{eq:gmm} implies that} $y_{k-\tau}-x_*=(1+\nu)e_{k-\tau}-\nu e_{k-\tau-1}$; %
taking the inner product of \sa{both sides of the first equality in \eqref{eq:gmm}} with $u_i$ therefore gives, for each $i$ separately,
\begin{equation}\label{eq:scalar-quadratic-recursion}
\varepsilon^{(i)}_{k+1}=(1+\beta)\varepsilon^{(i)}_k-\beta\varepsilon^{(i)}_{k-1}
-\alpha\lambda_i\bigl[(1+\nu)\varepsilon^{(i)}_{k-\tau}-\nu\varepsilon^{(i)}_{k-\tau-1}\bigr].
\end{equation}
Thus, the $d$-dimensional error recursion splits into $d$ independent scalar recursions, one for each eigenvalue of $Q$; this is the block-diagonalization in the proof of Theorem~\ref{thm-spec-rad-iter-matrix}. Substituting $\varepsilon^{(i)}_k=z^{k}$ into \eqref{eq:scalar-quadratic-recursion} shows that its characteristic polynomial is $P_{\tau,\delta_i}$ with $\delta_i=\alpha\lambda_i$.
In \eqref{eq:scalar-quadratic-recursion}, the stepsize $\alpha$ and the curvature $\lambda_i$ appear only through their product $\delta_i=\alpha\lambda_i$, which is the factor by which the recursion scales the delayed, extrapolated error before subtracting it from the current iterate. We call $\delta=\alpha\lambda$ the \emph{gain} associated with the eigenvalue $\lambda$. The name comes from control theory, where the update is viewed as a feedback loop (the iterate determines the point at which the gradient is evaluated, and the gradient is fed back to correct the iterate) and the factor multiplying the \mgd{gradient contribution} is called the \emph{loop gain}. In optimization terms, $\delta_i$ is the effective, dimensionless stepsize along the eigendirection $u_i$: for gradient descent ($\beta=\nu=0$, $\tau=0$), \eqref{eq:scalar-quadratic-recursion} reduces to $\varepsilon^{(i)}_{k+1}=(1-\delta_i)\varepsilon^{(i)}_k$. Since $P_{\tau,\delta}$ depends on $\alpha$ and $\lambda$ only through $\delta$, the stability and convergence properties of GMM on quadratics are functions of the gains $\delta\in[\alpha\mu,\alpha L]$ together with $\beta$ and $\nu$.}

Before turning to general smooth strongly convex objectives, %
we analyze the quadratic case, where \sa{the augmented error dynamics are \emph{linear} and} the worst-case convergence rate is governed exactly by the roots of the scalar characteristic polynomial $P_{\tau,\delta}(z)$. \mgd{This analysis establishes two complementary results, summarized next.}%

\mgd{First, we establish that the classical one-parameter
momentum families cannot retain acceleration under delay.
For every fixed positive delay, Schur-stable,
positive-momentum tunings of Nesterov's method
($\nu=\beta$) and heavy ball ($\nu=0$) cannot achieve
$\sqrt{\kappa}$-acceleration
(Theorems~\ref{thm:q-naglower} and~\ref{thm:q-hblower}).}
Specifically, the optimal spectral gap for heavy ball
is at most $\cO_\tau(\kappa^{-1})$ for $\tau\ge1$,
whereas for Nesterov it is at most
$\cO(\kappa^{-2/3})$ when $\tau=1$ and
$\cO_\tau(\kappa^{-1})$ when $\tau\ge2$.
\mgd{For $\tau=1$, Lemma~\ref{lem:q-nag-tau1-achievability}
provides an explicit Nesterov tuning attaining
a spectral gap of order $\Theta(\kappa^{-2/3})$,
showing that the upper bound is sharp in order,
although it still falls short of the accelerated
$\Theta(\kappa^{-1/2})$ scaling.}
\sa{The additional extrapolation freedom provided by the GMM framework therefore plays a substantive role in the convergence behavior, rather than merely amounting to a cosmetic reparametrization.}

Second, it shows that the delay-matched method has the right asymptotic scale.  With the very parameters of Theorem~\ref{thm:main}, \sa{i.e.,} $\alpha=1/(M^2L)$, $\beta=1-2q/M$, and $\nu=M-\tfrac12$ with $M=2\tau+1$, \mgd{a root-location argument} certifies acceleration on every quadratic for each fixed positive delay and sufficiently large $\kappa$ (Theorem~\ref{thm:q-alldelay}); a separately optimized one-delay tuning sharpens this at $\tau=1$ to the \emph{exact} worst-case spectral radius $1-\tfrac{1}{2}\kappa^{-1/2}$ (Proposition~\ref{prop:q-gmmaccel}). 
\mgd{
In particular, choosing $c=1/(2M)$ in
Theorem~\ref{thm:q-alldelay} yields
$
\rho\leq
1-{1}/{(2M\sqrt{\kappa})}$ for $M=2\tau+1$ and
for every fixed $\tau\ge1$ and all
sufficiently large $\kappa$. This is
consistent with the delay dependence
required by Corollary~\ref{cor:c-degrades}.
}

The main theorem then carries the positive half of this picture---acceleration with these parameters---from quadratics to the entire class $\Sml$, where root-location is unavailable and the analytic certificate in Lemma~\ref{lemm:ratecert} (Section~\ref{sec:master}) is required. Throughout the present section $f$ is a strongly convex quadratic, i.e., $f\in\QmuL$, the quadratics in $\Sml$ whose Hessian spectrum lies in $[\mu,L]$; the scalar characteristic polynomial defined in \eqref{eq:q-charpoly} governs the quadratic dynamics.%

\subsection[NAG and HB cannot achieve acceleration with sqrt-kappa scaling]{NAG and HB cannot achieve acceleration with $\sqrt{\kappa}$ scaling}
\label{sec:quad-setup}

Fix \wubtext{an integer delay $\tau\ge1$ and} a strongly convex quadratic $f\in\QmuL$ with condition number $\kappa=L/\mu\ge1$. For each
eigenvalue $\lambda\in[\mu,L]$ of the Hessian, recall the scalar characteristic polynomial
\begin{equation}
P_{\tau,\delta}(z)=z^{\tau}(z-1)(z-\beta)+\delta\bigl[(1+\nu)z-\nu\bigr],
\qquad \delta:=\alpha \lambda.
\end{equation}
defined earlier in \eqref{eq:q-charpoly}. \emph{Nesterov} is the one-parameter subfamily $\nu=\beta$, for which the characteristic polynomial $P_{\tau,\delta}(z)$ reduces to
\begin{equation}\label{eq:q-nag}
\nagP(z)=z^{\tau}(z-1)(z-\beta)+\delta\bigl[(1+\beta)z-\beta\bigr].
\end{equation}
\paragraph{Terminology.}
At \(\delta=0\), the characteristic polynomial \(P_{\tau,\delta}\) has
\(\tau\) roots at the origin, together with the roots \(z=\beta\) and
\(z=1\). Hence, when \(\delta\) is sufficiently small and \(\beta\) is
close to \(1\), \(\tau\) roots remain close to the origin while two roots
remain close to \(1\). The modes associated with the roots near the origin
decay rapidly and are non-binding for the asymptotic convergence rate in the
regime considered. We therefore refer to the roots of
\(P_{\tau,\delta}\) that remain close to \(1\) as the \emph{slow iteration
modes}, and to those near the origin as the \emph{fast iteration modes}, or
simply the \emph{non-binding roots}.

The Hessian spectral interval has endpoints \(\mu\) and \(L\). Since the
characteristic polynomial depends on a Hessian eigenvalue
\(\lambda\in\sigma(Q)\) only through
\[
\delta=\alpha\lambda,
\]
we write
\begin{equation}
\label{eq:s-Delta-def}
s:=1-\beta\in(0,1),
\qquad
\Delta:=\alpha L,
\qquad
\delta_\mu:=\alpha\mu=\frac{\Delta}{\kappa}.
\end{equation}
The quantity \(\delta=\alpha\lambda\) can be viewed as the \mgd{\emph{effective stepsize}} along the Hessian eigendirection associated with \(\lambda\). Thus,
as \(\lambda\) ranges over \([\mu,L]\),
$
\delta\in[\delta_\mu,\Delta],
$
with the low-curvature endpoint \(\lambda=\mu\) corresponding to the
smallest effective stepsize \(\delta_\mu\), and the high-curvature endpoint
\(\lambda=L\) corresponding to the largest effective stepsize \(\Delta\).
We use the terms \emph{low-curvature endpoint} and
\emph{high-curvature endpoint} for these two Hessian directions, reserving
the terminology \emph{slow} and \emph{fast iteration modes} for the roots
of \(P_{\tau,\delta}\). \mgd{Thus, the curvature endpoints
and the iteration modes describe different aspects
of the dynamics.}

\paragraph{Proof strategy.}
The two spectral endpoints play complementary roles in the lower-bound
argument. The low-curvature endpoint is used to obtain a lower bound on the
overall spectral radius. Indeed, since
\[
\rho
=
\sup_{\lambda\in[\mu,L]}
\rho\!\left(P_{\tau,\alpha\lambda}\right),
\]
we have
$
\rho
\geq
\rho\!\left(P_{\tau,\delta_\mu}\right).
$
Because
$
\delta_\mu={\Delta}/{\kappa}
$
is small when \(\kappa\) is large, the slow iteration modes of
\(P_{\tau,\delta_\mu}\) remain close to \(1\). Locating these roots yields
the lower bounds on the overall spectral radius established in
Lemmas~\ref{lem:q-floorA}, \ref{lem:q-floorB}, and~\ref{lem:q-hbslow}.
Importantly, the argument does not require the worst-case spectral radius
to be attained at \(\lambda=\mu\); the lower endpoint is used only to
produce a valid lower bound.\looseness=-1

The high-curvature endpoint is used instead to derive necessary stability
conditions. Any choice of parameters that is Schur stable for every
\(\lambda\in[\mu,L]\) must, in particular, make
$
P_{\tau,\Delta}
$
Schur stable. Evaluating the stability conditions at \(\lambda=L\)
therefore yields necessary restrictions on the largest effective stepsize
\(\Delta=\alpha L\); these are the bounds obtained in
Propositions~\ref{prop:q-fast} and~\ref{prop:q-hbfast}.

The lower-bound theorems combine these two endpoint arguments. For heavy
ball at every fixed delay, and for Nesterov when \(\tau\geq2\), the
high-curvature stability conditions force \(\Delta\) to be sufficiently
small relative to \(s:=1-\beta\). Consequently,
$
\delta_\mu=\frac{\Delta}{\kappa}
$
is too small to move the slow iteration modes of
\(P_{\tau,\delta_\mu}\) sufficiently far from \(1\), leading to the
\(\cO_\tau(\kappa^{-1})\) spectral-gap bound. For Nesterov with
\(\tau=1\), the high-curvature stability restriction is weaker, of the form
$
\Delta=\cO(\sqrt{s}),
$
and balancing this restriction with the low-curvature root bounds instead
gives the \(\cO(\kappa^{-2/3})\) spectral-gap barrier in
Theorem~\ref{thm:q-naglower}.

All statements below about heavy ball and Nesterov---including the quantifier ``every tuning'' in the theorems and their summaries---are for the standard positive-momentum regime $\beta\in(0,1)$, equivalently $s\in(0,1)$.
\mgd{Within a given momentum subfamily}---Nesterov ($\nu=\beta=1-s$) or heavy ball ($\nu=0$, treated in \S\ref{sec:quad-hb})---a tuning is fixed by the pair $(\alpha,s)$. \sa{For either algorithm family, $\mathrm{fam}\in\{\mathrm{NAG},\mathrm{HB}\}$, on problems with condition number $\kappa$ and gradient delay $\tau$, its \emph{worst-case spectral
radius} and the family's \emph{best certifiable rate}} in view of Theorem \ref{thm-spec-rad-iter-matrix} are
\beq
\rho(\alpha,\beta)&:=&\sup_{\lambda\in[\mu,L]}\ \max\bigl\{|z|:P_{\tau,\alpha \lambda}(z)=0\bigr\}, \\
\rhost_{\mathrm{fam}}&:=&\rhost_{\mathrm{fam}}(\tau,\kappa):=\inf\{\rho(\alpha,\beta): \alpha>0,~\beta\in (0,1), \mgd{\rho(\alpha,\beta)<1} \}.
\eeq
A tuning is \emph{Schur stable} iff $\rho=\sa{\rho(\alpha,\beta)}<1$. The family accelerates at delay $\tau$ iff
$1-\rhost_{\mathrm{fam}}=\Theta(\kappa^{-1/2})$. In the theorem below, we will show $\wubtext{1}\mathbin{\wubtext{-}}\rhost_{\mathrm{NAG}}$ never %
\sa{achieves this dependence on $\kappa$}. Before we present this result, we provide some lemmas about stability in the next section.

Note that for HB, we have $\nu=0$; so that the  polynomial \sa{$P_{\tau,\delta}$} given in \eqref{eq:q-charpoly} simplifies to
\beq \hbP(z) =  z \hbK(z), \quad
\mbox{where} \quad \hbK(z):=z^{\tau-1}(z-1)(z-\beta)+\sa{\delta}.
\label{eq:q-hb}
\eeq
\subsubsection{\mgd{Stability criteria  for polynomials}}\label{sec:quad-lemmas}
The analysis rests on locating the roots of \eqref{eq:q-nag} relative to the unit circle. The classical tool for this purpose is the \emph{Schur--Cohn} criterion \cite{Schur1917,Cohn1922,Marden,Jury}. The following lemma recalls a standard result: the Schur stability of a polynomial \(P\) is equivalent to the Hurwitz stability of its \emph{Cayley transform} \(\widetilde P\). %

\begin{lemma}[Cayley transformation]\label{lem:q-cayley}
Let $P$ be a real polynomial of degree $n$ with $P(-1)\neq0$, and define its Cayley transform as the polynomial
\[
\Pt(w):=(1-w)^{n}\,P\!\left(\frac{1+w}{1-w}\right).
\]
Then $\Pt$ is a real polynomial of degree $n$, and $P$ is \emph{Schur stable}
\textup{(}all roots in $|z|<1$\textup{)} if and only if $\Pt$ is \emph{Hurwitz}
\textup{(}all roots in $\operatorname{Re}w<0$\textup{)}.
\end{lemma}

\begin{proof}
The M\"obius map $w\mapsto z=\frac{1+w}{1-w}$ is a bijection of the Riemann sphere with inverse
$z\mapsto w=\frac{z-1}{z+1}$; it sends the open left half-plane $\{\sa{w:}~\operatorname{Re}w<0\}$ onto
the open unit disk $\{\sa{z:}~|z|<1\}$, \sa{and also} the \wubtext{extended} imaginary axis onto the unit circle, and $w=1\mapsto
z=\infty$, $w=-1\mapsto z=0$. Clearing denominators, $\Pt(w)=(1-w)^nP\!\big(\frac{1+w}{1-w}\big)$
is a polynomial of degree %
\sa{at most $n$}; its leading coefficient is
$\Pt$'s coefficient of $w^n$, which equals $(-1)^nP(-1)\neq0$ %
--equivalently, as $w\to\infty$ the transformed argument tends to $z=-1$, so this leading coefficient can vanish only if $P(-1)=0$; %
hence $\deg\Pt=n$.
Each finite root $z_0\neq-1$ of $P$ corresponds to the root $w_0=\frac{z_0-1}{z_0+1}$ of
$\Pt$, with $|z_0|<1\iff\operatorname{Re}w_0<0$. As the map is a bijection and degrees agree,
$P$ has all roots in the disk iff $\Pt$ has all roots in the left half-plane.
\end{proof}
When $\tau=1$, the characteristic polynomial $P_{\tau,\delta}$ is a cubic polynomial. The following lemma characterizes Schur stability of cubic polynomials and will be used in our analysis.
\begin{lemma}[cubic Schur--Cohn conditions for stability]\label{lem:q-cubic}
Let $P(z)=z^3+a_2z^2+a_1z+a_0$ be real with $P(-1)\neq0$. $P$ is Schur stable if and only if the four \emph{Jury conditions} hold:
\begin{equation}\label{eq:q-jury}
\textup{(i) } P(1)>0,\qquad
\textup{(ii) } P(-1)<0,\qquad
\textup{(iii) } |a_0|<1,\qquad
\textup{(iv) } |a_0a_2-a_1|<1-a_0^2 .
\end{equation}
\end{lemma}

\begin{proof}
The Schur--Cohn stability conditions for cubics are classical. In the notation
$P(z)=z^3+a_2z^2+a_1z+a_0$, they reduce to \cite[Theorem A]{gardini2021necessary}
\begin{equation}\label{eq:q-jury-bis}
\textup{(I) } P(1)>0,\qquad
\textup{(II) } P(-1)<0,\qquad
\textup{(III) } |a_0|<1,\qquad
\textup{(IV) } 
-(a_0a_2-a_1)<1-a_0^2 .
\end{equation}
Clearly $\textup{(I)}$--$\textup{(III)}$ coincide with
$\textup{(i)}$--$\textup{(iii)}$. It remains to show that, under these three
conditions, $\textup{(IV)}$ is equivalent to $\textup{(iv)}$. Set
\[
J:=a_0a_2-a_1,\qquad r:=a_2+a_0 .
\]
Condition $\textup{(IV)}$ gives the lower bound $J>-(1-a_0^2)$. We show that the upper bound $J<1-a_0^2$ follows from
$\textup{(I)}$--$\textup{(III)}$. Indeed,
\[
P(1)=1+a_1+r>0,
\qquad
P(-1)=-1-a_1+r<0;
\]
hence,
$
-1-a_1<r<1+a_1.
$
Thus $1+a_1>0$ and $|r|<1+a_1$. Since $|a_0|<1$, it follows that
$
|a_0r|<1+a_1,
$
and in particular $a_0r<1+a_1$. Now
$
J=a_0a_2-a_1=a_0(r-a_0)-a_1=a_0r-a_0^2-a_1,
$
\sa{which implies that}
\[
J\leq |a_0r|-a_0^2-a_1 <(1+a_1)-a_0^2-a_1=1-a_0^2.
\]
Therefore, under $\textup{(I)}$--$\textup{(III)}$, condition $\textup{(IV)}$ is equivalent to
$
-(1-a_0^2)<J<1-a_0^2.
$
Since $|a_0|<1$, we have $1-a_0^2>0$, and this double inequality is exactly
$
|J|<1-a_0^2.
$
Substituting $J=a_0a_2-a_1$ gives
$
|a_0a_2-a_1|<1-a_0^2,
$
which is condition $\textup{(iv)}$. The proof is complete.
\end{proof}
When $\tau=2$, the characteristic polynomial $P_{\tau,\delta}$ is a quartic polynomial. The following result characterizes Hurwitz stability of quartic polynomials.
\begin{lemma}[stability for quartic polynomials]\label{lem:q-quartic}
A real quartic $A_4w^4+A_3w^3+A_2w^2+A_1w+A_0$ with \mg{$A_i>0$ for $i=0,1,2,3,4$}
is Hurwitz if and only if its third Hurwitz determinant is positive:%
\begin{equation}\label{eq:q-Delta3}
\mg{H}_3:=A_3A_2A_1-A_1^2A_4-A_0A_3^2>0.
\end{equation}
\end{lemma}

\begin{proof}
The Hurwitz determinants of the quartic are $H_1=A_3$, $H_2=A_3A_2-A_4A_1$,
$H_3=A_1H_2-A_3^2A_0=A_3A_2A_1-A_1^2A_4-A_0A_3^2$, and $H_4=A_0H_3$. By the
Li\'enard--Chipart criterion \cite{lieenard1914signe,barnett1971new} once all coefficients are positive, it
suffices to test the odd-indexed Hurwitz determinants $H_1,H_3$; since $A_3=H_1>0$ is given,
Hurwitz stability is equivalent to $H_3>0$.
\end{proof}
\mgc{Before deriving stability restrictions for delayed NAG,
we observe that the condition $\beta<1$ is not an
additional assumption in the positive-momentum regime.
Indeed, as the following lemma shows, Schur stability
itself requires $\beta<1$ for every positive delay.}
\mgc{
\begin{lemma}[Necessary momentum bound for delayed NAG]
\label{lem:q-nag-beta-upper}
Let $\tau\ge1$ be an integer, $\delta>0$, and
$\beta \in (0,\infty)$. Consider the delayed NAG
characteristic polynomial
\[
P_{\tau,\delta}(z)
=z^\tau(z-1)(z-\beta)
+\delta\bigl[(1+\beta)z-\beta\bigr].
\]
If $P_{\tau,\delta}$ is Schur stable, then
$\beta<1$.
\end{lemma}}

\begin{proof}
\mgc{Suppose, for contradiction, that $\beta\ge1$.
Since $\tau\ge1$, the polynomial $P_{\tau,\delta}$
is monic and has constant coefficient
$-\delta\beta$. Schur stability therefore implies
\[
0<\delta\beta<1;
\qquad\text{hence,}\qquad
0<\delta<\frac1\beta\le1.
\]
We consider two cases.}

\mgc{\emph{Case 1: $\tau=1$.}
The characteristic polynomial is
\[
P_{1,\delta}(z)
=z^3-(1+\beta)z^2
+\bigl[\beta+\delta(1+\beta)\bigr]z
-\delta\beta.
\]
Writing $P_{1,\delta}(z)=z^3+a_2z^2+a_1z+a_0$,
we have
\[
a_2=-(1+\beta),\qquad
a_1=\beta+\delta(1+\beta),\qquad
a_0=-\delta\beta.
\]
By Lemma~\ref{lem:q-cubic}, Schur stability
requires $
|a_0a_2-a_1|<1-a_0^2$. In particular, since $a_0a_2-a_1<0$, this means 
\[
a_0a_2-a_1>-(1-a_0^2).
\]
Substituting the coefficients gives
\begin{equation}\label{eq:nag-beta-necessary-cubic}
1-\beta+\delta(\beta^2-1)-\delta^2\beta^2>0.
\end{equation}
However, setting $t=\delta\beta\in(0,1)$,
we obtain the identity
\[
1-\beta+\delta(\beta^2-1)-\delta^2\beta^2
=
-\frac{
t^2+(\beta-1)(\beta-t)(1-t)
}{\beta}<0,
\]
where the strict inequality follows from
$\beta\ge1$ and $0<t<1$.
This contradicts
\eqref{eq:nag-beta-necessary-cubic}.}

\mgc{\emph{Case 2: $\tau\ge2$.}
Consider the Cayley transform
\[
\widetilde P(w)
=(1-w)^{\tau+2}
P_{\tau,\delta}
\left(\frac{1+w}{1-w}\right).
\]
By Lemma~\ref{lem:q-cayley}, Schur stability
of $P_{\tau,\delta}$ implies that
$\widetilde P$ is Hurwitz stable.
Moreover,
$
\widetilde P(0)=P_{\tau,\delta}(1)=\delta>0.
$
Since a real Hurwitz polynomial with
positive constant coefficient has strictly
positive coefficients, necessarily
$\widetilde P'(0)>0$.} \mgc{On the other hand, direct differentiation gives
\begin{align*}
\widetilde P'(0)
&=-(\tau+2)P_{\tau,\delta}(1)
  +2P_{\tau,\delta}'(1)
=2(1-\beta)+\delta(2\beta-\tau).
\end{align*}
Using $\tau\ge2$, $\beta\ge1$, and
$0<\delta<1$, we obtain
$
\widetilde P'(0)
\le 2(1-\beta)+2\delta(\beta-1)
=2(\beta-1)(\delta-1)
\le0,
$
contradicting $\widetilde P'(0)>0$.}
\mgc{Thus $\beta\ge1$ is impossible in either
case, and consequently $\beta<1$.}
\end{proof}

\subsubsection{%
\mgd{High-curvature stability}
conditions \sa{for NAG}}
\label{sec:quad-fast}

At the %
\mgd{high-curvature} endpoint
$\lambda=L$, we have \sa{$\delta=\alpha L=\Delta$}. %
Schur stability of $P_{\tau,\Delta}$ caps
the admissible step. The following proposition provides a bound on $\Delta$ for Schur stability.

\begin{proposition}[\mgd{high-curvature stability bounds}]\label{prop:q-fast}
Let $0<s<\tfrac12$ \textup{(}equivalently $\beta>\tfrac12$\textup{)} and suppose the Nesterov
polynomial $\nagPfast$ of \eqref{eq:q-nag} is Schur stable. Then:
\begin{enumerate}[label=\textup{(\alph*)},nosep,leftmargin=2.2em]
\item \textup{(}$\tau=1$\textup{)}\quad $\Delta^2\beta^2<s$; hence, $\Delta<\sqrt{s}/\beta\le2\sqrt s$;
\item \textup{(}$\tau=2$\textup{)}\quad $\Delta<s$;
\item \textup{(}$\tau\ge3$\textup{)}\quad
$\displaystyle\Delta<\frac{2s}{\tau-2+2s}\le\frac{2s}{\tau-2}$.
\end{enumerate}
\end{proposition}

\begin{proof} For simplicity, throughout the proof we write $P_{\tau,\Delta}$ for the degree-$(\tau+2)$ Nesterov polynomial $\nagPfast$, suppressing the superscript. Thus, $\tau=1$ gives a cubic and $\tau=2$ a quartic.

\emph{(a) The cubic.} For $\tau=1$,
$P_{1,\Delta}(z)=z^3-(1+\beta)z^2+\bigl[\beta+\Delta(1+\beta)\bigr]z-\Delta\beta$, \sa{i.e.,}
$a_2=-(1+\beta)$, $a_1=\beta+\Delta(1+\beta)$, $a_0=-\Delta\beta$. Since Schur stability is assumed, all four Jury conditions hold. The first two are $P(1)=\Delta>0$ and $P(-1)=-2-2\beta-\Delta(1+2\beta)<0$, \sa{the third one implies $\Delta\beta<1$}, while the step restriction we need comes from condition~(iv).
Compute
\[
a_0a_2-a_1=\Delta\beta(1+\beta)-\beta-\Delta(1+\beta)=-\beta-\Delta(1+\beta)s,
\]
so $|a_0a_2-a_1|=\beta+\Delta(1+\beta)s$, and (iv) reads
\[
\beta+\Delta(1+\beta)s<1-\Delta^2\beta^2
\iff s-\Delta(1+\beta)s-\Delta^2\beta^2>0 .
\]
\mgd{Since $\Delta^2\beta^2>0$ and $s>0$, the preceding
inequality implies
$0<1-\Delta(1+\beta)<1$.
Consequently, $\Delta^2\beta^2<s$, yielding
$\Delta<\sqrt{s}/\beta\le2\sqrt{s}$,
since $\beta>\tfrac12$.}

\emph{(b) The quartic.} For $\tau=2$,
$P_{2,\Delta}(z)=z^4+(s-2)z^3+(1-s)z^2+\Delta(2-s)z-\Delta(1-s)$. By
Lemma~\ref{lem:q-cayley}, $P_{2,\Delta}$ is Schur iff its Cayley transform
$\Pt(w)=(1-w)^4P_{2,\Delta}\!\big(\tfrac{1+w}{1-w}\big)=A_4w^4+A_3w^3+A_2w^2+A_1w+A_0$ is
Hurwitz, with
\begin{equation}\label{eq:q-quartcoeff}
\;
\begin{aligned}
A_4&=4-2s-3\Delta+2\Delta s, &\qquad A_3&=8-2s+8\Delta-6\Delta s,\\
A_2&=4+2s-6\Delta+6\Delta s, &\qquad A_1&=2s(1-\Delta), &\qquad A_0&=\Delta .
\end{aligned}\;
\end{equation}
Since \(P_{2,\Delta}\) is Schur stable, Lemma~\ref{lem:q-cayley}
implies that its Cayley transform \(\widetilde P\) is Hurwitz. A real Hurwitz polynomial necessarily has coefficients of the same sign.\footnote{\mg{This follows directly from the root factorization of \(\widetilde{P}(w)\). Each real root \(-a\), with \(a>0\), contributes a factor \(w+a\), while each complex-conjugate pair \(-a\pm ib\) contributes
$
(w+a+ib)(w+a-ib)
=
w^2+2aw+a^2+b^2.
$
Each such factor has strictly positive coefficients, and therefore so does their product, which agrees with \(\widetilde{P}(w)\) \mgd{up to its leading coefficient.} %
\looseness=-1
}} Since \(A_0=\Delta>0\), all
coefficients of \(\widetilde P\) are positive. \mgd{Therefore, Lemma~\ref{lem:q-quartic}
implies that \mgd{Schur stability requires}
$\mg{H}_3>0$}. %
\sa{Expanding \mg{$H_3$} defined in \eqref{eq:q-Delta3}, \sa{and substituting} \eqref{eq:q-quartcoeff}, we obtain}
\begin{equation}\label{eq:q-nagquartic}
\mg{H}_3=64\Bigl[\,s-
\sa{g(s)}\Delta
-(2-s)\Delta^2-(1-s)^3\Delta^3\Bigr],
\end{equation}
\sa{where $g(s):=1+s(1-s)^2$}. 
\mg{Using %
$g(s)\geq 1$ on $[0,1]$, 
$\Delta>0$ and $s\in(0,1)$;} %
\sa{$H_3>0$ implies that}
\[
s>g(s)\,\Delta+(2-s)\Delta^2+(1-s)^3\Delta^3\ \ge\ g(s)\,\Delta\ \ge\ \Delta, \quad \mbox{i.e.,} \quad \Delta<s.
\]
\emph{(c) \mgd{The $\tau\geq 3$ case.}} Write $\Pt(w)=(1-w)^{\tau+2}P_{\tau,\Delta}\!\big(\tfrac{1+w}{1-w}\big)$. Then
$\Pt(0)=P_{\tau,\Delta}(1)=\Delta>0$, and
\begin{align*}
\Pt{}'(0):=\frac{d\Pt(w)}{dw}\Big |_{w=0} &= -(\tau+2)\,P_{\tau,\Delta}(1)+2\,P_{\tau,\Delta}'(1)
=-(\tau+2)\Delta+2\bigl[s+\Delta(1+\beta)\bigr]\\
&=2s+\Delta(2\beta-\tau)=2s-\Delta(\tau-2+2s),
\end{align*}
\sa{where we used} $P_{\tau,\Delta}'(1)=s+\Delta(1+\beta)$ and $\beta=1-s$. \sa{Since} a Hurwitz polynomial has all coefficients of
the \sa{same} sign %
\sa{with} its constant term, %
and $\Pt(0)=\Delta>0$, 
the coefficient of $w^1$, which equals $\Pt{}'(0)$ is also positive, %
\sa{implying that}
\begin{equation}
\Delta<\frac{2s}{\tau-2+2s}\le\frac{2s}{\tau-2}\quad \mbox{for}\footnote{Note that for $\tau=2$, the same analysis using this inequality gives the crude bound $\Delta<1$; therefore, for $\tau=2$, a more precise analysis based on the stability of quartic polynomials was given in part (b) of the proof.} \quad \tau\ge3,
\label{delta-bound-general-tau}
\end{equation}
which completes the proof.
\end{proof} 

\begin{remark}%
\label{rem:q-summary}
The three regimes use three different facets of the same stability machinery. 
For $\tau \ge 3$, the argument relies only on the signs of the coefficients 
of the Cayley transform $\Pt$. This analysis, however, is too loose for 
$\tau=2$ and does not apply to $\tau=1$. Accordingly, the case $\tau=1$ is 
handled using the cubic Schur--Cohn condition~(iv) in Lemma~\ref{lem:q-cubic}, 
whereas the case $\tau=2$ is handled using the quartic Li\'enard--Chipart 
determinant~\eqref{eq:q-nagquartic}.
\end{remark}

\subsubsection{\mgd{Low-curvature spectral-radius bounds for NAG and HB}
}%
\label{sec:quad-slow}
At the \mgd{low-curvature endpoint} $\lambda=\mu$, we have $\delta=\delta_\mu=\Delta/\kappa$ small. We next provide two lower bounds on
$\rho$ \sa{for NAG} based on the location of the slow modes. 
\begin{lemma}[Lower bound on $\rho$ for NAG]\label{lem:q-floorA}
Fix $\tau\ge1$, \mgd{let $s=1-\beta\in(0,1)$} and suppose $0<\delta_\mu\le s^2/8$ and $2\tau\delta_\mu/s\le\tfrac13$. Then
$\nagPslow$ %
has a real root in $\bigl(1-\tfrac{2\delta_\mu}{s},\,1\bigr)$; hence,
$
\rho\ > 1-\frac{2\delta_\mu}{s}.
$
\end{lemma}

\begin{proof}
Put $z_0=1-\tfrac{2\delta_\mu}{s}$. From $\delta_\mu\le s^2/8$, we get
$\tfrac{2\delta_\mu}{s}\le\tfrac s4$; \sa{therefore,} $z_0\in[1-\tfrac s4,1)\subset(\beta,1)$,
$z_0-\beta=s-\tfrac{2\delta_\mu}{s}\ge\tfrac{3s}{4}>0$, and $z_0>0$. Now
$\nagPslow(1)=\delta_\mu>0$; \sa{moreover,} using $z_0-1=-\tfrac{2\delta_\mu}{s}$ and
$(1+\beta)z_0-\beta<(1+\beta)-\beta=1$, \sa{we get}
\[
\nagPslow(z_0)
=z_0^\tau(z_0-1)(z_0-\beta)+\delta_\mu\bigl[(1+\beta)z_0-\beta\bigr]
<-z_0^\tau\,\frac{2\delta_\mu}{s}\cdot\frac{3s}{4}+\delta_\mu
=\delta_\mu\Bigl(1-\tfrac32 z_0^\tau\Bigr).
\]
\sa{Since $\tfrac{2\delta_\mu}{s}\leq\tfrac{s}{4}\leq 1$,} by Bernoulli inequality, $z_0^\tau=(1-\tfrac{2\delta_\mu}{s})^\tau\ge1-2\tau\tfrac{\delta_\mu}{s}\ge1-\tfrac13=\tfrac23$, %
where in the last inequality we used our assumption on $\delta_\mu$. Therefore, $1-\tfrac32 z_0^\tau\le0$, i.e., $\nagPslow(z_0)<0<\nagPslow(1)$. The
intermediate value theorem yields a real root in $(z_0,1)$; hence, Theorem~\ref{thm-spec-rad-iter-matrix} implies that $\rho> z_0$.
\end{proof}

\begin{lemma}[Lower bound on the spectral radius for $\tau=1$]\label{lem:q-floorB}
There exists $\delta_0>0$ such that the following holds for $0<\delta_\mu\le\delta_0$: For $\tau=1$ and $0<s\le\tfrac12$, the two roots $z_1,z_2$ of $\nagPslowone$ %
other than the one near $0$ are
either a positive real pair or a complex-conjugate pair, with product
\begin{equation}\label{eq:q-slowprod}
z_1z_2=1-s+\mathcal{O}(\delta_\mu^2),
\end{equation}
\mg{where the implied constant is uniform over
$
0<s\le\tfrac12$ and $
0<\delta_\mu\le\delta_0;
$
that is, it is independent of both $s$ and $\delta_\mu$.}
Consequently $\rho\ge\sqrt{|z_1z_2|}=\sqrt{1-s}+\mathcal{O}(\delta_\mu^2)$, and
$ 
1-\rho\ \le\ 1-\sqrt{1-s}+\mathcal{O}(\delta_\mu^2)
\ =\ \frac{s}{1+\sqrt{1-s}}+\mathcal{O}(\delta_\mu^2).
$
\end{lemma}
\begin{proof}
Substituting $z=1+\xi$ into \eqref{eq:q-nag} with $\tau=1$, $\delta=\delta_\mu$ and \sa{$\beta=1-s$} gives
\begin{equation}\label{eq:Qdef}
\nagPslowone(1+\xi)
=
Q(\xi,\delta_\mu,s)
:=
\xi^3+(1+s)\xi^2+\bigl[s+\delta_\mu(1+\beta)\bigr]\xi+\delta_\mu,
\end{equation}
which we regard as a polynomial---hence a $C^\infty$ function---jointly in the real
variables $(\xi,\delta_\mu,s)$, \sa{and we also note that} the momentum coefficient $\beta=\beta(s)$ is itself polynomial in $s$, so after substituting $\beta=1-s$, $Q$ remains a polynomial in $(\xi,\delta_\mu,s)$. Denote by $\xi_1,\xi_2,\xi_3$ the three
roots of $Q(\cdot,\delta_\mu,s)$; by Vieta's formulas,
\begin{equation}\label{eq:vieta}
\xi_1+\xi_2+\xi_3=-(1+s),
\qquad
\xi_1\xi_2\xi_3=-\delta_\mu .
\end{equation}
At the unperturbed value $\delta_\mu=0$,
\begin{equation}\label{eq:Qfactor}
Q(\xi,0,s)=\xi^3+(1+s)\xi^2+s\xi=\xi(\xi+1)(\xi+s),
\end{equation}
so the root corresponding to $z\mathrel{=}0$ is \sa{$\xi_3^0=-1$}, while the remaining two sit at
\sa{$\xi_{1}^0=0$ and $\xi_{2}^0=-s$}.

\medskip
\noindent\emph{Step 1: a uniformly isolated, smooth root near \sa{$\xi_{3}^0=-1$}.}
We first show that, for all small \sa{$\delta_\mu>0$} and all $s\in[0,\tfrac12]$, the polynomial
$Q(\cdot,\delta_\mu,s)$ has exactly one root in the disk 
\mgd{$\mathcal D_{-1}
:=\{\xi\in\mathbb C:|\xi+1|<\tfrac14\}$}
\sa{centered at the unperturbed root $\xi_3^0=-1$}, and that this root is real, simple,
and a $C^\infty$ function of $(\delta_\mu,s)$.\looseness=-1

Fix $s\in[0,\tfrac12]$ and view $\xi\mapsto Q(\xi,\delta_\mu,s)$ as a polynomial on
$\mathbb{C}$. On the circle $\mgd{\partial\mathcal{D}_{-1}}=\sa{\{\xi\in\mathbb{C}:|\xi-\xi_3^0|=\tfrac14\}}=\{\sa{\xi\in\mathbb{C}:}|\xi+1|=\tfrac14\}$, we compare $Q$ with
its unperturbed value $Q(\cdot,0,s)$ from \eqref{eq:Qfactor}. Their difference is
\[
Q(\xi,\delta_\mu,s)-Q(\xi,0,s)=\delta_\mu\bigl[(1+\beta)\xi+1\bigr],
\]
and %
$s\in[0,\tfrac12]$ %
\sa{implies that $\beta=1-s\in[\tfrac12,1]$, i.e., $1+\beta\le2$; thus,}
every \mgd{$\xi\in\partial\mathcal{D}_{-1}$} (for which $|\xi|\le\tfrac54$) satisfies
\begin{equation}\label{eq:rouche-pert}
\bigl|Q(\xi,\delta_\mu,s)-Q(\xi,0,s)\bigr|
\le \mg{|\delta_\mu|}\bigl(2\cdot\tfrac54+1\bigr)=\tfrac72\,\mg{|\delta_\mu|\leq \tfrac72\ \delta_0.}
\end{equation}
On the same circle, the distances from $\xi$ to the other two unperturbed roots, \sa{$\xi_{1}^0$ and $\xi_{2}^0$,} are
bounded below uniformly in $s$: $\sa{|\xi-\xi_{1}^0|=}|\xi|\ge 1-\tfrac14=\tfrac34$ and
$\sa{|\xi-\xi_{2}^0|=}|\xi+s|\ge(1-s)-\tfrac14\ge\tfrac14$; \sa{hence, \eqref{eq:Qfactor} implies that}
\begin{equation}\label{eq:rouche-main}
\bigl|Q(\xi,0,s)\bigr|=|\xi|\,|\xi+1|\,|\xi+s|\ \ge\ \tfrac34\cdot\tfrac14\cdot\tfrac14=\tfrac{3}{64},
\qquad \forall~\xi\in\partial\mathcal{D}_{-1}.
\end{equation}
\sa{Fix $\delta_0\in(0,\tfrac{3}{224})$, then we have} $\tfrac72\delta_\mu<\tfrac{3}{64}$ for
$0\le\delta_\mu\le\delta_0$, so \eqref{eq:rouche-pert}--\eqref{eq:rouche-main} give the
strict inequality $|Q(\xi,\delta_\mu,s)-Q(\xi,0,s)|<|Q(\xi,0,s)|$ on $\partial\mathcal{D}_{-1}$.
By Rouch\'e's theorem, $Q(\cdot,\delta_\mu,s)$ and $Q(\cdot,0,s)$ have the same number of zeros in $\mathcal{D}_{-1}$, counted with multiplicity; \sa{as $\xi_{3}^0=-1$ is the only zero of $Q(\cdot,0,s)$ in $\mathcal{D}_{-1}$, which is also simple, $Q(\cdot,\delta_\mu,s)$ has exactly one zero in
$\mathcal{D}_{-1}$ as well, and it is simple. Let $\xi_3(\delta_\mu,s)$ denote that zero as a function of $\delta_\mu$ and $s$.} Because $Q$ has real coefficients and $\mathcal{D}_{-1}$ is
symmetric about the real axis, the conjugate $\overline{\xi_3}$ is also a zero in
$\mathcal{D}_{-1}$; %
\sa{{uniqueness of the zero in $\mathcal{D}_{-1}$} implies that} $\xi_3=\overline{\xi_3}$, i.e., $\xi_3$ is \emph{real}.
 All the estimates above are strict and uniform in $s\in[0,\tfrac12]$, and depend
 continuously on $(\delta_\mu,s)$; hence, they persist on an open neighbourhood
 \sa{$\Omega\mg{\supset}\{(\delta_\mu,s)\in\mathbb{R}^2:~(\delta_\mu,s)\in[0,\delta_0]\times[0,\tfrac12]\}$}, 
\mgd{on which the isolated root remains simple. Thus $\partial_\xi Q(\xi_3(\delta_\mu,s),\delta_\mu,s)\neq0$} 
throughout $\Omega$, so the implicit
function theorem \cite[Section~3]{krantz2002implicit} applies at every point of $\Omega$:
near each point the unique root in $\mathcal{D}_{-1}$ is given by a $C^\infty$ function, and by
uniqueness these local representations all coincide with $\xi_3$. Therefore
\[
\xi_3:\Omega\to\mathcal{D}_{-1}\cap\mathbb{R}
\qquad\text{is of class } C^\infty,
\qquad \xi_3(0,s)=-1 .
\]
In particular $\xi_3$ stays within $\tfrac14$ of $-1$, so
\begin{equation}\label{eq:xi3-away}
\xi_3(\delta_\mu,s)\in[-\tfrac54,-\tfrac34],
\qquad\text{whence } |\xi_3|\ge\tfrac34>0 \text{ on } \Omega.
\end{equation}
\sa{In the rest, we adopt the notation $Q_\xi(\tilde\xi,\delta,s):=\partial_\xi Q(\xi,\delta,s)\mid_{\xi=\tilde\xi}$ and $Q_\delta(\xi,\tilde\delta,s):=\partial_\delta Q(\xi,\delta,s)\mid_{\delta=\tilde\delta}$.}
\medskip
\noindent\emph{Step 2: first-order \mgd{expansion of the isolated root}.}
\sa{Recall that $\xi_3(0,s)=\xi_3^0=-1$ for $s\in[0,\tfrac{1}{2}]$; hence, $$Q_\xi(-1,0,s)=3-2(1+s)+s=1-s \quad \mbox{and} \quad Q_{\delta_\mu}(-1,0,s)=-(1+\beta)+1=-\beta.$$ Thus, the implicit function theorem implies that}
\begin{equation}\label{eq:calibration}
\partial_{\delta_\mu}\xi_3(0,s)
=-\frac{Q_{\delta_\mu}(-1,0,s)}{Q_\xi(-1,0,s)}
=\frac{\beta}{1-s}
=1 ,
\end{equation}
\sa{where} the last equality is due to \sa{the definition of $s:=1-\beta$.}
Since
$\partial_{\delta_\mu}^2\xi_3\sa{(\cdot,\cdot)}$ is continuous, hence bounded, on the compact set
$[0,\delta_0]\times[0,\tfrac12]$, Taylor's theorem in $\delta_\mu$ yields
\begin{equation}\label{eq:xi3-expand}
\xi_3(\delta_\mu,s)=-1+\delta_\mu+\cO(\delta_\mu^2)
\qquad\text{uniformly in } s\in[0,\tfrac12].
\end{equation}

\medskip
\noindent\emph{Step 3: the product $z_1z_2$ and its $C^2$ dependence on $s$.}
Let $z_3=1+\xi_3$ be the root of $P_{1,\delta_\mu}$ near $0$, and $z_1=1+\xi_1$,
$z_2=1+\xi_2$ \sa{be} the other two \sa{roots of $\nagP(z)$ defined in \eqref{eq:q-nag}}. Eliminating $\xi_3$ from \eqref{eq:vieta},
\[
\xi_1+\xi_2=-(1+s)-\xi_3,
\qquad
\xi_1\xi_2=-\frac{\delta_\mu}{\xi_3},
\]
where \sa{the quotient is well-defined as %
the denominator $\xi_3(\delta_\mu,s)\neq 0$ on $\Omega$} by \eqref{eq:xi3-away}. Hence
\begin{equation}\label{eq:gdef}
g(s,\delta_\mu):=z_1z_2=(1+\xi_1)(1+\xi_2)=1+\xi_1+\xi_2+\xi_1\xi_2
=-s-\xi_3(\delta_\mu,s)-\frac{\delta_\mu}{\xi_3(\delta_\mu,s)}.
\end{equation}
\sa{Since $\xi_3$ belongs to $C^\infty$ and is bounded away from $0$ on $\Omega$, we can conclude that
$g$ also belongs to the class $C^\infty$ jointly in $(s,\delta_\mu)$ on $\Omega$.}
\sa{Define} $R(s,\delta_\mu):=g(s,\delta_\mu)-(1-s)$ \sa{for $(\delta_\mu,s)\in\Omega$. Since} $\xi_3(0,s)=-1$, \eqref{eq:gdef} implies that
\[
g(s,0)=-s-(-1)-0=1-s,
\quad\text{\sa{i.e.,}}\quad
R(s,0)\equiv0 \quad \sa{\forall}~s\in[0,\tfrac12].
\]
Differentiating \eqref{eq:gdef} in $\delta_\mu$ and using
$\partial_{\delta_\mu}\xi_3(0,s)=1$ from \eqref{eq:calibration}, \sa{we get}
\[
\partial_{\delta_\mu}g(s,0)
=-\partial_{\delta_\mu}\xi_3(0,s)-\frac{1}{\xi_3(0,s)}
=-1-\frac{1}{-1}=0,
\quad\text{\sa{i.e.,}}\quad
\partial_{\delta_\mu}R(s,0)\equiv0 .
\]
Thus $R(\cdot,\delta_\mu)$ vanishes to second order in $\delta_\mu$, uniformly in $s$.
Taylor's theorem with integral remainder gives, for all
$(s,\delta_\mu)\in[0,\tfrac12]\times[0,\delta_0]$, \sa{that}
\begin{equation}\label{eq:R-int}
R(s,\delta_\mu)
=\biggl(\underbrace{\int_0^1(1-t)\,\partial_{\delta_\mu}^2R(s,t\delta_\mu)\,dt}_{=:\,r(s,\delta_\mu)}\biggr)\delta_\mu^2 .
\end{equation}
Since $\partial_{\delta_\mu}^2R=\partial_{\delta_\mu}^2g$ is $C^\infty$, its
$s$-derivatives of orders $0,1,2$ are continuous, hence bounded, on the compact set
$[0,\tfrac12]\times[0,\delta_0]$; differentiating \eqref{eq:R-int} twice under the
integral sign is therefore justified, and shows that $s\mapsto r(s,\delta_\mu)$ is of
class $C^2$ (indeed $C^\infty$) with
\begin{equation}\label{eq:r-bounded}
\sup_{\substack{0\le s\le 1/2\\ 0\le\delta_\mu\le\delta_0}}
\bigl(|r|+|\partial_s r|+|\partial_s^2 r|\bigr)<\infty .
\end{equation}
Combining \eqref{eq:R-int} with %
\sa{ $R(s,\delta_\mu):=g(s,\delta_\mu)-(1-s)$ and $g(s,\delta_\mu)=z_1z_2$ for $(\delta_\mu,s)\in\Omega$, we get}
\begin{equation}\label{eq:prod-final}
z_1z_2=1-s+r(s,\delta_\mu)\,\delta_\mu^2 ,
\end{equation}
which establishes \eqref{eq:q-slowprod}; \sa{moreover, we also see that the
$\cO(\delta_\mu^2)$ term in \eqref{eq:q-slowprod} has the explicit form $r(s,\delta_\mu)\,\delta_\mu^2$ such that $s\mapsto r(s,\delta_\mu)$ is of
class $C^2$ and \eqref{eq:r-bounded} holds. Therefore, in addition to \eqref{eq:q-slowprod}, we can also conclude that $\frac{\partial}{\partial s}(z_1z_2)
=-1+\cO(\delta_\mu^2)$ and $\frac{\partial^2}{\partial s^2}(z_1z_2)=
\cO(\delta_\mu^2)$ uniformly in $s$.}

\medskip
\noindent\emph{Step 4: the spectral-radius bound.}
After shrinking $\delta_0$ if necessary so that $\delta_0^2\sup_s|r|<\tfrac14$, the
lower bound $1-s\ge\tfrac12$ and \eqref{eq:prod-final} give
\begin{equation}\label{eq:prod-pos}
z_1z_2=1-s+r\,\delta_\mu^2\ \ge\ \tfrac12-\tfrac14=\tfrac14>0 .
\end{equation}
Moreover, by \eqref{eq:xi3-expand},
$z_1+z_2=2+(\xi_1+\xi_2)=2-s-\delta_\mu+\cO(\delta_\mu^2)\ge\tfrac32-\delta_\mu+\cO(\delta_\mu^2)>0$
for $\delta_0$ small. Since $P_{1,\delta_\mu}$ has real coefficients, the pair
$(z_1,z_2)$ is either real or complex-conjugate; in the real case,
\eqref{eq:prod-pos} together with $z_1+z_2>0$ forces both to be positive. This proves the
dichotomy in the statement. In either case,
\[
\rho\ \ge\ \max\{|z_1|,|z_2|\}\ \ge\ \sqrt{|z_1z_2|}\ =\ \sqrt{\,1-s+r\,\delta_\mu^2\,}.
\]
Because $1-s+r\delta_\mu^2\ge\tfrac14$ by \eqref{eq:prod-pos}, the elementary identity
$\sqrt{a+h}-\sqrt a=h/(\sqrt{a+h}+\sqrt a)$ with $a=1-s$ and $h=r\delta_\mu^2$ gives
\begin{equation}\label{eq:rho-expand}
\sqrt{\,1-s+r\,\delta_\mu^2\,}=\sqrt{1-s}+\tilde r(s,\delta_\mu)\,\delta_\mu^2,
\qquad
\tilde r:=\frac{r}{\sqrt{1-s+r\delta_\mu^2}+\sqrt{1-s}} .
\end{equation}
The denominator in \eqref{eq:rho-expand} is bounded below by \mgd{$\tfrac12+1/\sqrt2>0$ where we used \eqref{eq:prod-pos} and $1-s\ge\tfrac12$. Furthermore, the denominator is $C^2$ in $s$} %
while the numerator $r$ obeys
\eqref{eq:r-bounded}; hence $\tilde r$ is $C^2$ in $s$ with $\tilde r,\partial_s\tilde
r,\partial_s^2\tilde r$ bounded uniformly in $(s,\delta_\mu)$. Therefore
$\rho\ge\sqrt{1-s}+\cO(\delta_\mu^2)$, and
\[
1-\rho\ \le\ 1-\sqrt{1-s}+\cO(\delta_\mu^2)\ =\ \frac{s}{1+\sqrt{1-s}}+\cO(\delta_\mu^2),
\]
\sa{where we use the elementary identity} $1-\sqrt{1-s}=\dfrac{s}{1+\sqrt{1-s}}$. \sa{Note that} 
the $\cO(\delta_\mu^2)$ remainder inherits the same $C^2$-in-$s$ dependence from
\eqref{eq:rho-expand}. This completes the proof.
\end{proof}

\begin{remark}\label{rem:q-floorBnum}
The proof of Theorem~\ref{thm:q-naglower} uses only the uniform $\cO(\delta_\mu^2)$ control in \eqref{eq:q-slowprod}; no sign information about the remainder is needed.
\end{remark}
\begin{lemma}[The regime $s=\cO(1/\kappa)$]
\label{lem:q-tinys-slowpair}
Fix \(\tau\ge1\) and \(\mgd{B_{\mathrm{bd}}}>0\), and define
\[
   \mgd{C_{\mathrm{bd}}}:=\mgd{B_{\mathrm{bd}}}+\sqrt{\mgd{B_{\mathrm{bd}}}}+2 ,
\]
\sa{Then,} there is a threshold
\(\kappa_0=\kappa_0(\tau,\mgd{B_{\mathrm{bd}}})\) such that, for every \(\kappa\ge\kappa_0\) and every pair
\((s,\delta_\mu)\) satisfying
\[
   0<s\le\frac{\mgd{B_{\mathrm{bd}}}}{\kappa},\qquad
   0\le\delta_\mu\le\frac{\mgd{B_{\mathrm{bd}}}}{\kappa^2},\qquad
   \beta=1-s,
\]
the delayed Nesterov polynomial \sa{\(P^{\mathrm{NAG}}_{\tau,\delta}\), defined in
\eqref{eq:q-nag},} %
has two roots\footnote{Counted with
multiplicity; in the case of a double root we set \(z_1=z_2\).} \(z_1,z_2\) \sa{for \(\delta=\delta_\mu\) such that}
\[
   \sa{\max\{|z_1-1|,~|z_2-1|\}}\le \frac{\mgd{C_{\mathrm{bd}}}}{\kappa}. %
\]
Consequently its spectral radius satisfies \(\rho\ge 1-\mgd{C_{\mathrm{bd}}}\kappa^{-1}\); equivalently,
\(1-\rho\le \mgd{C_{\mathrm{bd}}}\kappa^{-1}\). The same conclusion holds for the delayed heavy-ball polynomial
\sa{\(P^{\mathrm{HB}}_{\tau,\delta}\)}, defined in \eqref{eq:q-hb}, \sa{for \(\delta=\delta_\mu\)}: its spectral radius likewise
satisfies \(\rho\ge 1-\mgd{C_{\mathrm{bd}}}\kappa^{-1}\), with the same constant \(\mgd{C_{\mathrm{bd}}}\).
\end{lemma}

\begin{proof}
\sa{Given $\mgd{B_{\mathrm{bd}}}>0$, %
let \(\varepsilon:=\kappa^{-1}\), %
and we set the two small parameters, $s$ and $\delta_\mu$, and $\beta$ as follows:}
\[
   s=\sigma\varepsilon,\quad \delta_\mu=\eta\varepsilon^2,\quad \beta=1-\sigma\varepsilon,
   \quad%
   \text{\sa{for}}\quad
   \sigma\in(0,\mgd{B_{\mathrm{bd}}}],\quad \eta\in[0,\mgd{B_{\mathrm{bd}}}].
\]
We locate the roots near \(z=1\) through the substitution \(z=1+\varepsilon w\) \sa{for some $w\in\mathbb{C}$}. Since
\[
   z-1=\varepsilon w,\qquad z-\beta=\varepsilon(w+\sigma),
   \qquad (1+\beta)z-\beta=1+(2-\sigma\varepsilon)\varepsilon w,
\]
both characteristic polynomials factor through the same power of \(\varepsilon\):
\begin{align*}
   P^{\mathrm{NAG}}_{\tau,\delta_\mu}(1+\varepsilon w)
      &=\varepsilon^2\,G_\varepsilon(w),
      & G_\varepsilon(w)&:=(1+\varepsilon w)^{\tau}\,w(w+\sigma)
         +\eta\bigl[1+(2-\sigma\varepsilon)\varepsilon w\bigr],\\
   \hbKslow(1+\varepsilon w)
      &=\varepsilon^2\,\widetilde G_\varepsilon(w),
      & \widetilde G_\varepsilon(w)&:=(1+\varepsilon w)^{\tau-1}\,w(w+\sigma)+\eta,
\end{align*}
using $\hbPslow = z \hbKslow$, \sa{where $\hbKslow$ is defined in \eqref{eq:q-hb}}. As \(\varepsilon\to0\) both \(G_\varepsilon\) and \(\widetilde G_\varepsilon\) reduce to the %
\emph{same} \(\varepsilon\)-independent quadratic
\[
   g(w):=w(w+\sigma)+\eta .
\]

\smallskip
\noindent\emph{Step 1: a circle enclosing the roots of \(g\).}
If \(g(w)=0\), then \(|w|^2\le\sigma|w|+\eta\), so
\(|w|\le\tfrac12\bigl(\sigma+\sqrt{\sigma^2+4\eta}\bigr)\le\sigma+\sqrt{\eta}\le \mgd{B_{\mathrm{bd}}}+\sqrt{\mgd{B_{\mathrm{bd}}}}\).
Hence, with \(R:=\mgd{C_{\mathrm{bd}}}=\mgd{B_{\mathrm{bd}}}+\sqrt{\mgd{B_{\mathrm{bd}}}}+2\), both roots of \(g\) lie in \sa{\(\{w: |w|<R-1\}\). Moreover,} on the
circle \sa{\(\{w:~|w|=R\}\)}, using \(|w+\sigma|\ge R-\sigma\ge R-\mgd{B_{\mathrm{bd}}}\) and \(\eta\le \mgd{B_{\mathrm{bd}}}\),
\[
   |g(w)|\ \ge\ |w|\,|w+\sigma|-\eta\ \ge\ R(R-\mgd{B_{\mathrm{bd}}})-\mgd{B_{\mathrm{bd}}}=:m_R .
\]
Since \(R-\mgd{B_{\mathrm{bd}}}=\sqrt{\mgd{B_{\mathrm{bd}}}}+2\ge2\), we have \(m_R\ge 2R-\mgd{B_{\mathrm{bd}}}=\mgd{B_{\mathrm{bd}}}+2\sqrt{\mgd{B_{\mathrm{bd}}}}+4>0\).

\smallskip
\noindent\emph{Step 2: a uniform \(\cO(\varepsilon)\) perturbation bound on \(|w|=R\).}
For \(\varepsilon\le1/R\) we have \(|\varepsilon w|\le1\) on \sa{\(\{w:~|w|=R\}\)}, so for every integer
\(p\ge0\),
\[
   \bigl|(1+\varepsilon w)^{p}-1\bigr|
   \ \le\ \sum_{k=1}^{p}\binom{p}{k}|\varepsilon w|^{k}
   \ \le\ (2^{p}-1)\,\varepsilon R .
\]
\sa{Since} \(|w(w+\sigma)|\le R(R+\mgd{B_{\mathrm{bd}}})\) and \(|2-\sigma\varepsilon|\le2\), %
\sa{for $w$ such that \(|w|=R\), we get}
\beq
   |G_\varepsilon(w)-g(w)|
   \ &\le& \ (2^{\tau}-1)\varepsilon R\cdot R(R+\mgd{B_{\mathrm{bd}}})+2\mgd{B_{\mathrm{bd}}}\varepsilon R
   \ \le\ c_1\varepsilon,
 \\
   |\widetilde G_\varepsilon(w)-g(w)|
   \ &\le&\ (2^{\tau-1}-1)\varepsilon R\cdot R(R+\mgd{B_{\mathrm{bd}}})\ \le\ c_1\varepsilon,
\eeq
where \(c_1:=(2^{\tau}-1)R^{2}(R+\mgd{B_{\mathrm{bd}}})+2\mgd{B_{\mathrm{bd}}}R\) depends only on \((\tau,\mgd{B_{\mathrm{bd}}})\).

\smallskip
\noindent\emph{Step 3: \sa{Rouché's theorem} and conclusion.}
Set \(\kappa_0:=\max\{\,R,\ 1+c_1/m_R\,\}\). For \(\kappa\ge\kappa_0\), we have
\(\varepsilon\le1/R\), and using the definition $\varepsilon=1/\kappa$, we have \(c_1\varepsilon<m_R\). \sa{Hence, on \(\{w:|w|=R\}\), we have}
\[ 
   |G_\varepsilon(w)-g(w)|<m_R\le|g(w)|,
   \qquad
   |\widetilde G_\varepsilon(w)-g(w)|<m_R\le|g(w)| .
\]
By Rouch\'e's theorem each of \(G_\varepsilon,\widetilde G_\varepsilon\) has the same number
of zeros in \(\{w:~|w|<R\}\) as \(g\), \sa{counted with multiplicity. Since $g$ has two zeros in \(\{w:~|w|<R\}\), let
\(w_1,w_2\) be the zeros of $G_\varepsilon$ (resp. $\widetilde G_\varepsilon$) in \(\{w:~|w|<R\}\); hence,} \(z_i:=1+\varepsilon w_i\) \sa{for $i=1,2$} are roots of $\nagPslow$ \textup{(}resp.\ \(\hbPslow\)\textup{)}, and
\[
   |z_i-1|=\varepsilon|w_i|<\varepsilon R=\frac{R}{\kappa}=\frac{\mgd{C_{\mathrm{bd}}}}{\kappa},\quad i=1,2.
\]
In particular \(|z_i|\ge 1-\mgd{C_{\mathrm{bd}}}/\kappa\), so the spectral radius satisfies
\(\rho\ge\max_i|z_i|\ge 1-\mgd{C_{\mathrm{bd}}}\kappa^{-1}\).
\end{proof}

\subsubsection[Nesterov cannot sqrt-kappa-accelerate under delay]{Nesterov cannot $\sqrt\kappa$-accelerate under delay}\label{sec:quad-main}

\begin{theorem}[Nesterov cannot \sa{achieve $\sqrt\kappa$-acceleration} under delay]\label{thm:q-naglower}
For the delayed Nesterov method $(\nu=\beta)$ on $\QmuL$, \mgd{with $\alpha>0$ and $0<\beta<1$} the best certifiable rate
$\rhost_{\mathrm{NAG}}(\tau,\kappa)$ \sa{over Schur-stable tunings} satisfies
\begin{equation}\label{eq:q-barrier}
1-\rhost_{\mathrm{NAG}}\ \le
\begin{cases}
\bigl(2+o(1)\bigr)\,\kappa^{-2/3}, & \tau=1,\\[4pt]
\mathcal{O}(\kappa^{-1}), & \tau=2,\\[4pt]
\mathcal{O}_\tau(\kappa^{-1}), & \tau\ge3.
\end{cases}
\end{equation}
\sa{for every fixed \mgd{integer} delay \mgd{$\tau\geq 1$} and all sufficiently large
$\kappa$.} In particular,
\[
        1-\rhost_{\mathrm{NAG}}=o(\kappa^{-1/2}),
\]
\sa{i.e.,} a fixed delay prevents the classical Nesterov choice $\nu=\beta$ from retaining
$\sqrt\kappa$-acceleration. %
\end{theorem}

\begin{proof}
Fix a Schur-stable Nesterov tuning and write $s=1-\beta\in(0,1)$,
$\Delta=\alpha L$, and $\delta_\mu=\Delta/\kappa$. We prove an upper bound on
$1-\rho$ for every such tuning; taking the supremum over tunings is the same
as bounding $1-\rhost_{\mathrm{NAG}}$. 

We first handle the small-momentum regime
\mgd{$s\ge\tfrac12$ (equivalently, $\beta\le\tfrac12$),
uniformly for every fixed delay $\tau\ge1$,
including $\tau=1$.} %
In this case, as shown below, $\Delta=\cO_\tau(1)$ for every Schur-stable tuning. Indeed, the \mgd{high-curvature} polynomial $\nagPfast$ is monic of degree $n=\tau+2$ and all its roots
lie in the unit disk. Write
$
\nagPfast(z)=\prod_{i=1}^n (z-z_i)
=z^n+\sum_{j=0}^{n-1} a_{j}z^j,
$ with roots $|z_i|<1$. By Vieta formulas,
$
a_{n-k}=(-1)^k
\sum_{|I|=k}\prod_{i\in I}z_i.
$
\sa{Since $|z_i|<1$ for $i=1,\ldots,n$, we have}
$|a_{n-k}|<\binom{n}{k}$. In particular, %
\sa{as $n=\tau+2$, %
for $k=\tau+1$} we obtain that the coefficient $a_1$ (which is the coefficient of $z$ in 
$\nagPfast$) satisfies $\sa{|a_1|} < \tau+2$. 
If $\tau=1$, we have the coefficient $\sa{0<}a_1 = \beta+\Delta(1+\beta)<3$; therefore, $\Delta=\cO(1)$. If $\tau\ge2$, then $\sa{0<}a_1 = \Delta(1+\beta)\le \tau+2$; hence, we again have $\Delta=\cO_\tau(1)$. Consequently, \sa{using \eqref{eq:s-Delta-def},}
\[
        \delta_\mu=\frac{\Delta}{\kappa}=\cO_\tau(\kappa^{-1}).
\]
\mgd{For all sufficiently large $\kappa$, the conditions of
Lemma~\ref{lem:q-floorA} hold at the low-curvature
endpoint $\lambda=\mu$, yielding a slow mode near $1$
and hence}
\[
        1-\rho %
        \sa{<} \frac{2\delta_\mu}{s}\le 4\delta_\mu
        =\cO_\tau(\kappa^{-1}),
\]
because $s\ge\tfrac12$. \mgd{Thus, for every fixed $\tau\ge1$, all
Schur-stable tunings with $s\ge\tfrac12$
satisfy $1-\rho=\cO_\tau(\kappa^{-1})$.
It remains to consider $0<s<\tfrac12$,
where Proposition~\ref{prop:q-fast} applies.
We distinguish three cases according to $\tau$:
for $\tau\ge2$, we establish the same
$\cO_\tau(\kappa^{-1})$ upper bound,
whereas for $\tau=1$, we obtain the
$\cO(\kappa^{-2/3})$ upper bound.}

\medskip
\noindent\textbf{Case $\tau\ge3$.}
By Proposition~\ref{prop:q-fast}(c), \sa{due to Schur stability,
we have}
\[
        \Delta<\frac{2s}{\tau-2+2s}\le \frac{2s}{\tau-2}.
\]
First suppose that
\[
        s\ge \frac{16}{(\tau-2)\kappa}.
\]
Then
\[
        \delta_\mu=\frac{\Delta}{\kappa}
        <\frac{2s}{(\tau-2)\kappa}
        \le \frac{s^2}{8}
\quad \mbox{and} \quad  
        \frac{2\tau\delta_\mu}{s}
        <\frac{4\tau}{(\tau-2)\kappa},
\]
where $\frac{4\tau}{(\tau-2)\kappa} \leq 1/3$ 
for all sufficiently large $\kappa$ at fixed $\tau$. \sa{Thus, Lemma~\ref{lem:q-floorA} %
applies to the setting we consider here, i.e.,}
\beq
        1-\rho
        \sa{<} \frac{2\delta_\mu}{s}
        =\frac{2\Delta}{\kappa s}
        <\frac{4}{(\tau-2)\kappa}.
        \label{nag-rate-bound-large-s}
\eeq
\sa{Next,} we consider the complementary \sa{scenario for $s$, i.e.,} 
$
        s<\frac{16}{(\tau-2)\kappa}.
$
Then $s=\cO_\tau(\kappa^{-1})$, and the same \mgd{high-curvature stability bound} %
gives
\[
        \delta_\mu=\frac{\Delta}{\kappa}
        <\frac{2s}{(\tau-2)\kappa}
        =\cO_\tau(\kappa^{-2}).
\]
Thus, the hypotheses of Lemma~\ref{lem:q-tinys-slowpair} hold, \sa{and it shows that}
$
        1-\rho\le \mgd{C_{\mathrm{bd}}}(\tau)\kappa^{-1}
$
\mgd{for some $C_{\mathrm{bd}}(\tau)>0$ that depends only on
the fixed delay $\tau$, and is independent of
$\kappa$ and the chosen tuning.} From this inequality and \eqref{nag-rate-bound-large-s}, we conclude every Schur-stable tuning with
$\tau\ge3$ satisfies
$
        1-\rho=\cO_\tau(\kappa^{-1})
$
for any $0<s<1/2$.

\medskip
\noindent\textbf{Case $\tau=2$.}
By Proposition~\ref{prop:q-fast}(b), \sa{due to Schur stability, we have} 
$
        \Delta<s.
$
First, suppose that
$
        s\ge {16}/{\kappa}.
$
Then,
\[
        \delta_\mu=\frac{\Delta}{\kappa}<\frac{s}{\kappa}
        \le \frac{s^2}{16}<\frac{s^2}{8}
\quad
\mbox{and}
\quad
        \frac{2\tau\delta_\mu}{s}
        =\frac{4\delta_\mu}{s}
        <\frac{4}{\kappa},
\]
where $\frac{4}{\kappa} \leq 1/3$ for all sufficiently large $\kappa$. Lemma~\ref{lem:q-floorA}
therefore gives
$
        1-\rho
        \le \frac{2\delta_\mu}{s}
        =\frac{2\Delta}{\kappa s}
        <\frac{2}{\kappa}$. 
        \sa{On the other hand, if one has}
$
        s<\frac{16}{\kappa},
$
then $s=\cO(\kappa^{-1})$ and
$
        \delta_\mu=\frac{\Delta}{\kappa}<\frac{s}{\kappa}=\cO(\kappa^{-2});
$
\sa{hence, for this scenario Lemma~\ref{lem:q-tinys-slowpair} implies that}
$
        1-\rho\le \mgd{C_{\mathrm{bd}}}\kappa^{-1}
$
\mgd{for some constant $\mgd{C_{\mathrm{bd}}}>0$ independent of $\kappa$ and the chosen tuning.}
\sa{Thus, we can conclude that} every Schur-stable tuning with $\tau=2$ satisfies
$
        1-\rho=\cO(\kappa^{-1}).
$

\medskip
\noindent\textbf{Case $\tau=1$ (the $\kappa^{-2/3}$ barrier).}
By Proposition~\ref{prop:q-fast}(a), every Schur-stable tuning with
$0<s<\tfrac12$ satisfies
\begin{equation}\label{eq:q-tau1cap}
        \Delta<\frac{\sqrt s}{\beta}\le 2\sqrt s,
        \qquad
        \delta_\mu=\frac{\Delta}{\kappa}<\frac{\sqrt s}{\beta\kappa}
        \le \frac{2\sqrt s}{\kappa}.
\end{equation}
\sa{For sufficiently large $\kappa$, let $S_\kappa \in (0,\frac{1}{2})$ be the solution of}
\begin{equation}\label{eq:q-Skappa}
        (1-S_\kappa)\kappa S_\kappa^{3/2}=8.
\end{equation}
\sa{Note that for all sufficiently large $\kappa$, the solution $S_\kappa$ is unique}
because $s\mapsto(1-s)s^{3/2}$ is strictly increasing on $(0,\tfrac12)$; \sa{moreover, as $\kappa\to\infty$, it holds that}
\begin{equation}\label{eq:q-Skappa-asymp}
        S_\kappa=(8/\kappa)^{2/3}(1+o(1))=4\kappa^{-2/3}(1+o(1)),
\end{equation}
\sa{where $o(1)$ means a quantity tending to zero as $\kappa\to\infty$.}
Note that by \eqref{eq:q-tau1cap}, every Schur-stable tuning satisfies $\Delta < \Delta_{\max}:= \sqrt{s}/\beta$. Next, we will consider the cases $s\ge S_\kappa$ and $s< S_\kappa$ separately. 

\textbf{Case $s \ge S_\kappa$:} \sa{
For \sa{$s$ such that $\frac{1}{2}\geq s\geq S_\kappa$}, we can argue that $\sa{\delta_\mu<}\Delta_{\max}/\kappa \leq s^2/8$. %
Indeed, for $s=S_\kappa$, we have exactly $\Delta_{\max}/\kappa=\sa{S_\kappa^2}/8$. Therefore, using the fact that
$(1-s)s^{3/2}$ is increasing on $(0,\tfrac12)$, we get
\[
        \delta_\mu\le \frac{\Delta_{\max}}{\kappa}
        =\frac{\sqrt{s}}{(1-s)\kappa}\le \frac{s^2}{8}.
\]
Moreover, we also have $2\delta_\mu/s\le s/4\le1/3$ as $s<1/2$. Thus, Lemma~\ref{lem:q-floorA} is applicable here, and together with $\Delta<\Delta_{\max}$, Lemma~\ref{lem:q-floorA} implies that}
\[
        1-\rho
        \sa{<} \frac{2\delta_\mu}{s}
        =\frac{2\Delta}{\kappa s}
        \le G_\kappa(s):= \frac{2}{(1-s)\sqrt{s}\kappa}.
\]
\sa{The function $G_\kappa(s)$ is decreasing on $(0,1/3)$ and increasing on
$(1/3,1/2]$. %
Hence,
\[
\sup_{s\in[S_\kappa,1/2]} G_\kappa(s)
=
\max\left\{G_\kappa(S_\kappa),G_\kappa(1/2)\right\}.
\]
From \eqref{eq:q-Skappa}, we have
$G_\kappa(S_\kappa)
=
\frac{S_\kappa}{4}
=
(1+o(1))\kappa^{-2/3}$,
whereas
$G_\kappa(1/2)
=
\frac{4\sqrt{2}}{\kappa}
=
o(\kappa^{-2/3})$.
Therefore, for all sufficiently large $\kappa$,
\[
1-\rho
\le
\sup_{s\in[S_\kappa,1/2]}G_\kappa(s)
=
(1+o(1))\kappa^{-2/3}.
\]}

\sa{\textbf{Case $s<S_\kappa$:}
Since every admissible
tuning satisfies $\Delta<\Delta_{\max}$, for $0<s<S_\kappa$, we have
\begin{equation}
\label{eq:delta-mu-bound}
\delta_\mu=\frac{\Delta}{\kappa}
<
\frac{\Delta_{\max}}{\kappa}
=
\frac{\sqrt{s}}{(1-s)\kappa}
\le
\frac{2\sqrt{S_\kappa}}{\kappa}
=
\mathcal{O}(\kappa^{-4/3}),
\end{equation}
where we used \eqref{eq:q-tau1cap}. Thus, for all sufficiently large $\kappa$, we have
$\delta_\mu\le\delta_0$ uniformly for $0<s<S_\kappa$ and over all
admissible tunings. Lemma~\ref{lem:q-floorB} therefore gives
\[
1-\rho
\le
\frac{s}{1+\sqrt{1-s}}
+\mathcal{O}(\delta_\mu^2),
\]
where the $O(\delta_\mu^2)$ term is uniform in $s$ and over all
admissible tunings.} The first term is increasing in $s$, so over $s<S_\kappa$,
\sa{
\[
\frac{s}{1+\sqrt{1-s}}
\le
\frac{S_\kappa}{1+\sqrt{1-S_\kappa}}
=
\frac{S_\kappa}{2}(1+o(1)).
\]}%
Moreover, \sa{\eqref{eq:delta-mu-bound} implies that 
$\delta_\mu^2=O(\kappa^{-8/3})=o(\kappa^{-2/3})$.} Thus,
\[
1-\rho
\le
\frac{S_\kappa}{2}(1+o(1))
=
2(1+o(1))\kappa^{-2/3}.
\]
Combining the two subcases, every Schur-stable tuning with $\tau=1$ satisfies
\[
        1-\rho\le \bigl(2+o(1)\bigr)\kappa^{-2/3}.
\]

Finally, combining the three delay cases ($\tau=1$, $\tau=2$, $\tau\geq 3$) and then optimizing over Schur-stable tunings
gives \eqref{eq:q-barrier}. In every fixed-delay case,
$
        1-\rhost_{\mathrm{NAG}}=o(\kappa^{-1/2}),
$
so delayed Nesterov does not retain $\sqrt\kappa$-acceleration.
\end{proof}
{The previous result provides an upper bound for the spectral gap 
$1-\rhost_{\mathrm{NAG}}\ \le
\bigl(2+o(1)\bigr)\,\kappa^{-2/3}
$
for $\tau=1$, a natural question is whether this bound is tight or equivalently whether an $\Theta(\kappa^{-2/3})$ spectral gap is achievable by NAG. The following result shows that this is indeed the case. This shows that NAG is still faster than gradient descent for delay $\tau=1$ when $\kappa$ is sufficiently large but it is still not accelerated in the sense that it cannot achieve the spectral gap of $\Theta(\kappa^{-1/2})$.}
\mg{\begin{lemma}[Matching $\Theta(\kappa^{-2/3})$ spectral gap for NAG when $\tau=1$]
\label{lem:q-nag-tau1-achievability}
Let $\tau=1$ and $\kappa\ge 64$. Choose
\begin{equation}
\label{eq:q-nag-achievability-tuning}
        s=\kappa^{-2/3},
        \qquad
        \beta=1-s,
        \qquad
        \Delta=\alpha L=\frac{\sqrt{s}}{2}
        =\frac{1}{2}\kappa^{-1/3}.
\end{equation}
Then $\delta_\mu=\Delta/\kappa=s^2/2$, and for every
$\delta\in[\delta_\mu,\Delta]$, every root $z$ of
$P^{\scriptscriptstyle \mathrm{NAG}}_{1,\delta}$ satisfies
\begin{equation}
\label{eq:q-nag-achievability-radius}
        |z|
        \le
        1-\frac{7}{22}s
        =
        1-\frac{7}{22}\kappa^{-2/3}.
\end{equation}
Consequently, the tuning \eqref{eq:q-nag-achievability-tuning} is Schur stable and
\begin{equation}
\label{eq:q-nag-achievability-rho}
        \rho(\alpha,\beta)
        \le
        1-\frac{7}{22}\kappa^{-2/3},
        \qquad
        1-\rhost_{\mathrm{NAG}}(1,\kappa)
        \ge
        \frac{7}{22}\kappa^{-2/3}.
\end{equation}
In particular, together with Theorem~\ref{thm:q-naglower},
\[
        1-\rhost_{\mathrm{NAG}}(1,\kappa)
        =\Theta(\kappa^{-2/3})
        \qquad\text{as }\kappa\to\infty.
\]
\end{lemma}
\begin{proof}
Fix $\delta\in[\delta_\mu,\Delta]$. Since $\kappa\ge64$, we have
$s\le1/16$, and therefore
\[
        \beta=1-s\ge\frac{15}{16},
        \qquad
        \Delta=\frac{\sqrt{s}}{2}\le\frac18.
\]
For $\tau=1$, the NAG characteristic polynomial \eqref{eq:q-nag} is
\[
P^{\scriptscriptstyle \mathrm{NAG}}_{1,\delta}(z)
=
 z(z-1)(z-\beta)
 +\delta\bigl[(1+\beta)z-\beta\bigr].
\]
Evaluating it at $0$ and at $\delta$ gives
\[
P^{\scriptscriptstyle \mathrm{NAG}}_{1,\delta}(0)
=-\delta\beta<0,
\qquad
P^{\scriptscriptstyle \mathrm{NAG}}_{1,\delta}(\delta)
=\delta^3>0.
\]
Hence, by the intermediate value theorem, there exists a real root
$x_\delta\in(0,\delta)$. In particular,
$x_\delta<\delta\le\Delta\le1/8$.
Set
\[
        d_\delta
        :=\beta-(1+\beta)x_\delta
        =1-s-(2-s)x_\delta.
\]
Using $s\le1/16$ and $x_\delta<1/8$,
\begin{equation}
\label{eq:q-nag-achievability-d}
        d_\delta
        \ge
        1-\frac1{16}-2\cdot\frac18
        =\frac{11}{16}>\frac12.
\end{equation}
Since $P^{\scriptscriptstyle \mathrm{NAG}}_{1,\delta}(x_\delta)=0$,
\beq
 x_\delta(1-x_\delta)(\beta-x_\delta)
 =\delta\bigl[\beta-(1+\beta)x_\delta\bigr]
 =\delta d_\delta.
 \label{eq-x-delta-root}
\eeq
Moreover, expanding the product
$
        (1-x_\delta)(\beta-x_\delta)
        =\beta - (1+\beta)x_\delta+ x_\delta^2
    =
        d_\delta+x_\delta^2.
$
and inserting this equality back into \eqref{eq-x-delta-root} gives
\begin{equation}
\label{eq:q-nag-achievability-delta}
        \delta
        =x_\delta\left(1+\frac{x_\delta^2}{d_\delta}\right).
\end{equation}
Because $d_\delta>1/2$ and $x_\delta<\delta$, \eqref{eq:q-nag-achievability-delta} yields
$
        \delta
        <x_\delta(1+2\delta^2),
$
and hence, using $\delta\le\Delta=\sqrt{s}/2$ and $\delta\ge\delta_\mu=s^2/2$,
\begin{equation}
\label{eq:q-nag-achievability-xlower}
        x_\delta
        >\frac{\delta}{1+2\delta^2}
        \ge
        \frac{\delta_\mu}{1+2\Delta^2}
        =
        \frac{s^2}{2+s}.
\end{equation}
We now factor out this real root. 
Since
\[
P^{\scriptscriptstyle \mathrm{NAG}}_{1,\delta}(z)
=
z^3-(1+\beta)z^2+\bigl[\beta+\delta(1+\beta)\bigr]z-\delta\beta,
\]
Vieta's formulas give the sum and product of the remaining two roots as \(1+\beta-x_\delta\) and \(\delta\beta/x_\delta\), respectively; hence
we have
\begin{equation}
\label{eq:q-nag-achievability-factor}
P^{\scriptscriptstyle \mathrm{NAG}}_{1,\delta}(z)
= 
(z-x_\delta)
\left[
 z^2-(1+\beta-x_\delta)z
 +\frac{\delta\beta}{x_\delta}
\right].
\end{equation}
We next show that the quadratic factor has a complex-conjugate pair of roots.
By \eqref{eq:q-nag-achievability-delta}, its constant term is
\begin{equation}
\label{eq:q-nag-achievability-product}
        \frac{\delta\beta}{x_\delta}
        =
        1-s+\frac{(1-s)x_\delta^2}{d_\delta}.
\end{equation}
Therefore its discriminant equals
\begin{align}
        &(2-s-x_\delta)^2
        -4\left(1-s+\frac{(1-s)x_\delta^2}{d_\delta}\right)
        =
        (s+x_\delta)^2-4x_\delta
        -\frac{4(1-s)x_\delta^2}{d_\delta}
        <
        (s+x_\delta)^2-4x_\delta.
\label{eq:q-nag-achievability-disc}
\end{align}
The two zeros, as a function of $x$, of $\phi(x):=(s+x)^2-4x$ are
\[
        x_-=(1-\sqrt{1-s})^2,
        \qquad
        x_+=(1+\sqrt{1-s})^2.
\]
Since $\sqrt{1-s}\ge1-s$,
\[
        x_-
        =\frac{s^2}{(1+\sqrt{1-s})^2}
        \le
        \frac{s^2}{(2-s)^2}.
\]
For $0<s\le1/16$, we have
$
        (2-s)^2-(2+s)
        =2-5s+s^2>0,
$
so
$
        x_-<\frac{s^2}{2+s}<x_\delta
$
by \eqref{eq:q-nag-achievability-xlower}. On the other hand,
$x_\delta<1/8<1<x_+$. Therefore, $x_\delta \in (x_- , x_+)$. 
\sa{Moreover, the leading coefficient of the quadratic $\phi(x)$ is positive;
therefore, $\phi(x)<0$ for $x\in(x_-,x_+)$.}
Thus,
$\phi(x_\delta)=(s+x_\delta)^2-4x_\delta<0$, and
\eqref{eq:q-nag-achievability-disc} shows that the quadratic factor in
\eqref{eq:q-nag-achievability-factor} indeed has a complex-conjugate pair.
Let $z_1,z_2$ denote this pair. By \eqref{eq:q-nag-achievability-product},
Vieta's formula applied to this factor and \eqref{eq:q-nag-achievability-d} give
\begin{align*}
        |z_1|^2=|z_2|^2=z_1z_2
        &=1-s+\frac{(1-s)x_\delta^2}{d_\delta}
        <1-s+\frac{\Delta^2}{11/16}
        =1-\frac{7}{11}s,
\end{align*}
where we used $1-s<1$ and $x_\delta<\delta\leq\Delta=\sqrt{s}/2$.
Using $\sqrt{1-u}\le1-u/2$ for $u\in[0,1]$,
\[
        |z_1|=|z_2|
        <\sqrt{1-\frac{7}{11}s}
        \le
        1-\frac{7}{22}s.
\]
The remaining root satisfies
$x_\delta<1/8<1-7s/22$ for $s\le1/16$. \sa{All the preceding estimates are uniform over
$\delta\in[\delta_\mu,\Delta]$, since they use only the bounds
$\delta\le\Delta$, $x_\delta<\delta\le\Delta$, and
$d_\delta\ge11/16$. Hence, every root of
$P^{\scriptscriptstyle \mathrm{NAG}}_{1,\delta}$ satisfies
\eqref{eq:q-nag-achievability-radius}, uniformly over
$\delta\in[\delta_\mu,\Delta]$.}
\sa{Since $[\delta_\mu,\Delta]$ is precisely the range of
$\delta=\alpha\lambda$ for $\lambda\in[\mu,L]$, the definition of
$\rho(\alpha,\beta)$ gives the first inequality in \eqref{eq:q-nag-achievability-rho}. Taking the infimum over all Schur-stable NAG tunings then yields $1-\rho_{\rm NAG}^\star(1,\kappa)
\ge \frac{7}{22}\kappa^{-2/3}$. Combining this bound with Theorem~\ref{thm:q-naglower} proves $1-\rho_{\rm NAG}^\star(1,\kappa)
=\Theta(\kappa^{-2/3})$.}
\end{proof}
} 

\subsubsection{Heavy ball is non-accelerated under delay}\label{sec:quad-hb}

We now treat \sa{the
\emph{heavy-ball}~(HB) method, i.e., GMM with $\nu=0$, which also belongs to
the family with the characteristic polynomial in \eqref{eq:q-charpoly}.}
The conclusion is even simpler than for the Nesterov's method; indeed, we show that \emph{for every fixed \mgd{positive} delay}, %
the spectral gap is at most of order $\kappa^{-1}$, \mgd{uniformly over Schur-stable heavy-ball tunings with $\beta>0$}. 
Thus, the heavy-ball method also fails to retain the $\sqrt\kappa$ accelerated scaling. %

\paragraph{The factorization, and the contrast with Polyak.} Recall from \eqref{eq:q-hb} that the characteristic polynomial of the delayed HB admits the factorization
\beq\label{eq:q-hbfactor}
\hbP(z)
=z^{\tau}(z-1)(z-\beta)+\delta z = 
z \hbK(z),\quad \mbox{where} \quad
\hbK(z) = z^{\tau-1}(z-1)(z-\beta)+\delta. 
\eeq
The root $z=0$ is irrelevant for the convergence rate; the nontrivial rate is governed by the roots of $\hbK$ as $\delta$ ranges over $[\delta_\mu,\Delta]$.

\sa{Now consider the undelayed case $\tau=0$, where} $\hbPzero(z)=z^2-(1+\beta-\delta)z+\beta$; \sa{thus,} the product of the two roots is the constant \sa{$\beta>0$}. When the pair is complex, \sa{the modulus of each} root is exactly $\sqrt\beta$, \sa{independent}
of $\delta$; \sa{hence,} the Polyak choice keeps both root moduli equal to $\sqrt{\beta}$ throughout the spectral interval and gives the accelerated quadratic rate. \sa{On the other hand,} under \sa{positive} delay, already at $\tau=1$, the product of the nontrivial pair becomes $\beta+\delta$, which varies across the spectrum as $\lambda$ varies in $[\mu,L]$.

\begin{lemma}[Lower bound on $\rho$ for HB]\label{lem:q-hbslow}
Fix $\tau\ge1$ and \mgd{let $s =1-\beta \in (0,1)$}. Suppose $0<\delta_\mu\le \frac{s^2}{8}$ and $\frac{2(\tau-1)\delta_\mu}{s}\le\frac{1}{3}$.
Then $\hbKslow$ %
has a real root in 
$
        [1-\frac{2\delta_\mu}{s},1).
$
Consequently
$
        \rho\ge1-\frac{2\delta_\mu}{s}$, and hence
        $
        1-\rho\le \frac{2\delta_\mu}{s}$.
\end{lemma}

\begin{proof} Set
$
        z_0:=1-\frac{2\delta_\mu}{s}.
$
The condition $\delta_\mu\le s^2/8$ gives
$
        \frac{2\delta_\mu}{s}\le \frac{s}{4};
$
\sa{hence,}
\[
        z_0\in[1-s/4,1)\subset(\beta,1),
        \qquad
        z_0-\beta=s-\frac{2\delta_\mu}{s}\ge\frac{3s}{4}.
\]
Moreover
$
        \hbKslow(1)=\delta_\mu>0.
$
\sa{Therefore, at $z_0$, we have}
\[
\begin{aligned}
\hbKslow(z_0)=z_0^{\tau-1}(z_0-1)(z_0-\beta)+\delta_\mu  \le
-z_0^{\tau-1}\frac{2\delta_\mu}{s}\frac{3s}{4}+
\delta_\mu
=
\delta_\mu\left(1-\frac32 z_0^{\tau-1}\right).
\end{aligned}
\]
\sa{
Note that the first hypothesis implies $\frac{2\delta_\mu}{s}
\le \frac{s}{4}
\mathrel{\wubtext{<}} \frac1{\wubtext{4}}
<1$,
where we have used $0<s\mathrel{\wubtext{<}}\wubtext{1}$. Thus, for $\tau\ge 2$,
Bernoulli's inequality gives 
$z_0^{\tau-1}
=
\left(1-\frac{2\delta_\mu}{s}\right)^{\tau-1}
\ge
1-(\tau-1)\frac{2\delta_\mu}{s}
\ge \frac23$,
where the last inequality follows from the second hypothesis.
For $\tau=1$, the same bound holds trivially since
$z_0^{\tau-1}=1$. Therefore,
$K_{\tau,\delta_\mu}^{\rm HB}(z_0)
\le
\delta_\mu
\left(1-\frac32 z_0^{\tau-1}\right)
\le 0<K_{\tau,\delta_\mu}(1)$.}
By the intermediate value theorem, $K_{\tau,\delta_\mu}$ has a real root in $[z_0,1)$, and the displayed spectral-radius bound follows.
\end{proof}

At $\tau=1$, 
\sa{$\wubtext{\hbKone}(z)$ in \eqref{eq:q-hbfactor} is a quadratic, i.e.,} %
\begin{equation}\label{eq:q-hbquad}
\hbKone(z)=z^2-(1+\beta)z+(\beta+\delta).
\end{equation}
The following proposition gives the explicit bound $1-\rho<2/\kappa$ for every Schur-stable heavy-ball tuning with $\tau=1$.
\begin{proposition}[$\tau=1$ \sa{for} heavy ball]\label{prop:q-hb1}
For every Schur-stable heavy-ball tuning \sa{with $\tau=1$,}
$
        1-\rho<\frac{2}{\kappa}.
$
Consequently,
$
        1-\rhost_{\mathrm{HB}}(1,\kappa)<\frac{2}{\kappa}.
$
\end{proposition}

\begin{proof}
Let $z_1,z_2$ be the two roots of \eqref{eq:q-hbquad}. We will use three elementary inequalities. 

First, since
$
        z_1+z_2=1+\beta=2-s,
$
the triangle inequality gives
\begin{equation}\label{eq:q-hbsum}
        \rho\ge \sa{\max\{|z_1|,|z_2|\}}
        \ge \frac{|z_1+z_2|}{2}
        =1-\frac{s}{2}.
\end{equation}
Thus,
$
        1-\rho\le \frac{s}{2}.
$
Second, \mgd{at the high-curvature endpoint} $\delta=\Delta$, the root product is
$
        z_1z_2=\beta+\Delta=1-(s-\Delta).
$
Schur stability \sa{implies that $\beta+\Delta<1$; hence,} $\Delta<s$. Since
$
        \max\{|z_1|,|z_2|\}\ge \sqrt{|z_1z_2|}=\sqrt{\beta+\Delta},
$
we obtain
\begin{equation}\label{eq:q-hbfast}
        1-\rho
        \le
        1-\sqrt{1-(s-\Delta)}
        \le
        s-\Delta.
\end{equation}
Third, \mgd{at the low-curvature endpoint} $\delta=\delta_\mu$, if $\delta_\mu\le s^2/4$, the two roots are real and the larger one is
\[
        z_+
        =\frac12\left[(2-s)+\sqrt{s^2-4\delta_\mu}\right].
\]
Using $\sqrt{1-u}\ge1-u$ for $u\in[0,1]$ gives
$
        \sqrt{s^2-4\delta_\mu}
        \ge
        s-\frac{4\delta_\mu}{s};
$
hence,
\begin{equation}\label{eq:q-hbslow1}
        1-\rho\le \frac{2\delta_\mu}{s}
        =\frac{2\Delta}{\kappa s}
        \quad \mbox{for} 
    \quad \delta_\mu\le s^2/4.
\end{equation}We now combine these bounds. If $s<2/\kappa$, then \eqref{eq:q-hbsum} gives
$
        1-\rho<\frac{1}{\kappa}<\frac{2}{\kappa},
$ 
\sa{which leads to the desired result.} %
\sa{Now suppose} $s\ge2/\kappa$. If $\delta_\mu>s^2/4$, then $\Delta>\kappa s^2/4$, and \eqref{eq:q-hbfast} gives
$
        1-\rho
        \le
        s-\Delta
        <
        s-\frac{\kappa s^2}{4}.
$
The function $s\mapsto s-\kappa s^2/4$ is decreasing on $[2/\kappa,\infty)$; therefore, %
\sa{for $s=2/\kappa$ it achieves its maximum value on $[2/\kappa,\infty)$,} namely $1/\kappa$.
\sa{Finally,} it remains to consider \sa{the scenario with} $s\ge2/\kappa$ and $\delta_\mu\le s^2/4$. \sa{For this setting,} both \eqref{eq:q-hbfast} and \eqref{eq:q-hbslow1} apply, and we get
\[
        1-\rho
        \le
        \min\left\{\frac{2\Delta}{\kappa s},\,s-\Delta\right\}.
\]
The maximum of the right-hand side over $0\le\Delta\le s$ occurs when the two affine functions balance, i.e. at
$
        \Delta_\star=\kappa s^2/(2+\kappa s)\sa{\in[0,s]},
$
and the corresponding value is
$
        \frac{2s}{2+\kappa s}
        <
        \frac{2}{\kappa}$. Thus, for all cases $1-\rho<2/\kappa$. \mgd{
It remains to show that the strict inequality is preserved
when optimizing over all Schur-stable tunings.
For $\kappa>2$, the first two cases considered above,
namely (i) $s<2/\kappa$ and
(ii) $s\ge 2/\kappa$ with $\delta_\mu>s^2/4$,
both yield
$
1-\rho<1/\kappa<2/(\kappa+2).
$
In the remaining case,
(iii) $s\ge 2/\kappa$ with $\delta_\mu\le s^2/4$,
we have
\[
1-\rho
\le \frac{2s}{2+\kappa s}
\le \frac{2}{\kappa+2},
\]
where the last inequality follows from $0<s<1$.
Thus, $2/(\kappa+2)$ is an upper bound
that holds uniformly over all Schur-stable tunings.
For $1\le\kappa\le2$, the root-sum bound
\eqref{eq:q-hbsum} instead gives
$1-\rho\le s/2<1/2$.
Taking the supremum over all Schur-stable tunings
therefore yields
\[
1-\rhost_{\mathrm{HB}}(1,\kappa)
\le
\begin{cases}
\frac12, & 1\le\kappa\le2,\\[2pt]
\frac{2}{\kappa+2}, & \kappa>2.
\end{cases}
\]
In either case, the upper bound is strictly smaller
than $2/\kappa$, proving
$
1-\rhost_{\mathrm{HB}}(1,\kappa)
<\frac{2}{\kappa}$ for $\kappa\ge1$.
}
\end{proof}
\mgd{The following results bounds the high-curvature effective stepsize $\Delta$}
for any $\tau\geq 1$ \mgc{and shows that $\beta<1$ is necessary for Schur stability.}
\begin{proposition}[\mgd{high-curvature stability bound for heavy ball}]\label{prop:q-hbfast}
Fix $\tau\geq 1$. \mgc{Let $\Delta>0$ and $\beta >0$.} If  $\hbPfast$
is Schur stable, then
$
        \Delta<\frac{2s}{\tau}
$
where $s=1-\beta$. \mgc{In particular, Schur stability necessarily
implies $\beta<1$.}
\end{proposition}

\begin{proof}
Consider the Cayley transform
\[
        \widetilde \hbPfast(w)
        :=(1-w)^{\tau+2}\hbPfast\!\left(\frac{1+w}{1-w}\right).
\]
Since $\hbPfast$ is Schur stable, Lemma~\ref{lem:q-cayley} implies that $\widetilde \hbPfast$ is Hurwitz. Also
$
        \widetilde \hbPfast(0)=\hbPfast(1)=\Delta>0.
$
A real Hurwitz polynomial has all coefficients of %
\sa{the same} sign, so the coefficient of $w$ in $\widetilde \hbPfast$, which we denote by $\widetilde a_1$,  must be positive. By differentiating $\widetilde \hbPfast$ at $w=0$, we can compute 
\[
        \widetilde a_1 
        =-(\tau+2)\hbPfast(1)+2{\hbPfast}'(1),
\]
where
$\hbPfast(1)=\Delta$ and
        ${\hbPfast}'(1)=s+\Delta;
$
therefore,
$        \widetilde a_1 
        =-(\tau+2)\Delta+2(s+\Delta)
        =2s-\tau\Delta$. %
        Thus, $\widetilde a_1>0$ leads to the desired bound $2s-\tau\Delta>0$. \mgc{Finally, this implies $\beta=1-s<1$ as $\Delta>0$.} %
\end{proof}
\mgd{We are now ready to show that delayed heavy ball does not retain the accelerated $\kappa^{-1/2}$ scaling.}
\begin{theorem}[Heavy ball is non-accelerated under any fixed positive delay]\label{thm:q-hblower} For delayed heavy-ball on $\QmuL$ and any fixed $\tau\ge1$, \mgd{for every 
Schur-stable tuning with $\alpha>0$ and $\beta>0$, it holds that} $
        1-\rho=\cO_\tau(\kappa^{-1}),
$ 
as $\kappa\to\infty$.
Consequently,
\[
        1-\rhost_{\mathrm{HB}}(\tau,\kappa)=\cO_\tau(\kappa^{-1})=o(\kappa^{-1/2}),
\]
\sa{i.e.,} the delayed heavy ball never retains the accelerated $\kappa^{-1/2}$ scaling.
\end{theorem}

\begin{proof}
\mgd{By Proposition \ref{prop:q-hbfast}}, $\beta<1$ for a Schur-stable tuning. Fix a Schur-stable heavy-ball tuning and write $s=1-\beta\in(0,1)$, $\Delta=\alpha L$, and $\delta_\mu=\Delta/\kappa$. 

First suppose that
$
        s\ge \frac{16}{\tau\kappa}.  
$
By Proposition~\ref{prop:q-hbfast},
$        
\Delta<\frac{2s}{\tau}
$,
        and
$\delta_\mu=\frac{\Delta}{\kappa}<\frac{2s}{\tau\kappa}.
$
By the lower bound on $s$, we get
$
        \delta_\mu<\frac{2s}{\tau\kappa}\le \frac{s^2}{8};
$
moreover, we also have
\[
        \frac{2(\tau-1)\delta_\mu}{s}
        \mathrel{\le}
        \frac{4(\tau-1)}{\tau\kappa},
\]
which is at most $1/3$ for all sufficiently large $\kappa$ at fixed $\tau$; \sa{therefore, Lemma~\ref{lem:q-hbslow} applies here and it yields} 
\beq
        1-\rho
        \le
        \frac{2\delta_\mu}{s}
        =
        \frac{2\Delta}{\kappa s}
        <
        \frac{4}{\tau\kappa}.\label{ineq-spectral-bound}
\eeq
\sa{Next, consider the remaining scenario, i.e.,}
$
        s<\frac{16}{\tau\kappa}.
$
Then $s=\cO_\tau(\kappa^{-1})$, and \mgd{Proposition~\ref{prop:q-hbfast} gives}
\[
        \delta_\mu
        =
        \frac{\Delta}{\kappa}
        <
        \frac{2s}{\tau\kappa}
        =
        \cO_\tau(\kappa^{-2}).
\]
Thus, for a constant $\mgd{B_{\mathrm{bd}}}$, depending only on the fixed delay, the hypotheses of Lemma~\ref{lem:q-tinys-slowpair} hold for the heavy-ball slow factor, and  %
\sa{it implies that}
$
        1-\rho\le \mgd{C_{\mathrm{bd}}}\kappa^{-1},
$
\mgd{with $C_{\mathrm{bd}}=\mgd{B_{\mathrm{bd}}}+\sqrt{\mgd{B_{\mathrm{bd}}}}+2 $ that depends only on the delay}. Combining the two regimes proves $1-\rho=\cO_\tau(\kappa^{-1})$ for every Schur-stable heavy-ball tuning. Taking the infimum over Schur-stable tunings gives the stated bound on $1-\rhost_{\mathrm{HB}}$.
\end{proof}
 
\begin{remark}[Heavy ball versus Nesterov]\label{rem:q-hbvsnag}
In the GMM family, \mgd{we recall that} Nesterov corresponds to $\nu=\beta$, whereas heavy ball corresponds to $\nu=0$. \mgd{For $\tau=1$, Theorem~\ref{thm:q-naglower}
and Lemma~\ref{lem:q-nag-tau1-achievability}
show that the optimal spectral gap of delayed
NAG is $\Theta(\kappa^{-2/3})$, whereas
Theorem~\ref{thm:q-hblower} establishes
an upper bound of $\cO(\kappa^{-1})$
on the optimal spectral gap of delayed HB.
Thus, delayed NAG achieves a better
asymptotic spectral gap than delayed HB,
although neither retains the accelerated
$\Theta(\kappa^{-1/2})$ scaling.} This suggests that larger values of $\nu$ may improve performance in the delayed setting, and that the classical endpoints $\nu=0$ and $\nu=\beta<1$ are insufficient when $\tau$ is large. In Theorem~\ref{thm:q-alldelay}, we show that allowing $\nu$ to grow with $\tau$ at an appropriate rate recovers an accelerated regime.
\end{remark}

\subsection[GMM achieves the accelerated rate with sqrt-kappa scaling when tau=1]{GMM achieves the accelerated rate with $\sqrt{\kappa}$ scaling when $\tau=1$}
For a one-step delay ($\tau=1$), the next proposition gives explicit GMM parameters whose worst-case spectral radius is exactly $1-1/(2\sqrt{\kappa})$. Theorem~\ref{thm:q-alldelay}~later gives a different parameter choice that achieves acceleration for every fixed positive delay, provided $\kappa$ is sufficiently large.

Throughout this and the following subsection we write $\delta_L:=\alpha L$ for the \mgd{high-curvature effective stepsize} 
to emphasize its dependence on $L$; this quantity was denoted as $\Delta$ in the lower-bound analysis above.
\begin{proposition}[an accelerated delayed GMM for $\tau=1$]\label{prop:q-gmmaccel}
\sa{For} $\tau=1$, let $q:=\kappa^{-1/2}$ and \sa{$\rho:=1-\tfrac q2$.} 
With the \sa{following} constant parameters,
\[
\alpha=\frac{1}{4L}\Big(1-\frac{q^2}{2}\Big),\qquad
\beta=1-q+\frac{q^2}{2},\qquad
\nu=\frac{2\big(1-q+\tfrac{q^2}{4}\big)}{1-\tfrac{q^2}{2}},
\]
the delayed GMM is Schur stable on every quadratic with spectrum in $[\mu,L]$, and its worst-case spectral radius \sa{can be computed \emph{exactly} as follows:}
\[
\sup_{\lambda\in[\mu,L]}\max\{|z|:P_{1,\alpha\lambda}(z)=0\}=\rho=1-\frac{1}{2\sqrt\kappa} = 1-\frac{q}{2}.
\]
\end{proposition}
\begin{proof}
Write $\delta_L=\alpha L=\tfrac14(1-\tfrac{q^2}{2})$ and $t=\lambda/L\in[q^2,1]$; \sa{hence,
$\delta=t\delta_L$ where we used the definition $\delta:=\alpha\lambda$ for $\lambda\in[\mu,L]$.} After some computations,
\wubtext{\begin{equation}\begingroup\normalcolor
\begin{aligned}
        P_{1,t\delta_L}(z)
        =
        z^3-\Bigl(2-q+\frac{q^2}{2}\Bigr)z^2   +
        \Bigl(1-q+\frac{q^2}{2}
        +t\Bigl(\frac34-\frac q2\Bigr)\Bigr)z
        -t\Bigl(\frac12-\frac q2+\frac{q^2}{8}\Bigr).
\end{aligned}
\endgroup %
\label{eq:q-t-expanded}\end{equation}}If $q=1$, then the interval $t\in[q^2,1]$ degenerates to $t=1$ and
$P_{1,\delta_L}(z)=(z-\tfrac12)^3$; \sa{thus,} the claim holds. \sa{Next,} assume
$0<q<1$. At the endpoints one has the exact factorizations
$P_{1,q^2\delta_L}(z)=(z-\rho)^2(z-\tfrac{q^2}{2})$ and
$P_{1,\delta_L}(z)=(z-\tfrac12)(z^2-2az+\rho^2)$ with $a=\tfrac34-\tfrac q2+\tfrac{q^2}{4}$ \sa{--note that for $t=1$, since
$\rho-a=\tfrac{1-q^2}{4}>0$, the quadratic factor in $P_{1,\delta_L}$ has complex roots of modulus exactly $\rho$,} and
$\tfrac12<\rho$. Hence, all the roots $z$ of  $P_{1,t\delta_L}(z)$ satisfy $|z|\leq \rho$ at the endpoints $t=q^2$ or $t=1$. For interior \sa{$t\in (q^2,1)$}, set $z=\rho y$ and apply the Schur-Cohn conditions for stability to $\widehat p_t(y)=\rho^{-3}P_{1,t\delta_L}(\rho y)$. More specifically, with the cubic normalized as
$y^3+a_2y^2+a_1y+a_0$, Lemma~\ref{lem:q-cubic} requires the positivity of
$\widehat p_t(1)$, $-\widehat p_t(-1)$, $1-a_0^2$, and the two margins
$1-a_0^2\pm(a_0a_2-a_1)$.  After some straightforward computations, we get
{\small
\begin{equation}
\label{eq:hat-p}
        \widehat p_t(1)
        =
        \frac{(1-q)(t-q^2)}{(2-q)^2},
\  
        -\widehat p_t(-1)
        =
        \frac{16(1-q)+7q^2-q^3+(5-3q)t}{(2-q)^2},
\ 
        1-a_0^2
        =
        \frac{(2-q-t)(2-q+t)}{(2-q)^2},
\end{equation}
\begin{equation}
\label{eq:q-gmmaccel-jury-plus}
        1-a_0^2+(a_0a_2-a_1)
        =
        \frac{(t-q^2)(1-t)}{(2-q)^2},
\quad
        1-a_0^2-(a_0a_2-a_1)
        =
        \frac{8(1-q)+3q^2-q^2t-t-t^2}{(2-q)^2}.\end{equation}}All denominators are positive \sa{in \eqref{eq:hat-p} and \eqref{eq:q-gmmaccel-jury-plus}. The %
numerators in \eqref{eq:hat-p} and the first one in \eqref{eq:q-gmmaccel-jury-plus} are all} strictly
positive for \(0<q<1\) and \(q^2<t<1\). \sa{On the other hand, for the second numerator in \eqref{eq:q-gmmaccel-jury-plus},} define \[ N(t):=8(1-q)+3q^2-q^2t-t-t^2 . \] Then \(N'(t)=-q^2-1-2t<0\), so \(N\) is decreasing in \(t\). Since \(t<1\), \[ N(t)>N(1) = 8(1-q)+3q^2-q^2-1-1 = 2(1-q)(3-q)>0 . \] Thus, all Schur--Cohn inequalities hold, %
\sa{which implies that} all roots of \(\widehat p_t\) lie in the open unit disk. Equivalently, all roots of \(P_{1,t\delta_L}\) satisfy \(|z|<\rho\) for \(q^2<t<1\), while at each endpoint at least one root satisfies \(|z|\mathrel{=}\rho\). This completes the proof.
\end{proof}

\mgb{Proposition~\ref{prop:q-gmmaccel} provides a tuning
with $1-\rho=q/2$, where $q=\kappa^{-1/2}$.
Thus, the normalized spectral gap $(1-\rho)/q$
is exactly $1/2$. The next result establishes
that this leading constant is asymptotically
optimal and characterizes the parameter behavior
necessary to attain it.} To state the result, \mgd{we  next introduce some notation}. Fix \mgd{$L>0$ and $\tau=1$}.
For each $q\in(0,1]$, consider GMM parameters
$\alpha_q>0$, $\beta_q\ge0$, and $\nu_q\ge0$,
which may depend on $q$, and define the
normalized stepsize
\[
\Delta_q:=\alpha_q L.
\]
For each normalized curvature
\[
t:=\frac{\lambda}{L}\in[q^2,1],
\]
define the corresponding characteristic polynomial
\beq\label{def-pqt-optimal-tuning}
P_{q,t}(z)
:=
z(z-1)(z-\beta_q)
+t\Delta_q\bigl((1+\nu_q)z-\nu_q\bigr).
\eeq
Let $\theta_q:=(\Delta_q,\beta_q,\nu_q)$
denote the parameter triple, and define its
worst-case spectral radius by
\[
\rho(q;\theta_q)
:=
\sup_{t\in[q^2,1]}
\max\bigl\{
|z|:P_{q,t}(z)=0
\bigr\}.
\]
\mgd{Here, $\theta_q$ is held fixed, and the
supremum is taken only over the normalized
curvature $t$. Finally, define the normalized
spectral gap
\[
G(q;\theta_q)
:=
\frac{1-\rho(q;\theta_q)}{q}.
\]}
\mgb{\begin{proposition}[Sharp asymptotic optimality for \(\tau=1\)]
\label{prop:tau-one-optimality} 
Let $L>0$ be fixed, \(\tau=1\) and let \(\theta_q=(\Delta_q,\beta_q,\nu_q)\) be any family of parameters with
\(\Delta_q>0\), \(\beta_q\geq0\), and \(\nu_q\geq0\), such that
\(P_{q,t}(\cdot)\) defined in \eqref{def-pqt-optimal-tuning} is Schur stable for every
\(t\in[q^2,1]\). Then
\[
\limsup_{q\downarrow0}G(q;\theta_q)
\leq
\frac12.
\]
Hence the tuning of Proposition~\ref{prop:q-gmmaccel}, for which
$
\rho(q;\theta_q)=1-\frac q2,
$
is asymptotically optimal. Moreover, any asymptotically optimal family satisfying
\[
G(q;\theta_q)\to\frac12
\qquad\text{as }q\downarrow0
\]
must obey
\[
\Delta_q=\frac14+o(1),
\qquad
\beta_q=1-q+o(q),
\qquad
\nu_q=2+o(1).
\]
\end{proposition}}

\begin{proof}
For each \(n\geq1\), define
\[
s_n
:=
\sup_{0<q<1/n}G(q;\theta_q).
\]
Then
\[
\gamma
:=
\limsup_{q\downarrow0}G(q;\theta_q)
=
\lim_{n\to\infty}s_n
\in[0,+\infty].
\]
For each \(n\), choose \(q_n\in(0,1/n)\) such that
$
G(q_n;\theta_{q_n})
\geq
\min\{s_n,n\}-\frac1n.
$
Such a choice is possible by the definition of the supremum. Since
\(0<q_n<1/n\), we have \(q_n\to0\). Moreover,
\[
G_n
:=
G(q_n;\theta_{q_n})
\longrightarrow\gamma,
\]
where the case $(\gamma=+\infty)$ is understood in the extended-real sense.
For notational convenience, define
\[
\rho_n:=\rho(q_n;\theta_{q_n}),
\qquad
\beta_n:=\beta_{q_n},
\qquad
\Delta_n:=\Delta_{q_n},
\qquad
\nu_n:=\nu_{q_n},
\quad
\mbox{and}
\quad
\omega_n:=\Delta_{q_n}\nu_{q_n}.
\]
By the definition of \(G_n\),
$
1-\rho_n=q_nG_n.
$
\mg{The proof proceeds in several steps.}

\mg{\emph{Step 1: Boundedness of parameters.}}
We first show that the parameter sequences
$
\{\beta_n\}$, $
\{\Delta_n\}$ and
$
\{\omega_n\}
$
are bounded. 
At the \mg{upper} endpoint \(t=\lambda/L=1\) \mg{(when $\lambda=L$)}, define
\[
\mgd{\overline{\mathcal{P}_n}(z)}
:=
P_{q_n,1}(z).
\]
Using \(\omega_n=\Delta_n\nu_n\), this polynomial can be written as
\[
\mgd{\overline{\mathcal{P}_n}(z)}
=
z^3-(1+\beta_n)z^2
+(\beta_n+\Delta_n+\omega_n)z-\omega_n.
\]
Let its roots be
$
\zeta_{1,n}$, $
\zeta_{2,n}$ and $
\zeta_{3,n}$. Because \(\rho_n\) is the worst-case spectral radius over
\(t\in[q_n^2,1]\), we have
\[
|\zeta_{j,n}|
\leq
\rho_n
<1,
\qquad
j=1,2,3.
\]
By Vieta's formulas,
$
\zeta_{1,n}+\zeta_{2,n}+\zeta_{3,n}
=
1+\beta_n.
$
Since \(\beta_n\geq0\),
\[
1+\beta_n
=
\left|
\zeta_{1,n}+\zeta_{2,n}+\zeta_{3,n}
\right|
\leq
\sum_{j=1}^3|\zeta_{j,n}|
\leq
3\rho_n
<3.
\]
Therefore,
$
0\leq\beta_n<2.
$
On the other hand, the product-of-roots identity gives
$
\zeta_{1,n}\zeta_{2,n}\zeta_{3,n}
=
\omega_n.
$
Since \(\omega_n=\Delta_n\nu_n\geq0\),
\[
0\leq\omega_n
=
\left|
\zeta_{1,n}\zeta_{2,n}\zeta_{3,n}
\right|
\leq
\rho_n^3
<1.
\]
Finally,
$
\mgd{\overline{\mathcal{P}_n}(1)}=\mgd{P_{q_n,1}(1)=\Delta_n}.
$
Factoring the polynomial in terms of its roots therefore gives
\[
\Delta_n
=
\left|
(1-\zeta_{1,n})
(1-\zeta_{2,n})
(1-\zeta_{3,n})
\right|.
\]
Hence
\[
0<\Delta_n
\leq
\prod_{j=1}^3(1+|\zeta_{j,n}|)
\leq
(1+\rho_n)^3
<8.
\]
We have thus proved the uniform bounds
$
0\leq\beta_n<2$,
$
0<\Delta_n<8$ and
$
0\leq\omega_n<1.
$
By the Bolzano--Weierstrass theorem applied to the bounded sequence
$
\bigl\{(\beta_n,\Delta_n,\omega_n)\bigr\}_{n\geq1}
\subset\mathbb{R}^3,
$
there exist a strictly increasing sequence of indices
$
n_1<n_2<\cdots
$
and limits
$
\bar\beta,\bar\Delta,\bar\omega\in\mathbb{R}
$
such that
\[
\beta_{n_k}\longrightarrow\bar\beta,
\qquad
\Delta_{n_k}\longrightarrow\bar\Delta,
\qquad
\omega_{n_k}\longrightarrow\bar\omega
\qquad\text{as }k\to\infty.
\]
Since \(q_n\to0\) and \(G_n\to\gamma\), the corresponding subsequences
also satisfy
$
q_{n_k}\longrightarrow0$ and
$
G_{n_k}\longrightarrow\gamma.
$

\emph{Step 2: Finer analysis of the limit $\bar\beta$.} \mg{We next analyze possible values of the limit $\bar\beta$ and its consequences on the $\gamma$ value.}
\mg{To do this, we consider
\(t=q_{n_k}^2\) and analyze the lower endpoint of the interval for $t$}. Define
\[
\mg{\underline{\mathcal{P}_k}(z)}
:=
P_{q_{n_k},q_{n_k}^2}(z).
\]
Then
\begin{align}
\mg{\underline{\mathcal{P}_k}(z)}
=
z(z-1)(z-\beta_{n_k})
+
q_{n_k}^2
\bigl(
\Delta_{n_k}
+
(\Delta_{n_k}+\omega_{n_k})(z-1)
\bigr).
\label{eq:tau-one-small-polynomial}
\end{align}
Because
$
q_{n_k}\to0$, 
$
\beta_{n_k}\to\bar\beta,
$
and the sequences
\(\{\Delta_{n_k}\}\) and \(\{\omega_{n_k}\}\) are bounded, we have
the coefficient-wise convergence
\[
\mg{\underline{\mathcal{P}_k}(z)}
\longrightarrow
\mg{\underline{\mathcal{P}_\infty}(z)}:= z(z-1)(z-\bar\beta).
\]
Since \(\bar\beta\) is a root of the limiting polynomial
\(z(z-1)(z-\bar\beta)\), continuity of polynomial roots implies that
there exist roots \(z_{n_k}\) of $\mg{\underline{\mathcal{P}_k}(z)}$ such that
$
z_{n_k}\longrightarrow\bar\beta.
$
Because every root of $\mg{\underline{\mathcal{P}_k}(z)}$ satisfies
\[
|z_{n_k}|\leq\rho_{n_k}<1,
\]
taking limits gives
$
|\bar\beta|\leq1.
$
Since \(\bar\beta\geq0\), it follows that
$
0\leq\bar\beta\leq1.
$
There are two cases: 

\emph{Case $(i)$:} Suppose first that
$
\bar\beta<1.
$
Then \(z=1\) is a simple root of the limiting polynomial
$
\mg{\underline{\mathcal{P}_\infty}(z)} = z(z-1)(z-\bar\beta).
$
By continuity of roots, there exists a root \(z_{n_k}\) of
\(\mg{\underline{\mathcal{P}_k}(z)}\) such that
$
z_{n_k}\longrightarrow1.
$
Define
\[
\delta_{n_k}:=1-z_{n_k}.
\]
Since \(\mg{\underline{\mathcal{P}_k}(z_{n_k})=0}\), the identity
\[
\mg{\underline{\mathcal{P}_k}(z)}
=
z(z-1)(z-\beta_{n_k})
+
q_{n_k}^2
\bigl[
\Delta_{n_k}
+
(\Delta_{n_k}+\omega_{n_k})(z-1)
\bigr]
\]
gives
\[
\Bigl[
z_{n_k}(z_{n_k}-\beta_{n_k})
+
q_{n_k}^2(\Delta_{n_k}+\omega_{n_k})
\Bigr]\delta_{n_k}
=
q_{n_k}^2\Delta_{n_k}.
\]
Moreover,
\[
z_{n_k}(z_{n_k}-\beta_{n_k})
+
q_{n_k}^2(\Delta_{n_k}+\omega_{n_k})
\longrightarrow
1-\bar\beta>0.
\]
\sa{Therefore, the bracketed factor is bounded away from zero for all
sufficiently large $k$. Since $\{\Delta_{n_k}\}$ is bounded, it follows that
\[
|\delta_{n_k}|=\cO(q_{n_k}^2).
\]
Hence,
\[
|z_{n_k}|=|1-\delta_{n_k}|
\ge 1-|\delta_{n_k}|
=1-\cO(q_{n_k}^2).
\]
Consequently,
\[
0<G_{n_k}
=\frac{1-\rho_{n_k}}{q_{n_k}}
\le \frac{1-|z_{n_k}|}{q_{n_k}}
=\cO(q_{n_k})\to0.
\]}
Since \(G_{n_k}\to\gamma\), this implies
$
\gamma=0.
$
Thus, if \(\gamma=0\), the desired upper bound is immediate.

\emph{Case $(ii)$:} \mg{It remains
to consider the case $\bar\beta=1$}. \mg{We will argue that in this case $\gamma \leq 1/2$.} %
\mgd{Since $\beta_{n_k}\to 1$}, the %
polynomials $\mg{\underline{\mathcal{P}_k}(z)}$ satisfy
\[
\mg{\underline{\mathcal{P}_k}(z)}
\longrightarrow
\mg{\underline{\mathcal{P}_\infty}(z)} = z(z-1)^2.
\]
We may consequently label their roots as
$
r_{0,n_k}$, $r_{1,n_k}$ and $r_{2,n_k}$ so that
\[
r_{0,n_k}\longrightarrow0,
\qquad
r_{1,n_k}\longrightarrow1,
\qquad
r_{2,n_k}\longrightarrow1.
\]
Evaluating \(\mg{\underline{\mathcal{P}_k}(z)}\) at \(z=1\) gives
\[
q_{n_k}^2\Delta_{n_k}
=
\mg{\underline{\mathcal{P}_k}(1)}
=
(1-r_{0,n_k})
(1-r_{1,n_k})
(1-r_{2,n_k}).
\]
Taking absolute values and using
$
|r_{j,n_k}|\leq\rho_{n_k},
$
we obtain
\begin{align*}
q_{n_k}^2\Delta_{n_k}
=
|1-r_{0,n_k}|
|1-r_{1,n_k}|
|1-r_{2,n_k}|
\geq
|1-r_{0,n_k}|(1-\rho_{n_k})^2
=
|1-r_{0,n_k}|q_{n_k}^2G_{n_k}^2.
\end{align*}
After dividing by \(q_{n_k}^2\), we obtain
\beq
\Delta_{n_k}
\geq
|1-r_{0,n_k}|G_{n_k}^2.
\label{ineq-Deltank-lower-bound}
\eeq
Since \(\{\Delta_{n_k}\}\) is bounded and
$
|1-r_{0,n_k}|\longrightarrow1,
$
the sequence \(\{G_{n_k}\}\) is bounded. Hence
$
\gamma<+\infty.
$
We may now pass to the limit in the preceding inequality \mg{\eqref{ineq-Deltank-lower-bound}} to obtain
\beq
\bar\Delta\geq\gamma^2.
\label{ineq-gamma-lower-bound}
\eeq
\mgd{We next return to the high-curvature endpoint $t=1$ along the same subsequence. Write}
\[
\mgd{\overline{\mathcal{P}_{n_k}}(z)}=
z^3+a_{2,n_k}z^2+a_{1,n_k}z+a_{0,n_k},
\]
where
\[
a_{2,n_k}=-(1+\beta_{n_k}),
\qquad
a_{1,n_k}
=
\beta_{n_k}+\Delta_{n_k}+\omega_{n_k},
\qquad
a_{0,n_k}=-\omega_{n_k}.
\]
Because \(\mgd{\overline{\mathcal{P}_{n_k}}}\) is Schur stable, the fourth cubic
Schur--Cohn condition \mg{from Lemma \ref{lem:q-cubic}} gives
\[
|a_{0,n_k}a_{2,n_k}-a_{1,n_k}|
<
1-a_{0,n_k}^2.
\]
With a direct calculation, this is equivalent to%
$
\left|
\beta_{n_k}(\omega_{n_k}-1)-\Delta_{n_k}
\right|
<
1-\omega_{n_k}^2.
$
Passing to the limit as \(k\to\infty\), and using
\(\bar\beta=1\), gives
\beq
|\bar\omega-1-\bar\Delta|
\leq
1-\bar\omega^2.
\label{ineq-jury-cond-four}
\eeq
\mg{This bound and \eqref{ineq-gamma-lower-bound}} 
imply
\[
0\leq\bar\omega\leq1,
\qquad
\bar\Delta\geq0.
\]
Thus,
$
\bar\omega-1-\bar\Delta\leq0,
$
and hence
$
|\bar\omega-1-\bar\Delta|
=
1+\bar\Delta-\bar\omega.
$
It follows \mg{from \eqref{ineq-jury-cond-four}} that
$
1+\bar\Delta-\bar\omega
\leq
1-\bar\omega^2,
$
or equivalently,
\[
\bar\Delta
\leq
\bar\omega-\bar\omega^2
=
\bar\omega(1-\bar\omega)
\leq
\frac14.
\]
Combining this with the other bound obtained in \eqref{ineq-gamma-lower-bound} gives
$$
\gamma^2
\leq
\bar\Delta
\leq
\frac14 \quad
\mbox{and therefore} \quad
\gamma = \limsup_{q\downarrow0}G(q;\theta_q)\leq\frac12.
$$

\emph{Step 3: Achieving the optimal $\gamma=1/2$.} Proposition~\ref{prop:q-gmmaccel} provides a particular family of
tunings for which
$
\rho(q;\theta_q)=1-\frac q2.
$
For this family,
$
G(q;\theta_q)=\frac12
$
for every \(q\), proving that the constant \(1/2\) is sharp.

\mg{We next identify the parameter limits for any family whose normalized spectral gap converges to $1/2$}. Suppose that
$
G(q;\theta_q)\longrightarrow\frac12$
as $q\downarrow0,
$
and define
\[
\omega_q:=\Delta_q\nu_q.
\]
Let \(\{\widehat q_m\}_{m\geq1}\) be an arbitrary sequence satisfying
$
\widehat q_m\longrightarrow0.
$
The large-curvature endpoint estimates established above show that
the sequence
$
\left\{
\bigl(
\beta_{\widehat q_m},
\Delta_{\widehat q_m},
\omega_{\widehat q_m}
\bigr)
\right\}_{m\geq1}
$
is bounded in \(\mathbb{R}^3\). Therefore, there exists a strictly
increasing sequence of indices
$
m_1<m_2<\cdots
$
and limits
$
\widetilde\beta,\widetilde\Delta,\widetilde\omega
$
such that
\[
\beta_{\widehat q_{m_j}}
\longrightarrow
\widetilde\beta,
\qquad
\Delta_{\widehat q_{m_j}}
\longrightarrow
\widetilde\Delta,
\qquad
\omega_{\widehat q_{m_j}}
\longrightarrow
\widetilde\omega.
\]
Since
$
G(\widehat q_{m_j};\theta_{\widehat q_{m_j}})
\longrightarrow\frac12,
$
the endpoint argument proved above applies to this subsequence with
\(\gamma=1/2\). It gives
\[
\widetilde\beta=1,
\] and
\[
\mg{\gamma^2}=\frac14
\leq
\widetilde\Delta
\leq
\widetilde\omega(1-\widetilde\omega)
\leq
\frac14.
\] 
Therefore,
\[
\widetilde\Delta=\frac14,
\qquad
\widetilde\omega=\frac12,
\qquad
\widetilde\beta=1.
\]
\sa{Since the original sequence $\{\widehat q_m\}$ was arbitrary, the preceding
argument shows that every cluster point of
\[
(\beta_q,\Delta_q,\omega_q)
\]
as $q\downarrow0$ equals $(1,\tfrac14,\tfrac12)$. Together with the
boundedness established in Step~1, this implies
\[
\beta_q\to1,\qquad
\Delta_q\to\frac14,\qquad
\omega_q\to\frac12.
\]}
Since
$
\omega_q=\Delta_q\nu_q
$
and
$
\Delta_q\longrightarrow\frac14>0,
$
we obtain
$
\nu_q
=
\frac{\omega_q}{\Delta_q}
\longrightarrow2.
$

\mgd{\emph{Step 4: First-order asymptotic behavior
of $\beta_q$.}} It remains to establish that $\beta_q=1-q+o(q)$. At the lower endpoint \mg{associated with $t=q^2$},
\[
P_{q,q^2}(z)=z(z-1)(z-\beta) + q^2\Delta_q ((1+\nu_q)z-\nu_q)
\longrightarrow
z(z-1)^2,
\]
coefficientwise as \(q\downarrow0\). Hence, for all sufficiently
small \(q\), its roots may be labeled as
$
r_{0,q}$, $r_{1,q}$ and $r_{2,q}$
so that 
\[
r_{0,q}\longrightarrow0,
\qquad
r_{1,q}\longrightarrow1,
\qquad
r_{2,q}\longrightarrow1.
\]
For \(i=1,2\), define
\[
u_{i,q}
:=
\frac{1-r_{i,q}}{q},
\]
so that
$
r_{i,q}=1-qu_{i,q}.
$
Evaluating \(P_{q,q^2}\) at \(z=1\) gives
\[
q^2\Delta_q
=
(1-r_{0,q})(1-r_{1,q})(1-r_{2,q}).
\]
Since
$
1-r_{i,q}=qu_{i,q}$
for $i=1,2,
$
we obtain
$
\Delta_q
=
(1-r_{0,q})u_{1,q}u_{2,q}.
$
Using
$
\Delta_q\longrightarrow\frac14$ and
$
r_{0,q}\longrightarrow0,
$
we conclude that
\beq
u_{1,q}u_{2,q}
\longrightarrow
\frac14.
\label{eq-product-roots}
\eeq
Moreover, \mg{by assumption},
$
\rho(q;\theta_q)
=
1-qG(q;\theta_q)
=
1-\frac q2+o(q).
$
Therefore, for \(i=1,2\),
\[
|1-qu_{i,q}|
=
|r_{i,q}|
\leq
\rho(q;\theta_q)
=
1-\frac q2+o(q).
\]
Squaring gives
\[
1-2q\Ree(u_{i,q})+q^2|u_{i,q}|^2
\leq
1-q+o(q).
\]
After subtracting \(1\) and dividing by \(q>0\), we obtain
$
2\Ree(u_{i,q})-q|u_{i,q}|^2
\geq
1+o(1).
$
In particular,
\beq
\Ree(u_{i,q})
\geq
\frac12+o(1),
\qquad i=1,2.
\label{ineq-real-part-roots}
\eeq

 Because \(P_{q,q^2}\) has real coefficients and \(r_{0,q}\) is the
 unique root near zero, the roots \(r_{1,q}\) and \(r_{2,q}\) are
 either both real or form a complex-conjugate pair.
\mg{We now show that in either case
\[
u_{1,q}\longrightarrow \frac12 
\quad
\mbox{and} 
\quad
u_{2,q}\longrightarrow \frac12
\qquad\text{as }q\downarrow0.
\]
Suppose otherwise. Then there exist \(\varepsilon>0\) and a sequence
\(q_n\downarrow0\) such that
\[
\max_{i=1,2}\left|u_{i,q_n}-\frac12\right|\ge\varepsilon
\qquad\text{for all }n.
\]
Passing to a subsequence, we may assume that \(r_{1,q_n},r_{2,q_n}\)
are either both real for every \(n\), or form a complex-conjugate pair
for every \(n\).
In the complex-conjugate case, from \eqref{eq-product-roots}, 
\[
u_{2,q_n}=\overline{u_{1,q_n}},
\qquad
|u_{1,q_n}|^2
=
u_{1,q_n}u_{2,q_n}
\longrightarrow \frac14.
\]
Hence \(|u_{1,q_n}|\to 1/2\). Since
$
\operatorname{Re}(u_{1,q_n})\ge \frac12+o(1)
$ \mg{by \eqref{ineq-real-part-roots}}
and \(\operatorname{Re}(u_{1,q_n})\le |u_{1,q_n}|\), we obtain
\(\operatorname{Re}(u_{1,q_n})\to1/2\), and therefore
$
\operatorname{Im}(u_{1,q_n})\longrightarrow0.
$
Thus \(u_{1,q_n},u_{2,q_n}\to1/2\). 
In the real-root case, from \eqref{ineq-real-part-roots} and \eqref{eq-product-roots}, we have 
\[
u_{i,q_n}\ge \frac12+o(1),
\qquad i=1,2,
\qquad\text{and}\qquad
u_{1,q_n}u_{2,q_n}\longrightarrow\frac14.
\]
\mgd{The lower bounds make both factors positive and bounded away from zero. Their product limit then bounds each factor above by $1/2+o(1)$, so $u_{1,q_n},u_{2,q_n}\to1/2$. Both cases contradict the choice of the subsequence. Therefore,}
} 

\begin{equation}
\mg{u_{1,q}, u_{2,q}\to \frac{1}{2}} \quad \mbox{and} \quad u_{1,q}+u_{2,q}\longrightarrow1.
\label{eq-root-sum-limit}
\end{equation}

By Vieta's formula for the sum of the roots \mg{of $P_{q,q^2}$},
$
r_{0,q}+r_{1,q}+r_{2,q}
=
1+\beta_q.
$
Substituting
$
r_{1,q}=1-qu_{1,q},
$ and
$r_{2,q}=1-qu_{2,q},
$
gives
\[
\frac{1-\beta_q}{q}
=
u_{1,q}+u_{2,q}
-
\frac{r_{0,q}}{q}.
\]
On the other hand, the product-of-roots identity \mg{applied to the polynomial $P_{q,q^2}$} gives
$
r_{0,q}r_{1,q}r_{2,q}
=
q^2\omega_q.
$
Since
$
r_{1,q}r_{2,q}\longrightarrow1
$
and \(\omega_q\) is bounded, it follows that
$
r_{0,q}
=
\frac{q^2\omega_q}{r_{1,q}r_{2,q}}
=
O(q^2).
$
Therefore,
$
\frac{r_{0,q}}{q}\longrightarrow0.
$
Consequently, from \eqref{eq-root-sum-limit}, we obtain
$
\frac{1-\beta_q}{q}
\longrightarrow1,
$
which is equivalent to
$
\beta_q=1-q+o(q).
$
Finally, since
$
\Delta_q=\alpha_qL,
$
we conclude that
\[
\alpha_qL=\frac14+o(1),
\qquad
\beta_q=1-q+o(q),
\qquad
\nu_q=2+o(1),
\]
which completes the proof.
\end{proof}

\subsection[Acceleration with GMM for arbitrary delays tau >= 1]{Acceleration with GMM for arbitrary delays $\tau\geq 1$}
Recall that \begin{equation}\label{eq:q-Ptaudelta}
        P_{\tau,\delta}(z)
        =z^\tau(z-1)(z-\beta)+\delta[(1+\nu)z-\nu].
\end{equation}
The %
worst-case spectral radius \sa{for the quadratic setting} is
\begin{equation}\label{eq:q-rhoalg}
        \rho_{\tau,\kappa}(\alpha,\beta,\nu)
        :=\sup_{t\in[q^2,1]}\max\{|z|:P_{\tau,t\delta_L}(z)=0\}.
\end{equation}
Acceleration means that, for fixed \sa{\(\tau\in\mathbb{Z}_+\)},
\[
        \rho_{\tau,\kappa}\le 1-c(\tau)\kappa^{-1/2}
\]
\sa{holds for all sufficiently large $\kappa>0$} with \(c(\tau)>0\) independent of \(\kappa\). \mgd{At} $\delta=0$ and \mgd{for} $\beta$ near %
\sa{$1$}, $P_{\tau,\delta}$ has $\tau$ roots near %
\sa{$0$} and two roots near $1$; \sa{therefore,} the worst-case spectral radius is governed by the roots near $1$. As before, we refer to the roots of $P_{\tau,\delta}(z)$ that stay close to $1$ \wubtext{as} the \emph{slow modes}, and those near zero \wubtext{as} the \emph{fast modes}, or \emph{non-binding roots}.
\subsubsection{Slow-mode asymptotics}
The following result shows that the slow mode admits an expansion as a function of \sa{$q:=\kappa^{-1/2}$,} where the delay \sa{$\tau$} does not affect the first-order terms, but %
\sa{it affects only the} higher order terms.

\begin{lemma}[Slow-mode asymptotics]\label{lem:q-slow}
\sa{Fix \(\tau\ge1\), $L$, \(\delta_L>0\), and \(\nu\ge0\), and let $\{\mu_n\}$ \wubtext{be a sequence} such that $\mu_n\downarrow 0$. For $\kappa_n:=L/\mu_n$, let $q_n:=\kappa_n^{-1/2}$, and we set}
\[
        a:=\sqrt{\delta_L},
        \qquad
        \beta_n:=1-2a q_n.
\]
\sa{Let $t_n=q_n^2\ell_n$ for some $\ell_n$ such that $1\le \ell_n\le \mgd{\overline{B}}$
for some fixed \(\mgd{\overline{B}}<\infty\).} Consider the characteristic polynomial of GMM with $\delta=\wubtext{t_n}\delta_L$ and $\beta_n=1-2\sqrt{\delta_L}q_n$ for a fixed $\nu$, i.e.
\begin{equation}\label{eq:Pn}
        P_n(z)
        :=P_{\tau, t_n \delta_L}(z)=z^\tau(z-1)(z-\beta_n)
        +t_n\delta_L[(1+\nu)z-\nu].
\end{equation}
Then, for all sufficiently large \(n\), \(P_n\) has exactly two roots
\(z_n^\pm\) converging to \(1\). \sa{Indeed, for $y_n^\pm$ such that}
        $z_n^\pm=1+q_n y_n^\pm$,
one has %
\begin{equation}\label{eq:moving-root-distance}
        \dist\!\left(
            y_n^\pm,
            \left\{
            -a+ia\sqrt{\ell_n-1},\;
            -a-ia\sqrt{\ell_n-1}
            \right\}
        \right)
        =\cO_{\overline{B}}(q_n^{1/2}),
\end{equation}
where the $\cO(\cdot)$ constants \sa{are} also allowed to depend on the fixed parameters $\tau, \delta_L$ and $\nu$. Consequently,
\begin{equation}\label{eq:slow-modulus-replacement}
        |z_n^\pm|
        =1-aq_n+\cO_{\overline{B}}(q_n^{3/2})
        =1-\sqrt{\delta_L}\,q_n+o(q_n).
\end{equation}
The remaining \(\tau\) roots of \(P_n\) converge to \(0\) and are non-binding for
the spectral radius.
\end{lemma}

\begin{proof}
Set \(z=1+q_nw\). \sa{Then, we have}
\[
        z-1=q_nw,
        \qquad
        z-\beta_n=q_n(w+2a).
\]
\sa{Since} \(t_n=q_n^2\ell_n\) and \(\delta_L=a^2\), \sa{we also have}
\[
        t_n\delta_L[(1+\nu)z-\nu]
        =q_n^2\ell_n a^2[1+(1+\nu)q_nw].
\]
After division by \(q_n^2\), the rescaled polynomial is
\begin{equation}\label{eq:Fn}
        F_n(w)
        :=q_n^{-2}P_n(1+q_nw)
        =(1+q_nw)^\tau w(w+2a)
        +\ell_n a^2[1+(1+\nu)q_nw].
\end{equation}
Compare \(F_n\) with the moving quadratic
\begin{equation}\label{eq:Qn}
        Q_n(w)
        :=w(w+2a)+\ell_n a^2
        =(w+a)^2+a^2(\ell_n-1).
\end{equation}
The roots of \(Q_n\) are
\begin{equation}\label{eq:rn}
        r_n^\pm=-a\pm ia\sqrt{\ell_n-1}.
\end{equation}
Since \(1\le \ell_n\le \overline{B}\), the roots \(r_n^\pm\) remain in a fixed bounded
subset of the \(w\)-plane; \sa{indeed,} %
$|r_n^\pm| = a \sqrt{\ell_n} \leq a \sqrt{\overline{B}}$. Choose \(R=3a \sqrt{\overline{B}}\) so that \(|r_n^\pm|<R/2\) for every \(n\).  On the
circle \sa{\(\{w:|w|=R\}\)}, the quadratics \(Q_n\) are bounded away from zero uniformly
in \(n\), \sa{i.e.,} \(|Q_n(w)|=\sa{|w-r_n^+||w-r_n^-|}\ge m_R >0 \) for \(m_R:=(R-a\sqrt{\overline{B}})^2\). \sa{Moreover,} for \(|w|\le R\), uniformly in \(w\), we have
\begin{align*}
        F_n(w)-Q_n(w)
        =\bigl((1+q_nw)^\tau-1\bigr)w(w+2a)
          +\ell_n a^2(1+\nu)q_nw=\cO_{\overline{B}}(q_n),
\end{align*}
in the last equality we use the fact that  $R=3a\sqrt{\overline{B}}$. Note that here the constants hidden by $\cO(\cdot)$ are allowed to depend on the fixed parameters $\tau, \delta_L$ and $\nu$. Therefore, \(|F_n-Q_n|<m_R\) on \(|w|=R\) for all large
\(n\) \sa{as $q_n\searrow 0$ as $n\to\infty$.} By Rouch\'e's theorem, \(F_n\) has exactly two roots in \sa{\(\{w:|w|<R\}\)}.
These two roots correspond exactly to the two roots of \(P_n\) in the disk
\(|z-1|<q_nR\); in particular, they converge to \(1\) in the original
\(z\)-variable.  Denote them by \(z_n^\pm=1+q_ny_n^\pm\).

For each such bounded root \(y_n^\pm\), we have \(F_n(y_n^\pm)=0\).  Therefore,
using the uniform estimate above, \sa{we get}
\[
        |Q_n(y_n^\pm)|
        =|Q_n(y_n^\pm)-F_n(y_n^\pm)|
        =\cO_{\overline{B}}(q_n).
\]
Since $Q_n(w)=(w-r_n^+)(w-r_n^-)$, we get
\[
         |Q_n(y_n^\pm)|= |y_n^\pm-r_n^+|\,|y_n^\pm-r_n^-|=\cO_{\overline{B}}(q_n).
\]
Moreover, since
 $\dist\bigl(y_n^\pm,\{r_n^+,r_n^-\}\bigr)^2 \leq |y_n^\pm - r_n^+| |y_n^\pm - r_n^-|$, we have
\[
        \dist\bigl(y_n^\pm,\{r_n^+,r_n^-\}\bigr)=\cO_{\overline{B}}(q_n^{1/2}),
\]
which proves \eqref{eq:moving-root-distance}.  %

Note that both roots \(r_n^\pm\) have real part \(-a\); hence,
\[
        \Ree y_n^\pm=-a+\cO_{\overline{B}}(q_n^{1/2}),
        \qquad
        y_n^\pm=\cO_{\overline{B}}(1).
\]
Using \(z_n^\pm=1+q_ny_n^\pm\), we obtain
\begin{align*}
        |z_n^\pm|^2 =|1+q_ny_n^\pm|^2 =1+2q_n\Ree y_n^\pm+q_n^2|y_n^\pm|^2 =1-2aq_n+\cO_{\overline{B}}(q_n^{3/2}).
\end{align*}
Taking square roots and \sa{using $\sqrt{1+u}=1+\frac12u+O(u^2)$ as $u\to 0$} gives
\[
        |z_n^\pm|=1-aq_n+\cO_{\overline{B}}(q_n^{3/2}),
\]
which is \eqref{eq:slow-modulus-replacement}.

It remains to identify the non-binding roots.  Since \(q_n\to0\),
\(\beta_n\to1\), and \(t_n=q_n^2\ell_n\to0\), the coefficientwise limit of
\(P_n\) is
$
        z^\tau(z-1)^2.
$
By continuity of roots as a multiset, exactly \(\tau\) roots of \(P_n\)
converge to \(0\); thus, \wubtext{they} are non-binding for the spectral radius, while the remaining two converge to \(1\).  This accounts for all \(\tau+2\) roots and completes the proof.
\end{proof}

\sa{If we choose
$
        \delta_L={\widetilde c}/{(1+\tau)^2},
$
where \(\widetilde c\mathrel{\wubtext{>}}\wubtext{0}\) \wubtext{is fixed} as \(q=1/\sqrt{\kappa}\to0\), then} the previous lemma (Lemma \ref{lem:q-slow}) suggests that at $t=q^2$ (for $\ell_n = 1, \forall n$) the near 1-roots (slow modes) %
demonstrate an \emph{accelerated} behavior, i.e., we have
\[
        \rho(P_{\tau, q^2 \delta_L})
       =
        1-\sqrt{\delta_L}\,q + o (q) 
        =
        1-\wubtext{\Theta}\!\left(\frac{1}{(1+\tau)\sqrt{\kappa}}\right) \quad \mbox{as} \quad q\to 0 \quad \mbox{\wubtext{for fixed $\tau$}}.
\]
This is consistent with the accelerated GMM construction for \(\tau=1\) in Proposition \ref{prop:q-gmmaccel}. There, however, the choice is \(\nu= 2+\cO(q)\geq 1\) for $\tau=1$, which is larger
than the Nesterov-type choice \(\nu=\beta<1\). This suggests that larger
values of \(\nu\) \wubtext{can be beneficial}, %
at least for $\tau=1$. Motivated by this observation, we consider the \sa{following choice of parameters:}
\begin{equation}\label{eq:q-ansatz}
        \alpha=\frac{\delta_L}{L},
        \qquad
        \beta=1-2\sqrt{\delta_L}\,q,
        \qquad
        \nu=\nu_\tau,
\end{equation}
where \(\delta_L\) is independent of \(\kappa\) and $\wubtext{\nu}_\tau$ is chosen as a function of $\tau$. %
\wubtext{The closed-form choice in Theorem~\ref{thm:q-alldelay} is $\nu_\tau=2\tau+\tfrac12$}, \mg{where we show that} under this ansatz, the slow modes yield a rate gap of order \(\sqrt{\delta_L}\,q\) \sa{for fixed $\tau\geq 1$.} The remaining task is
to stabilize the fast and intermediate modes introduced by the delay. 

\mg{The choices of \(\alpha\) and \(\beta\) in \eqref{eq:q-ansatz} are also motivated by \sa{the parameter choice for the} classical Nesterov acceleration. In the delay-free case \(\tau=0\), NAG with stepsize \(\alpha=1/L\) has the standard choice of parameters for strongly convex functions
$
\beta=\frac{1-q}{1+q}
=
1-2q+\mathcal{O}(q^2),
$
see, e.g., \cite[(2.2.22)]{NesterovBook}. In addition, for a stepsize
$$
\alpha=\frac{\delta_L}{L},
\qquad
0<\delta_L\le1,
$$
the corresponding optimal momentum parameter (that leads to the fastest linear rate) for delay-free NAG on quadratic objectives is
known to be $$
\beta
=
\frac{1-\sqrt{\alpha\mu}}{1+\sqrt{\alpha\mu}}
=
\frac{1-\sqrt{\delta_L}\,q}{1+\sqrt{\delta_L}\,q}
=
1-2\sqrt{\delta_L}\,q+\mathcal{O}(q^2),
$$
\cite[Lemma 2.1]{aybat2019universally}, \cite{gurbuzbalaban2023robustly}. Thus, the choices of \(\alpha\) and \(\beta\) in \eqref{eq:q-ansatz} agree with the delay-free NAG scaling to leading order in \(q\). The extrapolation parameter \(\nu_\tau\) is then chosen separately to stabilize the effects of the delay.
}
\subsubsection{Asymptotics of spectral radius}\label{sec:quad-gate}
Letting \(q\downarrow0\) sends \(\beta\to1\).  The limiting characteristic polynomial is
\begin{equation}\label{eq:q-Pinf}
        \;
        P_\infty(z;\delta,\nu)
        :=z^\tau(z-1)^2+\delta[(1+\nu)z-\nu].
        \;
\end{equation}
\mgd{For each fixed gain $\delta>0$, the roots
of the characteristic polynomial converge
to those of $P_\infty$ as $q\downarrow 0$, counted
with multiplicity.
However, establishing a uniform convergence-rate
bound over the \emph{gain interval}
$[\delta_\mu,\delta_L]$, whose lower
endpoint $\delta_\mu = \delta_L/\kappa$ approaches zero as $\kappa\to\infty$ (when $\delta_L$ is fixed),
requires additional analysis.} \mgd{The regime of vanishing gain requires
special care because, at $\delta=0$,
the limiting polynomial becomes
$P_\infty(z;0,\nu)=z^\tau(z-1)^2$,
which has a double root at $z=1$
on the unit circle. Consequently,
pointwise root continuity for fixed
positive gains does not establish
the uniform accelerated convergence-rate
bound required over the entire gain interval.} The next result shows that if \(P_\infty\) is \wubtext{strictly} Schur stable for every \(0<\delta\le\delta_L\), then the spectral radius of GMM exhibits accelerated behavior: specifically, it is bounded above by \(1 - c/\sqrt{\kappa}\) for some constant \(c>0\), provided
\sa{\(\kappa = 1/q^2\)} is sufficiently large and the parameters $\alpha$ and $\beta$ are chosen according to \eqref{eq:q-ansatz}.\looseness=-1
\begin{theorem}[Asymptotics of spectral radius]\label{thm:q-gate}
Fix \sa{\(\tau\geq 1\)}.  Suppose \(\delta_L>0\) and \(\nu\ge0\) \sa{ensure that}
\begin{equation}\label{eq:q-Pinfstable}
        P_\infty(\cdot;\delta,\nu)
        \quad\hbox{is strictly Schur stable for every}\quad
        0<\delta\le\delta_L.
\end{equation}
Choose \(\alpha=\delta_L/L\), \(\beta=1-2\sqrt{\delta_L}\,q\), and %
\sa{\(\nu\) as given}. \mgd{Then, we \sa{necessarily have} $\nu>\tau$,
and for every \(\mgd{0<c}<\sqrt{\delta_L}\),
there \sa{exists}
\(q_0=q_0(\tau,\delta_L,\nu,c)>0\) such that}%
\begin{equation}\label{eq:q-gaterate}
        \rho_{\tau,\kappa}(\alpha,\beta,\nu)\le1-cq,
\end{equation}
whenever \(0<q\le q_0\), where $\rho_{\tau,\kappa}(\alpha,\beta,\nu)$ is \sa{defined} in \eqref{eq:q-rhoalg}.
\end{theorem}

\begin{proof}
We first show that the hypothesis in \eqref{eq:q-Pinfstable} %
\sa{implies that} \(\nu>\tau\). The only point requiring justification is the expansion of the
roots that bifurcate from the double root \(z=1\) of \(P_\infty(\cdot;0,\nu)\).
Let $\varepsilon=\sqrt{\delta}$ and $z=1+\varepsilon w$. Then
\begin{equation}\label{eq:G-def}
        P_\infty(1+\varepsilon w;\varepsilon^2,\nu)
        =\varepsilon^2 G(w,\varepsilon),
\end{equation}
where
\begin{equation}\label{eq:G-function}
        G(w,\varepsilon)
        :=(1+\varepsilon w)^\tau w^2+1+(1+\nu)\varepsilon w .
\end{equation}
The function \(G\) is analytic in a neighborhood of each point \((\pm i,0)\),
and
\[
        G(w,0)=w^2+1 .
\]
Since \sa{$G(\pm i,0)=0$ and} \(\partial_wG(\pm i,0)=\pm2i\ne0\), the analytic implicit function
theorem \cite{krantz2002implicit} gives two analytic branches
\begin{equation}\label{eq:w-branches}
        w_\pm(\varepsilon)
        =\pm i+a_\pm\varepsilon+
        b_\pm\varepsilon^2+
        c_\pm\varepsilon^3+
        \cO(\varepsilon^4),
\end{equation}
\mgd{which satisfy} \(G(w_\pm(\varepsilon),\varepsilon)=0\); thus, the corresponding roots of \(P_\infty\) admit convergent expansions\footnote{The existence of such expansions also follows from the classical Puiseux expansion theorem for roots of analytic polynomial families;
see Kato~\cite[Chapter II, \S1.7]{Kato1995}; here, we gave a direct argument to keep the proof self-contained.} in powers of
\(\varepsilon=\sqrt{\delta}\). Substituting \eqref{eq:w-branches} into \(G(w_\pm(\varepsilon),\varepsilon)=0\),
and %
\sa{using the fact that} the coefficient of \(\varepsilon\) is $0$ \sa{leads to}
\[
        2(\pm i)a_\pm+\tau(\pm i)^3+(1+\nu)(\pm i)=0 .
\]
Since \((\pm i)^2=-1\), this yields
\begin{equation}\label{eq:a-pm}
        a_\pm=\frac{\tau-1-\nu}{2} .
\end{equation}
Consequently, \sa{$z_\pm=1+\epsilon\left(
\pm i+\frac{\tau-1-\nu}{2}\epsilon+O(\epsilon^2)
\right)=1\pm i\epsilon
+\frac{\tau-1-\nu}{2}\epsilon^2
+O(\epsilon^3)$; hence,}
\begin{equation}\label{eq:z-pm-first}
        z_\pm(\delta)
        =1
        +\frac{\tau-1-\nu}{2}\,\delta\sa{\pm i\sqrt{\delta}}
        +\cO(\delta^{3/2}),
\end{equation}
\sa{which implies that}
\begin{equation}\label{eq:z-pm-mod-first}
        |z_\pm(\delta)|^2
        =1+\delta(\tau-\nu)+\cO(\delta^{3/2}).
\end{equation}
Thus, \(\nu<\tau\) is impossible under \eqref{eq:q-Pinfstable}.  The borderline
case \(\nu=\tau\) is also impossible.  \hamzatext{We can see this by} continuing the same analytic
expansion two orders further, which gives
\begin{equation}\label{eq:z-pm-border}
        |z_\pm(\delta)|^2
        =1+\frac{\tau(\tau+1)(2\tau+1)}{6}\,\delta^2
        +\cO(\delta^{5/2})
\end{equation}
\sa{when \(\nu=\tau\)}. For \(\tau>0\), the coefficient of \(\delta^2\) in \eqref{eq:z-pm-border} is
positive, contradicting strict Schur stability for all sufficiently small
positive \(\delta\). If \(\tau=0\) and \(\nu=0\), then
\(P_\infty(z;\delta,0)=z^2+(\delta-2)z+1\), whose two roots have product one;
for \(0<\delta<4\), they are non-real conjugates; \sa{hence, they} lie on the unit circle \sa{-- this} again contradicts strict Schur stability. \sa{Thus, we can conclude that}
\begin{equation}\label{eq:nu-gt-tau}
        \nu>\tau .
\end{equation}

We now prove the asserted uniform bound. 
\sa{For the sake of contradiction, suppose} that there
exist \(0<c<\sqrt{\delta_L}\), a sequence \(q_n\downarrow0\), numbers
\(t_n\in[q_n^2,1]\), and roots \(z_n\) of \(P_{\tau,t_n\delta_L}\) such that
\begin{equation}\label{eq:q-contra}
        |z_n|>1-cq_n,
\end{equation}
with $t_n = q_n^2 \ell_n$ \sa{for some $1\le \ell_n\le q_n^{-2}$}, where
$$ P_{\tau, t_n\delta_L}(z) = z^{\tau}(z-1)(z-1+2\sqrt{\delta_L}q_n) + t_n \delta_L ((1+\nu)z-\nu) $$
with our choice of $\alpha$ and $\beta$. To simplify the notation, we denote these polynomials by \(P_{t_n,q_n}\) in the rest of the proof. Passing to a subsequence, we may assume that
\[
        t_n\to t_0\in[0,1],
        \qquad
        z_n\to z_0,
\]
\sa{which is possible since the roots are uniformly bounded} %
as the polynomials \(P_{t_n,q_n}\) are monic of
fixed degree and have uniformly bounded coefficients.  From
\eqref{eq:q-contra}, we have \(|z_0|\ge1\).

\sa{Since $q_n\to0$ and $t_n\to t_0\in[0,1]$, $P_{t_n,q_n}\longrightarrow P_\infty(\cdot;t_0\delta_L,\nu)$ coefficientwise. Hence, \(z_0\) is a root of
\(P_\infty(\cdot;t_0\delta_L,\nu)\). If $t_0>0$, then $0<t_0\delta_L\le\delta_L$,} 
and the hypothesis \eqref{eq:q-Pinfstable} implies that
\(|z_0|<1\), which is a contradiction. Therefore, \sa{we must have}
\begin{equation}\label{eq:t-to-zero}
        t_n\to0 .
\end{equation}
\sa{From \eqref{eq:q-Pinf},} the coefficientwise limit is now \(z^\tau(z-1)^2\).  Since \(|z_0|\ge1\), it
follows that \(z_0=1\).  Write
\begin{equation}\label{eq:xi-sn-def}
        \xi_n:=z_n-1,
        \qquad
        s_n:=2\sqrt{\delta_L}\,q_n .
\end{equation}
The root equation \sa{$P_{t_n,q_n}(z_n)=0$} is exactly
\begin{equation}\label{eq:q-exactxi}
        (1+\xi_n)^\tau\xi_n(\xi_n+s_n)
        +t_n\delta_L\bigl(1+(1+\nu)\xi_n\bigr)=0 .
\end{equation}
Since \(\xi_n\to0\), \sa{\eqref{eq:q-exactxi} implies that there exists a constant $\mgd{K_{\xi}}>0$, independent of $n$, such that,} for all sufficiently large
\(n\), $|\xi_n|\,|\xi_n+s_n|\le \mgd{K_{\xi}} t_n$; \sa{hence, we also have} $|\xi_n|^2\le \mgd{K_{\xi}} t_n+s_n|\xi_n|$. Consequently,
\begin{equation}\label{eq:q-xibound}
        |\xi_n|=
        \cO(s_n+\sqrt{t_n}).
\end{equation}
Expanding
\((1+\xi_n)^\tau=1+\tau\xi_n+\cO(\sa{|\xi_n|^2})\) in
\eqref{eq:q-exactxi}, we obtain the local quadratic relation
\begin{equation}\label{eq:q-uniformquad}
        \xi_n^2+\bigl[s_n+t_n\delta_L(1+\nu)\bigr]\xi_n+t_n\delta_L
        =\cO\bigl(|\xi_n|^3+s_n|\xi_n|^2\bigr).
\end{equation}

The remaining analysis depends only on the size of the ratio
\(t_n/q_n^2\).  Since \(t_n\ge q_n^2\), this ratio lies in \([1,\infty)\).
After passing to a subsequence, there are exactly two possibilities: either
\(t_n/q_n^2\to\infty\), or the ratio is bounded, in which case it has a further
subsequential limit \(\ell\in[1,\infty)\).  We treat the unbounded alternative
first; the bounded alternative is the final case below.

\sa{Suppose that} \(t_n/q_n^2\to\infty\).  Then
\[
        \frac{s_n}{\sqrt{t_n}}
        =2\sqrt{\delta_L}\frac{q_n}{\sqrt{t_n}}
        =2\sqrt{\delta_L}\left(\frac{q_n^2}{t_n}\right)^{1/2}
        \longrightarrow0 .
\]
Thus \(s_n=o(\sqrt{t_n})\), and \eqref{eq:q-xibound} gives
\(|\xi_n|=\cO(\sqrt{t_n})\). \sa{Define}
\[
        \eta_n:=\frac{\xi_n}{\sqrt{t_n}} .
\]
The preceding estimate shows that the sequence \sa{\(\{\eta_n\}\subset\mathbb{C}\) is bounded; hence,}  
after passing to a further subsequence we may assume that
\(\eta_n\to\eta_0\).  To identify \(\eta_0\), substitute
\(\xi_n=\sqrt{t_n}\eta_n\) into the exact equation \eqref{eq:q-exactxi} and
divide by \(t_n\).  This gives the exact scaled equation
\begin{equation}\label{eq:q-etaexact}
        (1+\sqrt{t_n}\eta_n)^\tau
        \eta_n\left(\eta_n+\frac{s_n}{\sqrt{t_n}}\right)
        +\delta_L\bigl(1+(1+\nu)\sqrt{t_n}\eta_n\bigr)=0 .
\end{equation}
Since \(t_n\to0\), \(s_n/\sqrt{t_n}\to0\), and \sa{\(\{\eta_n\}\)} is bounded converging to $\eta_0$,
letting \(n\to\infty\) in \eqref{eq:q-etaexact} yields $\eta_0^2+\delta_L=0$; hence,
\begin{equation}\label{eq:eta0-roots}
        \eta_0=\pm i\sqrt{\delta_L} .
\end{equation}
We next expand \eqref{eq:q-etaexact} one order further.  Choose \(R>0\) such
that \(|\eta_n|\le R\) for all sufficiently large \(n\). 
\sa{Since $\eta_n\to\eta_0$ and $|\eta_0|=\sqrt{\delta_L}$, we may take, for example, $R=2\sqrt{\delta_L}$ for all sufficiently large $n$.}
Taylor's formula on
the fixed disk \(|\eta|\le R\) gives
\[
        (1+\sqrt{t_n}\eta)^\tau
        =1+\tau\sqrt{t_n}\eta+\cO_R(t_n),
\]
where the constant in \(\cO_R(t_n)\) is independent of \(n\) and of
\(|\eta|\le R\).  Applying this with \(\eta=\eta_n\) in
\eqref{eq:q-etaexact}, and using \(s_n/\sqrt{t_n}\to0\), yields
\begin{equation}\label{eq:q-etaeq}
        0=
        \eta_n^2+\delta_L
        +\frac{s_n}{\sqrt{t_n}}\eta_n
        +\sqrt{t_n}\bigl(\tau\eta_n^3
            +\delta_L(1+\nu)\eta_n\bigr)
        +o\!\left(\frac{s_n}{\sqrt{t_n}}+\sqrt{t_n}\right).
\end{equation}

\sa{Here and below, the little-$o$ estimates are understood along the present
subsequence, for which $\eta_n \to \eta_0$. The point of the preceding
$R$-argument is that the Taylor remainder is controlled uniformly on a
fixed disk containing all sufficiently large $\eta_n$; hence, the implicit
constant does not depend on $n$. Write $\eta_n := \eta_0 + h_n$,
where $\eta_0$ is as in \eqref{eq:eta0-roots}, and set
\[
    m_n := \frac{s_n}{\sqrt{t_n}} + \sqrt{t_n}.
\]
Since $s_n/\sqrt{t_n} \to 0$ and $t_n \to 0$, we have $m_n \to 0$.
From \eqref{eq:q-etaeq}, using that $\{\eta_n\}$ is bounded, we obtain
$\eta_n^2 + \delta_L = \mathcal{O}(m_n)$.
Since $\eta_0^2=-\delta_L$, it follows that
\[
    \eta_n^2+\delta_L
    =(\eta_n-\eta_0)(\eta_n+\eta_0)
    =h_n(2\eta_0+h_n).
\]
Moreover, $h_n\to0$ and $\eta_0=\pm i\sqrt{\delta_L}\neq0$, so
$|2\eta_0+h_n|$ is bounded away from zero for all sufficiently large $n$.
Consequently, $h_n=\mathcal{O}(m_n)$; hence $h_n^2=\mathcal{O}(m_n^2)=o(m_n)$.}

\sa{We can now expand \eqref{eq:q-etaeq} around the simple root $\eta_0$.
Using $\eta_n=\eta_0+h_n$, we have
\[
    \eta_n^2+\delta_L
    =2\eta_0h_n+o(m_n),
\]
and, since $h_n=\mathcal{O}(m_n)$, we also have $\frac{s_n}{\sqrt{t_n}}\eta_n
    =
    \frac{s_n}{\sqrt{t_n}}\eta_0+o(m_n)$.
Similarly,
\[
    \sqrt{t_n}
    \left(
        \tau\eta_n^3+\delta_L(1+\nu)\eta_n
    \right)
    =
    \sqrt{t_n}
    \left(
        \tau\eta_0^3+\delta_L(1+\nu)\eta_0
    \right)
    +o(m_n).
\]
Since $\eta_0^2=-\delta_L$, we have
\[
    \tau\eta_0^3+\delta_L(1+\nu)\eta_0
    =
    \delta_L(1+\nu-\tau)\eta_0.
\]
Therefore, \eqref{eq:q-etaeq} gives
\[
    2\eta_0h_n
    +\frac{s_n}{\sqrt{t_n}}\eta_0
    +\sqrt{t_n}\,\delta_L(1+\nu-\tau)\eta_0
    =o(m_n).
\]
Dividing by $2\eta_0\neq0$, we conclude that
\begin{equation}
\label{eq:h-expansion}
    h_n
    =
    -\frac{1}{2}
    \left(
        \frac{s_n}{\sqrt{t_n}}
        +\sqrt{t_n}\,\delta_L(1+\nu-\tau)
    \right)
    +
    o\left(
        \frac{s_n}{\sqrt{t_n}}+\sqrt{t_n}
    \right).
\end{equation}}
Based on the expansion \eqref{eq:h-expansion} and using $\xi_n = \eta_n \sqrt{t_n}$, every near-one root in this regime
satisfies
\begin{equation}\label{eq:q-xiexp-corrected}
        \xi_n
        =\pm i\sqrt{t_n\delta_L}
        -\frac{s_n}{2}
        -\frac{t_n\delta_L(1+\nu-\tau)}{2}
        +o(q_n+t_n),
\end{equation}
with the sign determined by the branch.  It follows that
\[
        2\operatorname{Re}\xi_n
        =-s_n-t_n\delta_L(1+\nu-\tau)+o(q_n+t_n),
\]
and
\[
        |\xi_n|^2=t_n\delta_L+o(q_n+t_n).
\]
Thus
\begin{align}\label{eq:q-modbranch}
        |1+\xi_n|^2
        &=1+2\operatorname{Re}\xi_n+|\xi_n|^2
          \notag\\
        &=1-s_n+t_n\delta_L(\tau-\nu)+o(q_n+t_n).
\end{align}
By \eqref{eq:nu-gt-tau}, the coefficient \(\delta_L(\tau-\nu)\) is negative. \mgd{Since $\delta_L(\tau-\nu)<0$, we distinguish two
cases, passing to a further subsequence if necessary.
If $t_n=\cO(q_n)$, then \eqref{eq:q-modbranch} yields}
\[
        |1+\xi_n|
        \le 1-\sqrt{\delta_L}\,q_n+o(q_n),
\]
contradicting \eqref{eq:q-contra} because \(c<\sqrt{\delta_L}\). If
\(t_n/q_n\to\infty\), then the negative term
\(t_n\delta_L(\tau-\nu)\) dominates the \(q_n\)-scale in
\eqref{eq:q-modbranch}, and \eqref{eq:q-contra} is again impossible.

It remains only to consider the bounded alternative for \(t_n/q_n^2\).  This
is the only remaining possibility because, as observed above, every subsequence
of the ratios either has a further subsequence tending to \(\infty\) or a
bounded further subsequence.  Passing to a subsequence in the bounded case, we
may assume that
\[
        \frac{t_n}{q_n^2}\to\ell\in[1,\infty).
\]
The estimate \eqref{eq:q-xibound} gives \(\xi_n=\cO(q_n)\).  Thus
\(z_n=1+\xi_n\) is one of the two near-one roots covered by
Lemma~\ref{lem:q-slow}.  Hence
\[
        |z_n|=1-\sqrt{\delta_L}\,q_n+o(q_n),
\]
which contradicts \eqref{eq:q-contra}, again because
\(c<\sqrt{\delta_L}\).

The contradiction proves that, for every \(0<c<\sqrt{\delta_L}\), there exists
\(q_0>0\) such that every root of \(P_{t,q}\), uniformly for
\(t\in[q^2,1]\), lies in \(|z|\le1-cq\) whenever \(0<q\le q_0\).  By the
definition of \(\rho_{\tau,\kappa}\), this is precisely
\eqref{eq:q-gaterate}.
\end{proof}

\subsubsection{Stability analysis via the Cayley transform}

We reuse the Cayley transform of Lemma~\ref{lem:q-cayley},
\begin{equation}\label{eq:q-cayley}
        z=\frac{1+w}{1-w},
\end{equation}
which sends \(\Ree w<0\) to \(|z|<1\).  Multiplying by \((1-w)^{\tau+2}\), the transformed polynomial becomes
\begin{align}
        \mathcal H_{\tau,\delta,\nu}(w)
        &:=(1-w)^{\tau+2}P_\infty\!\left(\frac{1+w}{1-w};\delta,\nu\right)=4w^2(1+w)^\tau+
          \delta[1+(1+2\nu)w](1-w)^{\tau+1}.\label{eq:q-Hpoly}
\end{align}
Based on Lemma \ref{lem:q-cayley}, \(P_\infty\) is Schur stable if and only if \(\mathcal H_{\tau,\delta,\nu}\) is Hurwitz stable provided that $P_\infty(\wubtext{-1};\delta,\wubtext{\nu})\neq 0$. Then, as a corollary to the previous theorem, we get the following result. %
\begin{corollary}\label{coro:q-smallgain}
If \(\nu>\tau\), then \(\mathcal H_{\tau,\delta,\nu}\) is Hurwitz stable for all sufficiently small \(\delta>0\).
\end{corollary}
\begin{proof}
    The proof of Theorem \ref{thm:q-gate} shows that for $\wubtext{\nu}>\tau$, the polynomial $P_\infty(z;\delta,\wubtext{\nu})$ is Schur stable for $\delta$ small, more specifically this follows from the expansion \eqref{eq:z-pm-mod-first}. \wubtext{The remaining $\tau$ roots stay near $0$; \sa{hence, they are} inside the unit disk for sufficiently small $\delta>0$, by continuity.} Note that \(P_\infty(-1;0,\wubtext{\nu})=4(-1)^\tau \neq 0\) so that $P_\infty(-1;\delta,\wubtext{\nu})\neq 0$ for $\delta$ around 0. Then, \(P_\infty\) is Schur stable if and only if \(\mathcal H_{\tau,\delta,\nu}\) is Hurwitz stable by Lemma \ref{lem:q-cayley}; and we conclude.
\end{proof}

\begin{lemma}[explicit no-imaginary-axis interval]\label{lem:q-nocross}
Consider the following parameter choice
$$\nu=M-\tfrac12 \quad \mbox{with} \quad M:=2\tau+1.$$ %
Then \(\mathcal H_{\tau,\delta,\nu}\) has no root on the imaginary axis for every
\begin{equation}\label{eq:q-deltasmall}
        0<\delta\le M^{-2}.
\end{equation}
\end{lemma}
\begin{proof} 
    Fix an arbitrary $\delta\in (0,M^{-2}]$. Suppose, for contradiction, that $\wubtext{\mathcal H}_{\tau, \delta, \nu}$ has a root on the imaginary axis. \sa{Since we have}
    \[
    \wubtext{\mathcal H}_{\tau,\delta,\nu}(0) = \delta >0,
    \]
    this root is non-zero. Moreover, $\wubtext{\mathcal H}_{\tau,\delta,\nu}$ has real coefficients, so its non-real roots occur in conjugate pairs. Hence, without loss of generality we can assume \sa{that}
\(w=i\omega\), \(\omega>0\), is a root.
Put \(u=\arctan\omega\in(0,\pi/2)\) and $a:=1+2\nu=2M$. From \eqref{eq:q-Hpoly},
\begin{equation}\label{eq:q-imagroot2}
        \delta(1+ia\omega)(1-i\omega)^{\tau+1}
        =4\omega^2(1+i\omega)^\tau.
\end{equation}
Taking arguments in \eqref{eq:q-imagroot2} requires keeping track of the branch.  Since
\(u=\arctan\omega\), we have
\[
        \arg(1-i\omega)=-u,
        \qquad
        \arg(1+i\omega)=u.
\]
Hence, equating the arguments of both sides in \eqref{eq:q-imagroot2},
\[
        \arg(1+ia\omega)-(\tau+1)u
        \equiv \tau u \pmod {2\pi},
\]
or equivalently
\[
        \arg(1+ia\omega)
        \equiv (2\tau+1)u=Mu \pmod {2\pi}.
\]
The principal argument on the left is
\[
        \arg(1+ia\omega)
        =\arctan(a\omega)
        =\arctan(2M\tan u)
        \in(0,\pi/2).
\]
If we put \(y:=Mu\), then \(y\) must be
congruent modulo \(2\pi\) to an angle in \((0,\pi/2)\).  Therefore there is an integer
\(j\ge0\) such that
\begin{equation}\label{eq:q-phasebranch}
        y\in(2j\pi,2j\pi+\pi/2),
        \qquad
        \tan y=\tan\bigl(\arg(1+ia\omega)\bigr)
        =a\omega=2M\tan u=2M\tan(y/M).
\end{equation}
Taking moduli of both sides in \eqref{eq:q-imagroot2}, noting %
\hamzatext{$|1\pm i\wubtext{\omega}| = (1+\omega^2)^{1/2}$}, we have also

\begin{hamzablock}
    \[
    \delta(1+a^2\omega^2)^{1/2}(1+\omega^2)^{(\tau+1)/2} = 4\omega^2(1+\omega^2)^{\tau/2}.
    \]
\end{hamzablock}
Reorganizing this and using $a=2M$, \sa{we get}
\begin{equation}\label{eq:q-delta-omega2}
\delta
=
\frac{4\omega^2}
{\sqrt{1+4M^2\omega^2}\sqrt{1+\omega^2}}.
\end{equation}
Now substitute \(\omega=\tan u\). Since \(u\in(0,\pi/2)\),
\[
\sqrt{1+\omega^2}
=
\sqrt{1+\tan^2 u}
=
\sec u
=
\frac1{\cos u}.
\]
Thus,
\[
\delta
=
\frac{4\tan^2 u \cdot \cos(u)}
{\sqrt{1+4M^2\tan^2 u}} = \frac{4\sin^2 u}
{\sqrt{\cos^2 u+4M^2\sin^2 u}}.%
\]\
From the argument calculation, with \(y=Mu\),
\[
\tan y=2M\tan u,
\qquad
y\in(2j\pi,2j\pi+\pi/2).
\]
Since \(y\) is 
congruent modulo \(2\pi\) to an angle in \((0,\pi/2)\), 
$
\cos y>0.
$
Therefore,
\[
\cos y
=
\frac{1}{\sqrt{1+\tan^2 y}}
=
\frac{1}{\sqrt{1+4M^2\tan^2 u}}.
\]
Using this in the previous expression for $\delta$ gives
\[
\delta
=
4\tan^2 u\cdot \cos u\cdot \cos y= 4\sin^2u\,\frac{\cos y}{\cos u}.
\]
Finally, \(y=Mu\), \sa{we get}
\begin{equation}\label{eq:q-deltau}
        \delta=4\sin^2u\,\frac{\cos(Mu)}{\cos u}
        =4\sin^2(y/M)\,\frac{\cos y}{\cos(y/M)}.
\end{equation}
\hamzatext{We next show that any $y$ \mg{satisfying} \eqref{eq:q-phasebranch}, with $\delta$ defined by \eqref{eq:q-deltau}, %
\mg{necessarily satisfies} $\delta > M^{-2}$.} There are two cases:

\smallskip
\noindent\textbf{Case 1: \(j=0\).}  Then \(0<y<\pi/2\), and \(0<y/M\le\pi/6\).  Since \(\tan x\ge x\), \eqref{eq:q-phasebranch} gives \(\tan y\ge2y\).  The function \(\tan y/y\) is increasing, and \(\tan(11/10)<2(11/10)\), so
\[
        y>\frac{11}{10}.
\]
On \([0,\pi/6]\), \(\tan x\le(4/3)x\); hence \(\tan y\le(8/3)y\).  Since \(\tan(13/10)>(8/3)(13/10)\),
\[
        y<\frac{13}{10}.
\]
Thus \(\cos y>\cos(13/10)>1/4\).  By concavity of sine on \([0,\pi/2]\),
\[
        \sin x\ge\frac{3}{\pi}x,
        \qquad 0\le x\le\pi/6.
\]
Therefore,
\[
        \delta
        \ge4\left(\frac{3}{\pi}\frac{y}{M}\right)^2\cos y
        >4\left(\frac{3}{\pi}\frac{11}{10M}\right)^2\frac14
        =\frac{1089}{100\pi^2}\frac1{M^2}
        >\frac1{M^2}.
\]

\smallskip
\noindent\textbf{Case 2: \(j\ge1\).}  This case is vacuous for \(\tau=1\), because then \(M=3\) and \(u<\pi/2\) imply \(y=Mu<3\pi/2<2\pi\).  Thus, \sa{\(j\ge1\) implies} \(M>4\).  Since \(y\ge2\pi\), we have \(u=y/M\ge2\pi/M\).  \sa{Using the concavity of $\sin(\cdot)$ in this region} gives
\begin{equation}\label{eq:q-singe4}
        \sin u\ge\frac{2}{\pi}u\ge\frac4M.
\end{equation}
Using \(\tan y=2M\tan u\) in \eqref{eq:q-deltau} is equivalent to the expression
\begin{equation}\label{eq:q-deltaalt}
        \delta
        =\frac{4\sin^2u}{\sqrt{\cos^2u+4M^2\sin^2u}}.
\end{equation}
Indeed, \(\cos y=[1+4M^2\tan^2u]^{-1/2}\) and \(\cos y>0\).  The right side of \eqref{eq:q-deltaalt} is increasing in \(s=\sin u\in(0,1)\).  Combining this monotonicity with \eqref{eq:q-singe4},
\[
        \delta
        \ge \frac{4(4/M)^2}{\sqrt{1+(4M^2-1)(4/M)^2}}
        =\frac{64}{M^2\sqrt{65-16/M^2}}
        >\frac1{M^2}.
\]
Both cases contradict \(\delta\le M^{-2}\), so no imaginary-axis root exists when $\wubtext{0}\mathrel{\wubtext{<}}\delta\leq M^{-2}$.
\end{proof}

\begin{theorem}[closed-form accelerated family for every fixed delay]\label{thm:q-alldelay}
Fix any integer \(\tau\ge1\), with \(M=2\tau+1\), set
\begin{equation}\label{eq:q-alldelayparams}
       \;
        \alpha=\frac{1}{M^2L},
        \qquad
        \beta=1-\frac{2}{M}q,
        \qquad
        \nu=M-\frac12=2\tau+\frac12.
        \;
\end{equation}
Then \mgd{\DREAM{}} on quadratics satisfies, for every fixed \(\tau\) and every \(\wubtext{0}\mathrel{\wubtext{<}}c<1/M\), the accelerated bound
\begin{equation}\label{eq:q-alldelayrate}
        \rho_{\tau,\kappa}(\alpha,\beta,\nu)\le 1-c\,q=1-c\,\kappa^{-1/2}
        \qquad\text{for all }\kappa\ge\kappa_0^{\mathrm{all}}(\tau,c),
\end{equation}
where \(\kappa_0^{\mathrm{all}}(\tau,c)<\infty\) is a finite threshold depending only on \(\tau\) and
\(c\).  In particular, taking \(c=\tfrac1{2M}\) gives \(\rho_{\tau,\kappa}\le 1-\tfrac{1}{2(2\tau+1)}\kappa^{-1/2}\).
\end{theorem}

\begin{proof}
By Corollary~\ref{coro:q-smallgain}, since \(\nu=M-1/2>\tau\), the Hurwitz image \(\mathcal H_{\tau,\delta,\nu}\) is
stable for all sufficiently small \(\delta>0\); fix such \sa{$\delta_\ast$ satisfying} \(\delta_\ast\in(0,M^{-2}]\).  The leading
coefficient of \(\mathcal H_{\tau,\delta,\nu}\) is
\[
        4+(-1)^{\tau+1}\delta(1+2\nu)=4+(-1)^{\tau+1}2M\delta,
\]
and for \(0\le\delta\le M^{-2}\) we have \(2M\delta\le 2/M\le 2/3<4\), so it never vanishes: the degree is
constant and no root escapes to infinity on \([\delta_\ast,M^{-2}]\).  By Lemma~\ref{lem:q-nocross} no root lies
on the imaginary axis for any \(\delta\in(0,M^{-2}]\).  Hence the number of right-half-plane roots of
\(\mathcal H_{\tau,\delta,\nu}\) is constant as \(\delta\) runs over \((0,M^{-2}]\) (a root count can change
only by a root crossing the axis or escaping to infinity, both excluded), and equals its value at
\(\delta_\ast\), namely \(0\).  Thus, \(\mathcal H_{\tau,\delta,\nu}\) is Hurwitz, i.e., \
\(P_\infty(\cdot;\delta,\nu)\) is Schur stable, for every \(0<\delta\le M^{-2}\).  \sa{Therefore, for every \(\wubtext{0}\mathrel{\wubtext{<}}c<1/M\),} Theorem~\ref{thm:q-gate}
with \(\delta_L=M^{-2}\) (so \(\sqrt{\delta_L}=1/M\)) then yields the bound \sa{in}
\eqref{eq:q-alldelayrate} for all \(q\le q_0(\tau,c)\), i.e., for all \(\kappa\ge\kappa_0^{\mathrm{all}}(\tau,c):=q_0(\tau,c)^{-2}\).
\end{proof}
\mgd{The threshold $\kappa_0^{\mathrm{all}}(\tau,c)$
in Theorem~\ref{thm:q-alldelay} is not explicit. However, in Theorem~\ref{thm:main}, we give a quantitative version of this theorem. More specifically, for general smooth strongly convex objectives, Theorem~\ref{thm:main} provides an admissible threshold satisfying $\kappa_0(\tau,c)=\mathcal O_c(M^8)$ under an alternative rate parameterization $1-cq/M$, where $c\in(0,2/3)$}. \mgd{We should also mention that the closed-form} \sa{parameter choice in} \eqref{eq:q-alldelayparams} is deliberately conservative:
it is designed so that Lemma~\ref{lem:q-nocross} can exclude every unit-circle
crossing of the roots by hand, at the cost of a rate constant $c<1/M$ that need not
be optimal.\looseness=-1

\paragraph{\mgd{Optimal parameter design for GMM on quadratics.}} \mgd{For $\nu>\tau$, let $\delta_L(\nu)$ denote the smallest positive
gain $\delta$ for which $P_\infty(\cdot;\delta,\nu)$ defined in \eqref{eq:q-Pinf} has a
unit-circle root. Write $\bar\delta:=\alpha L$ for the chosen
gain, denoted by $\delta_L$ in Theorem~\ref{thm:q-gate}.
With
\[
    \alpha=\frac{\bar\delta}{L},
    \qquad
    \beta=1-2\sqrt{\bar\delta}\,q,
    \qquad
    0<\bar\delta<\delta_L(\nu),
\]
that theorem guarantees $\rho\le1-c/\sqrt{\kappa}$ for every
$0<c<\sqrt{\bar\delta}$ and all sufficiently large $\kappa$.
Thus, for each $\nu>\tau$, the supremum of the certifiable rate
constants is $\sqrt{\delta_L(\nu)}$, and optimizing this
guarantee reduces to maximizing $\delta_L(\nu)$. Let $\nu_\tau^{\mathrm{crit}}$ denote the maximizing
extrapolation parameter, and define
\[
    \delta_\tau^{\mathrm{crit}}
        :=\delta_L(\nu_\tau^{\mathrm{crit}})
        =\max_{\nu>\tau}\delta_L(\nu),
    \qquad
    c_\tau^{\mathrm{crit}}
        :=\sqrt{\delta_\tau^{\mathrm{crit}}}.
\]
Every $0<c<c_\tau^{\mathrm{crit}}$ is therefore certified by
choosing $\nu=\nu_\tau^{\mathrm{crit}}$ and
$c^2<\bar\delta<\delta_\tau^{\mathrm{crit}}$.
Appendix~\ref{sec:design} characterizes this optimizer exactly
and establishes the following large-delay asymptotics
(Corollary~\ref{cor:q-designoptimal} and
Proposition~\ref{prop:q-scaling}):
\[
    c_\tau^{\mathrm{crit}}
        =\frac{a_\star\sqrt{\cos a_\star}}{\tau}\,(1+o(1))
        \approx\frac{0.7415}{\tau},
    \qquad
    \nu_\tau^{\mathrm{crit}}
        =\frac{2}{a_\star^2}\,\tau\,(1+o(1))
        \approx1.725\,\tau,
\]
where $a_\star\approx1.0769$ is the unique solution of
$a\tan a=2$ on $(0,\pi/2)$.}
Two features of this optimum
inform the parameters used throughout the paper: acceleration survives every fixed
delay but its constant is asymptotically proportional to $1/\tau$, and the optimal
extrapolation grows \emph{linearly} in the delay --- a look-ahead the classical
one-parameter families ($\nu=0$, $\nu\approx1$) cannot express, which is precisely
why they forfeit $\sqrt\kappa$-acceleration (Theorems~\ref{thm:q-naglower}
and~\ref{thm:q-hblower}).

The conservative choice $\nu=M-\tfrac12\sim2\tau$ of Theorem~\ref{thm:q-alldelay}
carries this same scaling and differs from the optimum only in constants. Indeed,
it yields $\rho\mathrel{\le}1-c/\sqrt\kappa$ with $c=\tfrac{1}{2(2\tau+1)}\approx 1/(4\tau)$ for
large $\tau$, so $c_\tau^{\mathrm{crit}}/c\to a_\star\sqrt{\cos a_\star}/(1/4)<3$ as $\tau\to\infty$;
the certified gap constant is thus asymptotically within a factor of three of this supremum. We adopt this
choice throughout because it extends to the
strongly convex case without requiring modification to the parameter choice, furnishing a single parameter setting valid for both
quadratic and smooth strongly convex objectives.

\section{Strongly convex smooth case: Acceleration with delays}\label{sec-strongly-convex}
\mgd{We now show that the parameter choice studied for quadratics also yields acceleration for every smooth strongly convex objective. For each fixed delay and all sufficiently large condition numbers, the resulting geometric rate has a gap of order $1/((2\tau+1)\sqrt{\kappa})$.}

\begin{theorem}[arbitrary fixed-delay acceleration]
\label{thm:main-nonexplicit-constants}
\mgd{Fix an integer $\tau\ge1$, let $M=2\tau+1$,
and choose any $c\in(0,2/3)$.
There exists a threshold $\kappa_0=\kappa_0(\tau,c)>1$
such that, for every $f\in\Sml$ with
$0<\mu\le L$ and condition number
$\kappa=L/\mu\ge\kappa_0$, \DREAM{},
namely delayed GMM~\eqref{eq:gmm} with parameters}
\mgd{\begin{equation*}
        \alpha\hamzatext{= \frac{1}{M^2L}},\qquad
        \beta\hamzatext{= 1-\frac{2}{M\sqrt{\kappa}}},\qquad
        \nu \hamzatext{= M-\frac{1}{2}}
\end{equation*}}
\mgd{and any prescribed finite initial history
$x_{-\tau-1},\ldots,x_0\in\mathbb R^d$
generates iterates satisfying}
\begin{equation}\label{eq:rate}
        \norm{x_k-x_*}\le
        \mgd{C_{\mathrm{rate}}}\left(1-\frac{c}{\mgd{M}\sqrt\kappa}\right)^k,
        \qquad k\ge0,
\end{equation}
\mgd{where $x_*$ is the unique minimizer of $f$
and $\mgd{C_{\mathrm{rate}}}\geq 0$ is a finite constant that may depend only on the constants $\mu, L$, algorithmic parameters
and on the initial history relative to $x_*$,
but not on the iteration index $k$.}
\end{theorem}

\mgd{The proof of Theorem~\ref{thm:main-nonexplicit-constants}, given in Section~\ref{sec:master}, rests on the following two building blocks.}
\subsection{
{Building Blocks for the Proof of Theorem~\ref{thm:main-nonexplicit-constants}}}\label{sec:certificate}\label{sec:proofs}

{The \textit{first} building block is the transfer function associated with the linear part of delayed GMM (Section~\ref{subsec:plant}). \mgd{The key structural observation is that GMM
can be represented as the interconnection of
a linear subsystem and the nonlinear gradient
mapping. The linear subsystem captures the
delay, momentum recursion, and extrapolation
from $x_k$ to $y_k$, while the nonlinear
component maps $y_k$ to $\nabla f(y_k)$. Thus, all nonlinearity in the GMM dynamics enters through the gradient mapping.}} %
{The \textit{second} building block is a hard integral quadratic constraint (IQC), valid on every finite horizon, satisfied by the gradient mapping of every \(f\in\Sml\) (Section~\ref{subsec:hardiqc}). The role of the IQC is to encode this nonlinear relation through a quadratic inequality implied by smoothness and strong convexity. For the purposes of the present analysis, the hard IQC may therefore be viewed as a finite-horizon consequence of these assumptions \mgd{in terms of quadratic inequalities}, one that remains valid after summation over any finite number of iterations.}

\paragraph{Standing normalization and parameters.}
The whole certificate is carried out under a fixed normalization and a fixed choice of parameters, which we record once here and use throughout the remainder of the paper.  After scaling $f$ by $1/\mu$ we work with
\begin{equation}\label{eq:norm}
        \mu=1,\qquad L=q^{-2},\qquad q=\kappa^{-1/2}\in(0,1],
\end{equation}
and, for a fixed $c\in(0,2/3)$,
\begin{equation}\label{eq:params}
        \alpha=\frac{q^2}{M^2},\qquad
        \beta=1-\frac{2q}{M},\qquad
        \nu=M-\frac12,
        \qquad
        \rho=1-\frac{cq}{M}.
\end{equation}
Here $\rho$ is the contraction rate to be certified, and \mgd{the large-condition-number regime corresponds to $q=\kappa^{-1/2}\downarrow0$.}%

\subsubsection{The transfer function of the linear \mgd{subsystem}%
}\label{subsec:plant}
We analyze \DREAM{} by viewing the method as a feedback interconnection between a linear time-invariant subsystem\footnote{A discrete-time system is said to be linear time-invariant (LTI) if its input-output map is linear and commutes with time shifts, see \cite{oppenheim1997signals}.} and the nonlinear gradient mapping. To expose this structure, we break the feedback relation \(u_k=\nabla f(y_k)\) and temporarily regard \(u=(u_k)\) as an \mgd{arbitrarily prescribed input, without imposing the gradient relation}. The resulting state recursion is
\begin{equation}\label{eq:g-derivation-xrec}
x_{k+1}=(1+\beta)x_k-\beta x_{k-1}-\alpha u_{k-\tau}.
\end{equation}
The linear subsystem therefore consists of the delay, the momentum recursion, and the extrapolation from \(x_k\) to the evaluation point \(y_k\), while the gradient map \(\nabla f\) is the %
\sa{only} nonlinear component \sa{in this dynamical system}. This feedback (Lur'e) decomposition is illustrated in Figure~\ref{fig:lure-separation}. In particular, the factor \(z^{-\tau}\) represents the delay that maps the current gradient sequence \(u_k\) to the \mgd{delayed gradient} \(u_{k-\tau}\) entering the momentum recursion.

\begin{figure}[t]
\centering

\definecolor{linearblue}{RGB}{0,170,220}
\definecolor{feedbackorange}{RGB}{245,90,10}

\begin{tikzpicture}[
    >=Latex,
    font=\small,
    innerblock/.style={
        draw=linearblue!80!black,
        rounded corners,
        minimum height=10mm,
        align=center,
        thick,
        fill=white
    },
    nl/.style={
        draw=feedbackorange!85!black,
        rounded corners,
        minimum width=32mm,
        minimum height=13mm,
        align=center,
        thick,
        fill=feedbackorange!10
    },
    wire/.style={->, thick},
]

\node[
    innerblock,
    minimum width=18mm
] (D) {
    \textbf{Delay}\\[-1pt]
    $\displaystyle z^{-\tau}$
};

\node[
    innerblock,
    minimum width=38mm,
    right=14mm of D
] (G) {
    \textbf{Momentum recursion}\\[-2pt]
    {\footnotesize $+$ \textbf{extrapolation} \sa{based on} \eqref{eq:g-derivation-xrec}}
};

\draw[
    wire,
    draw=linearblue
]
    (D.east)
    -- node[
        above,
        pos=0.5,
        text=linearblue!70!black
    ] {$u_{k-\tau}$}
    (G.west);

\begin{scope}[on background layer]
\node[
    draw=linearblue!80!black,
    rounded corners,
    thick,
    fill=linearblue!6,
    inner xsep=5mm,
    inner ysep=9mm,
    fit=(D)(G)
] (Linear) {};
\end{scope}

\node[
    anchor=north,
    font=\bfseries,
    text=linearblue!70!black
] at ($(Linear.north)+(0,-1.5mm)$) {
    Linear \mgd{subsystem} $\mgb{G_{\tau,q}(z)}:=g_{\tau,q}(z)\mgb{I_d}$
};

\node[
    nl,
    below=23mm of Linear
] (F) {
    \textbf{Gradient map}\\[2pt]
    {\footnotesize $u_k=\nabla f(y_k)$}
};

\coordinate (LinRight) at ($(Linear.east)+(25mm,0)$);
\coordinate (FRight)   at (LinRight |- F.east);

\coordinate (LinLeft)  at ($(Linear.west)+(-24mm,0)$);
\coordinate (FLeft)    at (LinLeft |- F.west);

\draw[
    wire,
    draw=linearblue
]
    (Linear.west) -- (D.west);

\draw[
    wire,
    draw=linearblue
]
    (G.east)
    -- (Linear.east)
    -- node[
        above,
        pos=0.45,
        text=linearblue!70!black
    ] {$y_k$}
    (LinRight)
    -- (FRight)
    -- (F.east);

\draw[
    wire,
    draw=feedbackorange
]
    (F.west)
    -- (FLeft)
    -- (LinLeft)
    -- node[
        above,
        pos=0.65,
        text=feedbackorange!85!black
    ] {$u_k$}
    (Linear.west);

\node[
    font=\bfseries,
    text=feedbackorange!85!black,
    fill=white,
    inner sep=1pt
] at ($(F.east)!0.55!(FRight)+(0,3mm)$) {
    Nonlinear (connection)
};

\node[
    align=center,
    anchor=west
] at ($(LinRight)!0.5!(FRight)+(2mm,0)$) {
    \footnotesize evaluation\\[-2pt]
    \footnotesize point
};

\node[
    align=center,
    anchor=east
] at ($(LinLeft)!0.5!(FLeft)+(-2mm,0)$) {
    \footnotesize gradient\\[-2pt]
    \footnotesize sequence
};

\end{tikzpicture}
\caption{Feedback representation of delayed GMM. Breaking the
feedback relation $u_k=\nabla f(y_k)$ separates the method into the interconnection of (i) \mgd{the linear subsystem with transfer matrix
$G_{\tau,q}(z)$} which incorporates the delay, momentum recursion,
and extrapolation, and (ii) the static nonlinear gradient map $\nabla f$.}
\label{fig:lure-separation}

\end{figure}
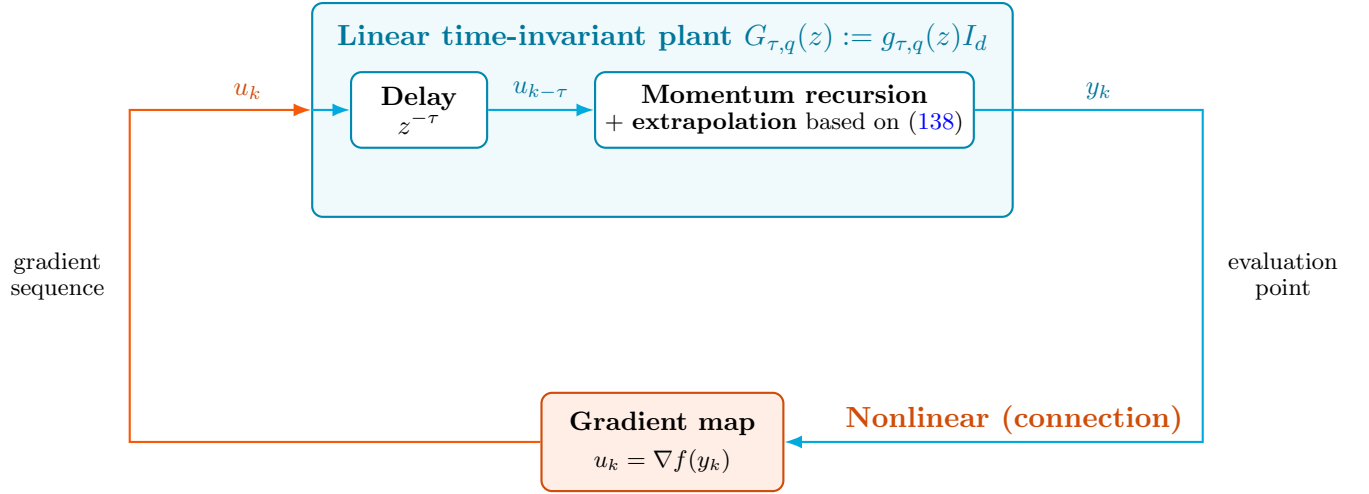

Taking the bilateral $z$-transform of \eqref{eq:g-derivation-xrec}, with zero initial conditions ($x_k=0$ for all $k\le 0$ \mgb{and $u_k=0$ for $k<0$}) for the input-output calculation, gives
\begin{equation}\label{eq:g-derivation-zx}
        \bigl(z-(1+\beta)+\beta z^{-1}\bigr)X(z)
        =-\alpha z^{-\tau}U(z).
\end{equation}
\mgd{Here $X(z)$ and $U(z)$ are the $z$-transforms of the zero-initial-state response and its prescribed input. The transform identities hold wherever these series converge; they do not assume square-summability of the optimization trajectory.}
Multiplying \eqref{eq:g-derivation-zx} by $z$ and factoring the quadratic \mgd{terms} yields
\begin{equation}\label{eq:g-derivation-xoveru}
        (z-1)(z-\beta)X(z)=-\alpha z^{1-\tau}U(z),
        \qquad
        {X(z)}=
        -\alpha\frac{z^{1-\tau}}{(z-1)(z-\beta)}\mgb{U(z)}.
\end{equation}
The evaluation point is the extrapolated point
\begin{equation}\label{eq:g-derivation-y}
        y_k=x_k+\nu(x_k-x_{k-1})=(1+\nu)x_k-\nu x_{k-1},
\end{equation}
so in the $z$-domain
\begin{equation}\label{eq:g-derivation-yz}
        Y(z)=\bigl(1+\nu-\nu z^{-1}\bigr)X(z)
        =\frac{(1+\nu)z-\nu}{z}X(z).
\end{equation}
Combining \eqref{eq:g-derivation-xoveru} and \eqref{eq:g-derivation-yz} gives
\begin{equation}\label{eq:g-derivation-general}
        {Y(z)}=
        -\alpha\frac{z^{-\tau}\bigl((1+\nu)z-\nu\bigr)}{(z-1)(z-\beta)}\mgb{U(z)}.
\end{equation}
\mgd{Finally, substituting the normalized parameters
$\alpha=q^2/M^2$ and $\nu=M-\frac12$
(so that $1+\nu=M+\frac12$) gives the
scalar transfer function from the input
sequence $(u_k)$ to the evaluation-point
sequence $(y_k)$:}
\begin{equation}\label{eq:g}
       \mgb{Y(z)=g_{\tau,q}(z)U(z)
       \quad
       \mbox{with}} \quad
       g_{\tau,q}(z):=
        -\frac{q^2}{M^2}\,
        \frac{z^{-\tau}\bigl((M+\frac12)z-(M-\frac12)\bigr)}{(z-1)(z-\beta)},
\end{equation}
\sa{where} each factor in $g_{\tau,q}(z)$ has a direct interpretation: $z^{-\tau}$ represents $\tau$-step delay of $u_k$, the factor $\bigl((M+\frac12)z-(M-\frac12)\bigr)$ is the extrapolation from $x_k$ to $y_k$, and $(z-1)(z-\beta)$ is the \mg{denominator} induced due to momentum recursion.
\mgb{We also introduce the transfer matrix}
\beq 
\mgb{G_{\tau,q}(z):=g_{\tau,q}(z)I_d \in \mathbb{C}^{d\times d}\quad
 \mbox{with} \quad
Y(z)=g_{\tau,q}(z)U(z)=G_{\tau,q}(z)U(z).}
\label{eq-transfer-function-g}
\eeq

\subsubsection{Integral quadratic constraints (IQCs) associated to the gradient map $\nabla f$}\label{subsec:hardiqc}
  In Section \ref{subsec:plant}, we isolate and study the linear part of delayed GMM by treating the gradient values $u_k$ as an external input, and derive the corresponding transfer function in \eqref{eq:g}. However, in the GMM algorithm, $u_k$ is not an arbitrary input but is generated by the non-linear gradient map $u_k=\nabla f(y_k)$. In this subsection, we restore this feedback relation and characterize the constraints that the gradient map imposes on the input-output pairs $(y_k,u_k)$. %

Two features will be important in what follows. First, the inequalities hold on \emph{every} finite horizon; this will later allow us to sum them directly within a finite-horizon Lyapunov estimate. Second, in the rest of the paper, without loss of generality, 
we translate the minimizer to the origin, so $\nabla f(0)=0$.  This can clearly be obtained by simply replacing $x_k$ by $x_k-x_*$ and $y_k$ by $y_k-x_*$ in the subsequent analysis. This normalization will simplify our analysis.

\mgd{We call $(y,u)\in\mathbb R^d\times\mathbb R^d$
a \emph{point--gradient pair} of $f$ if
$u=\nabla f(y)$. For a sequence of such pairs
$(y_k,u_k)$ indexed by $k\in\mathbb Z_+$, define
\begin{equation}\label{eq:rs}
        r_k:=u_k-\mu y_k,
        \qquad
        s_k:=L y_k-u_k.
\end{equation}
By \eqref{eq:interp}, for all $i,j\in\mathbb Z_+$,
\begin{equation}\label{eq:monotone-rs}
        \langle r_i-r_j,s_i-s_j\rangle\ge0,
        \qquad
        \langle r_i,s_i\rangle\ge0.
\end{equation}
The second inequality follows by comparing
$(r_i,s_i)$ with the pair $(0,0)$ corresponding
to the minimizer, since $x_*=0$ and
$\nabla f(0)=0$.}

\mgd{We next introduce a convex-potential
representation of these inequalities. Set
$\gamma_{\mathrm{diff}}:=L-\mu\ge0$.
We first consider $\mu<L$, so that
$\gamma_{\mathrm{diff}}>0$. When $L=\mu$,
we have $r_k=s_k=0$ for every point--gradient
pair, and the IQCs below hold trivially;
this case is treated separately in the proof
of Lemma~\ref{lem:hard}. For $\mu<L$, define
\[
        F(y):=f(y)-\frac{\mu}{2}\|y\|^2.
\]
}%
Then $F$ is convex and \mgd{$\gamma_{\mathrm{diff}}$-smooth}; hence, its Fenchel conjugate $F^*$ is \mgd{$(1/\gamma_{\mathrm{diff}})$-strongly convex} \mgd{as an extended real-valued function. Consequently, the following extended real-valued function is convex:}%
\begin{equation}\label{eq:psi-def}
        \psi(r):=\mgd{\gamma_{\mathrm{diff}}} F^*(r)-\frac12\|r\|^2 - \mgd{\gamma_{\mathrm{diff}}} F^*(0).
\end{equation}
Note that
$
F^*(0)=-\inf_y F(y)=-F(0).
$
Indeed, \(F\) is  convex and \(\nabla F(0)=0\), so zero is a minimizer of \(F\), \sa{where we used the fact that $\nabla f(0)=0$}.
\mgd{For every point--gradient pair, $r=\nabla F(y)$ and Fenchel equality gives $F^*(r)=\langle r,y\rangle-F(y)<\infty$. Thus $\psi(r)$ is finite \mgd{at every value $r_k$ arising from the point--gradient pairs considered below} even if $F^*$ is infinite elsewhere. This construction allows us to express $s_k$ as a subgradient of $\psi$ at $r_k$}.
\sa{Indeed,} since $r=\nabla F(y)$ implies $y\in\partial F^*(r)$, we have
\begin{equation}\label{eq:cyclic-potential}
        s=\mgd{\gamma_{\mathrm{diff}}} \cdot y-r\in\partial\psi(r).
\end{equation}
Moreover $0\in\partial\psi(0)$ and $\psi(0)=0$; \sa{hence,} convexity of $\psi$ gives \beq 
\psi(r)\ge0 \qquad \mbox{for all } r.
\label{eq:psi-nonnegative}
\eeq 
\mg{The next lemma shows $s_k$ is a subgradient of $\psi$ evaluated at $r_k$. Given a finite trajectory \(\{(y_k,u_k)\}_{k=0}^N\), this lemma also extends the sequence by padding it with \((0,0)\) for all ($k<0$) and ($k>N$). This convention simplifies the notation in the subsequent analysis by allowing us to avoid treating the boundary terms for \(k<0\) and \(k>N\) separately. }
\mgb{\begin{lemma}[padding a finite trajectory with the minimizer]\label{lem:zero-extension}
Let \mgd{$0<\mu<L$ and} $f\in\Sml$ with $\nabla f(0)=0$, let $(y_k,u_k)=(y_k,\nabla f(y_k))$ for $0\le k\le N$, and set $(y_k,u_k):=(0,0)$ for $k\in\mathbb Z\setminus\{0,\dots,N\}$. Define $r_k:=u_k-\mu y_k$ and $s_k:=Ly_k-u_k$ for $k\in\mathbb Z$. Then:
\begin{enumerate}[label=\textup{(\alph*)},nosep,leftmargin=2.2em]
\item $s_k\in\partial\psi(r_k)$ and $\psi(r_k)\ge0$ for every $k\in\mathbb Z$; consequently,
\[
\langle r_i-r_j,\,s_i-s_j\rangle\ge0
\quad\text{and}\quad
\langle r_i,\,s_i\rangle\ge0
\qquad\text{for all } i,j\in\mathbb Z .
\]
\item $r_k=s_k=0$ for $k\notin\{0,\dots,N\}$; hence every sum $\sum_{k\in\mathbb Z}\langle r_k,s_{k-j}\rangle$ has at most $N+1$ nonzero terms and, for every $j\in\mathbb Z$,
\[
\sum_{k\in\mathbb Z}\langle r_k,\,s_{k-j}\rangle=\sum_{k\in\mathbb Z}\langle r_{k+j},\,s_k\rangle .
\]
\end{enumerate}
\end{lemma}
\begin{proof}
Since $\nabla f(0)=0$, the pair $(0,0)$ is a point--gradient pair of $f$, so every pair of the extended sequence is of the form $(y,\nabla f(y))$ for some $y\in\mathbb R^d$ ($y=y_k$ for $0\le k\le N$ and $y=0$ otherwise). Hence \eqref{eq:cyclic-potential} and \eqref{eq:psi-nonnegative} give $s_k\in\partial\psi(r_k)$ and $\psi(r_k)\ge0$ for every $k\in\mathbb Z$. Monotonicity of the subdifferential of the convex function $\psi$ yields $\langle r_i-r_j,s_i-s_j\rangle\ge0$, and choosing $j\notin\{0,\dots,N\}$, for which $(r_j,s_j)=(0,0)$, yields $\langle r_i,s_i\rangle\ge0$; this proves (a). 
\sa{
For (b), $y_k=u_k=0$ gives $r_k=s_k=0$ for
$k\notin\{0,\ldots,N\}$. Hence, for every $j\in\mathbb Z$, the
sum $\sum_{k\in\mathbb Z}\langle r_k,s_{k-j}\rangle$
has at most $N+1$ nonzero terms and is therefore finite. The change
of variables $\ell=k-j$ then gives
$\sum_{k\in\mathbb Z}\langle r_k,s_{k-j}\rangle
=
\sum_{\ell\in\mathbb Z}\langle r_{\ell+j},s_\ell\rangle$,
which proves (b).}
\end{proof}}

\mg{Consider any finite cycle of indices $k_0,\dots,k_n,k_{n+1}$ where $k_{n+1}=k_0$. Based on the previous lemma, if we apply the subgradient inequality for the convex function $\psi$, \sa{we get}
\[
\psi(r_{k_{i+1}})
-
\psi(r_{k_i})
\ge
\left\langle
s_{k_i},r_{k_{i+1}}-r_{k_i}
\right\rangle,\quad \sa{\forall~i=0,\ldots,n,}
\]
\sa{i.e., it holds successively around the cycle. Summing over %
$i\in\{0,\ldots,n\}$, we obtain that} %
for every finite cycle of indices $k_0,\dots,k_n,k_{n+1}=k_0$, we have
$$\sa{\sum_{i=0}^{n}}\langle s_{k_i},r_{k_{i+1}}-r_{k_i}\rangle\le0.$$
This is referred to as the \emph{cyclic monotonicity}
property in the literature. By
Rockafellar~\cite{Rockafellar1966}, pairs $(r_k, s_k)$ that satisfy such a property arise from the subdifferential of a convex function, here $\psi$
of~\eqref{eq:psi-def}}.\looseness=-1 

\mg{The inequalities below involve finite sums of centered iterate--gradient pairs along a finite trajectory. Quadratic constraints of this form, which remain valid on every finite horizon, are referred to as \emph{hard integral quadratic constraints} (hard IQCs). This finite-horizon validity distinguishes hard IQCs from \emph{soft} IQCs, which are required to hold only \mgd{over the full infinite-horizon sequences}; see, e.g.,~\cite{MegretskiRantzer1997,LessardRechtPackard2016}.}

We use Zames--Falb IQCs to capture relations between
gradients evaluated at different iterates. The multiplier
family originates with Zames and Falb~\cite{ZamesFalb1968},
while the general IQC framework was developed by
Megretski and Rantzer~\cite{MegretskiRantzer1997}. For gradients of smooth strongly convex functions, Lessard,
Recht, and Packard~\cite[Lemma~9]{LessardRechtPackard2016}
establish a causal finite-horizon (hard) Zames--Falb IQC.
\mgd{They express this inequality as a finite sum over iteration
indices involving points and their gradients---its
\emph{time-domain} form. We restate it in our notation,
derive an equivalent \emph{frequency-domain} form using
the Fourier transforms of these sequences, and provide
a self-contained proof.}
    In this setting, the IQC is used as a tool for analyzing the stability of the feedback interconnection between the linear GMM dynamics and the nonlinear gradient map. Rather than modeling the nonlinear block exactly, the IQC specifies a quadratic constraint that every admissible gradient map must satisfy, so stability can be established uniformly over the entire class of nonlinearities consistent with the assumption on $f$. The frequency-domain formulation is useful because the linear dynamics were already represented in Section \ref{subsec:plant} by the transfer function $g_{\tau,q}(z)$. Using Parseval's identity, the time-domain IQC can be converted into an equivalent frequency-domain quadratic inequality, allowing the linear transfer function and the non-linear constraint to be combined directly in a single stability analysis. \mg{We now present a Zames-Falb IQC for the nonlinearity $\nabla f$.}%

\begin{lemma}[hard Zames--Falb IQC]\label{lem:hard}
Let \mgd{$f\in\Sml$ with $0<\mu\le L$, let $J\ge1$ be an integer, and let $h_1,\dots,h_J\ge0$ satisfy $\sum_{j=1}^{J}h_j\le1$. Define $h(z):=\sum_{j=1}^{J}h_jz^{-j}$ for $z\ne0$.} %
\mg{Assume the minimizer has been translated to the origin, so that $\nabla f(0)=0$.} Fix any horizon $N\ge0$ and any finite trajectory of
point--gradient pairs $\{(y_k,u_k)\}_{k=0}^{N}$ with $u_k=\nabla f(y_k)$, and
zero-extend it by $(y_k,u_k)=(0,0)$ for $k\notin\{0,\dots,N\}$, so that the
slope-shifted sequences $r_k=u_k-\mu y_k$ and $s_k=Ly_k-u_k$ have finite support and
all sums below are finite. Then, for this horizon,
\begin{equation}\label{eq:causal-hard-time}
        2\sum_{k}\langle r_k,s_k\rangle
        -2\sum_{j=1}^{J}h_j\sum_{k}\langle r_k,s_{k-j}\rangle
        \ge0 ,
\end{equation}
where each $\sum_k$ runs over $\mathbb Z$ but has at most $N+1$ nonzero terms.
Equivalently, by Parseval,
\begin{equation}\label{eq:causal-hard-frequency}
        \frac1{2\pi}\int_0^{2\pi}
        2\Ree\!\left[\bigl(1-h(e^{i\theta})\bigr)\,R(e^{i\theta})^*S(e^{i\theta})\right]d\theta\ge0 ,
\end{equation}
where $R,S$ are the discrete-time Fourier transforms of the zero-extended sequences
$r,s$, \sa{i.e., $R(e^{i\theta}):=\sum_{k\in\mathbb Z}r_k e^{-ik\theta}$ and
$S(e^{i\theta}):=\sum_{k\in\mathbb Z}s_k e^{-ik\theta}$.} Since \eqref{eq:causal-hard-time} holds for \emph{every} horizon $N$
separately, with no requirement on the trajectory beyond $N$, the constraint is
\emph{hard}: it may be applied to any finite truncation without a tail term.

Finally, substituting $r_k=u_k-\mu y_k$ and $s_k=Ly_k-u_k$ expresses the integrand
of \eqref{eq:causal-hard-frequency}, at each frequency $\theta\in[0,2\pi)$, as a
Hermitian form in the original sequences:
\[
        2\Ree\!\left[\bigl(1-h(e^{i\theta})\bigr)R(e^{i\theta})^*S(e^{i\theta})\right]
        =\begin{bmatrix} Y(e^{i\theta})\\ U(e^{i\theta})\end{bmatrix}^{\!*}
        \mgb{\Pi_{\mu,L,d}}\!\bigl(h(e^{i\theta})\bigr)
        \begin{bmatrix} Y(e^{i\theta})\\ U(e^{i\theta})\end{bmatrix},
\]
where $Y,U\mgb{\in \mathbb{C}^d}$ are the Fourier transforms of $y,u$, \sa{and} writing \mgb{the scalar} $h=h(e^{i\theta})$ for
brevity,
\begin{equation}\label{eq:Pi}
        \mgb{
        \sa{\Pi_{\mu,L,d}(h):=\Pi_{\mu,L}(h)\otimes I_d},\quad \mbox{where}} \quad \Pi_{\mu,L}(h):=
        \begin{bmatrix}
            -\mu L\,(2-h-\bar h) & L(1-\bar h)+\mu(1-h)\\[2pt]
            L(1-h)+\mu(1-\bar h) & -(2-h-\bar h)
        \end{bmatrix}.
\end{equation}
\end{lemma}
\begin{proof}\mgd{When $L=\mu$, the gradient map,
after translating the minimizer to the origin,
satisfies $u_k=\mu y_k$ for every $k$.
Consequently, the slope-shifted sequences
$r_k=u_k-\mu y_k$ and $s_k=Ly_k-u_k$
both vanish for every $k$, and the IQC holds
trivially. We may therefore assume $L>\mu$. All sums are finite by the zero extension and the finiteness of $\psi(r_k)$ established above (Lemma~\ref{lem:zero-extension}).}
For a fixed lag $j\ge1$, the subgradient inequality for $s_k\in\partial\psi(r_k)$ gives
\[
        \psi(r_{k+j})\ge \psi(r_k)+\langle s_k,r_{k+j}-r_k\rangle,
\]
\sa{which implies that}
\[
        \langle s_k,r_k-r_{k+j}\rangle\ge \psi(r_k)-\psi(r_{k+j}).
\]
Summing this inequality over all $k\in\mathbb Z$, the right-hand side telescopes to $\sum_k\psi(r_k)-\sum_k\psi(r_{k+j})=0$ since $k\mapsto k+j$ is a bijection of $\mathbb Z$ and \sa{$\{r_k\}$} has finite support; \sa{thus, we get}
\begin{equation}\label{eq:causal-lag-unweighted}
        \sum_k\langle s_k,r_k-r_{k+j}\rangle\ge0.
\end{equation}
Since
\[
        \sum_k\langle s_k,r_{k+j}\rangle
        =\sum_k\langle r_k,s_{k-j}\rangle,
\]
by Lemma \ref{lem:zero-extension}, the inequality \eqref{eq:causal-lag-unweighted} %
\sa{implies that}
\[
        A-B_j\ge0,
        \qquad
        A:=\sum_k\langle r_k,s_k\rangle,
        \quad
        B_j:=\sum_k\langle r_k,s_{k-j}\rangle .
\]
\sa{Moreover, by \eqref{eq:monotone-rs}, we also have $A\ge0$.} %
Therefore, with $H:=\sa{\sum_{j=1}^J} h_j\le1$, \sa{it follows that}
\[
        A-\sum_{j=1}^{J}h_jB_j
        =(1-H)A+\sum_{j=1}^{J}h_j(A-B_j)\ge0,
\]
\sa{which yields \eqref{eq:causal-hard-time} after multiplying it by $2$.}

Parseval's identity and the time-shift property of the discrete-time Fourier transform give
\[
        \sum_k\langle r_k,s_k\rangle
        =\frac1{2\pi}\int_0^{2\pi}R(e^{i\theta})^*S(e^{i\theta})\,d\theta,
\quad
\mbox{and}
\quad
        \sum_k\langle r_k,s_{k-j}\rangle
        =\frac1{2\pi}\int_0^{2\pi}e^{-ij\theta}
        R(e^{i\theta})^*S(e^{i\theta})\,d\theta .
\]
The time-domain expression is real, so after taking the real part of the integral we obtain \eqref{eq:causal-hard-frequency}. Because the zero-extended sequences are supported on $\{0,\dots,N\}$, their
discrete-time Fourier transforms are the finite trigonometric polynomials:
\[
        Y(e^{i\theta})=\sum_{k=0}^{N}y_k\,e^{-ik\theta},\qquad
        U(e^{i\theta})=\sum_{k=0}^{N}u_k\,e^{-ik\theta}, %
\]
\[
        R(e^{i\theta})=\sum_{k=0}^{N}r_k\,e^{-ik\theta},\qquad
        S(e^{i\theta})=\sum_{k=0}^{N}s_k\,e^{-ik\theta},
\]
with $r_k=u_k-\mu y_k$ and $s_k=Ly_k-u_k$; by linearity, $R=U-\mu Y$ and $S=LY-U$.
Then, we can express the integrand
of \eqref{eq:causal-hard-frequency}, at
each frequency $\theta\in[0,2\pi)$:
\begin{eqnarray}
        2\Ree\!\left[\bigl(1-h(e^{i\theta})\bigr)R(e^{i\theta})^*S(e^{i\theta})\right]
        =\begin{bmatrix} Y(e^{i\theta})\\ U(e^{i\theta})\end{bmatrix}^{\!*}
        \mgb{\Pi_{\mu,L,d}}\!\bigl(h(e^{i\theta})\bigr)
        \begin{bmatrix} Y(e^{i\theta})\\ U(e^{i\theta})\end{bmatrix},
\end{eqnarray}
which is the claimed Hermitian form. 
This completes the proof.
\end{proof}

The \mgb{Hermitian} matrix-valued \mgb{function} $\mgb{\Pi_{\mu,L,d}}(h(e^{i\theta}))$ is called a \emph{multiplier}: a
Hermitian weight assigned to each frequency $\theta\in[0,2\pi)$, in the sense of the
integral-quadratic-constraint framework~\cite{MegretskiRantzer1997}. The specific
multiplier used here, generated by the non-negative averaging weight $h$, is a
\emph{Zames--Falb multiplier} \cite{ZamesFalb1968,discrete-time-zames-falb-multiplier}. \mgb{Its scalar matrix representation \eqref{eq:Pi} coincides with
the standard Zames--Falb form; see
\cite[Proposition~14, Eq.~(6)]{BoczarLessardRecht2015}.}

\mgb{To certify exponential convergence rates, we use
rate-weighted Zames--Falb IQCs.
\mgd{Lessard, Recht, and
Packard~\cite[Lemma~10]{LessardRechtPackard2016}
establish a hard time-domain version involving a one-step lag,
while Boczar, Lessard, and
Recht~\cite[Theorem~16]{BoczarLessardRecht2015}
give a frequency-domain formulation allowing non-negative
combinations of multiple lags.
For gradients of smooth strongly convex functions, we
present the version with finitely many lags in both
time-domain and frequency-domain forms and provide a
self-contained proof of its validity on every finite horizon.}
We then specialize to uniform weights on the preceding
$\tau$ iterations.}

\begin{lemma}[rate-weighted hard Zames--Falb IQC]\label{lem:rhohard}
Let \mgd{{$f\in\Sml$ with $0<\mu\le L$ and $\nabla f(0)=0$, and let $0<\rho\le1$. Given a}} %
finite trajectory of point--gradient pairs
$(y_k,u_k)=(y_k,\nabla f(y_k))$, $0\le k\le N$, extend it to all integers by setting
$(y_k,u_k)=(0,0)$ for $k\notin\{0,\dots,N\}$, and put $\widehat y_k=\rho^{-k}y_k$,
$\widehat u_k=\rho^{-k}u_k$. \mgd{Fix an integer $J\ge1$.} Let
\begin{equation}\label{eq:hrate-general}
        h_\rho(z)=\sum_{j=1}^{J}\rho^{j}h_jz^{-j},
        \qquad h_j\ge0,\qquad \sum_{j=1}^{J}h_j\le1 .
\end{equation}
Then the weighted, zero-extended pair $(\widehat y,\widehat u)$ satisfies the hard
IQC generated by $\mgb{\Pi_{\mu,L,d}}(h_\rho)$, with $\mgb{\Pi_{\mu,L,d}}(\cdot)$ as in
\eqref{eq:Pi}: writing $\widehat Y,\widehat U$ for the Fourier transforms of
$\widehat y,\widehat u$,
\[
        \frac1{2\pi}\int_0^{2\pi}
        \begin{bmatrix}\widehat Y(e^{i\theta})\\ \widehat U(e^{i\theta})\end{bmatrix}^{\!*}
\mgb{\Pi_{\mu,L,d}}\!\bigl(h_\rho(e^{i\theta})\bigr)
        \begin{bmatrix}\widehat Y(e^{i\theta})\\ \widehat U(e^{i\theta})\end{bmatrix}
        d\theta\ge0 .
\]
In particular, taking $J=\tau$ and $h_j=1/\tau$ gives $h_\rho=h_{\rho,\tau}$, where
\[
        h_{\rho,\tau}(z)=\frac1\tau\sum_{j=1}^{\tau}\rho^{j}z^{-j}.
\]
\end{lemma}

\begin{proof} \mgd{The case $L=\mu$ gives zero sequences $y_k=u_k=0$ for all $k$ where the claims hold trivially. So, assume $L>\mu$.} Set $\widehat r_k=\rho^{-k}r_k$ and $\widehat s_k=\rho^{-k}s_k$. By Parseval, it suffices to prove the time-domain inequality
\begin{equation}\label{eq:rho-causal-time}
        2\sum_k\langle\widehat r_k,\widehat s_k\rangle
        -2\sum_{j=1}^{J}\rho^j h_j\sum_k
        \langle\widehat r_k,\widehat s_{k-j}\rangle
        \ge0.
\end{equation}
For a fixed lag $j$ define $\lambda_k=\rho^{-2k}$ for $k\ge0$ and $\lambda_k=0$ for $k<0$.  The zero extension makes all sums finite.  As in the proof of Lemma~\ref{lem:hard}, the subgradient inequality gives
\[
        \langle s_k,r_k-r_{k+j}\rangle\ge\psi(r_k)-\psi(r_{k+j}).
\]
Multiplying by $\lambda_k$ and summing over $k\ge0$ yields
\begin{align*}
        \sum_{k\ge0}\lambda_k\langle s_k,r_k-r_{k+j}\rangle
        &\ge
        \sum_{\ell\ge0}\bigl(\lambda_\ell-\lambda_{\ell-j}\bigr)\psi(r_\ell)\ge0,
\end{align*}
where $\lambda_{\ell-j}=0$ for $\ell<j$ and, for $\ell\ge j$,
$\lambda_\ell-\lambda_{\ell-j}=\rho^{-2\ell}(1-\rho^{2j})\ge0$.  Hence
\begin{equation}\label{eq:weighted-causal-lag}
        \sum_{k\ge0}\lambda_k\langle s_k,r_k\rangle
        \ge
        \sum_{k\ge0}\lambda_k\langle s_k,r_{k+j}\rangle .
\end{equation}
On the other hand, \sa{similar to $\sum_{k\ge0}\langle r_{k},s_{k-j}\rangle =\sum_{k\ge0}\langle r_{k+j},s_k\rangle$, we also have}
\[
        \rho^j\sum_k\langle\widehat r_k,\widehat s_{k-j}\rangle
        =\sum_{k\ge0}\lambda_k\langle r_{k+j},s_k\rangle.
\]
In addition, 
\[
        \sum_k\langle\widehat r_k,\widehat s_k\rangle
        =\sum_{k\ge0}\lambda_k\langle r_k,s_k\rangle .
\]
Then, by \eqref{eq:weighted-causal-lag}, each \sa{lag $j$} satisfies
\[
        \sum_k\langle\widehat r_k,\widehat s_k\rangle
        -\rho^j\sum_k\langle\widehat r_k,\widehat s_{k-j}\rangle\ge0.
\]
Multiplying this inequality by the non-negative weight $2h_j$ and summing over
$j=1,\dots,J$, then using $\sum_{j=1}^{J}h_j\le1$ together with
$\sum_k\langle\widehat r_k,\widehat s_k\rangle\ge0$, yields
\[
        2\sum_k\langle\widehat r_k,\widehat s_k\rangle
        -2\sum_{j=1}^{J}\rho^{j}h_j \sum_k\langle\widehat r_k,\widehat s_{k-j}\rangle\ge0 ,
\]
which \sa{yields} \eqref{eq:rho-causal-time}. The Fourier conversion is the same as in Lemma~\ref{lem:hard}, now applied to the weighted sequences.  If $\widehat Y$ and $ \widehat U$ are the Fourier transforms of $\widehat y$ and $\widehat u$, then \eqref{eq:rho-causal-time} is equivalent to
\[
        \frac1{2\pi}\int_0^{2\pi}
        \begin{bmatrix}\widehat Y\\ \widehat U\end{bmatrix}^{\!*}
        \mgb{\Pi_{\mu,L,d}}(h_\rho(e^{i\theta}))
        \begin{bmatrix} \widehat Y\\ \widehat U\end{bmatrix}
        d\theta\ge0.
\]
This is the desired hard IQC for the rate-weighted sequences. Finally, $h_\rho$ becomes $h_{\rho,\tau}$ for $J=\tau$ and $h_j = 1/\tau$.
\end{proof}

\section{Proving Theorem \ref{thm:main-nonexplicit-constants}: Accelerated Rates for \DREAM}\label{sec:master}%
\mgd{We prove Theorem~\ref{thm:main-nonexplicit-constants} by combining the gradient inequalities and the linear-recursion representation developed in Section~\ref{sec:certificate}},
Our analysis begins by decomposing the gradient as
    \[
    \nabla f(x) =mx+\sa{v(x)},\quad \mbox{i.e.,}\quad \sa{v(x):=\nabla f(x)-mx,}
    \]
    where $m := \frac{1}{2}(1+q^{-2})$ is the midpoint of the slope interval $[1,q^{-2}]$, and $v$ is the residual deviation from this constant slope. Recall that without loss of generality we take \mgb{$\mu=1$} %
    and $\nabla f(0)=0$
    in our analysis.\looseness=-1

    Absorbing the linear term $mx$ into the GMM dynamics yields a linear recursion, which we call the \textit{midpoint linear recursion}. This recursion is obtained by setting $v \equiv 0$. To certify contraction at a target rate $\rho$, we pass to rate-weighted coordinates, so that \mgd{a geometric bound with rate $\rho$ follows from boundedness of the weighted iterates}. %
    \mgd{Our objectives are twofold: (i) establish that the state-transition matrix associated with the rate-weighted midpoint linear recursion is Schur stable; and (ii) control the residual nonlinearity $v$ via a Zames--Falb IQC and show that, whenever the corresponding frequency-domain inequality holds, the full rate-weighted iteration is stable.}
Then, stability of the rate-weighted system  implies linear convergence of the original GMM iteration at rate $\rho$.

\begin{definition}[Midpoint \mgd{linear recursion}]\label{def:midpoint-plant}
Let $m=\tfrac12(1+q^{-2})$ denote the midpoint of the slope interval $[\mu,L]=[1,q^{-2}]$.
The \mgd{\emph{midpoint linear recursion}} is the linear time-invariant recursion obtained from delayed
GMM~\eqref{eq:gmm} by replacing the gradient with the constant slope $m$, that is, by
setting $u_k=m\,y_k$:
\begin{equation}\label{eq:midpoint-plant}
        x_{k+1}=(1+\beta)x_k-\beta x_{k-1}-\alpha m\,y_{k-\tau},
        \qquad
        y_k=(1+\nu)x_k-\nu x_{k-1}.
\end{equation}
This is exactly delayed GMM run on a quadratic whose curvature equals the
midpoint $m$. As we discussed in Section \ref{sec:quadratic}, its stability is governed by the characteristic polynomial
$P_{\tau,\delta_m}$ with gain $\delta_m:=\alpha m$, \sa{i.e.,}
\begin{equation}\label{eq:midpoint-charpoly}
        P_{\tau,\delta_m}(z)=z^{\tau}(z-1)(z-\beta)+\delta_m\bigl[(1+\nu)z-\nu\bigr].
\end{equation}
\end{definition}
Before presenting the proof of Theorem \ref{thm:main-nonexplicit-constants}, we first discuss sufficient conditions that would imply a given target convergence rate $\rho$. %
\subsection{\mg{A Convergence Rate Lemma: Sufficient Conditions for Linear Convergence}}\label{subsec-rate-lemma}
\hamzatext{Recall from Figure \ref{fig:lure-separation} that before closing the gradient feedback loop,} the linear part of~\eqref{eq:gmm} maps the gradient sequence $(u_k)$ to the query-point
sequence $(y_k)$. Under the $z$-transform this map can be expressed as a single rational function, the \emph{transfer function}
$g_{\tau,q}$. \sa{Indeed,} writing $Y(z),U(z)$ for the transforms of $(y_k),(u_k)$, under \sa{\textit{zero initial conditions}} \mg{($x_k=0$ for $k\leq0$ and $u_k=0$ for $k<0$)}, one has
$Y=g_{\tau,q}\,U$. To certify contraction \emph{at rate $\rho$} we pass to the rate-weighted coordinates
$\widehat{(\cdot)}_k=\rho^{-k}(\cdot)_k$, in which a sequence decaying like $\rho^k$
becomes bounded. %

\mgd{The following convergence-rate lemma builds on a
well-established IQC approach to certifying exponential
convergence. This approach combines stability of a
suitable linear subsystem with a strict frequency-domain
inequality (FDI) that couples the linear dynamics with
an IQC satisfied by the gradient mapping
\cite{MegretskiRantzer1997,BoczarLessardRecht2015,
LessardRechtPackard2016,Seiler2015}.
In our formulation, hypothesis \textbf{(H1)} requires
stability of the midpoint linear subsystem after rate
weighting, while hypothesis \textbf{(H2)} imposes the
corresponding strict FDI. Although these conditions are compactly stated, their
analytical verification for \DREAM{} requires substantial
work. %
In particular, we prove the lemma in the next subsection and carry out
the analytical verification of \textbf{(H1)} and
\textbf{(H2)} in
Sections~\ref{sec:stability}--\ref{sec:limit}.}
\begin{lemma}[Convergence Rate Lemma]\label{lemm:ratecert}
Consider \DREAM, namely delayed GMM~\eqref{eq:gmm} with the parameters of Theorem~\ref{thm:main-nonexplicit-constants}. \mgb{Let $\rho=1-cq/M \in (0,1)$ with $c\in(0,\tfrac23)$ as in \eqref{eq:params}} be the target rate  and let $m=\tfrac12(1+q^{-2})$ be the midpoint of
the slope interval $[\mu,L]=[1,q^{-2}]$. Assume \mgd{$0<q<M/(2M+1)$}. Suppose the following two conditions hold.
\begin{itemize}[leftmargin=2.6em]
\item[\textbf{\textup{(H1)}}] \emph{Midpoint-pole hypothesis (stability of the linear surrogate).} \mgd{Every root of the midpoint characteristic polynomial
$P_{\tau,\delta_m}$ defined in \eqref{eq:midpoint-charpoly} associated with the midpoint
linear recursion of Definition~\ref{def:midpoint-plant},
where $\delta_m=\alpha m$, lies in the open disk
$\{|z|<\rho\}$. %
} %
\item[\textbf{\textup{(H2)}}] \emph{Strict frequency-domain inequality (FDI).} With $z=\rho e^{i\theta}$,
the Hermitian form
\begin{equation}\label{eq:fdi}
\Phi_{\tau,q}(\theta):=
\begin{bmatrix}g_{\tau,q}(z)\\[2pt]1\end{bmatrix}^{\!*}
\Pi_{1,q^{-2}}\bigl(h_{\rho,\tau}(e^{i\theta})\bigr)
\begin{bmatrix}g_{\tau,q}(z)\\[2pt]1\end{bmatrix} < 0 \quad \mbox{for every} \quad \theta\in[0,2\pi).
\end{equation}
Here $g_{\tau,q}$ is the transfer function~\eqref{eq:g} of the linear part
of~\eqref{eq:gmm}, and $\Pi_{1,q^{-2}}(h_{\rho,\tau})$ is the Zames--Falb
multiplier~\eqref{eq:Pi} evaluated at the rate-weighted weight $h_{\rho,\tau}$ of
Lemma~\ref{lem:rhohard}.
\end{itemize}
\mg{Then the iterates of delayed GMM converge linearly with rate \(\rho\), \sa{i.e.,} there exists a finite constant \(\mgd{C_{\mathrm{rate}}}>0\) such that
\[
    \norm{x_k-x_*}\le \mgd{C_{\mathrm{rate}}}\,\rho^k,
    \qquad k\ge 0.
\]
Moreover, the rate is uniform over the function class \(\mathcal S_{1,q^{-2}}\). Specifically, for fixed algorithmic parameters, the same value of \(\rho\) applies to every objective \(f\in\mathcal S_{1,q^{-2}}\) with minimizer \(x_*\) and every finite initial history of the recursion, with \(\mgd{C_{\mathrm{rate}}}\) depending only on the fixed algorithmic parameters and on the initial history relative to \(x_*\).}
\end{lemma}
\subsection{Proof of the Convergence Rate Lemma}\label{subsec:master-proof}
This subsection proves the convergence rate lemma (Lemma~\ref{lemm:ratecert}). The proof is an energy argument, assembled from a small number of supporting lemmas, which we state here at the point of use; their proofs are collected in Appendix~\ref{subsec:lemma-proofs} so as not to interrupt the argument. The proof will assume the hypotheses \textbf{(H1)} and \textbf{(H2)} hold; these hypotheses themselves are verified afterwards, in Sections~\ref{sec:stability}--\ref{sec:limit}.

\sa{\mgd{We recall that} treating $u_k$ as an \mgd{arbitrarily prescribed} input to the linear dynamics in \eqref{eq:gmm}, the corresponding input-output map from $u_k$ to $y_k$ has transfer function $g_{\tau,q}$; that is, $Y(z)=g_{\tau,q}(z)U(z)$ under zero initial conditions; see \eqref{eq:g}.} If we introduce the rate-weighted sequences $\widehat y_k=\rho^{-k}y_k$ and
$\widehat u_k=\rho^{-k}u_k$, the following lemma, \sa{using the} basic properties of $z$-transforms, shows that the $z$-transforms $\widehat Y(z)$ and $\widehat U(z)$ of the weighted sequences are related through the rate-normalized transfer function $g_{\tau,q}(\rho z)$.
\begin{lemma}[rate-weighted \mgd{transfer function}]\label{lem:rate-weighted-plant}
Introduce the rate-weighted sequences $\widehat y_k=\rho^{-k}y_k$ and
$\widehat u_k=\rho^{-k}u_k$, and let $\widehat Y,\widehat U$ denote their
$z$-transforms. We have %
\mgb{$\widehat Y(z)=g_\rho(z)\,\widehat U(z)$}
under zero initial conditions, %
where
\beq
 \mgb{g_\rho(z)=g_{\tau,q}(\rho z)}.
        \label{def-G-open-loop}
\eeq
\mgb{The corresponding transfer matrix is $G_\rho(z)=g_\rho(z)I_d$;
thus $\widehat Y=G_\rho\widehat U$. The scalar \mgb{transfer} function $g_\rho(z)$
equals \mgb{$g_{\tau,q}(\rho z)$}}, \sa{i.e.,} %
rate weighting divides every pole of $g_{\tau,q}$ by $\rho$.\looseness=-1
\end{lemma}

\mgb{As in Figure \ref{fig:lure-separation}, we think of GMM as an interconnection of a linear system and a gradient nonlinearity; and we call the linear part \mgd{the \emph{open-loop} system}, whereas the whole GMM dynamics is referred to as \mgd{the \mgd{closed-loop} system obtained by imposing $u_k=\nabla f(y_k)$}.} Because rate weighting divides every pole of $g_{\tau,q}$ by $\rho$, the poles of \mgb{$g_{\rho}$}
may lie outside the unit circle; equivalently, the rate-weighted characteristic
polynomial need not be Schur. \mgd{
This poses a difficulty, since our IQC-based
analysis requires the associated linear subsystem
to be stable.} To overcome \sa{this issue}, we split the gradient into its midpoint slope and
a remainder, $\nabla f(y_k)=m\,y_k+v_k$, \sa{i.e., $v_k=v(y_k)$,} and pass to the rate-weighted sequences
$\widehat y_k=\rho^{-k}y_k$ and $\widehat v_k=\rho^{-k}v_k$. 
\mgd{We now describe the weighted midpoint-transformed
dynamics. %
}
 
\begin{lemma}[weighted midpoint dynamics]\label{lem:weighted-realization} 
\mgd{Let $\rho>0$, and let
$x_{-\tau-1},\ldots,x_0$ be the prescribed
initialization. Define the rate-weighted
variables
\beq
\widehat x_k=\rho^{-k}x_k,
\qquad
\widehat y_k=\rho^{-k}y_k,
\qquad
\widehat v_k=\rho^{-k}v_k, \label{eq-weighted-variables}
\eeq
for all indices at which the corresponding
quantities are defined, including the initialization required by the delayed recursion}
\[
        x_{k+1}=(1+\beta)x_k-\beta x_{k-1}-\alpha u_{k-\tau},
        \qquad
        y_k=(1+\nu)x_k-\nu x_{k-1},
\]
\sa{where $ u_k=my_k+v_k$.}
Then the weighted midpoint-transformed dynamics satisfy
\begin{align}\label{eq:weighted-realization-x-app}
\widehat x_{k+1}
={}&\frac{1+\beta}{\rho}\widehat x_k-
    \frac{\beta}{\rho^2}\widehat x_{k-1}
    -\frac{\alpha m(1+\nu)}{\rho^{\tau+1}}\widehat x_{k-\tau}
    +\frac{\alpha m\nu}{\rho^{\tau+2}}\widehat x_{k-\tau-1}
    -\frac{\alpha}{\rho^{\tau+1}}\widehat v_{k-\tau},\\
\widehat y_k
={}&(1+\nu)\widehat x_k-\frac{\nu}{\rho}\widehat x_{k-1}.\label{eq:weighted-realization-y-app}
\end{align}
\end{lemma}

\mg{The following lemma says that stability of the weighted midpoint-transformed dynamics provided in the previous lemma is equivalent to the hypothesis \textbf{(H1)}}. \mgd{The restriction $q<M/(2M+1)$ imposed in Lemma~\ref{lemm:ratecert} is used in the following lemma to establish that the numerator and denominator of $\widetilde g_\rho$ are \emph{coprime}, i.e., they have no common non-constant polynomial factor.}
\begin{lemma}[stability of the weighted-midpoint transformed dynamics]
\label{lem:midpoint-schur}
Recall from Lemma~\ref{lem:rate-weighted-plant} that \mgb{$g_\rho$} %
is the transfer function from
the rate-weighted gradient input to the rate-weighted query point, \sa{i.e.,}
$\widehat Y=\mgb{g_\rho} \widehat U$.
Write \mgb{$g_\rho=N/D$} %
as a ratio of polynomials, where \mgb{
$N(z)=-\alpha [(1+\nu)\rho z-\nu]$ and
$D(z)=(\rho z)^\tau(\rho z-1)(\rho z-\beta)$}. %
We next consider the midpoint change of variables $\widehat u=m\widehat y+\widehat v$, with
residual input $\widehat v_k=\widehat u_k-m\widehat y_k$. %
\mgd{For the zero-initial-state response, eliminating
$\widehat u$ yields the
\emph{midpoint-transformed transfer function}
\mgb{$\widetilde g_\rho$} from the residual input
$\widehat v$ to the query point $\widehat y$:}
\mgb{\beq \widehat Y=\widetilde g_\rho\,\widehat V,
        \qquad
        \widetilde g_\rho=\frac{g_\rho}{1-mg_\rho}
        =\frac{N}{D-mN}.
        \label{eq-midpoint-tf}
\eeq}%
The denominator is the rate-scaled midpoint characteristic polynomial:
$D-mN=P_{\tau,\delta_m}(\rho\,\cdot)$ with $\delta_m=\alpha m$, so its roots are the
points $r/\rho$ with $P_{\tau,\delta_m}(r)=0$. Assume $0<q<M/(2M+1)$ (with
$M=2\tau+1$, $\beta=1-2q/M$). Then the numerator $N$ and denominator $D-mN$ are
coprime, so \mgb{$\widetilde g_\rho$} has no pole--zero cancellation and its poles are \emph{exactly}
the roots of $D-mN$, i.e., the points $r/\rho$ with $P_{\tau,\delta_m}(r)=0$. Moreover, letting $A$ be the state matrix of the
finite-dimensional realization
\eqref{eq:weighted-realization-x-app}--\eqref{eq:weighted-realization-y-app} with state
\[
        \xi_k=
        \begin{bmatrix}
        \widehat x_k&
        \widehat x_{k-1}&
        \cdots&
        \widehat x_{k-\tau-1}&
        \widehat v_{k-1}&
        \cdots&
        \widehat v_{k-\tau}
        \end{bmatrix}^{\mathsf T}\mgb{\in\mathbb R^{(2\tau+2)d}},
\]\mgb{the matrix $A$} has the spectrum $\sigma(A)=\mathcal P(\widetilde g_\rho)\cup\{0\}$ , where
$\mathcal P(\widetilde g_\rho)$ denotes the poles of \mgb{${\widetilde g}_\rho$}. \mgd{Consequently, $\mgb{\widetilde g_\rho}$ is stable
and the state-transition matrix $A$ is Schur
stable if and only if every root of
$P_{\tau,\delta_m}$ lies in $\{|z|<\rho\}$,
that is, if and only if the midpoint-pole
hypothesis \textbf{(H1)} holds.}
\end{lemma}

In the next lemma, we transport the strict FDI \textbf{(H2)} from \mgd{the original linear subsystem} to the \mgd{Schur-stable midpoint linear system};
this transport is a \emph{congruence} of the underlying
quadratic form, an operation we first recall. A congruence is the way a Hermitian quadratic form changes under a linear change of
variables: if a form is written $x^{*}\Pi x$ for a Hermitian matrix $\Pi$ and one
substitutes $x=Ty$ with $T$ invertible, then
\[
        x^{*}\Pi x=(Ty)^{*}\Pi(Ty)=y^{*}\bigl(T^{*}\Pi T\bigr)y ,
\]
so in the new variable the form has the congruent matrix $T^{*}\Pi T$, and the value of
the form---hence its sign---is unchanged. %
Here the change of variables is the midpoint
substitution $\widehat u=m\widehat y+\widehat v$, which relates the old IQC variables
$(\widehat y,\widehat u)$ and the new variables $(\widehat y,\widehat v)$ by
\[
        \begin{bmatrix}\widehat y\\ \widehat u\end{bmatrix}
        =\mgb{T_{m,d}}\begin{bmatrix}\widehat y\\ \widehat v\end{bmatrix},
        \qquad
\mgb{T_{m,d}:=T_m\otimes I_d} \quad \mbox{where} \quad T_m=\begin{bmatrix}1&0\\ m&1\end{bmatrix},
\]
and the same relation holds for the Fourier transforms. 

\mgb{Let $\Pi(z):=\Pi_{1,q^{-2}}(h_{\rho,\tau}(z))$ denote the
$2\times2$ multiplier from \eqref{eq:Pi} and define}
\mgb{\beq\label{eq:midpoint-multiplier-lifts}
 \begin{aligned}
 \widetilde\Pi(z)&:=T_m^*\Pi(z)T_m,\qquad
 \Pi_d(z):=\Pi(z)\otimes I_d,\\
 \widetilde\Pi_d(z)&:=\widetilde\Pi(z)\otimes I_d
 =T_{m,d}^*\Pi_d(z)T_{m,d}.
 \end{aligned}
 \eeq
 }
\mgb{On the unit circle, \mgd{the corresponding vector-valued Fourier transforms satisfy the identity}%
\[
\begin{bmatrix}\widehat Y\\\widehat U\end{bmatrix}^{\!*}
 \Pi_d
 \begin{bmatrix}\widehat Y\\\widehat U\end{bmatrix}
 =
 \begin{bmatrix}\widehat Y\\\widehat V\end{bmatrix}^{\!*}
 \widetilde\Pi_d
 \begin{bmatrix}\widehat Y\\\widehat V\end{bmatrix}.
\]
}
\mgb{We have the scalar transfer functions
$g=g_\rho(e^{i\theta})$ and
$\widetilde g=\widetilde g_\rho(e^{i\theta})=g/(1-mg)$.
Whenever $g$ is finite and $1-mg\ne0$,
\[
T_m\begin{bmatrix}\widetilde g\\1\end{bmatrix}
 =\frac{1}{1-mg}\begin{bmatrix}g\\1\end{bmatrix}.
\]
Consequently, with both multipliers evaluated at $e^{i\theta}$,
\begin{equation}\label{eq:fdi-equivalence}
 \begin{bmatrix}\widetilde g\\1\end{bmatrix}^{\!*}
 \widetilde\Pi
 \begin{bmatrix}\widetilde g\\1\end{bmatrix}
 =\frac{1}{|1-mg|^2}
 \begin{bmatrix}g\\1\end{bmatrix}^{\!*}
 \Pi
 \begin{bmatrix}g\\1\end{bmatrix}.
\end{equation}
Here, $1-mg$ is a scalar.
For \mgd{the corresponding $d$-dimensional transfer matrix}, the Kronecker product gives
\begin{equation}\label{eq:fdi-scalar-lift}
 \begin{bmatrix}\widetilde G_\rho\\I_d\end{bmatrix}^{\!*}
 \widetilde\Pi_d
 \begin{bmatrix}\widetilde G_\rho\\I_d\end{bmatrix}
 =
 \left(
 \begin{bmatrix}\widetilde g_\rho\\1\end{bmatrix}^{\!*}
 \widetilde\Pi
 \begin{bmatrix}\widetilde g_\rho\\1\end{bmatrix}
 \right)I_d,
\end{equation}
where $\widetilde G_\rho:=\widetilde g_\rho I_d$. Thus a scalar margin $-\eta$ becomes the matrix margin $-\eta I_d$,
without a dimension-dependent factor: \mgd{a scalar upper bound $-\eta$ in \eqref{eq:fdi-equivalence} gives the matrix upper bound $-\eta I_d$} in \eqref{eq:fdi-scalar-lift}.}%

\begin{lemma}[FDI equivalence]\label{lem:fdi-congruence} \mg{Consider \DREAM, namely delayed GMM~\eqref{eq:gmm} with the parameters of Theorem~\ref{thm:main-nonexplicit-constants}} \mgb{where $0<\rho<1$ and $\rho\neq \beta$}.
\sa{Under} the midpoint-pole hypothesis \textbf{(H1)}, %
\mgb{$g_\rho$ is finite on the unit circle} and
$1-m \mgb{g_\rho}(e^{i\theta})\ne0$ for every $\theta\in[0,2\pi)$. In addition, the scalar \mgd{transformed transfer function} \mgb{$\widetilde g_\rho=g_\rho/(1-m g_\rho)$} is well defined on the unit circle and the scalar factor \mgb{$|1-m g|^2$} defined
in~\eqref{eq:fdi-equivalence} \mgb{with $g=g_\rho$} is finite and strictly positive. \mgb{Consequently the strict
FDI \sa{in}~\eqref{eq:fdi} is equivalent to a strict FDI for $\widetilde g_\rho$ with multiplier $\widetilde \Pi$, and hence, by \eqref{eq:fdi-scalar-lift} to the lifted matrix FDI for
$\widetilde G_\rho$ with multiplier $\widetilde\Pi_d$.}\looseness=-1
\end{lemma}

\mg{\mg{We next show that a strict FDI yields a positive-definite quadratic function $V$ of a suitably augmented state that satisfies the strict \emph{dissipation inequality}~\eqref{eq:positive-storage-kyp} below and thereby controls the iterates. This is a standard step in the IQC analysis of optimization algorithms~\cite{LessardRechtPackard2016} in its dissipativity form~\cite{Seiler2015,HuLacerdaSeiler2017,hu2017dissipativity}; the lemma below specializes it to our multipliers and systems and identifies the one-step quadratic quantity $\sigma_k$ explicitly. In the language of dissipativity theory~\cite{Willems1972}, $V$ is a \emph{storage function}, the energy stored in the system, and $-\sigma_k$ the \emph{supply rate}, a prescribed quadratic function of the input $\widehat v_k$ and output $\widehat y_k$ (and of their recent past, through the multiplier memory) representing the energy delivered to the system at step $k$; the inequality says that the stored energy can grow in one step by at most the energy delivered, minus a strict margin. Thus, unlike a Lyapunov function, $V$ need not decrease at every step; what keeps it bounded is the hard IQC of Lemma~\ref{lem:rhohard}, by which the gradient nonlinearity delivers no net energy over any finite horizon, $\sum_{k=0}^{N}\sigma_k\ge0$, so that summing \eqref{eq:positive-storage-kyp} bounds $V$ and hence the iterates.}}
\begin{lemma}[Existence of quadratic storage functions]
\label{lem:positive-storage-kyp}
Let  \mgb{$\mathcal G(z)\in\mathbb C^{d\times d}$} be a proper rational
transfer matrix \mgd{
from the $d$-dimensional rate-weighted residual
input $\widehat v_k$ to the rate-weighted
query point $\widehat y_k$, with realization
}
\begin{equation}\label{eq:plant-realization}
    \xi_{k+1}=A\xi_k+B\widehat v_k,
    \qquad
    \widehat y_k=C\xi_k+D\widehat v_k,
\end{equation}
\mgd{where $\xi_k\in\mathbb R^n$,
$\widehat v_k,\widehat y_k\in\mathbb R^d$,
$A\in\mathbb R^{n\times n}$ is Schur,
$B\in\mathbb R^{n\times d}$,
$C\in\mathbb R^{d\times n}$, and
$D\in\mathbb R^{d\times d}$. Thus
$\mathcal G(z)=C(zI_n-A)^{-1}B+D$}.
\mgd{Fix an integer $J\ge1$.} Let
\[
    h(z)=\sum_{j=1}^{J} c_j z^{-j},\quad\mbox{\sa{for some given}}
    \quad
    c_j\ge 0,
    \quad
    \mbox{\sa{such that}}
    \quad
    s_h\mathrel{:=}\sum_{j=1}^{J} c_j<1,
\]
be a finite-memory linear filter defined through its transfer function, so that
when it acts on the sequence $\widehat{s}_k$, it produces the output sequence %
$    (h\widehat s)_k:=\sum_{j=1}^{J} c_j\widehat s_{k-j}$.
Fix $a>0$ and define
\[
    \widehat r_k:=a\widehat y_k+\widehat v_k,
    \qquad
    \widehat s_k:=a\widehat y_k-\widehat v_k,
\]
together with 
\begin{equation}\label{eq:kyp-midpoint-supply}
    \sigma_k
    :=
    2\Rea\langle \widehat r_k,\widehat s_k\rangle
    -
    2\sum_{j=1}^{J} c_j
    \Rea\langle \widehat r_k,\widehat s_{k-j}\rangle .
\end{equation}

Define
\[
    \mgb{\Psi(z):=
    \begin{bmatrix}
        a(2-h(z)) & h(z)\\
        a h(z) & (2-h(z))
    \end{bmatrix},
    \qquad
    J_0:=
    \begin{bmatrix}
        1&0\\
        0&-1
    \end{bmatrix}},
\]
and
\mgb{\[
\Psi_d(z):=\Psi(z)\otimes I_d,\qquad J_{0,d}:=J_0\otimes I_d,
\]
\[   \widetilde\Pi(z):=\frac{1}{2}\Psi^{\sim}(z)J_0\Psi(z),\qquad
\widetilde\Pi_d(z):=\widetilde\Pi(z)\otimes I_d
    =\frac{1}{2}\Psi_d^{\sim}(z)J_{0,d}\Psi_d(z).
\]}
Assume that, for some $\eta>0$,
\mgb{\begin{equation}\label{eq:kyp-strict-fdi-input-margin}
    \begin{bmatrix}
        \mathcal G(e^{i\theta})\\
        I_d
    \end{bmatrix}^{\!*}
\widetilde\Pi_d(e^{i\theta})
    \begin{bmatrix}
        \mathcal G(e^{i\theta})\\
        I_d
    \end{bmatrix}
    \preceq -\eta I_d,
    \qquad
    0\le\theta<2\pi .
\end{equation}}
\mgd{Augment the state $\xi_k$ with the finite memory of the filter:}%
\[
    \Xi_k:=
    \begin{bmatrix}
        \xi_k\\
        q_k
    \end{bmatrix},
    \qquad
    q_k:=
    \begin{bmatrix}
        \widehat s_{k-1}\\
        \widehat s_{k-2}\\
        \vdots\\
        \widehat s_{k-J}
    \end{bmatrix}.
\]
\mgd{Here $q_k\in{\mathbb R}^{Jd}$ and $\Xi_k\in\mathbb R^{n_a}$,
where $n_a:=n+Jd$}.
Thus $\xi_k$ is the state of \mgb{$\mathcal{G}$}, while $q_k$ stores the past
values $\widehat s_{k-1}, \widehat s_{k-2}, \dots, \widehat s_{k-J}$ needed to evaluate $(h\widehat s)_k$ at time $k$. Then, there exist a positive-definite \mgd{symmetric} matrix \mgd{$X\in\mathbb{R}^{n_a\times n_a}$} and a real scalar $\varepsilon>0$ such that,
with
$    V(\Xi):=\Xi^\top X\Xi,
$
every trajectory of the augmented system satisfies
\begin{equation}\label{eq:positive-storage-kyp}
    V(\Xi_{k+1})-V(\Xi_k)+\sigma_k
    \le
    -\varepsilon
    \bigl(\|\Xi_k\|^2+\|\widehat v_k\|^2\bigr).
\end{equation}
\end{lemma}
\mgd{The complete proof of Lemma~\ref{lem:positive-storage-kyp}, including
positive definiteness of the storage matrix, is given in
Appendix~\ref{subsec:lemma-proofs}}. \sa{We are now ready to prove Lemma~\ref{lemm:ratecert}
by assembling these technical results.}
\subsubsection{\sa{Proof of Lemma~\ref{lemm:ratecert} (Convergence Rate Lemma)}}
As in Section~\ref{subsec:hardiqc}, translate the minimizer to the origin, so that
$x_*=0$ and $\nabla f(0)=0$. Without loss of generality, we take $\mu=1$, write $L=q^{-2}$ and set
\[
        m=\frac{\mu+L}{2}= \frac{1+\sa{q^{-2}}}{2},
        \qquad
        a=\frac{L-\mu}{2} =\frac{\sa{q^{-2}}-1}{2}.
\]
We show that the rate-weighted state is uniformly bounded.  This is sufficient because
$x_k=\rho^k\widehat x_k$. The proof consists of three main steps: $1)$ rate-weighting and analysis of the resulting weighted midpoint dynamics; $2)$ construction of a quadratic storage function; and $3)$ application of hard IQCs as we discuss next.

\emph{Step $1:$ Rate weighting and midpoint transformation.}
\mgd{As in \eqref{eq-weighted-variables} and Lemma \ref{lem:midpoint-schur}, consider}
\[
        \widehat x_k=\rho^{-k}x_k,\qquad
        \widehat y_k=\rho^{-k}y_k,\qquad
        \widehat u_k=\rho^{-k}u_k,
        \qquad
        \widehat v_k=\widehat u_k-m\widehat y_k.
\]
\mgd{For the associated zero-initial-state calculation, Lemma~\ref{lem:rate-weighted-plant} gives the transfer function from $\widehat u$ to
$\widehat y$ as}
\mgb{$g_{\rho}(z)=g_{\tau,q}(\rho z)$, i.e., $\widehat Y= g_\rho \widehat U$}. \sa{Moreover, according to Lemma~\ref{lem:midpoint-schur}, using the midpoint change of variables
$\widehat u=m\widehat y+\widehat v$, \eqref{eq-midpoint-tf}} gives %
\mgb{\[
\sa{\widehat Y} = \widetilde G_\rho(z) \widehat V, \quad \sa{\widetilde G_\rho}(z)=\sa{\tilde g_\rho}(z) I_d \quad \mbox{with} \quad
\widetilde g_\rho(z)= 
\frac{g_\rho(z)}{1-m g_\rho(z)}.
\]}%
\mgd{The transfer matrix $\widetilde G_\rho$ thus
characterizes the zero-initial-state response of
the weighted midpoint system. The corresponding
time-domain realization
\eqref{eq:weighted-realization-x-app}--\eqref{eq:weighted-realization-y-app}
also applies to arbitrary prescribed initial histories.}
\mgb{Hypothesis \textbf{(H1)} and Lemma~\ref{lem:midpoint-schur} imply that \mgb{the dynamics \eqref{eq:weighted-realization-x-app}--\eqref{eq:weighted-realization-y-app}} admit the transfer matrix
$\widetilde G_\rho$} %
\mgb{with the state vector $\xi_k$} 
described in Lemma~\ref{lem:midpoint-schur}, \mgb{and by Lemma \ref{lem:midpoint-schur}, the state matrix $A$ is Schur-stable}.
Moreover, by
Lemma~\ref{lem:fdi-congruence}, hypothesis \textbf{(H2)} gives a strict scalar FDI
for $\widetilde g_\rho$ with the $2\times 2$ multiplier
\[
        \widetilde\Pi
        =T_m^*\Pi_{1,\sa{q^{-2}}}(h_{\rho,\tau})T_m,
        \qquad
        T_m=\begin{bmatrix}1&0\\m&1\end{bmatrix}.
\]
\mgd{For fixed $(\tau,c,q)$, the scalar form is continuous and strictly negative on the unit circle; compactness gives a margin $\eta>0$ uniform in frequency. This compactness argument does not assert a margin uniform in $q$ or $\tau$.}
\mgb{By
\eqref{eq:fdi-scalar-lift}}, 
\mgb{\[
\begin{bmatrix}\widetilde G_\rho(e^{i\theta})\\I_d\end{bmatrix}^{\!*}
\widetilde\Pi_d(e^{i\theta})
\begin{bmatrix}\widetilde G_\rho(e^{i\theta})\\I_d\end{bmatrix}
 \preceq-\eta I_d.
\]}
\mgb{This is the matrix FDI required by Lemma~\ref{lem:positive-storage-kyp} with $\mathcal{G}=\widetilde{G}_\rho$.}

\emph{Step 2: Dissipation inequality.} For
\[
        h_{\rho,\tau}(z)=\frac1\tau\sum_{j=1}^{\tau}\rho^jz^{-j},
        \qquad c_j:=\frac{\rho^j}{\tau},
\]
we have $c_j\ge0$ and $\sum_{j=1}^{\tau}c_j<1$.  Introduce the slope-shifted weighted
sequences
\[
        \widehat r_k:=\widehat u_k-\widehat y_k,
        \qquad
        \widehat s_k:=L\widehat y_k-\widehat u_k.
\]
Because $\widehat u=m\widehat y+\widehat v$ and
$m-1=\sa{q^{-2}}-m=a$ \sa{(since $m$ is the midpoint)}, these satisfy
\[
        \widehat r_k=a\widehat y_k+\widehat v_k,
        \qquad
        \widehat s_k=a\widehat y_k-\widehat v_k.
\]
\mgd{The Schur-stable state-space realization
of $\widetilde G_\rho$ uses the delay state
introduced in Lemma~\ref{lem:midpoint-schur},}
\[
        \xi_k=
        \begin{bmatrix}
        \widehat x_k&
        \widehat x_{k-1}&
        \cdots&
        \widehat x_{k-\tau-1}&
        \widehat v_{k-1}&
        \cdots&
        \widehat v_{k-\tau}
        \end{bmatrix}^{\mathsf T}\in
\mgb{\mathbb{R}^{(2\tau+2)d}}.
\]
To realize the filter $h_{\rho,\tau}$ in state space, adjoin the multiplier
memory
\[
        q_k:=
        \begin{bmatrix}
        \widehat s_{k-1}&
        \widehat s_{k-2}&
        \cdots&
        \widehat s_{k-\tau}
        \end{bmatrix}^{\mathsf T} \mgb{\in \mathbb{R}^{\tau d}},
\]
and define the \mgd{augmented state}%
\begin{equation}\label{eq:Xi-state-rate-proof}
        \Xi_k:=
        \begin{bmatrix}\xi_k\\q_k\end{bmatrix}
        =
        \begin{bmatrix}
        \widehat x_k&\cdots&\widehat x_{k-\tau-1}&
        \widehat v_{k-1}&\cdots&\widehat v_{k-\tau}&
        \widehat s_{k-1}&\cdots&\widehat s_{k-\tau}
        \end{bmatrix}^{\mathsf T}\mgb{\in\mathbb{R}^{(3\tau+2)d}}.
\end{equation}
\mgd{
Thus, $\xi_k$ is the state of the midpoint-transformed linear system, while $\Xi_k$ augments this state with the $\tau$ previous values of $\widehat s$ needed to evaluate the delayed terms in the IQC.} In particular, $\widehat x_k$ is a $d$-dimensional block of both $\xi_k$ and
$\Xi_k$. The two histories should not be conflated: $\xi_0$ contains the prescribed
algorithmic prehistory, including the actual delayed residual inputs
$\widehat v_{-1},\ldots,\widehat v_{-\tau}$, whereas $q_k$ is an auxiliary filter state.
We initialize this multiplier memory at zero, $q_0=0$.  Accordingly, whenever a
negative-indexed $\widehat s_{k-j}$ appears through the multiplier realization below,
it is understood that it is set to zero; this does not
alter the algorithmic prehistory stored in $\xi_0$.

For this augmented realization, define
\begin{equation}\label{eq:pointwise-supply}
\sigma_k
:=2\langle\widehat r_k,\widehat s_k\rangle
-\frac{2}{\tau}\sum_{j=1}^{\tau}\rho^j
 \langle\widehat r_k,\widehat s_{k-j}\rangle,
\end{equation}
where $\widehat s_{k-j}=0$ whenever $k-j<0$, consistently with $q_0=0$.
With $J=\tau$ and $c_j=\rho^j/\tau$, the preceding identities put the transformed FDI
exactly in the notation of Lemma~\ref{lem:positive-storage-kyp}. Indeed, using $m= (\sa{q^{-2}}+1)/2, a=(q^{-2}-1)/2$, it is straightforward to verify that
\mgb{ \[
           \widetilde\Pi_d=T_{m,d}^*\Pi_{1,\sa{q^{-2}},d}(h)T_{m,d}
=\frac12\Psi_d^\sim J_{0,d}\Psi_d,
 \]}
 where \mgd{$h=h_{\rho,\tau}$ and \mgb{$\Psi_d$ and $J_{0,d}$} are}  %
 as in Lemma~\ref{lem:positive-storage-kyp}. Therefore, Lemma~\ref{lem:positive-storage-kyp} is directly applicable and 
there exist a
symmetric real matrix $X\succ0$ and a constant $\varepsilon>0$, depending only on the fixed
algorithmic and certificate parameters, such that, with $V(\Xi):=\Xi^*X\Xi$,
\begin{equation}\label{eq:dissipation}
        V(\Xi_{k+1})-V(\Xi_k)+\sigma_k
        \le
        -\varepsilon\bigl(\|\Xi_k\|^2+\|\widehat v_k\|^2\bigr)
        \qquad(k\ge0).
\end{equation}

\emph{Step 3: Finite-horizon use of the hard IQC.}
Fix $N\ge0$ and sum \eqref{eq:dissipation} from $k=0$ to $N$:
\begin{equation}\label{eq:kyp-finite}
 V(\Xi_{N+1})-V(\Xi_0)+\sum_{k=0}^{N}\sigma_k
 \le
 -\varepsilon\sum_{k=0}^{N}
 \bigl(\|\Xi_k\|^2+\|\widehat v_k\|^2\bigr).
\end{equation}
Next, we verify that the sums $\sum_{k=0}^N \sigma_k$ on the left are non-negative. Take the actual graph
points $(y_k,u_k)=(y_k,\nabla f(y_k))$ for $0\le k\le N$ and pad this finite segment
with the minimizer $(0,0)$ outside $\{0,\ldots,N\}$, as in
Lemma~\ref{lem:zero-extension}.  After rate weighting, let
$\widehat r_k^{(N)}$ and $\widehat s_k^{(N)}$ denote the corresponding slope-shifted
sequences. Lemma~\ref{lem:rhohard}, with $J=\tau$ and $h_j=1/\tau$, gives
\[
2\sum_{k\in\mathbb Z}
 \langle\widehat r_k^{(N)},\widehat s_k^{(N)}\rangle
-\frac{2}{\tau}\sum_{j=1}^{\tau}\rho^j
 \sum_{k\in\mathbb Z}
 \langle\widehat r_k^{(N)},\widehat s_{k-j}^{(N)}\rangle
\ge0.
\]
For $0\le k\le N$, the padded sequences agree with the actual weighted trajectory;
for $k<0$ or $k>N$, $\widehat r_k^{(N)}=0$.  Consequently the all-integer quadratic
form above is exactly the horizon sum in \eqref{eq:pointwise-supply} (the terms
with $k-j<0$ vanish by the zero padding, exactly as imposed by the zero multiplier
memory). \sa{Thus, we have}
\begin{equation}\label{eq:hard-finite-use}
        \sum_{k=0}^{N}\sigma_k\ge0.
\end{equation}
Since $X\succ0$, we \sa{also} have $V(\Xi_{N+1})\ge0$.  Combining
\eqref{eq:kyp-finite} and \eqref{eq:hard-finite-use} therefore yields
\begin{equation}\label{eq:weighted-energy-bound}
        \sum_{k=0}^{N}
        \bigl(\|\Xi_k\|^2+\|\widehat v_k\|^2\bigr)
        \le \varepsilon^{-1}V(\Xi_0),
        \qquad N\ge0.
\end{equation}
In particular, for every $k\ge0$,
$\|\Xi_k\|\le \sqrt{V(\Xi_0)/\varepsilon}$. The vector $\widehat x_k$ is a
\mgb{$d$-dimensional block} of \mgd{$\Xi_k$}; \sa{hence,}
\[
        \|\widehat x_k\|
        \le \sqrt{V(\Xi_0)/\varepsilon}
        \qquad(k\ge0).
\]
Therefore, in the translated coordinates, \sa{we get}
\beq
        \|x_k\|
        \le \mgd{C_{\mathrm{rate}}}\rho^k,
        \qquad
        \mgd{C_{\mathrm{rate}}}:=\sqrt{V(\Xi_0)/\varepsilon}.\label{eq-C-rate}
\eeq
The certificate $(X,\varepsilon)$ is independent of the particular
$f\in\mathcal S_{1,\sa{q^{-2}}}$.  It remains only to note that the initial augmented state is
uniformly controlled by the prescribed finite history.  Indeed, $\nabla f(0)=0$ and
\sa{$q^{-2}$}-smoothness imply, for the stored initialization indices $-\tau\le j\le-1$. %
\[
        \|v_j\|=\|\nabla f(y_j)-my_j\|
        \le (\sa{q^{-2}}+m)\|y_j\|,
\]
while $y_j=(1+\nu)x_j-\nu x_{j-1}$. \mgd{
Moreover, since $0<\rho<1$, we have
$\rho^{-j}\le1$ for every initialization index
$j\le0$. Thus, rate weighting does not increase
the norms of the initial history variables and the multiplier memory is initialized at zero.} %
Consequently there is a constant $K$, depending only on the fixed algorithmic and certificate
parameters, such that
\[
        V(\Xi_0)\le \lambda_{\max}(X)\|\Xi_0\|^2
        \le K\sum_{j=-\tau-1}^{0}\|x_j\|^2.
\]
\mgd{Thus, in light of \eqref{eq-C-rate}, one may choose $\mgd{C_{\mathrm{rate}}}=\sqrt{K/\varepsilon}\bigl(\sum_{j=-\tau-1}^{0}\|x_j\|^2\bigr)^{1/2}\ge0$}, uniformly over
$f\in\mathcal S_{1,\sa{q^{-2}}}$. Undoing the translation of the minimizer gives
$\|x_k-x_*\|\le \mgd{C_{\mathrm{rate}}}\rho^k$, as claimed.

\subsection{Proof of Theorem \ref{thm:main-nonexplicit-constants}}
The convergence rate lemma presented in Sections \ref{subsec-rate-lemma}--\ref{subsec:master-proof} assumes \textbf{(H1)} and \textbf{(H2)} and establishes Theorem~\ref{thm:main-nonexplicit-constants} for a constant \mgd{$\mgd{C_{\mathrm{rate}}}\geq 0$} and convergence rate \sa{$\rho\in(0,1)$}. It therefore suffices to verify that \textbf{(H1)} and \textbf{(H2)} hold for our choice of GMM parameters and for the rate
\[
\rho = 1-\frac{cq}{M}, \qquad c\in(0,2/3).
\]
The verification of \textbf{(H1)} is carried out in Section~\ref{sec:stability}, while that of \textbf{(H2)} is carried out in Sections~\ref{sec:factor}--\ref{sec:limit}, thereby completing the proof of Theorem~\ref{thm:main-nonexplicit-constants}.

\subsubsection{Verification of hypothesis (H1)}\label{sec:stability}

We now verify hypothesis \textbf{(H1)} of the Convergence Rate Lemma (Lemma~\ref{lemm:ratecert}), namely that \mgd{every root of the midpoint characteristic polynomial lies in $\{|z|<\rho\}$, or equivalently based on Lemma \ref{lem:midpoint-schur} that the state-transition matrix $A$ of the rate-weighted midpoint recursion is Schur stable.}%

Replacing the gradient with the midpoint slope
$
m=\frac{1+q^{-2}}{2}
$
turns the delayed recursion into the quadratic iteration corresponding to Hessian eigenvalue $\lambda=m$. Hence, the midpoint characteristic polynomial in~\eqref{eq:midpoint-charpoly} is exactly the quadratic characteristic polynomial in~\eqref{eq:q-charpoly}, evaluated at the single gain
\[
\delta=\alpha m=\frac{1+q^2}{2M^2},
\]
namely
\[
z^\tau (z-1)(z-\beta)
+\alpha m\bigl[(1+\nu)z-\nu\bigr]
=
P_{\tau,\alpha m}(z).
\]
\mgd{
Theorem~\ref{thm:q-alldelay} already establishes a
spectral-radius bound for the quadratic GMM recursion,
uniformly over all normalized curvature values
$t=\lambda/L\in[q^2,1]$, with the same parameters
$\alpha,\beta,\nu$ used here. In particular, the
midpoint slope $m=(1+q^{-2})/2$ corresponds to
\[
t_m=\frac{m}{L}=\frac{1+q^2}{2}\in[q^2,1].
\]
Thus, the midpoint characteristic polynomial is
covered by that theorem. Choose any constant
$\widetilde c$ satisfying
$c/M<\widetilde c<1/M$, which is possible since
$c<2/3$. The theorem then guarantees that, for
all sufficiently small $q$,
\[
\max\left\{|z|:P_{\tau,\alpha m}(z)=0\right\}
\le 1-\widetilde c q
<1-\frac{cq}{M}=\rho.
\]
This establishes hypothesis \textbf{(H1)}.}
Thus, the qualitative stability required in hypothesis \textbf{(H1)} is inherited directly from our previous quadratic analysis when $q$ is sufficiently small (or equivalently when $\kappa$ is sufficiently large). \mgd{
Proposition~\ref{prop:explicit-mid-threshold} in Appendix~\ref{sec:midpoint}
provides an explicit admissible bound on $q$; this quantitative refinement
is used in Section~\ref{subsec:main-proof}.}

\subsubsection{Verification of hypothesis (H2), part I: exact factorization of the FDI}\label{sec:factor}
By Lemma~\ref{lemm:ratecert} and the verification of \textbf{(H1)} in Section \ref{sec:stability}, it remains only to verify \textbf{(H2)} which says that the FDI \sa{in} \eqref{eq:fdi} is strictly satisfied at every frequency. 
The obstacle is now purely one of form. \mgd{As written, $\Phi_{\tau,q}$ in \eqref{eq:fdi}
is a scalar quadratic-form expression involving
the $2\times2$ Hermitian multiplier
$\Pi_{1,q^{-2}}$, the transfer function
$g_{\tau,q}$, and the filter $h_{\rho,\tau}$;
its sign cannot be read off directly.} This section removes that obstacle by reducing the form to a single scalar trigonometric polynomial. The reduction proceeds in three steps: first, a scalar factorization of the Hermitian form; second, a telescoping identity that reveals the factor $(z-1)$ hidden in the multiplier; and third, the combination of these two identities into a rational representation whose numerator, \(N_{\tau,q}\), is the quantity whose sign will be analyzed throughout the remainder of the paper.
\begin{lemma}[scalar factorization]\label{lem:scalar-factor}
With
$
h=h_{\rho,\tau}(e^{i\theta})$,
$g=g_{\tau,q}(\rho e^{i\theta})$,
$L=q^{-2}$, and with \(\bar g\) denoting the complex conjugate of \(g\), we have
\begin{equation}\label{eq:scalar-factor}
        \Phi_{\tau,q}(\theta)=
        2\Ree\bigl[(1-h)(1-\bar g)(Lg-1)\bigr].
\end{equation}
\end{lemma}

\begin{proof}
Define
\[
        W:=1-h .
\]
Then $1-\bar h=\bar W$ and
\[
        2-h-\bar h=(1-h)+(1-\bar h)=W+\bar W .
\]
Using $\mu=1$ and $L=q^{-2}$ in \eqref{eq:Pi}, the multiplier can therefore be
written as
\[
\Pi_{1,L}(h)=
\begin{bmatrix}
        -L(W+\bar W) & L\bar W+W\\
        LW+\bar W & -(W+\bar W)
\end{bmatrix}.
\]
Substituting this matrix into the scalar quadratic form in \eqref{eq:fdi}
gives
\begin{align*}
\Phi_{\tau,q}(\theta)
&=
\begin{bmatrix}\bar g&1\end{bmatrix}
\begin{bmatrix}
        -L(W+\bar W) & L\bar W+W\\
        LW+\bar W & -(W+\bar W)
\end{bmatrix}
\begin{bmatrix}g\\1\end{bmatrix}  \\
&= -L(W+\bar W)\abs g^2-(W+\bar W)
   +\bar g(L\bar W+W)+(LW+\bar W)g .
\end{align*}
The two off-diagonal terms are conjugates of each other, because $L$ is real:
\[
        \bar g(L\bar W+W)=\overline{(LW+\bar W)g}.
\]
Hence,
\begin{equation}\label{eq:Phi-expanded-W}
\Phi_{\tau,q}(\theta)
=-(W+\bar W)(L\abs g^2+1)
  +2\Ree\bigl[(LW+\bar W)g\bigr].
\end{equation}
We now collect the terms proportional to $L$, and from \eqref{eq:Phi-expanded-W}, \sa{these terms are}
\[
        -L(W+\bar W)\abs g^2+2L\Ree(Wg).
\]
Since
\[
        \Ree(W\abs g^2)=\abs g^2\Ree(W)
        =\frac12(W+\bar W)\abs g^2,
\]
we have
\begin{align*}
        -L(W+\bar W)\abs g^2+2L\Ree(Wg)
        &=2L\Ree(Wg)-2L\Ree(W\abs g^2)\\
        &=2L\Ree\bigl[Wg-Wg\bar g\bigr]\\
        &=2L\Ree\bigl[Wg(1-\bar g)\bigr].
\end{align*}
It remains to collect the terms independent of $L$.  These are
\[
        -(W+\bar W)+2\Ree(\bar W g).
\]
\sa{Since $\Ree(\bar W g)=\Ree(W\bar g)$,}
\begin{align*}
        -(W+\bar W)+2\Ree(\bar W g)
        &=-2\Ree(W)+2\Ree(W\bar g)\\
        &=-2\Ree\bigl[W(1-\bar g)\bigr].
\end{align*}
Combining the $L$-dependent and $L$-independent parts gives
\begin{align*}
\Phi_{\tau,q}(\theta)
&=2L\Ree\bigl[Wg(1-\bar g)\bigr]
  -2\Ree\bigl[W(1-\bar g)\bigr]\\
&=2\Ree\bigl[W(1-\bar g)(Lg-1)\bigr].
\end{align*}
Finally, substituting $W=1-h$ yields
\[
        \Phi_{\tau,q}(\theta)
        =2\Ree\bigl[(1-h)(1-\bar g)(Lg-1)\bigr],
\]
which is \eqref{eq:scalar-factor}.
\end{proof}

The factorization is the point at which all matrix notation collapses to a scalar expression.  For the remainder of this section we fix notation for the recurring factors.
Write, throughout this section,
\begin{equation}\label{eq:Ddefs}
        z=\rho e^{i\theta},\qquad
        D_1=z-1,
        \qquad
        D_2=z-\beta,
        \qquad
        P=(M+\tfrac12)z-(M-\tfrac12),
        \qquad
        \alpha=\frac{q^2}{M^2}.
\end{equation}
\sa{We also set}
\begin{equation}\label{eq:etah}
        \widetilde h(w)=\frac1\tau\sum_{k=0}^{\tau-1}(\tau-k)w^k \quad \mbox{for} \quad \tau\geq 1.
\end{equation}
\mg{The next result shows how $ W= 1 - h$ can be factored in terms of $\bar D_1$ and $\widetilde h(\bar z)$, for $h= h_{\rho,\tau}$.}
\begin{lemma}[Factorization of $W=1-h$]\label{lem:telescoped}
For $h= h_{\rho,\tau}$ and $z=\rho e^{i\theta}$, it holds that
\begin{equation}\label{eq:Wtel}
        W=1-h_{\rho,\tau}(e^{i\theta})=-\bar D_1\,\widetilde h(\bar z).
\end{equation}
\end{lemma}

\begin{proof}
Recall from \eqref{eq:Ddefs} that $z=\rho e^{i\theta}$. Hence,
\begin{equation*}
        w:=\bar z=\rho e^{-i\theta},
        \qquad
        w^j=\rho^j e^{-ij\theta}.
\end{equation*}
Therefore, the term $h_{\rho,\tau}(e^{i\theta})$ can be written as a polynomial in the variable $w$:
\begin{equation*}
        h_{\rho,\tau}(e^{i\theta})
        =\frac1\tau\sum_{j=1}^{\tau}\rho^j e^{-ij\theta}
        =\frac1\tau\sum_{j=1}^{\tau}w^j.
\end{equation*}
It follows that
\begin{equation*}
        1-h_{\rho,\tau}(e^{i\theta})
        =1-\frac1\tau\sum_{j=1}^{\tau}w^j
        =\frac1\tau\sum_{j=1}^{\tau}(1-w^j).
\end{equation*}
For each $j\ge1$ we use the finite geometric
identity
\begin{equation*}
        1-w^j=(1-w)(1+w+\cdots+w^{j-1})
        =(1-w)\sum_{k=0}^{j-1}w^k .
\end{equation*}
Substituting this identity into the previous display gives
\begin{align*}
        1-h_{\rho,\tau}(e^{i\theta})
        &=\frac{1-w}{\tau}
          \sum_{j=1}^{\tau}\sum_{k=0}^{j-1}w^k .
\end{align*}
We now interchange the order of summation.  For a fixed power $w^k$ with
$0\le k\le \tau-1$, the index $j$ must satisfy $j\ge k+1$ and
$j\le\tau$; \sa{hence,} this power appears exactly $\tau-k$ times. Therefore,
\begin{equation*}
        \sum_{j=1}^{\tau}\sum_{k=0}^{j-1}w^k
        =\sum_{k=0}^{\tau-1}(\tau-k)w^k .
\end{equation*}
Using the definition \eqref{eq:etah} of $\widetilde h$, we obtain
\begin{equation*}
        1-h_{\rho,\tau}(e^{i\theta})
        =(1-w)\frac1\tau\sum_{k=0}^{\tau-1}(\tau-k)w^k
        =(1-w)\widetilde h(w).
\end{equation*}
Finally, since $D_1=z-1$, we have
\begin{equation*}
        \bar D_1=\bar z-1=w-1,
        \qquad
        1-w=-\bar D_1 .
\end{equation*}
Evaluating the preceding factorization at $w=\bar z$ gives
\begin{equation*}
        W:=1-h_{\rho,\tau}(e^{i\theta})
        =(1-\bar z)\widetilde h(\bar z)
        =-\bar D_1\,\widetilde h(\bar z),
\end{equation*}
which proves \eqref{eq:Wtel}.
\end{proof}
\mg{Building on the previous two lemmas, the next \mgd{proposition} %
shows that $\Phi_{\tau,q}(\theta)$ can be expressed as the ratio of a real even trigonometric polynomial $N_{\tau,q}(\theta)$ and a strictly positive scalar term for every angle $\theta$. \mgd{Thus}, the question about the negativity of $\Phi_{\tau,q}(\theta)$ reduces to the negativity of the polynomial
$N_{\tau,q}(\theta)$.} 
\begin{proposition}[Representation of $\Phi_{\tau,q}$ as a ratio]\label{prop:rational}
For every $q\in(0,1]$ and $\theta\in[0,2\pi)$,
\begin{equation}\label{eq:Ndef}
        \Phi_{\tau,q}(\theta)=\frac{N_{\tau,q}(\theta)}
        {M^2\abs{D_1}^2\abs{D_2}^2},
        \qquad
        N_{\tau,q}(\theta):=2\Ree(\bar D_1 S),
\end{equation}
where, \mg{with $z=\rho e^{i\theta}$ and} $\widetilde h=\widetilde h(\bar z)$,
\begin{align}\label{eq:S-correct}
S={}&
\bar D_1\widetilde h\, z^{-\tau}P\bar D_2
+M^2\abs{D_1}^2\widetilde h\abs{D_2}^2
+\alpha\widetilde h\abs P^2\rho^{-2\tau}
+\alpha M^2{\widetilde h} \bar z^{-\tau}\bar P D_1D_2,
\end{align}
\mg{and $\alpha=q^2/M^2$}. %
\mgd{The numerator $N_{\tau,q}$ is an even real
trigonometric polynomial of degree at most
$2\tau+2$. Under the standing assumptions
$0<q\le1$ and $0<c<2/3$, the denominator
in \eqref{eq:Ndef} is strictly positive
for every $\theta\in[0,2\pi)$.}
\end{proposition}

\begin{proof}
We start from the scalar factorization in Lemma~\ref{lem:scalar-factor}.  With
\(W=1-h_{\rho,\tau}(e^{i\theta})\), it reads
\begin{equation*}
        \Phi_{\tau,q}(\theta)
        =2\Ree\bigl[W(1-\bar g)(Lg-1)\bigr].
\end{equation*}
Expanding the product gives four terms:
\begin{equation}\label{eq:rational-four-terms}
        \Phi_{\tau,q}(\theta)
        =2\Ree\{WLg-W-W\bar g Lg+W\bar g\}.
\end{equation}
We now rewrite each of these four terms using the notation in
\eqref{eq:Ddefs}.  Since \(L=q^{-2}\) and \(\alpha=q^2/M^2\), \mg{using the definition of $g$ given in \eqref{eq:g}}, we obtain
\begin{equation}\label{eq:rational-lg-gbar}
        Lg=-\frac{z^{-\tau}P}{M^2D_1D_2},
        \qquad
        \bar g=-\alpha\frac{(\bar z)^{-\tau}\bar P}{\bar D_1\bar D_2}.
\end{equation}
Moreover, Lemma~\ref{lem:telescoped} gives
\begin{equation}\label{eq:rational-W-use}
        W=-\bar D_1\widetilde h,
        \qquad
        \widetilde h=\widetilde h(\bar z).
\end{equation}
Substituting \eqref{eq:rational-lg-gbar}--\eqref{eq:rational-W-use} into
\eqref{eq:rational-four-terms}, we place the four summands over the common
complex denominator
\begin{equation*}
        M^2D_1\abs{D_2}^2=M^2D_1D_2\bar D_2 .
\end{equation*}
The four contributions are as follows: First,
\begin{align*}
        WLg
        &=(-\bar D_1\widetilde h)
          \left(-\frac{z^{-\tau}P}{M^2D_1D_2}\right) \sa{=\frac{\bar D_1\widetilde h\,z^{-\tau}P\bar D_2}
                {M^2D_1\abs{D_2}^2}.}
\end{align*}
Second,
\begin{align*}
        -W
        =\bar D_1\widetilde h =\frac{M^2D_1\bar D_1\widetilde h\abs{D_2}^2}
                {M^2D_1\abs{D_2}^2}
          =\frac{M^2\abs{D_1}^2\widetilde h\abs{D_2}^2}
                {M^2D_1\abs{D_2}^2}.
\end{align*}
Third, using \(z^{-\tau}(\bar z)^{-\tau}=\abs z^{-2\tau}=\rho^{-2\tau}\),
\begin{align*}
        -W\bar g Lg
        =-(-\bar D_1\widetilde h)
          \left(-\alpha\frac{(\bar z)^{-\tau}\bar P}{\bar D_1\bar D_2}\right)
          \left(-\frac{z^{-\tau}P}{M^2D_1D_2}\right) 
        =\frac{\alpha\widetilde h\abs P^2\rho^{-2\tau}}
                {M^2D_1\abs{D_2}^2}.
\end{align*}
Finally,
\begin{align*}
        W\bar g
        =(-\bar D_1\widetilde h)
          \left(-\alpha\frac{(\bar z)^{-\tau}\bar P}{\bar D_1\bar D_2}\right) 
        =\alpha\frac{\widetilde h\bar z^{-\tau}\bar P}{\bar D_2} 
        =\frac{\alpha M^2\widetilde h\bar z^{-\tau}\bar P D_1D_2}
                {M^2D_1\abs{D_2}^2}.
\end{align*}
This calculation also explains the sign of the last term: both factors in
\(W\bar g\) carry a minus sign; \sa{hence,} the two minus signs cancel.
Adding the four displayed fractions yields
\begin{equation}\label{eq:rational-S-denom}
        WLg-W-W\bar gLg+W\bar g
        =\frac{S}{M^2D_1\abs{D_2}^2},
\end{equation}
where \(S\) is exactly the expression in \eqref{eq:S-correct}.

It remains only to convert the denominator in \eqref{eq:rational-S-denom}
into a real positive denominator.  Since
\begin{equation*}
        \frac{1}{D_1}=\frac{\bar D_1}{\abs{D_1}^2},
\end{equation*}
we have
\begin{equation*}
        \frac{S}{M^2D_1\abs{D_2}^2}
        =\frac{\bar D_1S}{M^2\abs{D_1}^2\abs{D_2}^2}.
\end{equation*}
Taking twice the real part gives
\begin{equation*}
        \Phi_{\tau,q}(\theta)
        =\frac{2\Ree(\bar D_1S)}{M^2\abs{D_1}^2\abs{D_2}^2},
\end{equation*}
which proves \eqref{eq:Ndef} with
\(N_{\tau,q}(\theta)=2\Ree(\bar D_1S)\).

We next record why this numerator is a finite trigonometric polynomial of the
claimed degree.  The quantities \(D_1,D_2,P\) are affine polynomials in
\(z=\rho e^{i\theta}\), while their conjugates are affine polynomials in
\(\bar z=\rho e^{-i\theta}\).  The factor \(\widetilde h(\bar z)\) is a
polynomial of degree \(\tau-1\) in \(\bar z\). \sa{Hence,} every term in
\(\bar D_1S\) is a finite linear combination of monomials
\(z^a\bar z^b\), possibly also multiplied by \(z^{-\tau}\) or
\(\bar z^{-\tau}\).  Therefore, after writing
\(z=\rho e^{i\theta}\) and \(\bar z=\rho e^{-i\theta}\), each monomial is a
constant times \(e^{i\ell\theta}\) for some integer \(\ell\).  Inspecting the
largest possible powers in the four terms of \(\bar D_1S\) shows that
\(\abs{\ell}\le 2\tau+2\).  The worst negative exponent can occur in the
first term, where the factors \(z^{-\tau}\), \(\widetilde h(\bar z)\),
and the three affine conjugate factors contribute down to
\(-2\tau-2\); the largest positive exponent is smaller, at most
\(\tau+2\), and the remaining two terms lie within these bounds. Thus,
\(N_{\tau,q}\) is a real trigonometric polynomial of degree at most
\(2\tau+2\).

The numerator is real by definition, because it is twice a real part.  It is
even because the underlying transfer functions have real coefficients.  More
explicitly,
\begin{equation*}
        h_{\rho,\tau}(e^{-i\theta})=\bar h_{\rho,\tau}(e^{i\theta}),
        \qquad
        g_{\tau,q}(\rho e^{-i\theta})=\bar g_{\tau,q}(\rho e^{i\theta}).
\end{equation*}
Substituting these identities into the scalar factorization shows that
\(\Phi_{\tau,q}(-\theta)=\Phi_{\tau,q}(\theta)\).  The denominator
\(M^2\abs{D_1}^2\abs{D_2}^2\) is also even in \(\theta\); \sa{therefore,
\(N_{\tau,q}\) is even.}

\mgd{Finally, we prove strict positivity of the denominator. Under the standing
assumptions $0<q\le1$, $M\ge3$, and $0<c<2/3$ (in particular, $c<2$), we have}
\begin{equation*}
        \rho=1-\frac{cq}{M}<1.
\end{equation*}
Thus, \(D_1=z-1\) cannot vanish, since \(\abs z=\rho<1\).  For \(D_2\), note
that
\begin{equation*}
        D_2=0
        \quad\Longrightarrow\quad
        z=\mg{\rho e^{i\theta}} = \beta=1-\frac{2q}{M}.
\end{equation*}
Since \(\tau\ge1\), we have \(M=2\tau+1\ge3\); together with
\(q\in(0,1]\), this gives
\[
        \beta=1-\frac{2q}{M}\ge 1-\frac{2}{3}>0 .
\]
Thus, \(\beta\) is real and positive, \sa{and \(D_2=0\) would imply that}
\(\theta=0\) and \(\rho=\beta\). \sa{On the other hand, \(\rho=\beta\) is} equivalent to
\begin{equation*}
        1-\frac{cq}{M}=1-\frac{2q}{M},
\end{equation*}
which gives \(c=2\), contradicting \(c<2\). Hence, \(D_2\ne0\) as well, and
therefore
\begin{equation*}
        M^2\abs{D_1}^2\abs{D_2}^2>0,\quad \sa{\forall~\theta\in[0,2\pi).}
\end{equation*}
\end{proof}
\subsubsection{Verification of hypothesis \textbf{(H2)}, part II: the all-delay positivity certificate}\label{sec:limit}

\mgd{
By \eqref{eq:Ndef}, the denominator of the FDI
expression is strictly positive for every $q>0$ under
our standing assumptions. Thus, verifying
\textbf{(H2)} reduces to proving
$N_{\tau,q}(\theta)<0$ at every frequency.
We aim to establish this inequality for each fixed
delay $\tau\ge1$ and $c\in(0,2/3)$, whenever
$0<q\le q_0(\tau,c)$ for some positive threshold
$q_0(\tau,c)$. Since $q=\kappa^{-1/2}$, this is
equivalent to requiring a sufficiently large condition
number $\kappa$, with a threshold depending on
$(\tau,c)$.}

\mgd{Our argument proceeds in two stages. First, we
examine the limit $q\downarrow0$ at each fixed nonzero
frequency and show that the sign of the limiting FDI
is determined by a polynomial $r_\tau(\cos\theta)$.
Establishing the positivity of this polynomial,
together with continuity, yields strict negativity
for sufficiently small $q$, uniformly over frequencies
bounded away from zero. However, this argument does
not extend directly to frequencies approaching zero,
where both the numerator and denominator of the
limiting rational expression that identifies the FDI vanish. We therefore
treat this low-frequency region separately, using
a joint expansion in $q$ and $\theta$. The argument establishes the existence of an
admissible threshold $q_0(\tau,c)$ but does not
provide an explicit formula. The quantitative result
in Theorem~\ref{thm:main} of
Section~\ref{subsec:main-proof} addresses this
limitation, using the explicit estimates developed
in Appendix~\ref{sec:finiteq} to characterize the
dependence of the threshold on $\tau$ and $c$.
}

For fixed $\theta\in(0,\pi]$ and $q\downarrow0$, \mg{we have $\rho = 1 - cq/M \to 1$, $\beta\to 1$} as well as
\begin{equation}\label{eq:limit-objects}
        h_{\rho,\tau}(e^{i\theta})\to H_\tau(\theta):=\frac1\tau\sum_{j=1}^{\tau}e^{-ij\theta}.
\end{equation}
\sa{Moreover, by the definition in \eqref{eq:g},} we have  $g_{\tau,q}(\rho e^{i\theta}) \to 0$ and $q^{-2}g_{\tau,q}(\rho e^{i\theta})\to A_{\tau,0}(\theta)$, where
\begin{equation}\label{eq:A0}
        A_{\tau,0}(\theta)=
        -\frac{e^{-i\tau\theta}\bigl((M+\frac12)e^{i\theta}-(M-\frac12)\bigr)}
        {M^2(e^{i\theta}-1)^2}.
\end{equation}
Thus, by Lemma \ref{lem:scalar-factor}, \sa{we have}
\begin{equation}\label{eq:Phi0}
        \Phi_{\tau,0}(\theta):=\lim_{q\downarrow0}\Phi_{\tau,q}(\theta)
        =2\Ree\bigl[(1-H_\tau(\theta))(A_{\tau,0}(\theta)-1)\bigr].
\end{equation}
The endpoint $\theta=0$ is defined by continuity. In particular, the following calculation makes this continuous endpoint value explicit. Although
$A_{\tau,0}(\theta)$ itself has a second-order pole at the origin, the pole is
cancelled in the real part of the product in \eqref{eq:Phi0}; the resulting
endpoint magnitude is $\Theta(1)$ as $\tau\to\infty$. %
\begin{lemma}[endpoint value of the limiting FDI]\label{lem:Phi0-endpoint}
The continuous extension of $\Phi_{\tau,0}$ to $\theta=0$ satisfies
\begin{equation}\label{eq:Phi0-endpoint-value}
        \Phi_{\tau,0}(0)=\lim_{\theta\to 0}\Phi_{\tau,0}(\theta)
        =-\frac{\tau+1}{3M}.
\end{equation}
\end{lemma}

\begin{proof}
We expand the two factors in \eqref{eq:Phi0}.  First,
recall that 
$        H_\tau(\theta)=\frac1\tau\sum_{j=1}^{\tau}e^{-ij\theta}$.
Using $e^{-ij\theta}=1-ij\theta-j^2\theta^2/2+\mathcal{O}(\theta^3)$ and the sums
\begin{equation*}
        \frac1\tau\sum_{j=1}^{\tau}j=\frac{\tau+1}{2},
        \qquad
        \frac1\tau\sum_{j=1}^{\tau}j^2=\frac{(\tau+1)(2\tau+1)}{6}
        =\frac{(\tau+1)M}{6},
\end{equation*}
we obtain
\begin{equation}\label{eq:one-minus-H-expansion}
        1-H_\tau(\theta)
        = i\frac{\tau+1}{2}\theta
        +\frac{(\tau+1)M}{12}\theta^2
        +\mathcal{O}(\theta^3).
\end{equation}
Next consider $A_{\tau,0}$.  The numerator in \eqref{eq:A0} has the expansion
\begin{align*}
&e^{-i\tau\theta}\bigl((M+\tfrac12)e^{i\theta}-(M-\tfrac12)\bigr) \\
&\qquad
=1+i\Bigl(M+\frac12-\tau\Bigr)\theta+\cO(\theta^2)
=1+i\Bigl(\tau+\frac32\Bigr)\theta+\cO(\theta^2),
\end{align*} where $M=2\tau+1$ was used in the last equality.  Also,
\begin{equation*}
        (e^{i\theta}-1)^2
        =-\theta^2-i\theta^3+\cO(\theta^4),
        \qquad
        \frac{1}{(e^{i\theta}-1)^2}
        =-\frac{1}{\theta^2}\bigl(1-i\theta+\cO(\theta^2)\bigr).
\end{equation*}
Substituting these expansions into \eqref{eq:A0} gives
\begin{align}\label{eq:A0-origin-expansion}
        A_{\tau,0}(\theta)
        &=\frac{1}{M^2\theta^2}
          \biggl(1+i\Bigl(\tau+\frac32\Bigr)\theta+\cO(\theta^2)\biggr)
          \bigl(1-i\theta+\cO(\theta^2)\bigr) \notag\\
        &=\frac{1}{M^2\theta^2}+\frac{i}{2M\theta}+\cO(1).
\end{align}
The subtraction of $1$ in $A_{\tau,0}(\theta)-1$ only changes the bounded
$\cO(1)$ term. Combining \eqref{eq:one-minus-H-expansion} and
\eqref{eq:A0-origin-expansion}, we find
\begin{align*}
&(1-H_\tau(\theta))(A_{\tau,0}(\theta)-1) \\
&\quad =
\biggl(i\frac{\tau+1}{2}\theta
        +\frac{(\tau+1)M}{12}\theta^2
        +\cO(\theta^3)\biggr)
\biggl(\frac{1}{M^2\theta^2}+\frac{i}{2M\theta}+\cO(1)\biggr) \\
&\quad =
        \frac{i(\tau+1)}{2M^2\theta}
        -\frac{\tau+1}{4M}
        +\frac{\tau+1}{12M}
        +\cO(\theta) \\
&\quad =
        \frac{i(\tau+1)}{2M^2\theta}
        -\frac{\tau+1}{6M}
        +\cO(\theta).
\end{align*}
\sa{Although $(1-H_\tau(\theta))(A_{\tau,0}(\theta)-1)$ has a purely imaginary $\mathcal{O}(\theta^{-1})$ singularity, its real part has the finite limit
$-\frac{\tau+1}{6M}$. Therefore,
$$\Phi_{\tau,0}(0):=\lim_{\theta\to0}\Phi_{\tau,0}(\theta)=-\frac{\tau+1}{3M}.$$}
This is the value assigned to $\Phi_{\tau,0}(0)$ in the continuous extension.
\end{proof}

We now express the limiting function $\Phi_{\tau,0}(\theta)$ in terms of a polynomial $r_\tau$. 
\begin{lemma}[explicit limiting polynomial]\label{lem:rpoly}
For every $\tau\ge1$ there is a real polynomial $r_\tau$ of degree $2\tau$ such that
\begin{equation}\label{eq:Phi-r}
        \Phi_{\tau,0}(\theta)=-\frac{r_\tau(\cos\theta)}{4M^2},
        \qquad 0\le\theta\le
         \mg{2}\pi.
\end{equation}
It is given explicitly by
\begin{equation}\label{eq:rcheb}
\begin{aligned}
r_\tau(\cos\theta)={}&c_0+2\sum_{m=1}^{\tau}c_m\cos(m\theta)
      +2\sum_{m=\tau+1}^{2\tau}c_m\cos(m\theta),\\
 c_0={}&8M^2-2(\tau+1),\\
 c_m={}&-\frac{4M^2}{\tau}-2(\tau+1),\qquad 1\le m\le\tau,\\
 c_m={}&\frac{2}{\tau}(M^2-m^2),\qquad \tau+1\le m\le2\tau.
\end{aligned}
\end{equation}
Moreover we have $r_\tau(1)>0$ and $r_\tau(-1)>0$ where
\begin{equation}\label{eq:endpoints}
        r_\tau(1)=\frac{2M(M+1)}{3},
        \qquad
        r_\tau(-1)=\begin{cases}
         4M(2M-1), & \mbox{if} \quad \tau \mbox{ is even}, \\
          4M (2M+1)(1+1/\tau),  & \mbox{if} \quad \tau \mbox{ is odd.}
        \end{cases}
\end{equation}
\end{lemma}

\begin{proof}
\sa{\mgd{Recall that $        H_\tau(\theta)=\frac1\tau\sum_{j=1}^{\tau}e^{-ij\theta}$ and $A_{\tau,0}(\theta)$ is defined by \eqref{eq:A0}}. Since $H_\tau(-\theta)=\overline{H_\tau(\theta)}$ and $A_{\tau,0}(-\theta)=\overline{A_{\tau,0}(\theta)}$, the real-valued function $\Phi_{\tau,0}$ is even and $2\pi$-periodic. Moreover, $r_\tau(\cos\theta)$ is even and $2\pi$-periodic. Hence, it suffices to establish \eqref{eq:Phi-r} for $0\le\theta\le\pi$.}

Let $z=e^{i\theta}$ with $0\mathrel{<} \theta \leq \pi$.
\mgd{Setting $\rho=1$ in Lemma \ref{lem:telescoped}}, leads to \begin{equation}\label{eq:W0tel}
        W=1-H_\tau=(1-z^{-1})E(z^{-1}),
        \qquad
E(z^{-1})=\frac1\tau\sum_{k=0}^{\tau-1}(\tau-k)z^{-k}.
\end{equation}
Put $a=M+\frac12$ and $b=M-\frac12$.  Then
\begin{equation*}
A_{\tau,0}=-\frac{z^{-\tau}(az-b)}{M^2(z-1)^2},
\end{equation*}
and therefore
\begin{equation}\label{eq:WA}
WA_{\tau,0}=-\frac{E(z^{-1})z^{-\tau-1}(az-b)}{M^2(z-1)}.
\end{equation}
Adding the conjugate expression, i.e. substituting $z^{-1}$ for $z$, gives
\begin{align}\label{eq:WA-plus}
WA_{\tau,0}+\overline{WA_{\tau,0}}
={}&\frac1{\tau M^2}\sum_{k=0}^{\tau-1}(\tau-k)
\frac{a(z^n-z^{-(n-1)})-b(z^{n+1}-z^{-n})}{z-1},
\end{align}
where $n=\tau+1+k$.  Since
\begin{equation}\label{eq:dirichlet-div}
\frac{z^n-z^{-(n-1)}}{z-1}=\sum_{r=-(n-1)}^{n-1}z^r,
\qquad
\frac{z^{n+1}-z^{-n}}{z-1}=\sum_{r=-n}^{n}z^r,
\end{equation}
we obtain
\begin{equation}\label{eq:WA-dirichlet}
WA_{\tau,0}+\overline{WA_{\tau,0}}
=\frac1{\tau M^2}\sum_{k=0}^{\tau-1}(\tau-k)
\left(\sum_{|r|\le n-1}z^r-b(z^n+z^{-n})\right).
\end{equation}
On the other hand,
\begin{equation}\label{eq:W-plus}
W+\bar W=2-\frac1\tau\sum_{j=1}^{\tau}(z^j+z^{-j}).
\end{equation}
    We observe from \eqref{eq:WA-dirichlet} and \eqref{eq:W-plus} that both $W+\bar{W}$ and $WA_{\tau,0}+\overline{W A_{\tau,0}}$ are symmetric Laurent polynomials in $z$. Hence, 
    \begin{equation*}
        -4M^2\Phi_{\tau,0}=4M^2\bigl[(W+\bar W)-(WA_{\tau,0}+\overline{WA_{\tau,0}})\bigr],
    \end{equation*}
    is also a symmetric Laurent polynomial. We now collect the coefficients of each pair $z^j$ and $z^{-j}$ together with the constant term. The symmetry ensures that the coefficients of $z^j$ and $z^{-j}$ coincide, so the expression can be written in the cosine basis $z^j + z^{-j}=2\cos(j\theta)$. Substituting the coefficients obtained from \eqref{eq:W-plus} and \eqref{eq:WA-dirichlet} then yields the representation given in \eqref{eq:rcheb}. 
    
    \mg{\sa{It remains to establish the endpoint values in \eqref{eq:endpoints}, corresponding to $\theta=0$ and $\theta=\pi$.}
    At $\theta=0$,  Lemma~\ref{lem:Phi0-endpoint} gives
$\Phi_{\tau,0}(0)=-(\tau+1)/(3M)$ so that we get $r_\tau(1) = - 4 M^2 \Phi_{\tau,0}(0) = \frac{2M(M+1)}{3}$ where we used $M = 2\tau+1$. 
\sa{It remains to establish the formula for $r_\tau(-1)$.}
Recall from \eqref{eq:Phi0} that \begin{equation}\label{eq-recall-Phi}
        \Phi_{\tau,0}(\theta)
        =2\Ree\bigl[(1-H_\tau(\theta))(A_{\tau,0}(\theta)-1)\bigr],
\end{equation}
where $H_\tau(\theta):=\frac1\tau\sum_{j=1}^{\tau}e^{-ij\theta}$ and $A_{\tau,0}(\theta)$ is as in \eqref{eq:A0}. At $\theta=\pi$, one has
$A_{\tau,0}(\pi)=(-1)^\tau/(2M)$. There are two cases depending on whether $\tau$ is even or odd: If $\tau$ is even, then
$H_\tau(\pi)=\tau^{-1}\sum_{j=1}^{\tau}(-1)^j=0$, so
$1-H_\tau(\pi)=1>0$ and $A_{\tau,0}(\pi)-1=1/(2M)-1<0$. Thus, by \eqref{eq-recall-Phi},  $\Phi_{\tau,0}(\pi)=1/M - 2 <0$, equivalently $r_\tau(-1)=4M(2M-1)>0$. 
On the other hand, if
$\tau$ is odd, then $H_\tau(\pi)=-1/\tau$, so
$1-H_\tau(\pi)=1+1/\tau>0$ and $A_{\tau,0}(\pi)-1=-1/(2M)-1<0$.
Thus, again $\Phi_{\tau,0}(\pi)=-2(1+1/\tau)(1/(2M)+1)<0$, equivalently $r_\tau(-1)=(1+1/\tau) 4M (2M+1)>0$. Using $M=2\tau+1$, this proves the formula for $r_\tau(-1)$ and we conclude.}

\end{proof}
\begin{figure}[t!]
    \centering

    \begin{subfigure}[t]{0.48\textwidth}
        \centering
        \includegraphics[width=\linewidth]
        {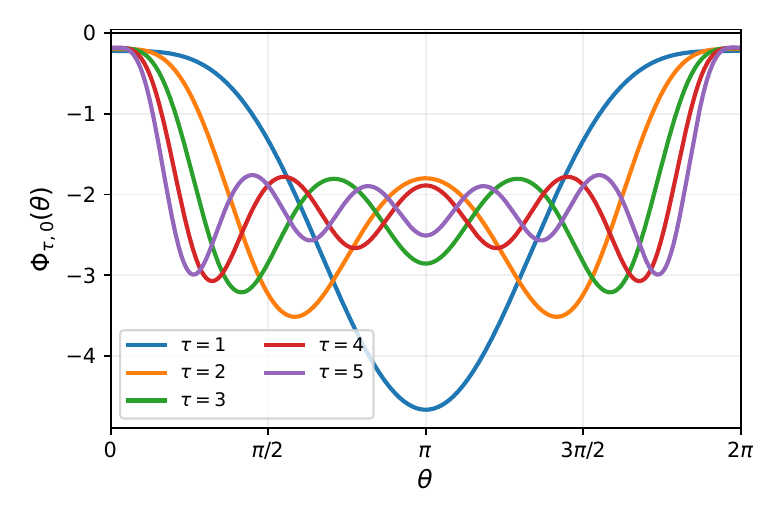}
        \caption{Small delays, $\tau=1,\ldots,5$.}
        \label{fig:phi-small-tau}
    \end{subfigure}
    \hfill
    \begin{subfigure}[t]{0.48\textwidth}
        \centering
        \includegraphics[width=\linewidth]
        {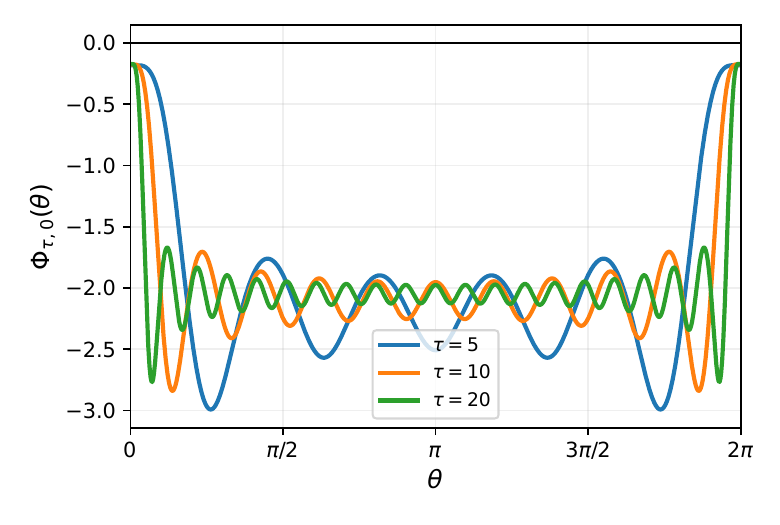}
        \caption{Increasing oscillations for larger $\tau$.}
        \label{fig:phi-large-tau}
    \end{subfigure}

    \caption{
    The function $\Phi_{\tau,0}(\theta)$ becomes increasingly
    oscillatory as $\tau$ grows while remaining strictly negative.
    }
    \label{fig:phi-oscillations}
\end{figure}

\mg{In Figure \ref{fig:phi-oscillations}, we plot the function 
$\Phi_{\tau,0}$ which arises in the $q\downarrow 0$ limit for various values of $\tau$. We observe that this function stays strictly negative, and that it becomes more oscillatory as $\tau$ increases. The fact that the oscillations increase is expected, because $r_\tau$, that represents $\Phi_{\tau,0}$ up to a constant by Lemma \ref{lem:rpoly}, is a degree-$2\tau$ polynomial and its degree increases when $\tau$ increases.} 

\mg{In the next proposition, we also prove rigorously that $\Phi_{\tau,0}<0$ for every $\theta$ as suggested by Figure \ref{fig:phi-oscillations}. The main idea is to show that the even trigonometric polynomial \(r_\tau(\cos\theta)\) is strictly positive for every \(\theta\), and then invoke Lemma~\ref{lem:rpoly}. By evenness and periodicity, it suffices to verify positivity for \(\theta\in[0,\pi]\). Because these polynomials exhibit highly oscillatory behavior, we control their sign uniformly over all delays by splitting \([0,\pi]\) into two regions: a region away from the origin, where a Dirichlet-kernel bound applies, and a region near the origin, where we use Taylor series approximations and exploit the cosine-series representation
\eqref{eq:rcheb}. In particular, differentiation with respect to $\theta$ introduces powers of the frequencies,
so the Taylor coefficients are combinations of the moments
$\sum_m c_m m^{2k}$. Since $c_m$ is constant in $m$ for
$1\le m\le\tau$ and quadratic in $m$ for $\tau+1\le m\le2\tau$,
these moments reduce to power sums $\sum_{m=1}^{\tau} m^p$, with
$p\le 8$ in the estimates below. Faulhaber's formulas then provide
explicit expressions for these sums in terms of $\tau$ which, in return, allows us to control the sign of $r_\tau$.}

\begin{proposition}[all-delay positivity of the limiting polynomial]\label{prop:rpositive}
For every integer $\tau\ge1$,
\begin{equation}\label{eq:rpositive}
        r_\tau(x)>0\quad \mbox{for} \quad -1\le x\le1.
\end{equation}
Consequently, $\Phi_{\tau,0}(\theta)<0$ for all $0\le\theta\le \mg{2}\pi$.
\end{proposition}

\begin{proof}
The idea is to split $[0,\pi]$ into an outer and an inner region around the origin and bound $r_\tau$ below on each.  In the outer region $\theta\ge3/(2\tau)$ we control the cosine sum $\sum_{k=1}^n\cos(k\theta)$ by the Dirichlet kernel and compare it to an explicit threshold $T_n$; in the inner region we use sharp Taylor envelopes for $1-\cos(\theta)$ together with Faulhaber's power-sum formulas.  Put $n=\tau$.  It is enough to prove positivity for $x=\cos\theta$, $0\le\theta\le\pi$.  From
\eqref{eq:rcheb}, write
\begin{equation}\label{eq:r-for-proof}
 r_n(\cos\theta)=c_0-2a\sum_{k=1}^{n}\cos(k\theta)
        +2\sum_{\ell=1}^{n}d_\ell\cos((n+\ell)\theta),
\end{equation}
where
\begin{equation}\label{eq:adcoeffs}
        c_0=32n^2+30n+6,
        \qquad
        a=\frac{18n^2+18n+4}{n},
        \qquad
        d_\ell=\frac{6n^2+8n+2-4n\ell-2\ell^2}{n}.
\end{equation}
The sequence $d_\ell$ is positive: it is decreasing in $\ell$, and
$d_n=8+2/n>0$.  Also
\begin{equation}\label{eq:DsumT}
        D_n:=\sum_{\ell=1}^{n}d_\ell=\frac{5(n+1)(2n+1)}{3},
        \qquad
        T_n:=\frac{c_0-2D_n}{2a}
        =\frac{n(19n^2+15n+2)}{3(3n+1)(3n+2)}.
\end{equation}
We split the interval into two regions.

\smallskip
\noindent\emph{Outer region: $\theta\ge 3/(2n)$.}
Let $C_n(\theta)=\sum_{k=1}^{n}\cos(k\theta)$.  The Dirichlet formula gives
\begin{equation}\label{eq:Cn-bound}
        C_n(\theta)
        =\frac{\sin((n+\frac12)\theta)}{2\sin(\theta/2)}-\frac12
        \le \frac{1}{2\sin(\theta/2)}-\frac12.
\end{equation}
\sa{Since
\[
\frac{3}{4n}\le \frac{\theta}{2}\le \frac{\pi}{2},
\]
and $\sin t$ is increasing on $[0,\pi/2]$, we have
\[
\sin\left(\frac{\theta}{2}\right)
\ge
\sin\left(\frac{3}{4n}\right)
\ge
\frac{3}{4n}
-\frac{1}{6}\left(\frac{3}{4n}\right)^3
=
\frac{3(32n^2-3)}{128n^3},
\]
where we used $\sin t\ge t-t^3/6$ for $t\ge0$. Therefore,}
\begin{equation}\label{eq:Cn-outer}
        C_n(\theta)\le\mgd{ \overline{C}_n}:=
        \frac{64n^3}{3(32n^2-3)}-\frac12.
\end{equation}
This upper bound is strictly smaller than $T_n$.  Indeed,
\begin{align}\label{eq:Tminus}
\mgd{g_n^{\mathrm{out}}}:=T_n-\mgd{\overline{C}_n}
={}&
\frac{64n^5+672n^4+622n^3+21n^2-93n-18}
{6(3n+1)(3n+2)(32n^2-3)}  \\
={}&
\frac{64m^5+992m^4+3950m^3+6559m^2+4823m+1268}
{6(3n+1)(3n+2)(32n^2-3)}>0,
\end{align}
where $m=n-1\ge0$.  Therefore, using $d_\ell>0$ and $\cos((n+\ell)\theta)\ge-1$,
\begin{equation}\label{eq:outer-positive}
        r_n(\cos\theta) 
        \ge c_0-2aC_n(\theta)-2D_n \geq \mgd{2a g_n^{\mathrm{out}}}
        >c_0-2aT_n-2D_n=0.
\end{equation}

\smallskip
\noindent\emph{Inner region: $0\le\theta\le 3/(2n)$.}
We use two elementary Taylor inequalities, valid for all $u\ge0$:
\begin{equation}\label{eq:taylor-bounds-cos}
        1-\cos u\ge \frac{u^2}{2}-\frac{u^4}{24},
        \qquad
        \cos u-1\ge -\frac{u^2}{2}+\frac{u^4}{24}-\frac{u^6}{720}.
\end{equation}
For completeness we recall the proof. The difference of both sides in the first inequality $e_1(u):= 1-\cos u - (\frac{u^2}{2}-\frac{u^4}{24})$ has fourth derivative $1-\cos u\ge0$ and 
\sa{its derivatives of orders $0,\ldots,3$ vanish at zero.}
Similarly, $e_2(u):=  \cos u-1 - (-\frac{u^2}{2}+\frac{u^4}{24}-\frac{u^6}{720})$ has sixth derivative $1-\cos u\ge0$ and 
\sa{its derivatives of orders $0,\ldots,5$ vanish at zero.}
Repeated integration gives \eqref{eq:taylor-bounds-cos}.

Using \eqref{eq:taylor-bounds-cos} in
\begin{equation*}
 r_n(\cos\theta)=r_n(1)
 +2a\sum_{k=1}^{n}(1-\cos(k\theta))
 +2\sum_{\ell=1}^{n}d_\ell\bigl(\cos((n+\ell)\theta)-1\bigr)
\end{equation*}
gives
\begin{equation}\label{eq:inner-poly-bound}
        r_n(\cos\theta)
        \ge r_n(1)+R_2\theta^2+R_4\theta^4-R_6\theta^6,
\end{equation}
where the coefficients before simplification are
\begin{align}
R_2&=a\sum_{k=1}^{n}k^2-\sum_{\ell=1}^{n}d_\ell(n+\ell)^2,\label{eq:R2-sum}\\
R_4&=-\frac{a}{12}\sum_{k=1}^{n}k^4+
\frac{1}{12}\sum_{\ell=1}^{n}d_\ell(n+\ell)^4,\label{eq:R4-sum}\\
R_6&=\frac{1}{360}\sum_{\ell=1}^{n}d_\ell(n+\ell)^6.\label{eq:R6-sum}
\end{align}
Since $d_\ell$ is quadratic in $\ell$, the last identity uses power sums up to degree eight.  Applying \textit{Faulhaber's formula}:
\[
\sum_{j=1}^{n} j^p
=
\frac{1}{p+1}\sum_{r=0}^{p}
\binom{p+1}{r}B_r^+\,n^{p+1-r},
\qquad 0\le p\le 8,
\]
in the Bernoulli convention determined by
\[
\frac{t}{1-e^{-t}}
=
\sum_{r=0}^{\infty}B_r^+\frac{t^r}{r!},
\]
so that $B_1^+=1/2$ \cite[Sec.~6.5, Eq.~(6.78)]{GrahamKnuthPatashnik}, gives
\begin{align}
 r_n(1)&=\frac{4(n+1)(2n+1)}{3}, \label{eq:rn1}\\
 R_2&=-\frac{2(n-1)(n+1)(n+2)(2n+1)}{15}, \label{eq:R2}\\
 R_4&=\frac{(n+1)(2n+1)(510n^4+1020n^3+539n^2+29n+2)}{1260}, \label{eq:R4}\\
 R_6&=\frac{(n+1)(2n+1)(4975n^6+14925n^5+13390n^4+1905n^3-1583n^2-48n+36)}{113400}. \label{eq:R6}
\end{align}

These identities are ordinary polynomial identities in $n$: substituting
\eqref{eq:adcoeffs} into \eqref{eq:R2-sum}--\eqref{eq:R6-sum}, replacing
$\sum_{\ell=1}^n \ell^j$ by Faulhaber's formula for $j\le8$, and clearing the
positive common denominator reproduces the four displayed numerators.  Equivalently,
a purely finite check is obtained by clearing denominators and comparing
coefficients of the resulting polynomials in $n$.  Thus, no
inequality is being used in \eqref{eq:rn1}--\eqref{eq:R6}; the inequalities enter
only in the sign checks that follow.
The coefficient $R_2$ is nonpositive.  The quartic and sixth-order terms satisfy, for all $n\ge1$,
\begin{align}\label{eq:R4R6}
R_4-\frac{9}{4n^2}R_6
={}&\frac{(n+1)(2n+1)}{50400n^2}
\bigl(15425n^6+25875n^5+8170n^4  \\
&\hspace{3.0cm}-745n^3+1663n^2+48n-36\bigr)>0.
\end{align}
The positivity is immediate after substituting $n=m+1$; the polynomial in parentheses becomes
\begin{equation*}
15425m^6+118425m^5+368920m^4+599185m^3+538573m^2+255744m+50400.
\end{equation*}
\sa{Moreover, $R_6>0$ by \eqref{eq:R6-sum}, since $d_\ell>0$ for
$1\le \ell\le n$. Since $\theta^2\le 9/(4n^2)$ in the inner
region, \eqref{eq:R4R6} yields
\[
R_4-R_6\theta^2
\ge
R_4-\frac{9}{4n^2}R_6
>0.
\]
Consequently,
\[
R_4\theta^4-R_6\theta^6
=
\theta^4(R_4-R_6\theta^2)
\ge0.
\]
Furthermore, since $R_2\le0$ and $\theta^2\le9/(4n^2)$,
\[
R_2\theta^2\ge \frac{9}{4n^2}R_2.
\]
Hence,}
\begin{align}\label{eq:inner-positive}
        r_n(\cos\theta)
        &\ge r_n(1)+R_2\theta^2
        \ge r_n(1)+\frac{9}{4n^2}R_2  \\
        &=\frac{(n+1)(2n+1)(31n^2-9n+18)}{30n^2}>0.
\end{align}
The inner and outer regions cover $[0,\pi]$.  Thus $r_\tau>0$ on $[-1,1]$.  The final assertion follows from
\eqref{eq:Phi-r}.
\end{proof}

\mg{The previous result establishes the limiting FDI, \(\Phi_{\tau,0}<0\), obtained in the limit \(q\downarrow0\).}
{In the following corollary, we first show that this strict negativity persists on each fixed frequency interval bounded away from zero when $q>0$ is sufficiently small. We then treat the neighborhood of zero separately using a joint expansion in $q$ and $\theta$.}
\mg{\begin{corollary}[outer limiting negativity]\label{cor:outer}
For every $\theta_\star\in(0,\pi]$, there exist constants $m_1>0$ and
$q_1>0$, depending on $(\tau,c,\theta_\star)$, such that
\begin{equation}\label{eq:outerN}
        N_{\tau,q}(\theta)\le -m_1<0,
        \qquad
\theta\in[\theta_\star,\pi],\quad 0<q\le q_1.
\end{equation}
\end{corollary}
\begin{proof}
By Proposition~\ref{prop:rpositive} and \eqref{eq:Phi-r},
$\Phi_{\tau,0}(\theta)<0$ on $[0,\pi]$. For fixed
$\theta_\star>0$, the denominators in the rational representation of
$\Phi_{\tau,q}$ are uniformly bounded away from zero on
$[\theta_\star,\pi]$ for all sufficiently small $q$. Consequently,
$\Phi_{\tau,q}\to\Phi_{\tau,0}$ uniformly on $[\theta_\star,\pi]$ as
$q\downarrow0$. Since $\Phi_{\tau,0}$ is strictly negative on this
compact interval, there exist $\eta>0$ and $q_1>0$ such that
\[
        \Phi_{\tau,q}(\theta)\le-\eta,
        \qquad
        \theta\in[\theta_\star,\pi],\quad 0<q\le q_1.
\]
The denominator in \eqref{eq:Ndef} is positive and uniformly bounded
below on the same interval. Multiplying by this denominator therefore
yields \eqref{eq:outerN}, after possibly decreasing $q_1$.
\end{proof}}

\mgd{Corollary~\ref{cor:outer} establishes the required negativity
uniformly on every frequency interval bounded away from the
origin. It therefore remains to control the low-frequency
regime, which is singular because the limiting denominator of the FDI vanishes at $\theta=0$. This requires a joint analysis in
$q$ and $\theta$.}\looseness=-1

\mgd{For fixed $\tau$ and $c$, the numerator $N_{\tau,q}(\theta)$
in \eqref{eq:Ndef}--\eqref{eq:S-correct} is an even
trigonometric polynomial whose Fourier coefficients are
analytic in $q$ near $q=0$. Indeed, $\rho(0)=1$, so the
negative powers of $\rho$, $z$, and $\bar z$ appearing in
its representation remain nonsingular for all sufficiently
small $q$. Consequently, there exist $q_0>0$ and
$C_6<\infty$, depending on $(\tau,c)$, such that
\[
\sup_{\substack{0\le q\le q_0\\
                0\le\theta\le 1/2}}
\left|\partial_\theta^6N_{\tau,q}(\theta)\right|
\le 6! \cdot C_6.
\]
Taylor's theorem and the evenness of $N_{\tau,q}$ therefore
give, for $0\le\theta\le 1/2$ and $0\le q\le q_0$,
\begin{equation}\label{eq:tay6-main}
N_{\tau,q}(\theta)
=
a_0(q)+a_2(q)\theta^2+a_4(q)\theta^4
+R_6(\theta,q),
\qquad
\abs{R_6(\theta,q)}\le C_6\theta^6.
\end{equation}
In particular, the remainder bound is uniform in $q$
for fixed $(\tau,c)$.}

\mgd{The leading Taylor coefficients are identified in
Lemma~\ref{lem:coeffs} of Appendix~\ref{sec:finiteq}.
Specifically, as $q\downarrow0$,
\begin{equation}\label{eq:coeff-asymp-main}
a_0(q)=-\mgd{B_{\mathrm{cub}}}q^3+\cO(q^4),
\qquad
a_2(q)=-\lambda q+\cO(q^2),
\qquad
a_4(q)=-A_0+\cO(q),
\end{equation}
where
\begin{equation}\label{eq:constants-coeff-main}
\mgd{B_{\mathrm{cub}}}=\frac{(\tau+1)c(1-c)^2}{M^3},
\qquad
\lambda=\frac{(\tau+1)(2-3c)}{M},
\qquad
A_0=\frac{M(\tau+1)}{3}.
\end{equation}
Here, the expansion of $a_4(q)$ follows from its
analyticity in $q$ and the identity
\[
a_4(0)
=
M^2\Phi_{\tau,0}(0)
=
-\frac{M(\tau+1)}{3}
=
-A_0,
\]
where $\Phi_{\tau,0}(0)$ denotes the continuous extension
at the origin established in
Lemma~\ref{lem:Phi0-endpoint}. The detailed derivation
of the coefficients is provided in
Lemma~\ref{lem:coeffs}.}

\mgd{Since $c\in(0,2/3)$, all three constants in
\eqref{eq:constants-coeff-main} are strictly positive.
Therefore, by \eqref{eq:coeff-asymp-main}, there exists
$q_\ast\in(0,q_0]$ such that
\[
a_0(q)\le-\frac{\mgd{B_{\mathrm{cub}}}}{2}q^3,
\qquad
a_2(q)\le-\frac{\lambda}{2}q,
\qquad
a_4(q)\le-\frac{A_0}{2},
\qquad
0<q\le q_\ast.
\]
Choose $0<\theta_\star\le1/2$ sufficiently small that
$
C_6\theta_\star^2\le\frac{A_0}{4}.
$
Combining these inequalities with the uniform Taylor
remainder bound in \eqref{eq:tay6-main}, we obtain, for
$0\le\theta\le\theta_\star$ and $0<q\le q_\ast$,
\begin{equation}\label{eq:near-neg}
\begin{aligned}
N_{\tau,q}(\theta)
\le
-\frac{\mgd{B_{\mathrm{cub}}}}{2}q^3
-\frac{\lambda}{2}q\theta^2
-\frac{A_0}{2}\theta^4
+C_6\theta^6
\le
-\frac{\mgd{B_{\mathrm{cub}}}}{2}q^3
-\frac{\lambda}{2}q\theta^2
-\frac{A_0}{4}\theta^4
<0.
\end{aligned}
\end{equation}}

\mgd{Thus, $N_{\tau,q}$ is strictly negative in a
neighborhood of the origin for all sufficiently small
$q>0$. On the complementary interval
$\theta\in[\theta_\star,\pi]$,
Corollary~\ref{cor:outer} gives
$N_{\tau,q}(\theta)<0$ uniformly for all sufficiently
small $q>0$. Taking the minimum of the admissible
$q$-bounds for the two frequency regions, we obtain
a threshold $q_{\rm fdi}(\tau,c)>0$ such that
$N_{\tau,q}(\theta)<0$ for all
$\theta\in[0,\pi]$ and $0<q\le q_{\rm fdi}(\tau,c)$.}

\mgd{Finally, by the evenness and $2\pi$-periodicity of
$N_{\tau,q}$, together with the positivity of the
denominator in \eqref{eq:Ndef} for $q>0$, we conclude that
\begin{equation}\label{ineq-Ntauq-expansion}
\Phi_{\tau,q}(\theta)<0,
\qquad
\theta\in[0,2\pi),\quad
0<q\le q_{\rm fdi}(\tau,c).
\end{equation}
Hence hypothesis \textbf{(H2)} holds for all sufficiently
small $q>0$.}
\mgd{Taking a common positive $q$-bound for \textbf{(H1)} and \textbf{(H2)}, and reducing it to lie \mgd{strictly below} $M/(2M+1)$, makes Lemma~\ref{lemm:ratecert} applicable. Undoing the normalization and translation then proves Theorem~\ref{thm:main-nonexplicit-constants}, with $\alpha=1/(M^2L)$ and the stated rate $\rho=1-c/(M\sqrt\kappa)$.} 
\mgd{While this argument is asymptotic in $q$, we also give a detailed (non-asymptotic) finite-$q$ argument in
Appendix~\ref{sec:finiteq}. In particular,  Definition~\ref{def:qfdi-explicit} and
Proposition~\ref{prop:explicit-fdi-threshold}
provide an explicit admissible threshold $q_{\rm fdi}(\tau,c)$, while
Corollary~\ref{cor:Phi-uniform-margin}
establishes a quantitative uniform negativity margin
for the FDI.}

\section{Iteration Complexity with Explicit Constants}\label{subsec:main-proof}
\mgd{
Although Theorem~\ref{thm:main-nonexplicit-constants}
establishes an accelerated convergence rate, it
guarantees only the existence of a condition-number
threshold $\kappa_0$ and a finite transient prefactor
$\mgd{C_{\mathrm{rate}}}$, without specifying admissible values for either.
The following theorem strengthens this result by
establishing polynomial upper bounds on $\kappa_0$
in terms of $\tau$ and on $\mgd{C_{\mathrm{rate}}}$ in terms of $\tau$,
$\kappa$, and the prescribed initial history.
While the theorem states these bounds in big-$O$
notation, explicit admissible formulas for both
quantities are provided in the appendices. We provide a proof sketch below, summarizing the
main ideas at a high level. The explicit verification of hypotheses
\textbf{(H2)} and \textbf{(H1)}, including their
admissible condition-number thresholds, is given
in Appendices~\ref{sec:finiteq}
and~\ref{sec:midpoint}, respectively.
The explicit prefactor bound is established in
Appendix~\ref{sec:explicit}, using the auxiliary
estimates in Appendix~\ref{app:gains}.
}
\begin{theorem}[Polynomial bounds on the threshold and prefactor]
\label{thm:main}
In the setting of Theorem~\ref{thm:main-nonexplicit-constants},
\mgd{let $M=2\tau+1$. For an arbitrary prescribed history
$x_{-\tau-1},\ldots,x_0$, one may choose
\beq
\kappa_0(\tau,c)
=
\mathcal O_c(M^8),
\qquad
C_{\mathrm{rate}}
=
\mathcal O_c(M^{11/2}\kappa^2)\Psi_0,\label{eq-kappa-zero}
\eeq
where
\[
\Psi_0=
\left(
\sum_{j=-\tau-1}^{0}\|x_j-x_*\|^2
\right)^{1/2}
\]
is the initial-history norm defined in
\eqref{def-Psi-zero}. The threshold $\kappa_0(\tau,c)$
is independent of the initial history.} \mgd{Furthermore, under the constant-history initialization}
\[
x_{-\tau-1}=\cdots=x_0=x_{\rm init},
\]
\mgd{for some $x_{\rm init}\in\mathbb R^d$, the prefactor
admits the improved bound}
\begin{equation}\label{eq:main-blackbox}
\mgd{
C_{\mathrm{rate}}
=
\mathcal O_c(M^5\kappa^2)
\|x_{\rm init}-x_*\|.
}
\end{equation}
\end{theorem}
 \begin{proof}[Proof sketch.] The detailed estimates are given in the appendices.
The threshold $\kappa_0(\tau,c)$ is defined in Definition~\ref{def:q0-explicit},
and its admissibility is established in
Proposition~\ref{prop-admissibility-combined-threshold}.
The prefactor bounds are given in
Proposition~\ref{prop:normalized-direct-prefactor}.
We outline the argument below.%
 
 \mgd{Lemma~\ref{lemm:ratecert} gives $\|x_k-x_*\|\le C_{\mathrm{rate}}\rho^k$, with $\rho=1-cq/M$
and $q=\kappa^{-1/2}$, under two hypotheses: \textbf{(H1)} the roots of the midpoint polynomial
$P_{\tau,\delta_m}$ in \eqref{eq:midpoint-charpoly} lie in $\{|z|<\rho\}$, and \textbf{(H2)} the
function $\Phi_{\tau,q}$ in \eqref{eq:fdi} is negative on $[0,2\pi)$. The proof of
Theorem~\ref{thm:main-nonexplicit-constants} verifies each hypothesis only for all sufficiently
small $q$, by a limiting argument as $q\downarrow0$, and expresses $C_{\mathrm{rate}}$ in
\eqref{eq-C-rate} through a storage function that is known to exist but is not quantified. We make
these three points explicit: Steps~1 and~2 treat the hypotheses and give the threshold $\kappa_0$,
whereas Step~3 concerns quantifying the prefactor. Without loss of generality, we work in the
normalization \eqref{eq:norm}--\eqref{eq:params} with the minimizer at the origin, so that
$\mu=1$, $L=\kappa$, and $x_*=0$.}

\smallskip
\mgd{\noindent\emph{Step 1: \textup{Showing \textbf{(H1)}} holds for $q\le q_{\rm mid}=\Theta(M^{-2})$.} For this purpose, we compare $P_{\tau,\delta_m}$ with its coefficientwise limit
$Q_\tau$ as $q\downarrow0$. Lemma~\ref{lem:midpoint-poly-gap} establishes Schur
stability of $Q_\tau$, while
Proposition~\ref{prop:midpoint-unit-floor} gives
$|Q_\tau(z)|\ge 1/(10M^2)$ on the unit circle. %
We then show that moving from the unit circle into the annulus $\rho\le|z|\le1$, or
from $Q_\tau$ to $P_{\tau,\delta_m}$, changes these values by $\cO(q)$, with explicit
estimates (equations \eqref{eq:Q-floor-on-rho} and \eqref{eq:PminusQ-bound} in Proposition \ref{prop:explicit-mid-threshold}). Consequently, for $q\le q_{\rm mid}$ with an
explicit $q_{\rm mid}=\Theta(M^{-2})$, we can argue that $Q_\tau$ has no root in this annulus and
$|P_{\tau,\delta_m}-Q_\tau|<|Q_\tau|$ on the circle $|z|=\rho$, so Rouch\'e's theorem
places all roots of $P_{\tau,\delta_m}$ in $\{|z|<\rho\}$, which is \textbf{(H1)}
(Proposition~\ref{prop:explicit-mid-threshold}).}

\smallskip
\mgd{\noindent\emph{Step 2: Verifying \textup{\textbf{(H2)}} with a uniform FDI margin
for $0<q\le q_{\rm fdi}=\Theta_c(M^{-4})$.}
By Proposition~\ref{prop:rational}, $\Phi_{\tau,q}$ is the quotient
of an even trigonometric polynomial $N_{\tau,q}$ by a positive
function. Thus, \textbf{(H2)} is equivalent to
$N_{\tau,q}(\theta)<0$ for all $\theta\in[0,\pi]$.
The limiting numerator $N_{\tau,0}$ is negative for
$0<\theta\le\pi$ by Proposition~\ref{prop:rpositive}, but its
leading term near $\theta=0$ is proportional to $-\theta^4$.
Its negativity at each fixed nonzero frequency therefore does
not by itself give a single admissible bound on $q$ valid
at all frequencies. To control the region where both $q$ and $\theta$ are small,
we also examine the perturbation $N_{\tau,q}-N_{\tau,0}$.
Equations~\eqref{eq:coeff-asymp} and \eqref{eq:constants-coeff}
in Lemma~\ref{lem:coeffs} identify negative leading terms of
orders $q^3/M^2$ and $q\,\theta^2$; the coefficient of
$q\,\theta^2$ is negative because $c<2/3$.
Combining these terms with the negative $\theta^4$ contribution
and controlling the remainders using derivative bounds
polynomial in $M$ (Lemma~\ref{lem:closed-fdi-envelope}),
Proposition~\ref{prop:explicit-fdi-threshold} gives an explicit
$q_{\rm fdi}=\Theta_c(M^{-4})$ such that
\begin{equation}\label{eq:outline-two-scale}
N_{\tau,q}(\theta)
\le -a_c\Bigl(\frac{q^3}{M^2}+q\,\theta^2+M^2\theta^4\Bigr),
\qquad
0\le\theta\le\pi,\quad 0<q\le q_{\rm fdi},
\end{equation}
where $a_c>0$ depends only on $c$; see
\eqref{eq:global-margin-explicit} and
\eqref{eq:Gamma-uniform-lower}. The denominator of $\Phi_{\tau,q}$ is bounded above by
$36M^2\max\{\theta,q/M\}^4$. Comparing this bound with
\eqref{eq:outline-two-scale} in the two regions
$\theta\ge q/M$ and $\theta<q/M$ yields
\[
-\Phi_{\tau,q}(\theta)\ge\eta_c^\Phi>0,
\]
where $\eta_c^\Phi$ depends only on $c$
(Corollary~\ref{cor:Phi-uniform-margin}). Thus, the FDI margin
is uniform in $\theta$, $q$, and $M$ over the stated range.
Taking
$q_0:=\min\{q_{\rm mid},q_{\rm fdi}\}=\Theta_c(M^{-4})$,
both \textbf{(H1)} and \textbf{(H2)} hold whenever
$\kappa\ge\kappa_0(\tau,c):=q_0^{-2}=\cO_c(M^8)$.}
 
\smallskip
\mgd{\noindent\emph{Step 3: Bounding the prefactor.}}
\mgb{To control the transient prefactor, we define
$
m_q:=\frac{1+\kappa}{2}$ and
$
\omega_q:=\frac{\kappa-1}{2}$,
the midpoint and half-width of the interval $[1,\kappa]$. We separate the normalized gradient into a linear component with slope $m_q$ and a remainder by setting
$
\widehat v_k:=\widehat u_k-m_q\widehat y_k.
$
Thus, $\widehat v_k$ measures the deviation of the normalized gradient from its midpoint-slope linear approximation. Smoothness and strong convexity imply
$
\|\widehat v_k\|\leq\omega_q\|\widehat y_k\|.
$
We therefore introduce the rescaled remainder
\[
\widehat w_k:=\frac{\widehat v_k}{\omega_q},
\qquad
\|\widehat w_k\|\leq\|\widehat y_k\|.
\]
This rescaling allows us to work with a bound on the nonlinear remainder that is independent of $\kappa$, while retaining the condition-number dependence explicitly in the linear relation $\widehat u_k=m_q\widehat y_k+\omega_q\widehat w_k$. This separation facilitates sharper estimates of the transient prefactor \mgd{(Remark~\ref{remark-improving-kappa-scaling})}.}

The uniform FDI margin and midpoint-stability estimate then yield
quantitative gain bounds for the rate-weighted recursion driven by
\(\widehat w\). \mgd{Combining these estimates with the hard IQC and the
arbitrary-history bound on the normalized augmented
initial state gives
(Proposition~\ref{prop:normalized-direct-prefactor})}
\[
\|x_k-x_*\|
\leq
C_{\rm cl}^{\rm norm}\rho^k,
\qquad
\forall k\geq0,
\quad
\mbox{with}
\quad
C_{\rm cl}^{\rm norm}
=
\cO_c\!\left(
M^{11/2}\kappa^2 \Psi_0
\right),
\]
\mgd{for an arbitrary initialization} $x_{-\tau-1},\ldots,x_0$. \mgd{Under constant initialization, the differences
$x_j-x_{j-1}$ vanish for $j\leq 0$, and the norm of the extrapolated sequence $y_j = x_j + \nu (x_j - x_{j-1})$ can be bounded with smaller constants; i.e. extrapolation does not amplify the norm squared of the initial errors
by the factor $\nu=M-\tfrac12$. This in turn yields slightly better scaling in $M$:}
\[
\|x_k-x_*\|
\leq
C_{\rm cl}^{\rm norm}\rho^k,
\qquad
\forall k\geq0,
\quad
\mbox{with}
\quad
C_{\rm cl}^{\rm norm}
=
\cO_c\!\left(
M^5\kappa^2
\|x_{\rm init}-x_*\|
\right).
\]
\mgd{The details are given in
Proposition~\ref{prop:normalized-direct-prefactor}}.  Finally, undoing the normalization gives
\begin{equation}\label{eq:proof-unscaled-params}
\alpha
=
\frac{q^2}{M^2\mu}
=
\frac{1}{M^2L},
\qquad
\beta
=
1-\frac{2q}{M}
=
1-\frac{2}{M\sqrt{\kappa}},
\qquad
\rho
=
1-\frac{c}{M\sqrt{\kappa}}.
\end{equation}
Thus the claimed rate, threshold, and prefactor bounds follow.

\end{proof}

\begin{corollary}[Iteration complexity]
\label{cor:main-complexity}
\mgd{Fix an integer $\tau\ge1$, $c\in(0,2/3)$, and
$\kappa\ge\kappa_0(\tau,c)$ as in
Theorem~\ref{thm:main}, and let $M=2\tau+1$.
For $f\in\Sml$, run \DREAM{} with the parameters
of that theorem and constant-history initialization
$x_{-\tau-1}=\cdots=x_0=x_{\rm init}$.
Define $R:=\|x_{\rm init}-x_*\|$, and choose
the prefactor $C_{\mathrm{rate}}$ from
Theorem~\ref{thm:main} so that
$C_{\mathrm{rate}}=\cO_c(M^5\kappa^2)R$,
as in \eqref{eq:main-blackbox}.}

\mgd{For every $\epsilon>0$, the bound
$\|x_k-x_*\|\le\epsilon$ holds for every
integer $k\ge N_\epsilon$, where
\begin{equation}\label{eq:main-complexity}
\begin{aligned}
N_\epsilon
&:=
\left\lceil
\frac{M\sqrt{\kappa}}{c}
\log\max\left\{1,\frac{C_{\mathrm{rate}}}{\epsilon}\right\}
\right\rceil
=
\cO_c\!\left(
M\sqrt{\kappa}
\left[
\log M+\log\kappa
+\log\max\left\{1,\frac{R}{\epsilon}\right\}
\right]
\right).
\end{aligned}
\end{equation}
Thus, the polynomial prefactor contributes only
logarithmic overhead in $M$ and $\kappa$.
If $R=0$, the iterates remain at $x_*$ and
no updates are needed.}

\mgd{For every $\epsilon_f>0$, a sufficient number of
updates to obtain $f(x_k)-f(x_*)\le\epsilon_f$ is
\[
\cO_c\!\left(
M\sqrt{\kappa}
\left[
\log M+\log\kappa
+\log\max\left\{1,\frac{\mu R^2}{\epsilon_f}\right\}
\right]
\right).
\]
For $R>0$, this matches, as $\epsilon_f\downarrow0$,
the leading scaling of the delayed first-order
span-model lower bound in Theorem~\ref{thm:asr-lower},
up to the additional $\log M+\log\kappa$ terms,
within the certified regime
$\kappa\ge\kappa_0(\tau,c)$.}
\end{corollary}

\begin{proof}
\mgd{If $R=0$, the prescribed history is identically
$x_*$, and the iterates remain at $x_*$.
Assume henceforth that $R>0$. By Theorem~\ref{thm:main}, there exists $A_c\ge1$,
depending only on $c$, such that
$C_{\mathrm{rate}}\le A_cM^5\kappa^2R$.
The constant $A_c$ can be chosen explicitly in our
analysis; see Proposition~\ref{prop:normalized-direct-prefactor}.
Using $1-t\le e^{-t}$, we obtain
\[
\|x_k-x_*\|
\le
C_{\mathrm{rate}}
\left(1-\frac{c}{M\sqrt{\kappa}}\right)^k
\le
C_{\mathrm{rate}}
\exp\!\left(-\frac{ck}{M\sqrt{\kappa}}\right).
\]
The definition of $N_\epsilon$ therefore ensures
$\|x_k-x_*\|\le\epsilon$ for every integer
$k\ge N_\epsilon$.
Moreover,
\[
N_\epsilon
\le
1+\frac{M\sqrt{\kappa}}{c}
\left[
\log A_c+5\log M+2\log\kappa
+\log\max\left\{1,\frac{R}{\epsilon}\right\}
\right].
\]
Since $M\ge3$, the constant terms can be absorbed
into the $\log M$ term, proving
\eqref{eq:main-complexity}.}

\mgd{For objective-value accuracy, $L$-smoothness and
$\nabla f(x_*)=0$ give
$
f(x_k)-f(x_*)
\le L\|x_k-x_*\|^2/2.
$
Thus, for every $\epsilon_f>0$, it suffices to apply
the distance guarantee with
$\epsilon=\sqrt{2\epsilon_f/L}$.
The corresponding sufficient update count is
\[
N_{\epsilon}
=
\left\lceil
\frac{M\sqrt{\kappa}}{2c}
\log\max\left\{
1,\frac{LC_{\mathrm{rate}}^2}{2\epsilon_f}
\right\}
\right\rceil.
\]
Using $L=\mu\kappa$ and the prefactor bound,
\[
\begin{aligned}
\log\max\left\{
1,\frac{LC_{\mathrm{rate}}^2}{2\epsilon_f}
\right\}
\le{}
2\log A_c+10\log M+5\log\kappa
+\log\max\left\{1,\frac{\mu R^2}{\epsilon_f}\right\}.
\end{aligned}
\]
This proves the stated objective-value complexity
bound for every $\epsilon_f>0$.}

\mgd{By Lemma~\ref{lem:dream-span}, the iterate $x_k$
corresponds to iteration $t=k+\tau$ in
Theorem~\ref{thm:asr-lower}, which gives the worst-case
objective-error lower bound
$(\mu R^2/4)\exp\!\bigl(-5(k+\tau)/
[(\tau+1)(\sqrt{\kappa}-1)]\bigr)$.
A uniform guarantee of error at most $\epsilon_f$
requires this lower bound to be at most $\epsilon_f$;
taking logarithms and rearranging yields
\[
k+\tau
\ge
\frac{(\tau+1)(\sqrt{\kappa}-1)}{5}
\log\frac{\mu R^2}{4\epsilon_f}.
\]
Here, $M=\Theta(\tau+1)$ and
$\sqrt{\kappa}-1=\Theta(\sqrt{\kappa})$.
As $\epsilon_f\downarrow0$, the additive index shift
and the constant factor inside the logarithm do not
affect the leading scaling. Comparing this lower
bound with the objective-value upper bound proves
the claimed agreement up to the additional
$\log M+\log\kappa$ terms.}\looseness=-1
\end{proof}

\begin{remark}[Quantitative bounds]
\mgd{
The main significance of the quantitative estimates
in Theorem~\ref{thm:main} lies in establishing
explicit polynomial dependence on the delay
$\tau$ and condition number $\kappa$.
Obtaining these bounds requires delicate
frequency-domain estimates and careful control
of transient effects. Although the resulting
bounds are conservative, our numerical
experiments in Section~\ref{sec-num-exp} demonstrate accelerated
behavior at condition numbers far below the
sufficient thresholds established by the
explicit analysis. Moreover, the polynomial
bound on the transient prefactor $\mgd{C_{\mathrm{rate}}}$ ensures that
it contributes only logarithmic overhead to
the iteration complexity as argued in Corollary \ref{cor:main-complexity}, preserving the accelerated scaling up to logarithmic factors.
}
\end{remark}
\section{Acceleration under delay for smooth convex objectives}\label{sec:convex-restart}

Our strongly convex analysis can be extended to smooth convex objectives that are not strongly convex. Let $\mathcal{S}_{0,L}$ denote the class of functions that are $L$-smooth and convex (but not necessarily strongly convex) \sa{on $\R^d$}. Let $f\in S_{0,L}$, assume that \sa{$\arg\min_x f(x)$} is nonempty, and fix \sa{an arbitrary} $x_*\in\sa{\arg\min_x f(x)}$.  Write
\[
        R_\star:=\norm{x_0-x_*}.
\]
The quantity $R_\star$ will appear only in the convergence bound, and it is not assumed to be known to the algorithm. %
\mgd{We adopt the relative-error stopping criterion
of the Catalyst framework~\cite{lin2015,lin2018catalyst}
and recall the corresponding accelerated
proximal-point convergence guarantee.}
\begin{lemma}[relative-error accelerated proximal point]\label{lem:convex-catalyst-outer}
\sa{Let $f\in \mathcal{S}_{0,L}$ \mgd{with $L>0$, fix $x_*\in\arg\min f$, and let} $\lambda>0$. Given $v_0\in\mathbb{R}^d$,} set $x_0=v_0$, $a_0=1$ and \mgd{$R_* = \|x_0 - x_*\|$}. For each $s\ge1$, define
\begin{equation}\label{eq:catalyst-subproblem}
        H_s(x):=f(x)+\frac{\lambda}{2}\norm{x-v_{s-1}}^2,
        \qquad
        H_s^\star:=\min_x H_s(x),
\end{equation}
and suppose that $x_s$ satisfies the relative accuracy condition
\begin{equation}\label{eq:catalyst-relative-criterion}
        H_s(x_s)-H_s^\star
        \le
        \frac{\delta_s\lambda}{2}\norm{x_s-v_{s-1}}^2,
        \qquad
        \delta_s:=\frac{1}{(s+1)^2}.
\end{equation}
Choose $a_s\in(0,1)$ as the positive solution of
\begin{equation}\label{eq:catalyst-a-recurrence}
        a_s^2=(1-a_s)a_{s-1}^2,
\end{equation}
and set
\begin{equation}\label{eq:catalyst-extrapolation}
        \theta_s
        :=
        \frac{a_{s-1}(1-a_{s-1})}{a_{s-1}^2+a_s},
        \qquad
        v_s:=x_s+\theta_s(x_s-x_{s-1}).
\end{equation}
Then, for every $s\ge1$,
\begin{equation}\label{eq:catalyst-outer-rate}
        f(x_s)-f(x_*)
        \le
        \frac{4\lambda R_\star^2}{(s+1)^2}.
\end{equation}
\end{lemma}
\begin{proof}
The result is a direct specialization of Theorem~7 and Proposition~9 of Lin, Mairal, and Harchaoui~\cite{lin2018catalyst}. We verify the correspondence and constants for completeness.

In the notation of \cite{lin2018catalyst}, the proximal regularization parameter denoted by $\kappa$ corresponds to $\lambda$ here. \mgd{
Since no positive strong-convexity constant
is assumed for the outer objective, we set
their $\mu=0$, and hence their
$q=\mu/(\mu+\kappa)=0$.
Their Algorithm~2 then uses $\alpha_0=1$
and $\alpha_s=a_s$.
Here $q$ denotes the Catalyst parameter,
which is distinct from the condition-number
parameter $q=\kappa^{-1/2}$ used in
Section~\ref{sec:quadratic}.} %
Equations~(10)--(11) of that algorithm reduce exactly to
\eqref{eq:catalyst-a-recurrence}--\eqref{eq:catalyst-extrapolation}. Likewise, their relative stopping criterion (C2) coincides with
\eqref{eq:catalyst-relative-criterion}.

Theorem~7 of \cite{lin2018catalyst} yields
\beq
f(x_s)-f(x_*)
\le
\frac{A_{s-1}}
{\prod_{j=1}^s(1-\delta_j)}
\frac{\lambda}{2}\norm{x_0-x_*}^2,
\label{ineq-thm7-bound}
\eeq
where
$
A_s:=\prod_{j=1}^s(1-a_j),
$ and $A_0:=1$. By \cite[Lemma 4]{lin2018catalyst},

$$
A_s
\le
\frac{4}
{\left(2+s\sqrt{\gamma_0/\lambda}\right)^2},
$$
where
$
\gamma_0
:=
(\lambda+\mu)a_0(a_0-q).
$
Since $\mu=0$, $q=0$, and $a_0=1$, we obtain $\gamma_0=\lambda$, and hence
$
A_s
\le
{4}/{(s+2)^2}.
$
For the choice
$
\delta_j=\frac{1}{(j+1)^2},
$
the product in the denominator telescopes:
$$
\prod_{j=1}^s(1-\delta_j)
=
\prod_{j=1}^s
\frac{j}{(j+1)} \cdot \frac{(j+2)}{(j+1)}
=
\frac{s+2}{2(s+1)}
\ge
\frac12.
$$
Substituting these estimates into \eqref{ineq-thm7-bound} \mgd{and using $A_{s-1}\leq 4/(s+1)^2$}, we  conclude.
\end{proof}

We now instantiate each proximal subproblem with delayed GMM.  Fix $\tau\ge1$ and $c\in(0,2/3)$, and recall
$
        M:=2\tau+1.
$
For a regularization parameter $\lambda>0$, define the \mgd{condition number bound for the proximal subproblems} %
\begin{equation}\label{eq:catalyst-inner-kappa}
        \chi_\lambda
        :=
        \frac{L+\lambda}{\lambda}
        =
        1+\frac{L}{\lambda}.
\end{equation}
Assume $\chi_\lambda\ge\kappa_0(\tau,c)$, where $\kappa_0(\tau,c)$ is the threshold of Theorem~\ref{thm:main}.  By Theorem~\ref{thm:main}, under constant-history initialization there is a dimensionless prefactor $B_{\tau,c}(\chi_\lambda)$ such that \DREAM{} applied to any $\lambda$-strongly convex and $(L+\lambda)$-smooth objective satisfies
\begin{equation}\label{eq:catalyst-inner-distance-rate}
        \norm{z_t-z_\star}
        \le
        B_{\tau,c}(\chi_\lambda)
        \left(1-\frac{c}{M\sqrt{\chi_\lambda}}\right)^t
        \norm{z_0-z_\star}.
\end{equation}
Moreover, Theorem~\ref{thm:main} gives
\begin{equation}\label{eq:catalyst-inner-prefactor}
        B_{\tau,c}(\chi_\lambda)
        =
        \cO_c\!\left(M^5\chi_\lambda^2\right).
\end{equation}
\mgd{Although this estimate is stated in $\cO(\cdot)$ notation,
an explicit admissible prefactor is also available.
Indeed, applying \eqref{eq:Cnorm-growth-constant-history}
of Appendix~\ref{sec:explicit} to the regularized
subproblem, \mgd{using the condition number bound} $\chi_\lambda$,
we may choose
$B_{\tau,c}(\chi_\lambda)
=2K_c^{\rm norm}M^5\chi_\lambda^2$,
where $K_c^{\rm norm}$ is defined in
\eqref{eq:Cnorm-growth}.} \mgd{We use this choice below; in particular,
$B_{\tau,c}(\chi_\lambda)\ge1$.}

At outer iteration $s$, initialize the entire inner delayed history at the current proximal center $v_{s-1}$ and apply \mgd{\DREAM{}} to $H_s$ in \eqref{eq:catalyst-subproblem} with parameters
\begin{equation}\label{eq:catalyst-inner-parameters}
        \alpha_{\rm in}
        =
        \frac{1}{M^2(L+\lambda)},
        \qquad
        \beta_{\rm in}
        =
        1-\frac{2}{M\sqrt{\chi_\lambda}},
        \qquad
        \nu_{\rm in}
        =
        M-\frac12.
\end{equation}
Thus, an inner delayed gradient evaluation at a query point $u$ uses
\begin{equation}\label{eq:catalyst-inner-gradient}
        \nabla H_s(u)
        =
        \nabla f(u)+\lambda(u-v_{s-1}).
\end{equation}
More specifically, \DREAM~as an inner solver for the subproblem \eqref{eq:catalyst-subproblem} consists of the updates
$$
y_j
=
z_j+\nu_{\rm in}(z_j-z_{j-1})
\quad
\mbox{followed by}
\quad
z_{j+1}
=
z_j
-
\alpha_{\rm in}\nabla H_s(y_{j-\tau})
+
\beta_{\rm in}(z_j-z_{j-1}),
$$
with initialization 
$z_{-\tau-1}=\dots=z_{-1}=z_0=v_{s-1}$. \mgd{All initial query points therefore equal $v_{s-1}$, and their gradients
can be initialized using the single value
$\nabla H_s(v_{s-1})=\nabla f(v_{s-1})$.
The inner iteration counts below start after this value is available;
we account separately for the startup delay at each outer iteration.
}

\begin{theorem}[accelerated delayed rate for smooth convex objectives]\label{thm:convex-regularized}
\mgd{Let $f\in\mathcal S_{0,L}$ with $L>0$ and assume that $f$ admits a minimum over $\mathbb{R}^d$. Fix $x_*\in\arg\min f$, and initialize $x_0=v_0$, $a_0=1$. Furthermore,} fix $\tau\ge1$ and $c\in(0,2/3)$, and choose $\lambda>0$ such that
\[
        \chi_\lambda=1+\frac{L}{\lambda}\ge\kappa_0(\tau,c).
\]
For each outer iteration $s\ge1$, run \mgd{\DREAM{}} on $H_s$ with the parameters \eqref{eq:catalyst-inner-parameters}, with constant inner history initialized at $v_{s-1}$, for
\begin{equation}\label{eq:catalyst-inner-budget}
        T_s
        :=
        \left\lceil
        \frac{M\sqrt{\chi_\lambda}}{2c}
        \log\!\left(
        4\chi_\lambda B_{\tau,c}(\chi_\lambda)^2(s+1)^2
        \right)
        \right\rceil
\end{equation}
inner iterations, and denote the resulting point by $x_s$. Update $a_s$ and $v_s$ according to \eqref{eq:catalyst-a-recurrence}--\eqref{eq:catalyst-extrapolation}.

\mgd{The iterates and the inner budgets $T_s$ do} not require any a priori bound on $R_\star=\norm{x_0-x_*}$. For every $s\ge1$,
\begin{equation}\label{eq:convex-catalyst-rate}
        f(x_s)-f(x_*)
        \le
        \frac{4\lambda R_\star^2}{(s+1)^2}.
\end{equation}
\mgb{Consequently, if
\begin{equation}\label{eq:convex-catalyst-outer-count}
        K_\epsilon
        :=
        \max\left\{1,
        \left\lceil
        2\sqrt{\frac{\lambda R_\star^2}{\epsilon}}
        \right\rceil
        \right\},
\end{equation}
then $f(x_{K_\epsilon})-f(x_*)\le\epsilon$. The total number
$
        N_\epsilon:=\sum_{s=1}^{K_\epsilon}T_s
$
of \mgd{\DREAM{}} inner \mgd{updates} satisfies
\begin{equation}\label{eq:convex-catalyst-complexity-general}
\begin{split}
        N_\epsilon
        \le{}
        K_\epsilon
        &+
        \frac{M\sqrt{\chi_\lambda}}{2c}
        K_\epsilon
        \left[
        \log\!\left(4\chi_\lambda B_{\tau,c}(\chi_\lambda)^2\right)
        +2\log(K_\epsilon+1)
        \right].
\end{split}
\end{equation}
Since
\[
        K_\epsilon
        \le
        1+2\sqrt{\frac{\lambda R_\star^2}{\epsilon}}
        \qquad\text{and}\qquad
        \lambda \cdot \chi_\lambda=L+\lambda,
\]
it follows that
\begin{equation}\label{eq:convex-catalyst-complexity-tilde}
        N_\epsilon
        =
        \widetilde{\cO}_c\!\left(
        (2\tau+1)
        \left[
        \sqrt{\chi_\lambda}
        +
        \sqrt{\frac{(L+\lambda)R_\star^2}{\epsilon}}
        \right]
        \right),
\end{equation}
where $\widetilde{\cO}$ hides constants depending on $c$ and logarithmic factors in $\tau$, $\chi_\lambda$, and $1+\lambda R_\star^2/\epsilon$.}
\end{theorem}

\begin{proof}
Fix an outer iteration $s$ and let
\[
        p_s:=\arg\min_x H_s(x).
\]
Since $H_s$ is $\lambda$-strongly convex and $(L+\lambda)$-smooth, $\chi_\lambda$ \mgd{is a valid condition number bound.} Theorem~\ref{thm:main}, applied to the inner run initialized at $z_0=v_{s-1}$ with constant initialization (with $z_{-\tau-1}= \dots =z_{-1}=z_0=v_{s-1}$) gives
\begin{equation}\label{eq:catalyst-proof-distance}
        \norm{z_t-p_s}
        \le
        B_{\tau,c}(\chi_\lambda)
        \rho_\lambda^t
        \norm{v_{s-1}-p_s},
        \qquad
        \rho_\lambda
        :=
        1-\frac{c}{M\sqrt{\chi_\lambda}}.
\end{equation}
By $(L+\lambda)$-smoothness and $\nabla H_s(p_s)=0$,
\begin{equation}\label{eq:catalyst-proof-gap-from-distance}
\begin{split}
        H_s(z_t)-H_s^\star
        \le
        \frac{L+\lambda}{2}\norm{z_t-p_s}^2
        \le
        \frac{\lambda\chi_\lambda}{2}
        B_{\tau,c}(\chi_\lambda)^2
        \rho_\lambda^{2t}
        \norm{v_{s-1}-p_s}^2.
\end{split}
\end{equation}
Since $1-u\le e^{-u}$ for $u\in(0,1)$, the choice \eqref{eq:catalyst-inner-budget} implies
$
        B_{\tau,c}(\chi_\lambda)^2
        \rho_\lambda^{2T_s}
        \le
       {\delta_s}/{(4\chi_\lambda)}
$
with $\delta_s = 1/((s+1)^2)$. Hence, 
\begin{equation}\label{eq:catalyst-proof-gap-small}
        H_s(x_s)-H_s^\star
        \le
        \frac{\delta_s\lambda}{8}
        \norm{v_{s-1}-p_s}^2,
\end{equation}
where $x_s = z_{T_s}$. Strong convexity also gives
\[
        \norm{x_s-p_s}^2
        \le
        \frac{2}{\lambda}\bigl(H_s(x_s)-H_s^\star\bigr)
        \le
        \frac{\delta_s}{4}\norm{v_{s-1}-p_s}^2.
\]
Because $\delta_s\le1$,
\[
        \norm{x_s-v_{s-1}}
        \ge
        \norm{v_{s-1}-p_s}-\norm{x_s-p_s}
        \ge
        \frac12\norm{v_{s-1}-p_s}.
\]
Combining this inequality with \eqref{eq:catalyst-proof-gap-small} yields
\[
        H_s(x_s)-H_s^\star
        \le
        \frac{\delta_s\lambda}{2}
        \norm{x_s-v_{s-1}}^2.
\]
Thus the relative accuracy condition \eqref{eq:catalyst-relative-criterion} holds at every outer iteration.  Lemma~\ref{lem:convex-catalyst-outer} therefore gives \eqref{eq:convex-catalyst-rate}.  The choice \eqref{eq:convex-catalyst-outer-count} then implies $f(x_{K_\epsilon})-f(x_*)\le\epsilon$.

It remains to count inner iterations.  Summing \eqref{eq:catalyst-inner-budget} and using $s+1\le K_\epsilon+1$ gives
\[
\begin{split}
        N_\epsilon
        &\le
        K_\epsilon
        +
        \frac{M\sqrt{\chi_\lambda}}{2c}
        \sum_{s=1}^{K_\epsilon}
        \left[
        \log\!\left(4\chi_\lambda B_{\tau,c}(\chi_\lambda)^2\right)
        +2\log(s+1)
        \right]\\
        &\le
        K_\epsilon
        +
        \frac{M\sqrt{\chi_\lambda}}{2c}
        K_\epsilon
        \left[
        \log\!\left(4\chi_\lambda B_{\tau,c}(\chi_\lambda)^2\right)
        +2\log(K_\epsilon+1)
        \right],
\end{split}
\]
which is \eqref{eq:convex-catalyst-complexity-general}.  Since
\[
        \sqrt{\chi_\lambda} \cdot \sqrt{\lambda}
        =
        \sqrt{L+\lambda},
\]
substituting \eqref{eq:convex-catalyst-outer-count} gives \eqref{eq:convex-catalyst-complexity-tilde}.
\end{proof}

\begin{remark}[Initialization of the inner runs]
\mgd{Each inner \DREAM{} run uses constant-history
initialization at the current proximal center
$v_{s-1}$. Consequently, all gradient evaluations
required by the initial history coincide with
$\nabla H_s(v_{s-1})=\nabla f(v_{s-1})$,
so a single gradient evaluation suffices.
Once this gradient is available, \DREAM{}
performs an update at every iteration without
waiting for a fresh gradient. In a fixed-delay oracle implementation,
obtaining the initial gradient at a new proximal
center may require up to $\tau+1$ additional
waiting steps per inner run.
Since $B_{\tau,c}(\chi_\lambda)\ge1$,
$\chi_\lambda\ge1$, $s\ge1$, and $c<2/3$,
the prescribed inner budget satisfies}
$
T_s\ge\frac{3M}{4}\log16>\tau+1
$. Hence, the total number of steps,
including both \DREAM{} updates and
initialization waiting steps, is at most
\[
N_\epsilon+(\tau+1)K_\epsilon
\le 2N_\epsilon.
\]
\mgd{The non-idling property is preserved within
each inner \DREAM{} run: once the initial gradient
is available, the method performs an update at
every iteration using previously available delayed
gradients. Waiting steps arise only between inner
runs, when a newly computed proximal center
requires a fresh gradient evaluation.
Since these initialization steps are explicitly
accounted for above, they do not affect the
asymptotic complexity bound in
\eqref{eq:convex-catalyst-complexity-tilde}. Moreover, for quadratic objectives with
zero initialization, the restarted construction
satisfies the delayed first-order span condition
of Theorem~\ref{thm:asr-lower}.
Indeed, concatenating the inner \DREAM{}
trajectories and inserting $\tau+1$ copies
of each new proximal center before the next
run yields a sequence $\{X_n\}_{n\ge0}$
with $X_0=\cdots=X_\tau=0$ and}
\[
X_{n+1}\in
\operatorname{span}
\{\nabla f(X_0),\ldots,\nabla f(X_{n-\tau})\},
\qquad n\ge\tau.
\]
\mgd{This follows by induction, using the
affine-gradient identity for quadratic objectives,
the linearity of the outer extrapolation, and
the identity
$\nabla H_s(y)=\nabla f(y)+\lambda(y-v_{s-1})$. Consequently, the delayed span-model lower
bound applies to the restarted construction
when the initialization steps are included
in the iteration count.}
\end{remark}

\mgb{
\begin{corollary}[Iteration complexity and near-optimality]\label{cor:main-complexity-convex} \mgd{In the setting of Theorem~\ref{thm:convex-regularized}}, fix $\tau\ge1$ and $c\in(0,2/3)$. Define
\beq\label{eq-catalyst-param-choice}
        \bar\kappa_0(\tau,c)
        :=
        \max\{2,\kappa_0(\tau,c)\},
        \qquad
        \lambda
        :=
        \frac{L}{\bar\kappa_0(\tau,c)-1}.
\eeq
Then $\chi_\lambda=\bar\kappa_0(\tau,c)$ and $L+\lambda\le 2L$. Hence, by \eqref{eq:convex-catalyst-complexity-tilde}, for every target accuracy $\epsilon>0$, the proximal-point-type method of Theorem~\ref{thm:convex-regularized}, with \mgd{\DREAM{}} applied to the inner problems, produces an outer iterate $x_{K_\epsilon}$ satisfying
\[
        f(x_{K_\epsilon})-f(x_*)\le\epsilon
\]
using a total of
\mgd{\begin{equation}\label{eq:main-complexity-convex}
N_\epsilon
=
\widetilde{\cO}_c\!\left(
\tau\left[
\sqrt{\bar\kappa_0(\tau,c)}
+
\sqrt{\frac{LR_\star^2}{\epsilon}}
\right]
\right)
\end{equation}}
\mgd{\DREAM{} inner updates}. Moreover, in the high-accuracy regime
\begin{equation}\label{eq:convex-catalyst-high-accuracy}
        \epsilon
        \le
        \lambda R_\star^2
        =
        \frac{LR_\star^2}{\bar\kappa_0(\tau,c)-1},
\end{equation}
the first term in \eqref{eq:main-complexity-convex} is absorbed by the second, and the complexity simplifies to
\mgd{\begin{equation}\label{eq:convex-complexity}
N_\epsilon
=
\widetilde{\cO}_c\!\left(
\tau\sqrt{\frac{LR_\star^2}{\epsilon}}
\right).
\end{equation}}
Thus, in this regime \mgd{and for fixed $c$}, the iteration complexity matches, up to logarithmic factors, the delayed first-order span-model lower-bound scale
\[
        \Omega\!\left(
        \tau\sqrt{\frac{LR_\star^2}{\epsilon}}
        \right)
\]
of Arjevani, Shamir, and Srebro on quadratic instances; see Theorem~\ref{thm:asr-lower}.
\end{corollary}}

\begin{proof}
\mgb{By the choice of parameters \eqref{eq-catalyst-param-choice},
\[
        \chi_\lambda
        =
        \bar\kappa_0(\tau,c),
        \qquad
        L+\lambda
        =
        L\frac{\bar\kappa_0(\tau,c)}
        {\bar\kappa_0(\tau,c)-1}
        \le 2L,
\]
where the last inequality uses $\bar\kappa_0(\tau,c)\ge2$. Substituting these relations, together with $M=2\tau+1$, into the general complexity bound \eqref{eq:convex-catalyst-complexity-tilde} gives
\[
        N_\epsilon
        =
        \widetilde{\cO}\!\left(
        \tau
        \left[
        \sqrt{\bar\kappa_0(\tau,c)}
        +
        \sqrt{\frac{LR_\star^2}{\epsilon}}
        \right]
        \right),
\]
which is \eqref{eq:main-complexity-convex}. The polynomial dependence of the inner prefactor $B_{\tau,c}$ in \eqref{eq:catalyst-inner-prefactor} contributes only logarithmic factors to this bound.}

\mgb{Now suppose that $\epsilon$ satisfies \eqref{eq:convex-catalyst-high-accuracy}. Then
$
        {LR_\star^2}/{\epsilon}
        \ge
        \bar\kappa_0(\tau,c)-1.
$
Since $\bar\kappa_0(\tau,c)\ge2$, we have
\[
        \bar\kappa_0(\tau,c)
        \le
        2\bigl(\bar\kappa_0(\tau,c)-1\bigr)
        \le
        2\frac{LR_\star^2}{\epsilon},
\]
and therefore
\[
        \sqrt{\bar\kappa_0(\tau,c)}
        \le
        \sqrt{2}\,
        \sqrt{\frac{LR_\star^2}{\epsilon}}.
\]
Thus the first term in \eqref{eq:main-complexity-convex} is absorbed by the second, yielding \eqref{eq:convex-complexity}. Comparing this with the lower bound of Theorem~\ref{thm:asr-lower} gives the claimed near-optimality up to logarithmic factors.}
\end{proof}
\section{Numerical Experiments}\label{sec-num-exp}

We use numerical experiments to test three consequences of the preceding analysis. First, on strongly convex quadratics, we verify the predicted loss of \(\sqrt{\kappa}\)-acceleration for classical heavy-ball (HB) and Nesterov accelerated gradient (NAG) under delayed gradients. Second, we examine whether the delay-matched generalized momentum method, referred to as \DREAM, preserves the \(\kappa^{-1/2}\) spectral-gap scaling predicted by the theory. Third, we test whether the same behavior is visible on a nonlinear smooth strongly convex objective and whether the proposed parameters remain effective when the design delay is only a conservative upper bound. In all comparisons, the classical methods are numerically tuned, whereas \DREAM\ uses the prescribed closed-form parameters; \mgd{this comparison tests the closed-form choice against tuned baselines.}
\subsection{Quadratic objectives and spectral-radius diagnostics}
\label{subsec:numerics-quadratic}
For $x\in\mathbb{R}^d$, we consider the quadratic objective
$
    f_q(x) = \frac{1}{2}x^TQx
$
where $Q\succ 0$. This centered form entails no loss of generality: any strongly convex quadratic can be written as \sa{$f(x)=\frac{1}{2}(x-x_*)^\top Q(x-x_*)+f_*$}, where $x_*$ is its minimizer and $f_*=f(x_*)$, and the methods we consider are invariant under translations, so shifting the minimizer to the origin moves the iterates by $x_*$ without altering the optimization dynamics. %
For \(x\in\mathbb{R}^d\), consider
\begin{equation}
    f_q(x)=\frac12 x^\top Qx,
    \qquad Q\succ0,
    \qquad \sigma(Q)\subseteq[\mu,L],
    \qquad \kappa=L/\mu.
\end{equation}
For a curvature mode \(\lambda\in[\mu,L]\), the delayed GMM recursion is governed by the characteristic polynomial
\begin{equation}
    P_{\tau,\delta}(z)
    =z^\tau(z-1)(z-\beta)
      +\delta\bigl((1+\nu)z-\nu\bigr),
    \qquad \delta=\alpha\lambda.
\end{equation}
Given a finite curvature grid \(\Lambda\subset[\mu,L]\), we evaluate the empirical worst-case spectral radius
\begin{equation}
    \widehat\rho(\theta;\kappa,\tau)
    =\max_{\lambda\in\Lambda}\max\{|z|:P_{\tau,\alpha\lambda}(z)=0\},
\end{equation}
where \(\theta=\alpha\) for gradient descent (GD) \mgd{with $\beta=\nu=0$}, and \(\theta=(\alpha,\beta)\) for HB and NAG, with \(\nu=0\) for HB and \(\nu=\beta\) for NAG. \mgd{Let $\Theta_{\rm search}$ be the finite
set of parameter choices examined for the relevant method and instance. We}
report the smallest stable value found \mgd{that is stable on the sampled grid}, %
\begin{equation}
    \widehat\rho^*(\kappa,\tau)
    =\min_{
 \substack{\mgd{\theta\in\Theta_{\rm search}}\\\widehat\rho(\theta;\kappa,\tau)<1} }
      \widehat\rho(\theta;\kappa,\tau).
\end{equation}

For \DREAM, no numerical minimization is performed; we \mgd{estimate} the spectral radius at the prescribed delay-matched parameters.

To distinguish accelerated from non-accelerated scaling, suppose \mgd{that,
at fixed delay, $1-\rho=\Theta(\kappa^{-p})$.
Since \(-\log\rho= 1-\rho + \mathcal{O}\bigl((1-\rho)^2\bigr)\)}
as \(\rho \to 1\), a least-squares fit of
$
    \log(-\log\rho)$
    against $
    \log\kappa$
has slope approximately \(-p\). Thus, a slope \sa{close to} \(-1/2\) indicates accelerated \(\sqrt\kappa\)-dependence, whereas a slope %
\sa{approximately} \(-1\) indicates the non-accelerated \(\kappa\)-dependence.

\noindent

\begin{figure}[h]
    \centering
    \includegraphics[width=\textwidth]{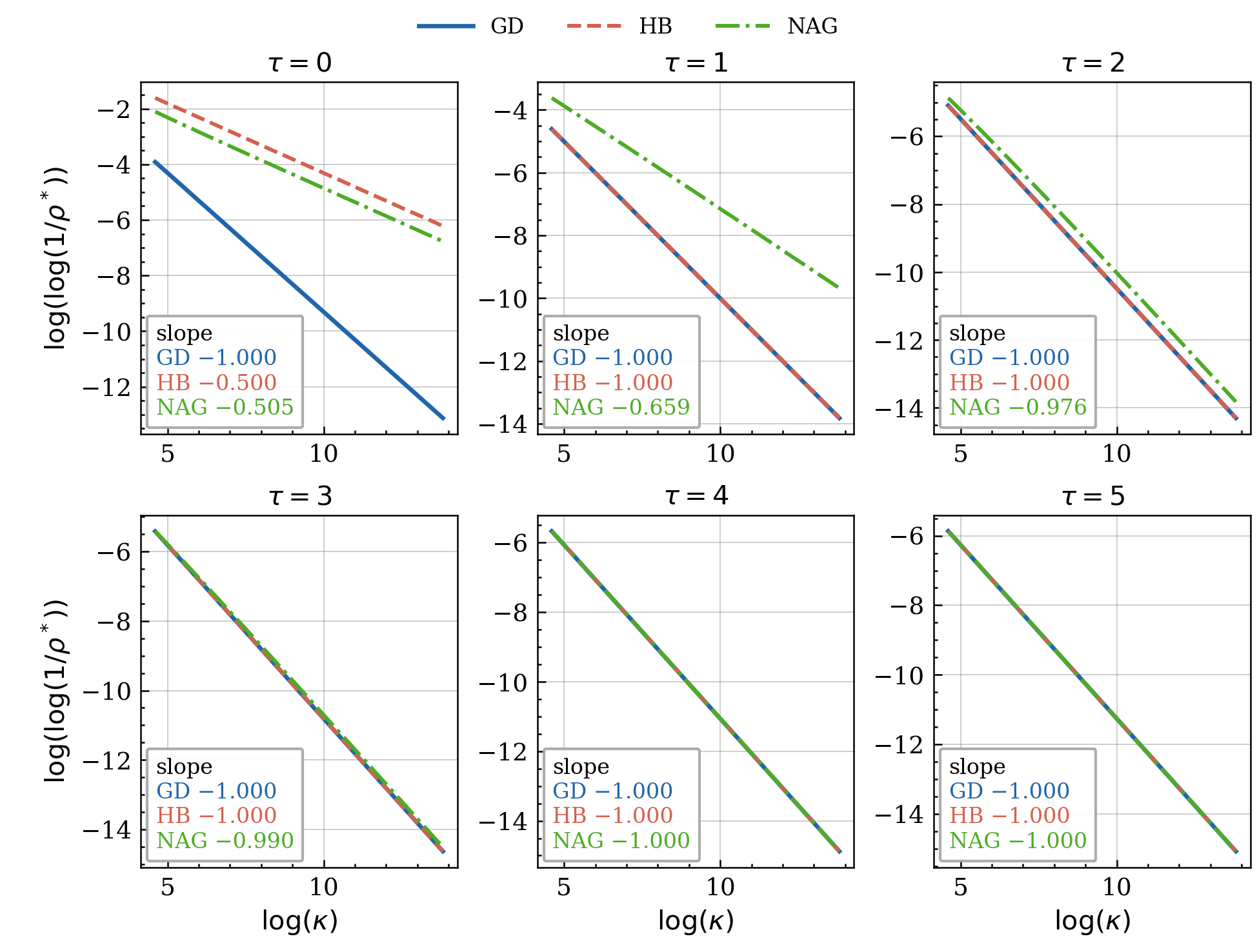}
   \caption{\mgb{Empirical worst-case contraction factor for GD, HB, and NAG after numerical parameter tuning. The vertical axis is \(\log(-\log \widehat\rho^*)\), so the fitted slope estimates the condition-number exponent of the spectral gap. \mgd{The fitted HB slope changes from approximately
$-1/2$ to $-1$ when a one-step delay is introduced. NAG has an intermediate slope near $-2/3$
for $\tau=1$ and slopes near $-1$ for the larger tested delays.} %
   }}
\label{fig:Delay_kills_Acceleration_plot}
\end{figure}

\subsubsection{Delay destroys classical momentum acceleration}%

We first optimize GD, HB, and NAG separately for each delay. We set \(\mu=1\), use 33 logarithmically spaced condition numbers in \([10^2,10^6]\), and consider \(\tau\in\{0,1,\ldots,5\}\). For each parameter choice, we approximate the worst-case spectral radius of $P_{\tau,\delta}(z)$, with $\delta=\alpha\lambda$, over $10$ equally spaced values of $\lambda\in[\mu,L]$, including the endpoints. We then minimize this worst-case radius over the candidate algorithm parameters using a grid search. 
{For GD, we search over $N_\alpha=500$ equally spaced stepsizes in $[10^{-12},\alpha_{\mathrm{hi}}]$, where $\alpha_{\mathrm{hi}}$ is determined by starting from $\alpha_0=0.01/L$, halving until stability is achieved, and then doubling until instability occurs. The stepsize yielding the smallest worst-case spectral radius is selected.} {For HB and NAG, $\beta$ is selected through an initial grid search over $25$ equally spaced values of $\mgd{\ell}=\log(1-\beta)$ in $[\log(10^{-7}),0]$. For each candidate $\beta$, we repeat the stepsize search described above using $N_\alpha=150$ stepsizes and select the $(\alpha,\beta)$ pair with the smallest worst-case spectral radius. We then perform $2$ successive refinements, each evaluating $25$ candidate values in a smaller interval \mgd{$[\ell_{\mathrm{best}}-\mgd{h_{\mathrm prev}},\ell_{\mathrm{best}}+\mgd{h_{\mathrm prev}}]\cap[\log(10^{-7}),0]$}, where \mgd{$\ell_{\mathrm{best}}$} corresponds to the best parameter pair found and \mgd{$h_{\mathrm prev}$} is the previous grid spacing. The $\alpha$ search is repeated for each candidate $\beta$ in both refinements, and the best stable $(\alpha,\beta)$ pair across all $3$ rounds is retained.}

Figure~\ref{fig:Delay_kills_Acceleration_plot} shows a sharp change in scaling once delay is introduced. At \(\tau=0\), HB and NAG have fitted slopes \(-0.500\) and \(-0.505\), respectively, while GD has slope \(-1.000\), reproducing the standard accelerated/non-accelerated dichotomy. {With one-step delay, HB already has slope $-1.000$, while NAG has the intermediate slope $-0.659$.} 
\sa{For \(\tau=2\), the NAG slope %
degrades} to \(-0.976\), and for \(\tau\ge3\) all three methods are essentially indistinguishable at the level of the condition-number exponent. These computations are consistent with Theorem~\ref{thm:q-naglower} \sa{and Theorem~\ref{thm:q-hblower}, which show} that classical momentum cannot retain the \(\sqrt{\kappa}\)-accelerated rate in the presence of any fixed positive delay, even 
\sa{in the case that its free parameters are optimally tuned.} Moreover, Lemma~\ref{lem:q-nag-tau1-achievability} establishes a matching (non-accelerated) \(\Theta(\kappa^{-2/3})\) \mgb{spectral gap} for NAG when \sa{$\tau=1$}.
The numerical results closely track this theoretical prediction, \mgd{supporting the predicted $\kappa^{-2/3}$ scaling over the tested range.} %

\begin{figure}[h]
  \centering
  \includegraphics[width=12.4cm]{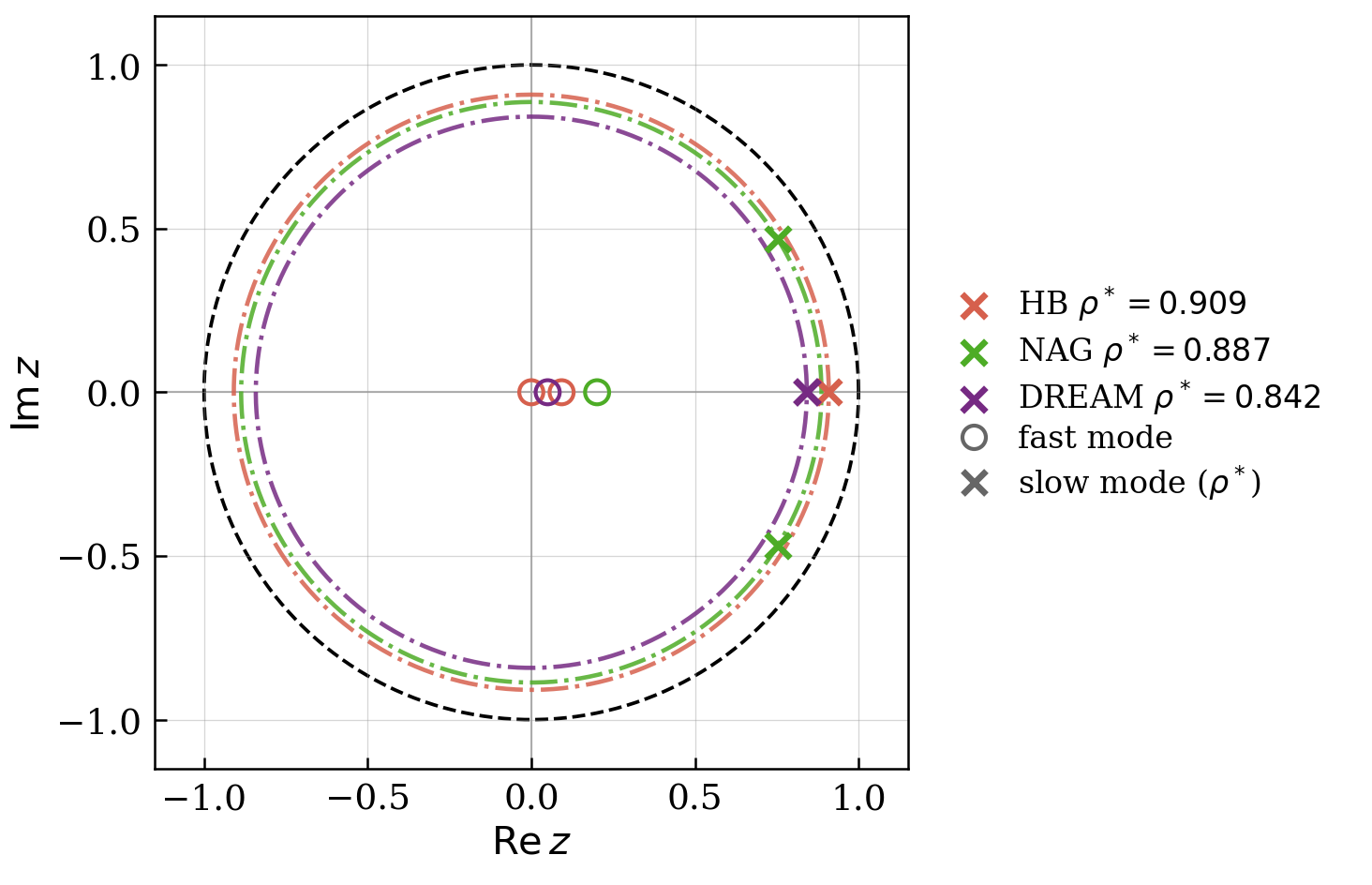}
  \caption{\mgd{Pole locations for $\tau=1$ and $\kappa=10$.
The method-specific dash-dotted circles have
radii equal to the corresponding reported
spectral radii. The one-delay GMM tuning of
Proposition~\ref{prop:q-gmmaccel} has the
smallest spectral radius among the parameter
choices compared.}}
  \label{fig:spectra-tau1}
\end{figure}

\subsubsection{Convergence rate comparison for one-step delay}

For the one-step delay ($\tau=1$), we also provide a direct numerical comparison of the \mgd{spectral radii for HB, NAG, and the one-delay GMM tuning.} %
We set \(\tau=1\) and \(\kappa=10\). \mgd{HB and NAG use the best parameters found in the
numerical search over a grid
$\Lambda\subset[\mu,L]$ consisting of $10$
logarithmically spaced curvature values, whereas
the curve labeled \DREAM\ uses the closed-form
one-delay GMM tuning of
Proposition~\ref{prop:q-gmmaccel}.} %
\mgd{This quadratic-specific
tuning differs from the general-delay \DREAM\ parameters; its guarantee
here is for quadratics, not the full smooth strongly convex class}. Figure~\ref{fig:spectra-tau1} plots the roots of the characteristic polynomial \(P_{1,\delta}(z)\) at the empirically worst-case curvature mode \mgd{\(\lambda_\star\in \Lambda\)} %
chosen to maximize the spectral radius of \(P_{1,\alpha\lambda}(z)\) \mgd{over the curvature grid $\Lambda$}. The dashed black circle is the unit circle, and each dash-dotted circle has radius equal to the corresponding worst-case contraction factor.

For the three momentum methods shown in the figure, \mgd{the reported spectral radii are}
\begin{equation}
 \mgb{\widehat\rho^*_{\rm HB}} \approx 0.909,
    \qquad
   \mgb{\widehat\rho^*_{\rm NAG}} \approx 0.887,
    \qquad
   \widehat\rho^*_{\mathrm{DREAM}}\approx 0.842.%
\end{equation}%
\mgd{For \DREAM, the last value agrees with the
exact one-step-delay formula
\[
\rho_{\mathrm{DREAM}}
=1-\frac{1}{2\sqrt{\kappa}},
\]
established in Proposition~\ref{prop:q-gmmaccel},
at $\kappa=10$.
\mgd{For this instance, the closed-form one-delay
GMM tuning has a smaller spectral
radius than the numerically tuned HB and NAG
methods. This comparison concerns the parameter
choices tested and does not establish the global
optimality of the numerical HB and NAG tunings.
However, Proposition~\ref{prop:tau-one-optimality}
establishes that the one-delay GMM tuning is
asymptotically optimal among all Schur-stable,
constant-parameter delayed GMM tunings with
$\alpha>0$ and $\beta,\nu\ge0$: it achieves
the largest possible leading-order normalized
spectral gap of $1/2$ as $\kappa\to\infty$.}
For the parameter choices considered, NAG
has a smaller spectral radius than HB, while
the one-delay GMM tuning has the smallest. This observation is consistent with our
theoretical results: for $\tau=1$, delayed NAG
achieves a spectral gap of order
$\Theta(\kappa^{-2/3})$ by
Theorem~\ref{thm:q-naglower} and
Lemma~\ref{lem:q-nag-tau1-achievability},
whereas the best achievable spectral gap
of delayed HB is at most
$\cO(\kappa^{-1})$ by
Theorem~\ref{thm:q-hblower}.
\mg{We observe a similar ordering for larger
tested values of $\kappa$.}
}

\begin{figure}[h]
  \centering
  \includegraphics[width=14cm]{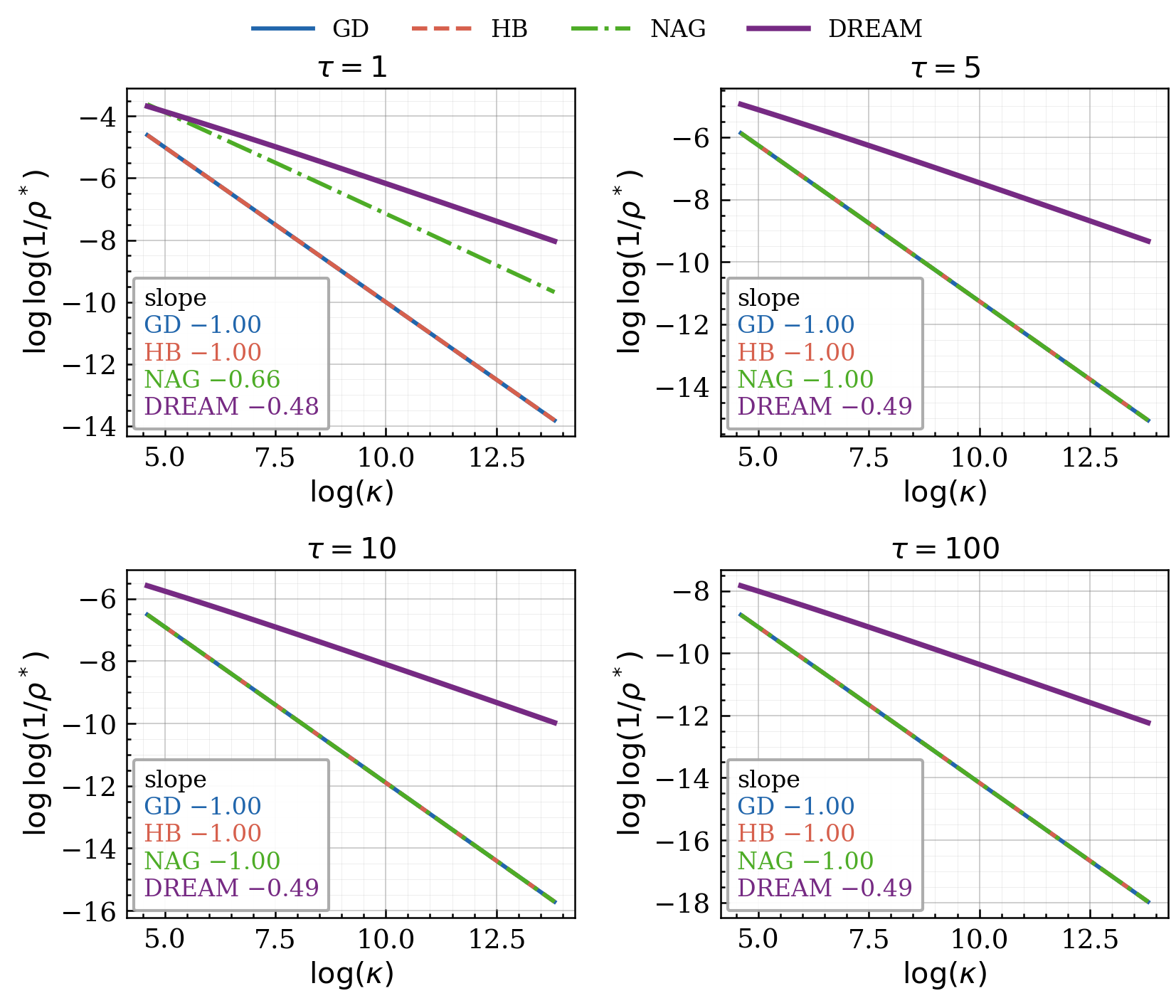}
  \caption{\mgb{Condition-number scaling on quadratic objectives for delays up to \(\tau=100\). GD, HB, and NAG are numerically tuned; \DREAM\ uses the fixed parameters in \eqref{eq:numerics-dream-params}. The proposed method retains a slope close to \(-1/2\) for every tested delay, while the classical methods \mgd{have slopes near $-1$, except for the intermediate NAG exponent at $\tau=1$}.%
  }}
  \label{fig:gmm-slopes}
\end{figure}
\subsubsection{\mgb{Accelerated behavior of \DREAM\ at larger delays}}
\label{sec:GMM-recover-large-delay}
We next test whether the accelerated exponent persists well beyond the one-step setting. We consider \(\tau\in\{1,5,10,100\}\) and \(\kappa\in[10^2,10^6]\). GD, HB, and NAG are again tuned numerically\hamzatext{, using $50$ logarithmically spaced values of $\lambda\in[\mu,L]$,} while \DREAM\ is evaluated without tuning at
\begin{equation}
    M=2\tau+1,
    \qquad
    \alpha=\frac{1}{M^2L},
    \qquad
    \beta=1-\frac{2}{M\sqrt\kappa},
    \qquad
    \nu=M-\frac12,
    \label{eq:numerics-dream-params}
\end{equation}
based on the parameters of Theorem~\ref{thm:q-alldelay}.
\mgb{For consistency across delays, we use the general parameterization~\eqref{eq:numerics-dream-params} for all $\tau\in\{1,5,10,100\}$, including $\tau=1$. This differs from the sharper one-step parameterization of Proposition~\ref{prop:q-gmmaccel}, used in the pole comparison of Figure~\ref{fig:spectra-tau1}}.

Figure~\ref{fig:gmm-slopes} shows the resulting rate diagnostic, and Table~\ref{tab:large-delay-slopes} summarizes the fitted slopes displayed in the figure. The classical methods approach the non-accelerated value \(-1\); NAG retains an intermediate exponent at \(\tau=1\) but is effectively non-accelerated by \(\tau=5\). In contrast, the \DREAM\ slope remains between \(-0.48\) and \(-0.49\) from \(\tau=1\) through \(\tau=100\). \mgd{Increasing the delay changes the multiplicative constant in the spectral gap, consistent with the
$1/(2\tau+1)$ factor in the theoretical rate bound,
but does not change its observed $\kappa^{-1/2}$
dependence. Thus, over the tested range, the
numerical results support the predicted accelerated
scaling.%
}%

\subsection{Strongly convex setting: $\ell_2$-regularized logistic regression}
We next test whether the quadratic picture persists for a nonlinear smooth strongly convex problem. We use \(\ell_2\)-regularized logistic regression,
\begin{equation}
    f(x)
    =\frac1m\sum_{i=1}^m
      \log\!\left(1+\exp(-y_i a_i^\top x)\right)
      +\frac{\mu_r}{2}\|x\|^2,
    \qquad y_i\in\{-1,1\}.
    \label{eq:numerics-logistic}
\end{equation}
Writing \(A\in\mathbb{R}^{m\times d}\) for the feature matrix, the Hessian satisfies
\begin{equation}
    \mu_r I
    \preceq \nabla^2 f(x)
    \preceq \frac{1}{4m}A^\top A+\mu_r I;
\end{equation}
hence, a valid global smoothness and strong-convexity pair is
\begin{equation}
    \mu_{\rm global}=\mu_r,
    \qquad
    L_{\rm global}=\lambda_{\max}\!\left(\frac{A^\top A}{4m}\right)+\mu_r.
\end{equation}
\begin{hamzablock}

    To vary conditioning in a controlled way, we use \(m=1000\), \(d=10\), and \(\mu_r=0.1\), and construct the feature matrix as
    \begin{equation}
        A=\sqrt{4m}\,U\operatorname{diag}(\sqrt{\mgd{\zeta} \tilde{\mathbf g}})V^\top + y \mathbf{1}_d^\top,
    \end{equation}
    \mgd{where $U\in\mathbb{R}^{m\times d}$ has orthonormal columns and $V\in\mathbb{R}^{d\times d}$ is orthogonal, both obtained from QR factorizations of random matrices of dimensions $m\times d$ and $d\times d$, respectively, with i.i.d.\ standard normal entries.}
    Here, $y\in\{-1,1\}^m$ is the vector of class labels, and $\mathbf{1}_d\in\mathbb{R}^d$ denotes the vector of ones. The class labels are assigned deterministically, with $y_i=+1$ for $i=1,\ldots,m/2$ and $y_i=-1$ for $i=m/2+1,\ldots,m$. The class-dependent offset term $y\mathbf{1}_d^\top$ introduces a separation between the two classes. Since this additional term changes the spectrum of $A^\top A/(4m)$, the entries of the vector $\tilde{\mathbf{g}}$ are scaled by the constant \mgd{$\zeta$} to achieve the prescribed global condition-number bound.

    \mgb{To set the \mgd{global} condition number to a \mgd{target}
$\kappa_{\mathrm{target}}$} \mgd{in the plotted range $[10^2,10^4]$}, we initialize $\tilde{\mathbf g}\in\mathbb{R}^d$ with logarithmically spaced entries $\tilde g_j=\mu_r(\kappa_{\mathrm{target}}-1)^{(j-1)/(d-1)}$, for $j=1,\ldots,d$, and set \mgd{$\mathbf g=\zeta\tilde{\mathbf g}$}, where \mgd{$\zeta>0$} is determined numerically, to a relative tolerance of $10^{-8}$, so that
    
    $$
    \mgd{\bigg| \lambda_{\max}\left(\frac{A(\zeta)^\top A(\zeta)}{4m}\right)
    - \mu_r(\kappa_{\mathrm{target}}-1)\bigg| \leq 
    10^{-8} \mu_r(\kappa_{\mathrm{target}}-1)}.
    $$
    \begin{table}[ht!]
    \centering
    \caption{Fitted slopes of \(\log(-\log\rho)\) versus \(\log\kappa\) in the larger-delay quadratic experiment. Values near \(-1/2\) indicate accelerated scaling; values near \(-1\) indicate non-accelerated scaling.}
    \label{tab:large-delay-slopes}
    \begin{tabular}{lrrrr}
        \toprule
        Method & \(\tau=1\) & \(\tau=5\) & \(\tau=10\) & \(\tau=100\) \\
        \midrule
        GD     & -1.00 & -1.00 & -1.00 & -1.00 \\
        HB     & -1.00 & -1.00 & -1.00 & -1.00 \\
        NAG    & -0.66 & -1.00 & -1.00 & -1.00 \\
        \DREAM & -0.48 & -0.49 & -0.49 & -0.49 \\
        \bottomrule
    \end{tabular}
\end{table}

\begin{figure}[ht!]
    \centering
    \includegraphics[width=0.9\textwidth]{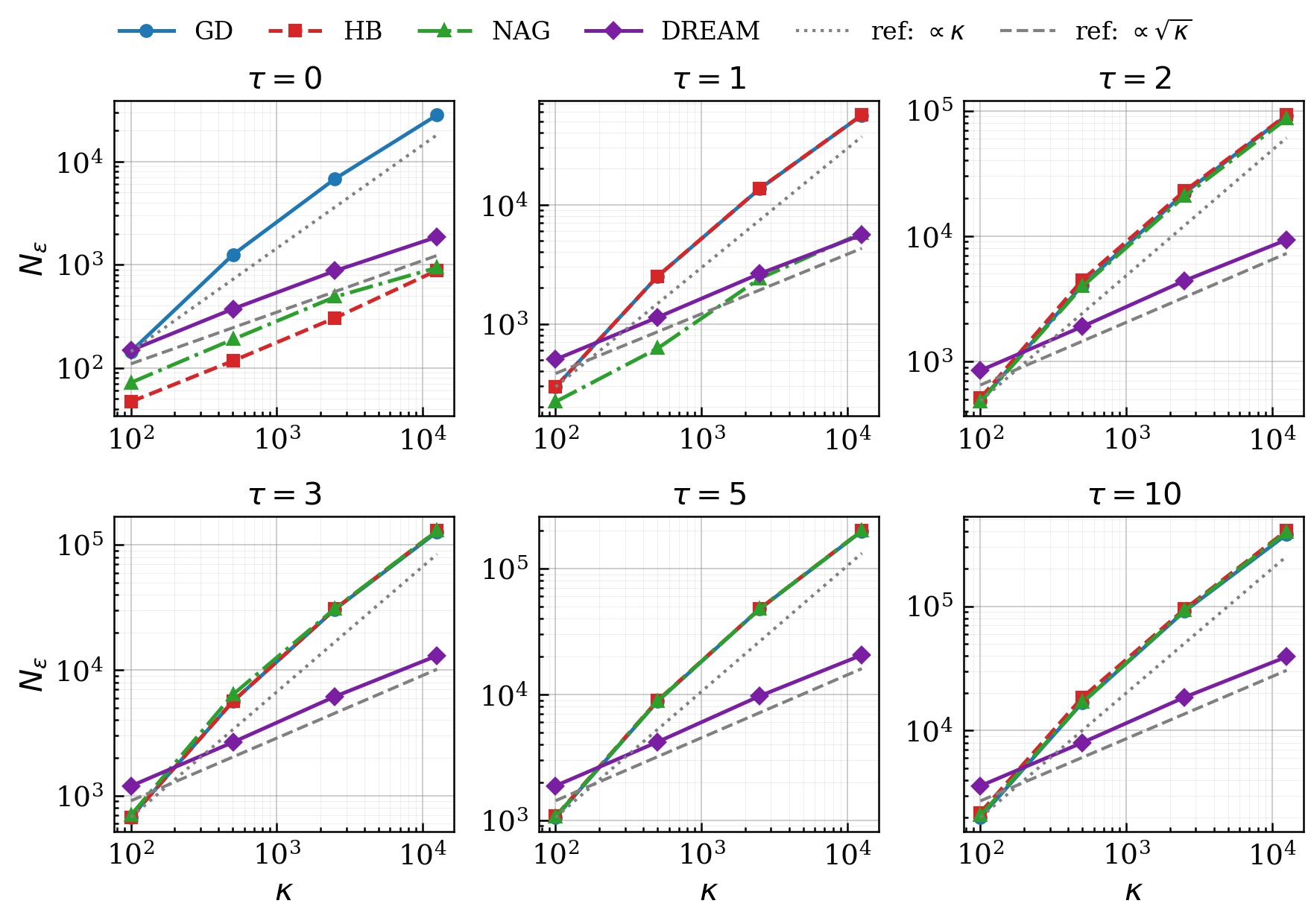}
    \caption{\mgb{First-hit stopping time \(N_\varepsilon\) for controlled-condition-number regularized logistic regression with \(\varepsilon=10^{-8}\). Dotted reference lines show 
   \mgd{linear and square-root dependence on $\kappa$.} %
   Classical momentum \sa{schemes, i.e., HB and NAG,} progressively approach the non-accelerated trend as the delay grows, whereas \DREAM\ remains close to the accelerated reference over the tested range.}}
    \label{fig:strongly-convex-stopping-time-plot}
\end{figure}
    \mgd{A feasible positive scale exists for these targets: at $\zeta=0$,
$\lambda_{\max}(A(0)^\top A(0)/(4m))=d/4=2.5$, below
$\mu_r(\kappa_{\mathrm{target}}-1)$ for every plotted target,
whereas this eigenvalue grows without bound as $\zeta\to\infty$.
Continuity therefore guarantees a crossing.} \mgd{Without the class-dependent offset
$y\mathbf{1}_d^\top$, the eigenvalues of
$A^\top A/(4m)$ would be exactly
$\zeta\widetilde g_1,\ldots,
\zeta\widetilde g_d$.
The offset introduces additional terms that
modify these eigenvalues. We therefore adjust
$\zeta$ numerically using the full feature matrix
$A$ to achieve the desired global
condition-number bound. The parameters
use the achieved bound
$\kappa_{\mathrm{global}}:=L_{\mathrm{global}}/\mu_r$, which equals
$\kappa_{\mathrm{target}}$ to the specified tolerance.}\end{hamzablock} \mgd{For $0<\varepsilon<1$ and $x_0\ne x_*$, define the
relative error and its first-hit iteration count\footnote{\mgd{We overload the notation $N_\varepsilon$ to denote
the empirically observed first-hit iteration count
under the relative-error stopping criterion in \eqref{eq:numerics-stopping-time},
rather than the theoretical sufficient iteration
counts considered in Sections~\ref{subsec:main-proof}
and~\ref{sec:convex-restart}}.} by}
\begin{equation}
    r_k=\frac{\|x_k-x_*\|}{\|x_0-x_*\|},
    \qquad
    N_\varepsilon\sa{:=}\min\{k:r_k\le\varepsilon\}.
    \label{eq:numerics-stopping-time}
\end{equation}
\mgd{Here $x_*$ denotes the exact minimizer.
In the numerical experiments, we approximate
$x_*$ using the limited-memory quasi-Newton
method L-BFGS-B, followed by damped Newton
iterations to refine the solution.} \\

At fixed \(\varepsilon\), \mgd{square-root and linear reference curves provide
empirical diagnostics of accelerated and non-accelerated dependence on
$\kappa$; the theoretical complexity bounds also include logarithmic
factors and transient effects.}
We use \(\varepsilon=10^{-8}\) and delays \(\tau\in\{0,1,2,3,5,10\}\). \hamzatext{For each instance and delay, GD, HB, and NAG are tuned by evaluating a grid of candidate parameters and selecting those with the smallest empirically estimated contraction factor, obtained by fitting the decay of $\log\|x_k-x_*\|$ against the iteration index $k$. To reduce computational cost, large candidate grids are evaluated in two stages: candidates are first ranked using short trajectories covering $12\%$ of the iteration budget, after which the $48$ highest-ranked candidates are evaluated over the full horizon. This tuning criterion reflects the geometric convergence rates studied in our theoretical analysis, while the resulting iteration counts $N_\varepsilon$ measure the corresponding finite-accuracy performance.}
\DREAM{} is not tuned and uses \eqref{eq:numerics-dream-params} with \((L,\kappa)=(L_{\rm global},\kappa_{\rm global})\). \mgd{For $\tau=0$, this is an empirical extrapolation of the formula,
not an application of our positive-delay theorem.}

Figure~\ref{fig:strongly-convex-stopping-time-plot} displays the first-hit iteration counts. With no delay, HB and NAG follow the \(\sqrt\kappa\) reference trend while GD follows the \(\kappa\) trend. \mgd{As the delay increases, HB and subsequently NAG
approach the non-accelerated reference; by
$\tau=2$--$3$, both exhibit iteration-count scaling
comparable to GD over the tested range.
In contrast, \DREAM\ continues to track the
$\sqrt{\kappa}$ reference for every tested delay.
These results reproduce on a nonlinear objective
the qualitative separation observed for quadratics:
increasing the delay reduces the magnitude of the
spectral gap $1-\rho$, consistent with its theoretical
scaling of order $1/((2\tau+1)\sqrt{\kappa})$,
while the observed iteration counts retain their
square-root dependence on $\kappa$.} %

\begin{figure}[h]
    \centering
    \includegraphics[width=15.8cm]{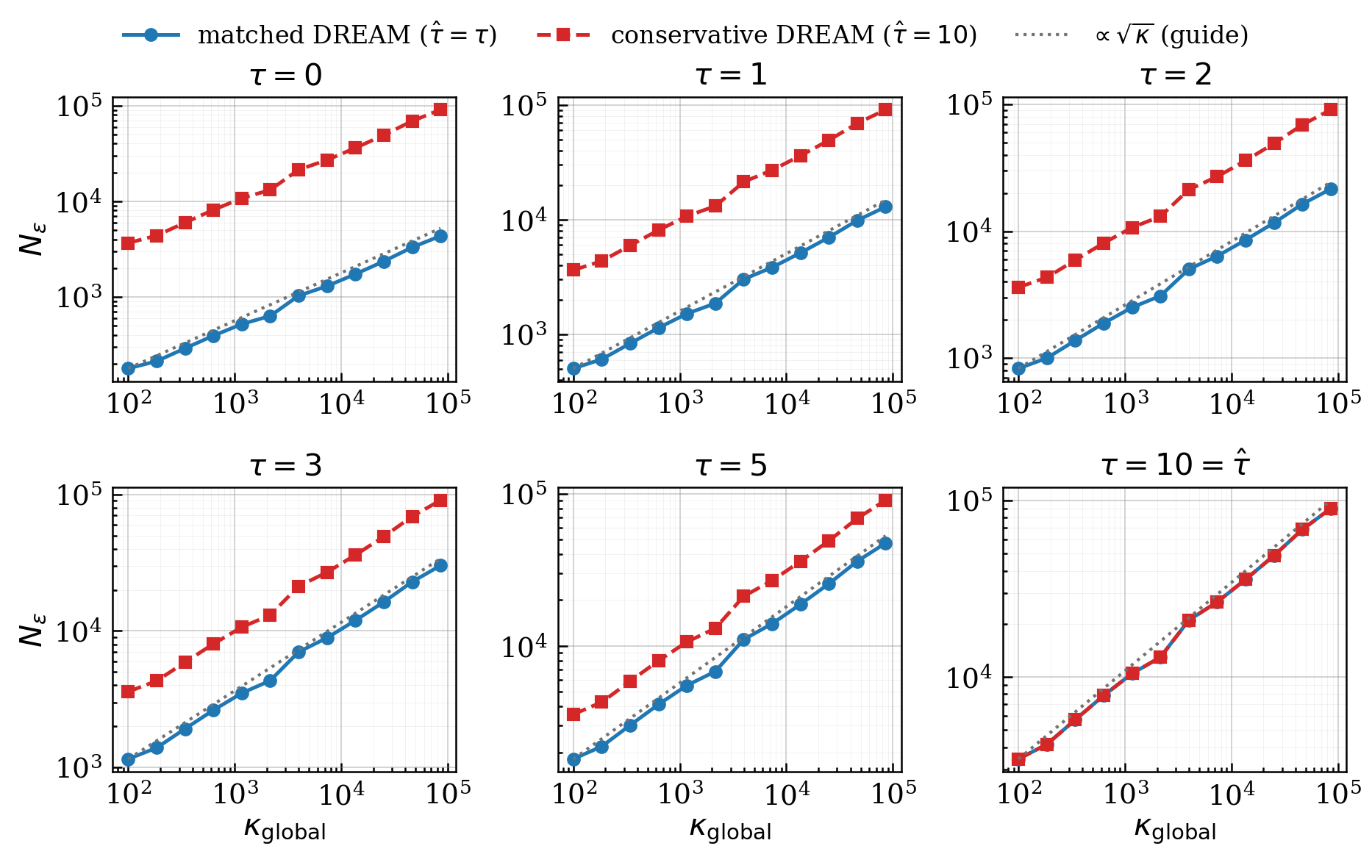}
    \caption{Robustness to delay overestimation on controlled-condition-number least squares. The matched design uses \(\widehat\tau=\tau\), while the conservative design uses \(\widehat\tau=10\) for every actual delay. Both exhibit approximately \(\sqrt\kappa\) scaling over the tested range; conservatism primarily increases the iteration-count prefactor. \mgb{A smaller \(N_\epsilon\) indicates that fewer iterations are required to reach \(\epsilon\)-accuracy and therefore corresponds to better performance}.}
    \label{fig:delay-overestimation-ls}
\end{figure}

\begin{figure}[h]
    \centering
    \includegraphics[width=15.8cm]{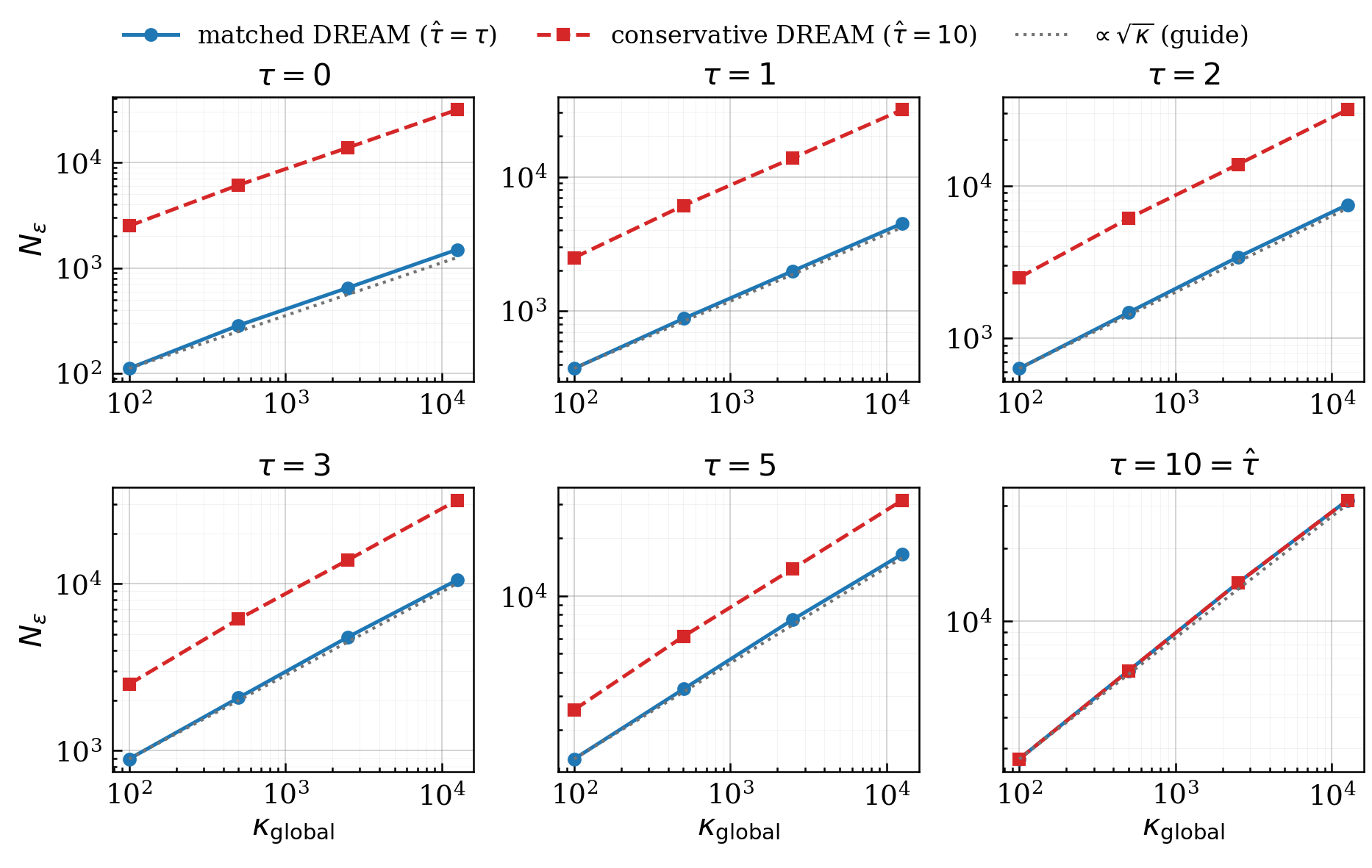}
    \caption{Robustness to delay overestimation on controlled-condition-number regularized logistic regression. The matched and conservative designs have nearly parallel log-log trends, and they coincide when \(\tau=\widehat\tau=10\). \mgb{A smaller \(N_\epsilon\) indicates that fewer iterations are required to reach \(\epsilon\)-accuracy and therefore corresponds to better performance}.}
    \label{fig:delay-overestimation-logistic}
\end{figure}
\subsection{Robustness to conservative delay estimates}
\label{subsec:numerics-delay-mismatch}

The theoretical parameterization uses the true delay \(\tau\). In practice, one may only know an upper bound. To assess this mismatch empirically, let \(\widehat\tau\) denote the delay used to choose the parameters while the recursion itself uses the true delay \(\tau\). We compare the matched choice \(\widehat\tau=\tau\) with the conservative fixed choice \(\widehat\tau=10\) for
\begin{equation}
    \tau\in\{0,1,2,3,5,10\}.
\end{equation}
In each case, the parameters are obtained from \eqref{eq:numerics-dream-params} after replacing \(\tau\) by \(\widehat\tau\). Thus, only the parameter design uses the conservative delay; the gradient in the recursion remains delayed by the actual \(\tau\). \mgd{The matched design at $\tau=0$ again uses the formula empirically;
Theorem~\ref{thm:main} does not cover that case}. We test the two designs on controlled-condition-number least-squares and regularized logistic-regression instances and use the stopping rule \eqref{eq:numerics-stopping-time} with \(\varepsilon=10^{-6}\). Figures~\ref{fig:delay-overestimation-ls} and~\ref{fig:delay-overestimation-logistic} show that designing for \(\widehat\tau=10\) increases the iteration count when the actual delay is smaller, but the log-log slope with respect to \(\kappa\) remains close to the accelerated value. The fitted exponents for the conservative design are approximately \(0.49\) for least squares and \(0.52\) for logistic regression. When \(\tau=10\), the matched and conservative designs coincide, as they should.

This final experiment is intended as a robustness check rather than a theoretical guarantee. Our analysis assumes that the delay is known and establishes convergence for the corresponding matched-delay parameterization. The numerical results indicate that the accelerated \(\mathcal{O}(\sqrt{\kappa})\) scaling persists even when only an upper bound on the actual delay is available. 
\mgd{Indeed, over the tested range, using a conservative
delay estimate increases the iteration counts,
while the observed dependence on $\kappa$
remains approximately of order $\sqrt{\kappa}$.
This increase is consistent with the theoretical
rate, whose spectral gap $1-\rho$ decreases
inversely with the delay parameter used in the
algorithm. The transient prefactor $\mgd{C_{\mathrm{rate}}}$ may also
increase with this parameter, further affecting
the iteration counts. Although both effects
may contribute, the main observation is that
the approximately square-root dependence on
$\kappa$ is preserved even when the design
delay overestimates the actual delay.
}%

\section{Conclusion}
\mgd{We studied acceleration of first-order methods under
fixed gradient delays. Our results show that classical
momentum schemes are fundamentally sensitive to stale
gradient information: under any fixed positive delay,
neither heavy-ball nor Nesterov's accelerated gradient
method can retain $\sqrt{\kappa}$-acceleration on
strongly convex quadratics, even when their stepsize
and momentum parameters are tuned to the delay and
condition number. This impossibility result applies
to convergent tunings with parameters held constant
across iterations, in the standard positive-momentum
regime.}
\sa{In contrast, by 
decoupling extrapolation from momentum and matching the extrapolation scale 
to the delay, \DREAM{} retains the accelerated $\sqrt{\kappa}$ dependence while 
updating the iterate at every iteration using the available delayed gradient. For \mgd{every fixed positive delay and
$\kappa\ge\kappa_0(\tau,c)$, constant-history initialization yields
near-optimal iteration complexity on} 
smooth strongly convex objectives, i.e. the resulting iteration complexity has the near-optimal scaling ${\mathcal O}(\tau\sqrt{\kappa}\log(1/\epsilon))$ up to logarithmic factors. \mgd{An accelerated proximal-point construction with \DREAM\ as its inner
solver extends the approach to smooth convex objectives, with the
near-optimal delay and accuracy dependence holding in the high-accuracy
regime of \eqref{eq:convex-catalyst-high-accuracy}}.
Our analysis also provides an analytic, all-delay IQC certificate that avoids numerical semidefinite programs or frequency-domain searches. The numerical experiments support the predicted separation between classical and delay-matched momentum and indicate that the accelerated scaling can be robust to conservative estimates of the delay.}

\sa{Several directions remain open. Our theory assumes a fixed deterministic delay  known to the algorithm, whereas practical asynchronous systems may exhibit time-varying, heterogeneous, or stochastic delays. Extending delay-matched momentum and the accompanying analytic certificates to these settings is a natural next step. It would also be interesting to study stochastic gradients and nonconvex objectives; more broadly, our results suggest that the adverse effect of delayed information need not preclude acceleration: appropriately adapting the algorithmic dynamics to the structure of the delay  can preserve accelerated condition-number dependence even when fresh gradients are unavailable.}

\section*{Acknowledgements} \mgb{This work was supported in part by the Office of Naval Research under Grant N00014-24-1-2628 and N00014-24-1-2666. The authors used generative artificial intelligence (AI) tools to assist with language editing, improving the exposition, drafting and revising portions of the mathematical proof, \hamzatext{developing and refining code for numerical experiments}, and refining constants in certain estimates. All AI-assisted material was carefully reviewed, corrected where necessary, and further developed by the authors. The authors assume full responsibility for the mathematical correctness and all other content of the final manuscript.}

\bibliographystyle{alpha}
\IfFileExists{Delayed_Acceleration_Arxiv/references_main.bib}
    {\bibliography{Delayed_Acceleration_Arxiv/references_main}}
    {\bibliography{references_main}}
    
\appendix
\section{Proofs of the supporting lemmas}\label{subsec:lemma-proofs}

We prove the supporting lemmas of Section~\ref{subsec:master-proof} in the order they were stated. Throughout, fix the rate-weighted sequences
\[
        \widehat x_k=\rho^{-k}x_k,
        \qquad \widehat y_k=\rho^{-k}y_k,
        \qquad \widehat u_k=\rho^{-k}u_k,
        \qquad \widehat v_k=\widehat u_k-m\widehat y_k .
\]

\subsection{Proof of Lemma~\ref{lem:rate-weighted-plant}}
If \(Y(z)=g_{\tau,q}(z)U(z)\) denotes the unweighted
zero-state input-output relation, then rate weighting corresponds to
evaluating the old transform at \(\rho z\).  Namely,
\[
        \widehat Y(z)
        =\sum_{k\ge0}\rho^{-k}y_k z^{-k}
        =\sum_{k\ge0}y_k(\rho z)^{-k}
        =Y(\rho z),
        \qquad
        \widehat U(z)=U(\rho z).
\]
Thus
\[
        \widehat Y(z)
        =Y(\rho z)
        =g_{\tau,q}(\rho z)U(\rho z)
        =g_{\tau,q}(\rho z)\widehat U(z).
\]
Consequently the transfer from \(\widehat u\) to \(\widehat y\) is
\mgb{
\[
        g_\rho(z):=g_{\tau,q}(\rho z),
        \qquad G_\rho(z):=g_\rho(z)I_d.
\]}
Equivalently, rate weighting divides every open-loop pole by \(\rho\):
a pole \(p\) of \(g_{\tau,q}\) becomes a pole \(p/\rho\) of
\(g_{\tau,q}(\rho z)\).
 
\subsection{Proof of Lemma~\ref{lem:weighted-realization}}
The midpoint change of variables gives
\[
        u_{k-\tau}=m y_{k-\tau}+v_{k-\tau}.
\]
Using the formula for \(y\) at the delayed time \(k-\tau\),
\[
        y_{k-\tau}=(1+\nu)x_{k-\tau}-\nu x_{k-\tau-1},
\]
the state recursion becomes
\begin{align*}
        x_{k+1}
        ={}&(1+\beta)x_k-\beta x_{k-1}
        -\alpha m y_{k-\tau}-\alpha v_{k-\tau}  \\
        ={}&(1+\beta)x_k-\beta x_{k-1}
        -\alpha m(1+\nu)x_{k-\tau}
        +\alpha m\nu x_{k-\tau-1}
        -\alpha v_{k-\tau}.
\end{align*}
Multiply this identity by \(\rho^{-(k+1)}\).  Since
\(x_j=\rho^j\widehat x_j\) and \(v_j=\rho^j\widehat v_j\), the terms are
rescaled as
\[
\begin{aligned}
        \rho^{-(k+1)}x_k &= \rho^{-1}\widehat x_k,
        &\qquad
        \rho^{-(k+1)}x_{k-1} &= \rho^{-2}\widehat x_{k-1},\\
        \rho^{-(k+1)}x_{k-\tau} &= \rho^{-(\tau+1)}\widehat x_{k-\tau},
        &
        \rho^{-(k+1)}x_{k-\tau-1} &= \rho^{-(\tau+2)}\widehat x_{k-\tau-1},\\
        \rho^{-(k+1)}v_{k-\tau} &= \rho^{-(\tau+1)}\widehat v_{k-\tau}.
\end{aligned}
\]
Substitution gives the first identity in \eqref{eq:weighted-realization-x-app}.  For the output, multiply
\(y_k=(1+\nu)x_k-\nu x_{k-1}\) by \(\rho^{-k}\).  Then
\[
        \rho^{-k}x_k=\widehat x_k,
        \qquad
        \rho^{-k}x_{k-1}=\rho^{-1}\widehat x_{k-1},
\]
which gives \eqref{eq:weighted-realization-y-app}.

\subsection{Proof of Lemma~\ref{lem:midpoint-schur}}
Throughout, $\alpha=q^2/M^2$, $\nu=M-\tfrac12$, $m=\tfrac12(1+q^{-2})$, and
$\delta_m=\alpha m$; recall $\alpha,m,\nu>0$ and $\tau\ge1$.

\emph{Step 1: loop elimination and the denominator.}
Clearing the delay factor from the open-loop transfer gives
\beq
        g_{\tau,q}(z)=\frac{N_0(z)}{D_0(z)},
        \qquad
        N_0(z)=-\alpha\bigl[(1+\nu)z-\nu\bigr],
        \qquad
        D_0(z)=z^{\tau}(z-1)(z-\beta),
        \label{def-N0-D0}
\eeq
so the rate-weighted \mgd{transfer function} is \mgb{$g_\rho(z)=g_{\tau,q}(\rho z)=N(z)/D(z)$} 
with
$N(z)=N_0(\rho z)$ and $D(z)=D_0(\rho z)$. Substituting
$\widehat U=m\widehat Y+\widehat V$ into $\widehat Y=\mgb{g_\rho}\,\widehat U$ and collecting the
$\widehat Y$ terms yields \mgb{$(1-mg_\rho)\widehat Y=g_\rho\,\widehat V$. Note
\[
1-mg_\rho=\frac{D-mN}{D},
\]
\begin{align*}
D(z)-mN(z)
 &=(\rho z)^\tau(\rho z-1)(\rho z-\beta)
   +\delta_m[(1+\nu)\rho z-\nu]=P_{\tau,\delta_m}(\rho z).
\end{align*}}
\mg{This} is a non-zero polynomial of degree $\tau+2$. Noting that the function $1-m\mgb{g_\rho}$ is not
identically zero, we obtain
\mgb{$$ \widehat Y=\widetilde g_\rho\,\widehat V,
        \qquad
        \widetilde g_\rho=\frac{g_\rho}{1-mg_\rho}=\frac{N}{D-mN}$$}with denominator the rate-scaled midpoint characteristic polynomial
$P_{\tau,\delta_m}(\rho\,\cdot)$. In particular the roots of $D-mN$ are
$\{r/\rho: P_{\tau,\delta_m}(r)=0\}$.

\emph{Step 2: no pole--zero cancellation.}
The numerator $N_0$ has the single finite zero
\[
        z_\ast=\frac{\nu}{1+\nu}=\frac{M-\tfrac12}{M+\tfrac12}=\frac{2M-1}{2M+1}
        \in(0,1),
\]
while the roots of $D_0$ are $0$, $1$, and $\beta$. Clearly $z_\ast\neq0$ and
$z_\ast\neq1$. The equality $z_\ast=\beta=1-\tfrac{2q}{M}$ would force
$\tfrac{2q}{M}=1-\tfrac{2M-1}{2M+1}=\tfrac{2}{2M+1}$, i.e.\ $q=\tfrac{M}{2M+1}$,
which the standing hypothesis $q<\tfrac{M}{2M+1}$ excludes. Hence
\beq \gcd(N_0,D_0)=1, \label{gcd-N0-DO}
\eeq
\mgd{where, $\gcd(N_0,D_0)$ denotes the monic greatest
common divisor of the polynomials $N_0$ and $D_0$.
Thus, $\gcd(N_0,D_0)=1$ means that these
polynomials are coprime, i.e., they have no
common nonconstant polynomial factor.} Rescaling $z\mapsto\rho z$ preserves coprimeness, so
$\gcd(N,D)=1$. Finally, any common zero $z_0$ of $N$ and $D-mN$ would satisfy
$N(z_0)=0$ and $D(z_0)=\bigl(D(z_0)-mN(z_0)\bigr)+mN(z_0)=0$, contradicting
$\gcd(N,D)=1$; therefore
\[
        \gcd\bigl(N,\;D-mN\bigr)=1 .
\]
Thus $\mgb{\widetilde g_\rho}=N/(D-mN)$ is in lowest terms, and its poles are
\emph{exactly} the roots of $D-mN=P_{\tau,\delta_m}(\rho\,\cdot)$. \mgd{Since $\deg N=1<\tau+2=\deg(D-mN)$,
the transfer $\widetilde g_\rho(z)$ tends to zero
as $z\to\infty$.}%

\emph{Step 3: spectrum of the realization.}
Collecting the terms of \eqref{eq:weighted-realization-x-app} by lag, the
homogeneous $\widehat x$-recursion reads
\beq\widehat x_{k+1}=\sum_{j=0}^{\tau+1}a_j\widehat x_{k-j} 
\label{eq-homogen-recursion}\eeq with
\mgb{\[
a_j=\frac{1+\beta}{\rho}\mathbf1_{\{j=0\}}
 -\frac{\beta}{\rho^2}\mathbf1_{\{j=1\}}
 -\frac{\alpha m(1+\nu)}{\rho^{\tau+1}}\mathbf1_{\{j=\tau\}}
 +\frac{\alpha m\nu}{\rho^{\tau+2}}\mathbf1_{\{j=\tau+1\}},
 \quad 0\le j\le\tau+1.
\]}\mgb{Note that, for $\tau=1$, there are two terms that contribute to $a_1$. Let $F\in\mathbb R^{(\tau+2)\times(\tau+2)}$ be the scalar companion matrix of the
recursion $\widehat x_{k+1}=\sum_{j=0}^{\tau+1}a_j\widehat x_{k-j}$, i.e., the matrix
\[
F:=\begin{bmatrix}
a_0 & a_1 & a_2 & \cdots & a_{\tau} & a_{\tau+1}\\
1   & 0   & 0   & \cdots & 0        & 0\\
0   & 1   & 0   & \cdots & 0        & 0\\
\vdots & \ddots & \ddots & \ddots & \vdots & \vdots\\
0   & \cdots & 0 & 1 & 0 & 0\\
0   & \cdots & 0 & 0 & 1 & 0
\end{bmatrix},
\]
whose first row implements the recursion and whose subdiagonal of ones shifts the
stored history. Expanding along the first row gives
$\det(zI-F)=z^{\tau+2}-\sum_{j=0}^{\tau+1}a_jz^{\tau+1-j}$.}
\mgb{Let
$E\in\mathbb R^{(\tau+2)\times\tau}$ have its only nonzero entry
$E_{1,\tau}=-\alpha/\rho^{\tau+1}$.}

Splitting the state as
$\xi_k=(X_k^{\mathsf T},W_k^{\mathsf T})^{\mathsf T}$ with
$X_k=(\widehat x_k,\dots,\widehat x_{k-\tau-1})^{\mathsf T}$ and
$W_k=(\widehat v_{k-1},\dots,\widehat v_{k-\tau})^{\mathsf T}$, the state matrix is
block upper triangular, and given by 
\[
        \mgb{A=A_0\otimes I_d, \quad A_0=\begin{bmatrix} F & E\\ 0 & J\end{bmatrix}},
\]
where $E$ records the delayed input $\widehat v_{k-\tau}$ in \mgd{\eqref{eq:weighted-realization-x-app}}, $J$ is the $\tau\times\tau$ down-shift with
\mgb{$$W_{k+1}=JW_k+e_1\widehat v_k,
\qquad
J:=\begin{bmatrix}
0 & 0 & \cdots & 0 & 0\\
1 & 0 & \cdots & 0 & 0\\
0 & 1 & \ddots & \vdots & \vdots\\
\vdots & \ddots & \ddots & 0 & 0\\
0 & \cdots & 0 & 1 & 0
\end{bmatrix}\in\mathbb R^{\tau\times\tau}$$
where $e_1\in\mathbb{R}^\tau$ is the first basis vector. For the matrix $A$,} the lower-left block
vanishes because the stored inputs evolve independently of the
$\widehat x$-states. %
The matrix $J$ is
nilpotent with \mgb{$\det (zI_{\tau}-J)=z^{\tau}$}, so\mgb{
 \[
\det(zI_{2\tau+2}-A_0)
 =z^\tau\det(zI_{\tau+2}-F),\qquad
 \det(zI_{(2\tau+2)d}-A)
 =\bigl [\det(zI_{2\tau+2}-A_0)\bigr]^d.
\]
}
We conclude that the spectrum \mgb{$\sigma(A)=\sigma(A_0)=\sigma(F)\cup\{0\}$}.

For the companion block, $\det(zI-F)=z^{\tau+2}-\sum_{j=0}^{\tau+1}a_jz^{\tau+1-j}$.
Multiplying by $\rho^{\tau+2}$ and inserting the coefficients above gives
\[
        \rho^{\tau+2}\det(zI-F)
        =\rho^{\tau+2}z^{\tau+2}-(1+\beta)\rho^{\tau+1}z^{\tau+1}
        +\beta\rho^{\tau}z^{\tau}
        +\alpha m(1+\nu)\rho z-\alpha m\nu
        =P_{\tau,\delta_m}(\rho z),
\]
by expanding $(\rho z)^\tau(\rho z-1)(\rho z-\beta)+\delta_m[(1+\nu)\rho z-\nu]$.
Thus $\det(zI-F)$ equals $D-mN$ up to the nonzero factor $\rho^{-(\tau+2)}$, and by
Step~2 its roots are exactly the poles of \mgb{$\widetilde g_\rho$}; that is,
$\sigma(F)=\mathcal P(\widetilde G)$. Moreover
$P_{\tau,\delta_m}(0)=-\delta_m\nu\neq0$, so $0\notin\operatorname{spec}(F)$ and the
union is disjoint:
\[
        \sigma(A)=\sigma(A_0)=\mathcal P(\widetilde G)\cup\{0\},
\]
\mgb{as sets. Each algebraic multiplicity in $A_0$ is multiplied by $d$ in $A$}.

\mgd{\emph{Step 4: Stability equivalences.}
Since $0$ lies in the open unit disk, the
state-transition matrix $A$ is Schur stable
if and only if
$\mathcal P(\widetilde g_\rho)
\subset\{|z|<1\}$, that is, if and only if
the transfer function $\widetilde g_\rho$
is stable. By Steps~1--2, its poles are
$\{r/\rho:P_{\tau,\delta_m}(r)=0\}$.
Consequently, $A$ is Schur stable if and
only if every root of $P_{\tau,\delta_m}$
lies in $\{|z|<\rho\}$, which is precisely
the midpoint-pole hypothesis \textbf{(H1)}.}

\subsection{Proof of Lemma~\ref{lem:fdi-congruence}}
By the identity~\eqref{eq:fdi-equivalence}, it suffices to show that
$1-m \mgb{g_\rho(e^{i\theta})}\ne0$ for all $\theta$; multiplication by the resulting positive
scalar $\lvert 1-m \mgb{g_\rho}\rvert^{-2}$ cannot change the sign of the quadratic form, and the
equivalence of the two strict FDIs follows.

First, \mgb{$g_\rho$} is finite on the unit circle: with \mgd{$g_{\tau,q}=N_0/D_0$ as in \eqref{def-N0-D0} where $\gcd(N_0,D_0)=1$ by \eqref{gcd-N0-DO}}. Hence, the
poles of \mgb{$g_\rho(z)=g_{\tau,q}(\rho z)$}
lie at $z\in\{0,\,1/\rho,\,\beta/\rho\}$, and none of these poles have modulus one, since $0<\beta=1-2q/M<1-cq/M=\rho<1$ \mgb{and $\rho \neq \beta$}. Next, from
\[
        1-m \mgb{g_\rho(z)}=\frac{D_0(\rho z)-mN_0(\rho z)}{D_0(\rho z)}
        =\frac{P_m(\rho z)}{D_0(\rho z)},
        \qquad
        P_m:=D_0-mN_0,
\]
a zero $1-m \mgb{g_\rho}(e^{i\theta})=0$ would force $P_m(\rho e^{i\theta})=0$, i.e.\ a root of the
midpoint characteristic polynomial of modulus exactly $\rho$. This contradicts
\textbf{(H1)}, which places every root of $P_m$ in $\{|z|<\rho\}$. Hence
$1-m \mgb{g_\rho}(e^{i\theta})\neq0$ for all $\theta\in[0,2\pi)$, and the claim follows.

\subsection{Proof of Lemma~\ref{lem:positive-storage-kyp}}
\mgd{
The main idea for controlling the prefactor is to reformulate the original
frequency-domain inequality through an invertible
change of input variables, so that it becomes
a standard inequality comparing input and output
norms. We then use this inequality, together with
the stability of the transformed system, to
establish the existence of a positive-definite
quadratic storage function.} The proof has five steps. 

\noindent\emph{Step 1: Exact factorization of the supply.}
Set $w:=1-h$ and define
\[
        e_{+,k}:=\widehat r_k+(w \widehat s)_k,
        \qquad
        e_{-,k}:=\widehat r_k-(w \widehat s)_k,
\]
where
\[
        (w \widehat s)_k=\widehat s_k-(h\widehat s)_k
        =\widehat s_k-\sum_{j=1}^{J}c_j \widehat s_{k-j}.
\]
The polarization identity
\[
        \|p+q\|^2-\|p-q\|^2=4\Rea\langle p,q\rangle
\]
yields, for every $k$,
\begin{align}
\frac12\bigl(\|e_{+,k}\|^2-\|e_{-,k}\|^2\bigr)
&=2\Rea\langle \widehat r_k,(w \widehat s)_k\rangle \notag\\
&=2\Rea\langle \widehat  r_k,\widehat s_k\rangle
  -2\sum_{j=1}^{J}c_j\Rea\langle \widehat r_k,\widehat s_{k-j}\rangle \notag\\
&=\sigma_k.\label{eq:kyp-factor-time}
\end{align}
Thus the original supply is \emph{exactly} the difference-of-squares supply associated with $(e_+,e_-)$. No summation in time, boundary term, or zero-initial-condition assumption is involved in \eqref{eq:kyp-factor-time}.

Using $\widehat r=a \widehat y+\widehat v$ and $\widehat s=a \widehat y-\widehat v$ gives
\begin{align*}
        e_+
        &=\widehat r+ \widehat s-h \widehat s
          =a(2-h)\widehat y+h \widehat v,\\
        e_-
        &=\widehat r-\widehat s+h \widehat s
          =ah \widehat y+(2-h) \widehat v.
\end{align*}
Hence
\begin{equation}\label{eq:Psi-factor}
        \begin{bmatrix}e_+\\e_-\end{bmatrix}
        =\mgb{\Psi_d}
\begin{bmatrix}\widehat y\\ \widehat v\end{bmatrix}.
\end{equation}
For zero initial state, $\widehat y=\mgb{\mathcal{G}}\widehat v$, and therefore the transfer \mgb{matrices} from $\widehat v$ to $e_+$ and $e_-$ are
\begin{equation}\label{eq:Hpm-def}
        H_+:=a(2-h) \mgb{\mathcal G}+h\mgb{\I_d},
        \qquad
        H_-:=ah \mgb{\mathcal G}+(2-h)\mgb{\I_d}.
\end{equation}
On the unit circle, \eqref{eq:Psi-factor} implies
\mgb{\begin{align*}
\begin{bmatrix} \mathcal G\\I_d\end{bmatrix}^{\!*}
\widetilde\Pi_d
\begin{bmatrix}\mathcal G\\I_d\end{bmatrix}
&=
\frac{1}{2}
\begin{bmatrix}H_+\\H_-\end{bmatrix}^{\!*}
J_{0,d}
\begin{bmatrix}H_+\\H_-\end{bmatrix}\\
&=\frac{1}{2}\bigl(H_+^*H_+-H_-^*H_-\bigr).
\end{align*}}
Consequently, \eqref{eq:kyp-strict-fdi-input-margin} is equivalent to
\begin{equation}\label{eq:kyp-Hpm}
        H_+(e^{i\theta})^*H_+(e^{i\theta})
        -H_-(e^{i\theta})^*H_-(e^{i\theta})
        \preceq -2\eta\mgb{\I_d}
\end{equation}
for every $\theta$.

\medskip
\mg{\noindent\emph{Step 2: Stable invertibility of $H_-$.} 
We show that $H_-^{-1}$ exists as a stable, proper
rational transfer matrix. \mgd{
Since $\mathcal G$ is stable and proper, and
$h(z)=\sum_{j=1}^J c_jz^{-j}$ has poles only
at $z=0$ and satisfies
$\lim_{z\to\infty}h(z)=0$,
the transfer matrix
$H_-=ah\mathcal G+(2-h)I_d$
is stable and proper, with
$\lim_{z\to\infty}H_-(z)=2I_d$}. Therefore, it suffices to establish that $H_-(z)$
is nonsingular for every $|z|\ge 1$, including at the infinity limit.}
For $t\in[0,1]$, define the homotopy
\[
        H_+^{(t)}:=ta(2-h) \mgb{\mathcal{G}} +h\mgb{\I_d},
        \qquad
        H_-^{(t)}:=tah \mgb{\mathcal{G}}+(2-h)\mgb{\I_d},
\]
and, for fixed $\theta$, 
\mgb{\[
        \Phi_t(\theta):=
        \begin{bmatrix}t \mathcal G(e^{i\theta})\\I_d\end{bmatrix}^{\!*}
        \widetilde\Pi_d(e^{i\theta})
        \begin{bmatrix}t \mathcal G(e^{i\theta})\\I_d\end{bmatrix}.
\]}
\mgb{Partition $\widetilde\Pi_d=[(\widetilde\Pi_d)_{ij}]_{i,j=1}^2$
into $d\times d$ blocks.} %
Directly from \mgb{$\widetilde\Pi_d=\frac{1}{2}\Psi_d^{\sim}J_{0,d}\Psi_d$,}
\begin{equation}\label{eq:Pi-diagonal-blocks}
    \mgb{(\widetilde\Pi_d)_{11}
        =2a^2\Rea(1-h)I_d,
        \qquad
        (\widetilde\Pi_d)_{22}
        =-2\Rea(1-h)I_d,}
\end{equation}
on the unit circle. Moreover,
\[
        \Rea(1-h(e^{i\theta}))
        \ge 1-|h(e^{i\theta})|
        \ge 1-s_h>0.
\]
Thus \mgb{$(\widetilde\Pi_d)_{11}(e^{i\theta})\succ0$ and
$(\widetilde\Pi_d)_{22}(e^{i\theta})\prec0$}  %
for every $\theta$.
For any fixed $u\in\C^{d}$, the scalar function
\[
        f_{\theta,u}(t):=u^*\Phi_t(\theta)u
\]
is a real quadratic polynomial whose $t^2$-coefficient is \mgb{$u^*\mathcal G^*(\widetilde\Pi_d)_{11} \mathcal Gu\ge0$}.
Hence $f_{\theta,u}$ is convex. Since
\mgb{\[
        \Phi_0(\theta)=(\widetilde\Pi_d)_{22}(e^{i\theta})\prec0,
        \qquad
        \Phi_1(\theta)\preceq-\eta I_d\prec0,
\]}
convexity gives, for $0\le t\le1$,
\[
        u^*\Phi_t(\theta)u
        \le (1-t)u^*\Phi_0(\theta)u+t u^*\Phi_1(\theta)u<0
\]
whenever $u\ne0$. Therefore
\begin{equation}\label{eq:Phi-negative}
        \Phi_t(\theta)\prec0,
        \qquad t\in[0,1],\quad 0\le\theta<2\pi.
\end{equation}
As in Step~1,
\begin{equation}\label{eq:Phi-Hpm}
        \Phi_t(\theta)
        =\frac12\left(
          H_+^{(t)}(e^{i\theta})^*H_+^{(t)}(e^{i\theta})
          -H_-^{(t)}(e^{i\theta})^*H_-^{(t)}(e^{i\theta})
        \right).
\end{equation}
It follows that $H_-^{(t)}(e^{i\theta})$ is nonsingular for every $(t,\theta)$. Indeed, if
$H_-^{(t)}(e^{i\theta})u=0$ with $u\ne0$, then \eqref{eq:Phi-Hpm} would imply
\[
        u^*\Phi_t(\theta)u
        =\frac12\|H_+^{(t)}(e^{i\theta})u\|^2\ge0,
\]
contrary to \eqref{eq:Phi-negative}.

We now exclude zeros in the whole exterior disc, not merely on its boundary. Because \mgb{$\mathcal G$} is stable and proper, it is analytic on a neighborhood of $\{z:|z|\ge1\}\cup\{\infty\}$. Introduce $\lambda=z^{-1}$ and define
\[
        \widehat H_-^{(t)}(\lambda):=
        \begin{cases}
        H_-^{(t)}(\lambda^{-1}),&\lambda\ne0,\\
        2\mgb{\I_d},&\lambda=0.
        \end{cases}
\]
The value at $\lambda=0$ is the removable limit because $h(\infty)=0$; \mgd{i.e. $\lim_{z\to\infty} h(z)=0$}. Hence, $\widehat H_-^{(t)}$ is analytic on a neighborhood of the closed unit disc $\overline\D$, uniformly in $t\in[0,1]$. Set
\[
        d_t(\lambda):=\det \widehat H_-^{(t)}(\lambda).
\]
By the preceding paragraph, $d_t(\lambda)\ne0$ for $|\lambda|=1$, for every $t$. Therefore the winding number
\[
        N(t):=\frac{1}{2\pi i}
        \int_{|\lambda|=1}\frac{d_t'(\lambda)}{d_t(\lambda)}\,d\lambda
\]
is well-defined. By the argument principle, $N(t)$ is the number of zeros of $d_t$ in $\D$, counted with multiplicity. Since $(t,\lambda)\mapsto d_t(\lambda)$ is continuous and never vanishes on $[0,1]\times\partial\D$, the integer-valued function $N(t)$ is constant.

At $t=0$,
\[
        \widehat H_-^{(0)}(\lambda)
        =\left(2-\sum_{j=1}^J c_j\lambda^j\right)\mgb{\I_d}.
\]
For $|\lambda|\le1$,
\[
        \left|\sum_{j=1}^J c_j\lambda^j\right|
        \le s_h<1,
\]
so $\widehat H_-^{(0)}(\lambda)$ is nonsingular throughout $\overline\D$. Hence $N(0)=0$, and therefore $N(t)=0$ for all $t$. Together with nonsingularity on $\partial\D$ and the identity
$\widehat H_-^{(t)}(0)=\mgb{2\I_d}$, this shows that
\[
        \det H_-^{(t)}(z)\ne0
        \qquad\text{for every }|z|\ge1,
\]
and that there is no zero at infinity. \mg{
In particular,
\begin{equation}\label{eq:Hminus-minphase}
    \det H_-(z)\ne 0 \quad (|z|\ge 1),
    \qquad H_-(\infty)=2I_d,
\end{equation}
\mgd{where $H_-(\infty):= \lim_{z\to\infty} H_-(z)$}. Since $H_-$ is stable and proper, these conditions imply
that $H_-^{-1}$ is also stable and proper.
}

\medskip
\noindent\emph{Step 3: Augmented realization and the change of input $\widehat v\mapsto e_-$.} 
\mgb{We use the Schur realization $(A,B,C,D)$ of
$\mathcal G$ specified in~\eqref{eq:plant-realization}}:
\begin{equation}\label{eq:plant-realization-kyp-proof}
        \xi_{k+1}=A\xi_k+B \widehat v_k,
        \qquad
        \widehat y_k=C\xi_k+D \widehat v_k,
\end{equation}
with $A$ Schur. Since
\[
        \widehat s_k=aC\xi_k+(aD-\mgb{\I_d}) \widehat v_k,
        \qquad
        (h \widehat s)_k=\sum_{j=1}^{J}c_j \widehat s_{k-j},
\]
augment $\xi_k$ by the filter states
\[
        q_k:=
        \begin{bmatrix}
          \widehat s_{k-1}\\ \widehat s_{k-2}\\\vdots\\\widehat s_{k-J}
        \end{bmatrix},
        \qquad
        \Xi_k:=\begin{bmatrix}\xi_k\\q_k\end{bmatrix}.
\]
Let
\[
        L_h:=\begin{bmatrix}c_1\I_d&c_2\I_d&\cdots&c_J\I_d\end{bmatrix},
        \qquad
        S:=
        \begin{bmatrix}
          0&0&\cdots&0\\
          \I_d&0&\cdots&0\\
          0&\I_d&\ddots&\vdots\\
          \vdots&\ddots&\ddots&0
        \end{bmatrix},
        \qquad
        E:=\begin{bmatrix}\I_d\\0\\\vdots\\0\end{bmatrix}.
\]
\mgd{Here $L_h\in\mathbb R^{d\times Jd}$, $S\in\mathbb R^{Jd\times Jd}$,
$E\in\mathbb R^{Jd\times d}$, and each displayed zero block is $d\times d$.}
Then $(h \widehat s)_k=L_hq_k$ and
\[
        q_{k+1}=Sq_k+E \widehat s_k
        =Sq_k+EaC\xi_k+E(aD-\I_d)\widehat  v_k.
\]
Thus
\begin{equation}\label{eq:augmented-state}
        \Xi_{k+1}=A_a\Xi_k+B_a \widehat v_k,
\end{equation}
where
\begin{equation}\label{eq:Aa-Ba}
        A_a=
        \begin{bmatrix}
          A&0\\
          EaC&S
        \end{bmatrix},
        \qquad
        B_a=
        \begin{bmatrix}
          B\\
          E(aD-\I_d)
        \end{bmatrix}.
\end{equation}
\mgb{The matrices $A_a$ and $B_a$ have dimensions $n_a\times n_a$ and
$n_a\times d$, respectively.} Because $A_a$ is block lower triangular with diagonal blocks $A$ and the nilpotent matrix $S$, it is Schur.

From Step~1,
\[
        e_+=2a \widehat y- (h \widehat s),
        \qquad
        e_-=2 \widehat v+ (h \widehat s).
\]
Therefore
\begin{equation}\label{eq:epm-realization}
        e_{+,k}=C_+\Xi_k+D_+\widehat v_k,
        \qquad
        e_{-,k}=C_-\Xi_k+D_- \widehat v_k,
\end{equation}
with
\begin{equation}\label{eq:Cpm-Dpm}
\begin{aligned}
        C_+&=\begin{bmatrix}2aC&-L_h\end{bmatrix},
        &D_+&=2aD,\\
        C_-&=\begin{bmatrix}0&L_h\end{bmatrix},
        &D_-&=2\mgb{\I_d}.
\end{aligned}
\end{equation}
\mgd{Here $C_\pm\in\mathbb R^{d\times n_a}$ and
$D_\pm\in\mathbb R^{d\times d}$.} In particular, $D_-$ is invertible. Hence, pointwise in time,
\begin{equation}\label{eq:v-from-eminus}
        \widehat v_k=D_-^{-1}(e_{-,k}-C_-\Xi_k).
\end{equation}
Thus replacing $\widehat v_k$ by $e_{-,k}$ is an invertible static change of input coordinates on the augmented state space. Substitution of \eqref{eq:v-from-eminus} into \eqref{eq:augmented-state} and \eqref{eq:epm-realization} yields the equivalent realization
\begin{equation}\label{eq:bar-realization}
        \Xi_{k+1}=\bar A\Xi_k+\bar B e_{-,k},
        \qquad
        e_{+,k}=\bar C\Xi_k+\bar D e_{-,k},
\end{equation}
where\footnote{\mgd{We use the \emph{overbar notation} in two senses: $\bar z$ denotes complex conjugation, whereas $\bar A$, $\bar B$, $\bar C$, and $\bar D$ denote the transformed system matrices defined in \eqref{eq:bar-matrices}, not their complex conjugates.}}
\begin{equation}\label{eq:bar-matrices}
\begin{aligned}
        \bar A&=A_a-B_aD_-^{-1}C_- ,
        &\bar B&=B_aD_-^{-1},\\
        \bar C&=C_+-D_+D_-^{-1}C_- ,
        &\bar D&=D_+D_-^{-1}.
\end{aligned}
\end{equation}

\medskip
\noindent\emph{\mgd{Step 4: Stability of the transformed system
and verification of the strict frequency-domain inequality.}}
For zero initial state,
\[
        e_+=H_+ \widehat v,
        \qquad
        e_-=H_- \widehat v.
\]
By~\eqref{eq:Hminus-minphase}, $H_-^{-1}$ is stable and proper. Hence the transfer from the new input $e_-$ to the output $e_+$ is
\begin{equation}\label{eq:K-def}
        K=H_+H_-^{-1}.
\end{equation}
The realization in \eqref{eq:bar-realization} has precisely this transfer.

\mgb{We next verify that the state matrix is Schur stable. This does not follow from stability of the transfer matrix alone, since some eigenvalues of the state matrix may not appear as poles of the transfer matrix.}\footnote{\mgb{Indeed, the adjugate formula for the matrix inversion gives $G(z)=D+C\operatorname{adj}(zI-A)B/\det(zI-A)$. Although the roots of $\det(zI-A)$ are precisely the eigenvalues of $A$, factors of this denominator may cancel with corresponding factors in the numerators of all entries of $G$. Consequently, some eigenvalues of $A$ may not appear as poles of $G$, and a stable transfer matrix need not have a Schur-stable state matrix.}}
For $|z|\ge1$, $z\mgb{\I_{n_a}}-A_a$ is invertible \mgb{where $n_a=n+Jd$}. Using \eqref{eq:bar-matrices}, the determinant identity and Sylvester's determinant theorem give
\begin{align}
\det(zI_{n_a}-\bar A)
&{=\det(zI_{n_a}-A_a)
  \det\!\left(I_d+D_-^{-1}C_-(zI_{n_a}-A_a)^{-1}B_a\right)\notag}\\
&=\det(zI_{n_a}-A_a)\det\!\left(D_-^{-1}H_-(z)\right).
\label{eq:Abar-determinant}
\end{align}
The first factor in \eqref{eq:Abar-determinant} is nonzero because $A_a$ is Schur; the second is nonzero by \eqref{eq:Hminus-minphase}. Hence
$\det(z\mgb{\I_{n_a}}-\bar A)\ne0$ for every $|z|\ge1$, so $\bar A$ is Schur.

Finally, rewrite \eqref{eq:kyp-Hpm} as
\[
        H_-^*H_- - H_+^*H_+\succeq2\eta\mgb{\I_d}
\]
on the unit circle and pre- and post-multiply by $H_-^{-*}$ and $H_-^{-1}$. Using \eqref{eq:K-def},
\begin{equation}\label{eq:strict-br-margin}
        \I_d-K(e^{i\theta})^*K(e^{i\theta})
        \succeq
        2\eta\,H_-(e^{i\theta})^{-*}H_-(e^{i\theta})^{-1}
        \succ0.
\end{equation}
Since the unit circle is compact and $H_-$ is continuous and nonsingular there, \eqref{eq:strict-br-margin} has a uniform positive margin. In particular,
\[
        \|K\|_{\mathcal{H}_\infty}:=\sup_\theta\|K(e^{i\theta})\| <1.
\]
\mgb{\mgd{We apply the \emph{strict discrete-time bounded-real
lemma}, a particular form of the
Kalman--Yakubovich--Popov (KYP) lemma
\cite[Theorem~2.2(b)]{de1992discrete}}: If $\bar A$ is Schur and $K(z)=\bar C(zI-\bar A)^{-1}\bar B+\bar D$ satisfies $\|K\|_{\mathcal H_\infty}<1$, then\footnote{\mgd{The strictness of the inequality matters.
In the nonstrict case, the existence of
$X_0\succ0$ satisfying the corresponding
KYP inequality may require additional assumptions
beyond the stability of $\bar A$
\cite[Theorem~1.1]{ball2018standard}.
In contrast, in the strict case used here,
the strict frequency-domain inequality together
with the Schur stability of $\bar A$ guarantees
such a positive-definite solution \cite[Theorem~1.2]{ball2018standard}.}} there exists a \mgd{real symmetric} matrix $X_0$ such that
\beq
R_0:=\bar B^{\top}X_0\bar B+\bar D^{\top}\bar D-I_d\prec0
\label{ineq-R0}
\eeq
and
\beq
\bar A^{\top}X_0\bar A-X_0+\bar C^{\top}\bar C
-\bigl(\bar A^{\top}X_0\bar B+\bar C^{\top}\bar D\bigr)R_0^{-1}\bigl(\bar B^{\top}X_0\bar A+\bar D^{\top}\bar C\bigr)\prec0 .
\label{ineq-schur-complement}
\eeq
The left-hand side of \eqref{ineq-schur-complement} is the Schur complement of the block $R_0$ in the \mgd{real symmetric} matrix
\begin{equation}\label{eq:br-lmi}
\mathcal M_0:=
\begin{bmatrix}
\bar A^{\top}X_0\bar A-X_0+\bar C^{\top}\bar C & \bar A^{\top}X_0\bar B+\bar C^{\top}\bar D\\[1mm]
\bar B^{\top}X_0\bar A+\bar D^{\top}\bar C & R_0
\end{bmatrix},
\end{equation}
and a \mgd{real symmetric} matrix with a negative definite block $R_0$ is negative definite if and only if the Schur complement of $R_0$ is negative definite. Hence \eqref{ineq-R0}--\eqref{ineq-schur-complement} are equivalent to $\mathcal M_0\prec0$, \mgd{and we have necessarily $\mathcal M_0\prec0$}.}

\noindent\emph{\mgd{Step 5: Positive-definite storage
and the dissipation inequality for the original input
$\widehat v_k$.}} \mgd{We first show that $X_0\succ0$.
Indeed, as $\mathcal M_0\prec0$, the upper-left block of
\eqref{eq:br-lmi} implies
$Q:=X_0-\bar A^\top X_0\bar A\succ0$. The identity
$X_0=Q+\bar A^\top X_0\bar A$
can be iterated to give}%
\mgb{$$X_0=Q+\bar A^{\top}X_0\bar A=Q+\bar A^{\top}Q\bar A+(\bar A^{\top})^{2}X_0\bar A^{2}=\cdots=\sum_{k=0}^{N-1}(\bar A^{\top})^{k}Q\bar A^{k}+(\bar A^{\top})^{N}X_0\bar A^{N}.$$
Since $\bar A$ is Schur, ${\bar A}^N \to 0$ as $N\to \infty$ and we get
$X_0=\sum_{k\ge0} {({\bar A}^{\top})}^{k}Q{\bar A}^{k}$, and every 
summand is positive semidefinite with the $k=0$ summand equal to $Q$; hence $X_0\succeq Q\succ0$. \mgd{Recall from \eqref{eq:kyp-factor-time} that the supply satisfies
$\sigma_k=\tfrac12(\|e_{+,k}\|^2-\|e_{-,k}\|^2)$. We set $X:=\tfrac12X_0\succ0$ to match the scaling of the supply. Scaling the matrix inequality
$\mathcal M_0\prec0$ in \eqref{eq:br-lmi}
by $\tfrac12$ gives
}
 
 \begin{equation}\label{eq:strict-kyp-matrix}
\mathcal M_X:=
\begin{bmatrix}
\bar A^\top X\bar A-X+\frac12\bar C^\top\bar C
&
\bar A^\top X\bar B+\frac12\bar C^\top \bar D\\[1mm]
\mgb{\bar B^\top X \bar A+ \frac{1}{2}\bar D^\top \bar C}
&
\bar B^\top X\bar B+\frac12(\bar D^\top \bar D-\I_d)
\end{bmatrix}
\prec0.
\end{equation}%
}
Since \eqref{eq:strict-kyp-matrix} is a strict finite-dimensional inequality, there exists $\delta>0$ such that
$\mathcal M_X\preceq-\delta\mgb{\I_{n_a+d}}$. Using this inequality and evaluating this quadratic form at $w_k:=
\begin{bmatrix}
\Xi_k^\top & 
e_{-,k}^\top
\end{bmatrix}^\top$
gives, along every trajectory of \eqref{eq:bar-realization}, 
\mgb{$$
\begin{aligned}
w_k^{\top}\mathcal M_X w_k
&=\Xi_k^{\top}(\bar A^{\top}X\bar A-X)\Xi_k+2\,\Xi_k^{\top}\bar A^{\top}X\bar B e_{-,k}+e_{-,k}^{\top}\bar B^{\top}X\bar B e_{-,k}\\
&\quad+\tfrac12\bigl(\Xi_k^{\top}\bar C^{\top}\bar C\Xi_k+2\,\Xi_k^{\top}\bar C^{\top}\bar D e_{-,k}+e_{-,k}^{\top}\bar D^{\top}\bar D e_{-,k}\bigr)-\tfrac12\|e_{-,k}\|^{2}\\
&=(\bar A\Xi_k+\bar Be_{-,k})^{\top}X(\bar A\Xi_k+\bar Be_{-,k})-\Xi_k^{\top}X\Xi_k+\tfrac12\|\bar C\Xi_k+\bar De_{-,k}\|^{2}-\tfrac12\|e_{-,k}\|^{2}\\
&=V(\Xi_{k+1})-V(\Xi_k)+\tfrac12\bigl(\|e_{+,k}\|^{2}-\|e_{-,k}\|^{2}\bigr) \leq -\delta \|w_k\|^2.
\end{aligned}
$$}
\mgb{Since $\|w_k\|^2 = \|\Xi_k\|^2+\|e_{-,k}\|^2$}, this leads to
\begin{equation}\label{eq:kyp-br-dissipation}
        V(\Xi_{k+1})-V(\Xi_k)
        +\frac12\bigl(\|e_{+,k}\|^2-\|e_{-,k}\|^2\bigr)
        \le
        -\delta\bigl(\|\Xi_k\|^2+\|e_{-,k}\|^2\bigr),
\end{equation}
where $V(\Xi)=\Xi^\top X\Xi$.

\medskip

It remains only to express the strict dissipation bound in terms of the original input $
\widehat v_k$. From \eqref{eq:epm-realization},
\[
        e_{-,k}=C_-\Xi_k+D_- \widehat v_k,
        \qquad D_-=2\I_d.
\]
Hence
\[
        \begin{bmatrix}\Xi_k\\e_{-,k}\end{bmatrix}
        =
        T\begin{bmatrix}\Xi_k\\ \widehat v_k\end{bmatrix},
        \qquad
T:=\begin{bmatrix}\I_{n_a}&0\\C_-&D_-\end{bmatrix} \mgb{\in \mathbb{R}^{(n_a+d)\times(n_a+d)}}.
\]
The block triangular matrix $T$ is invertible because $D_-$ is invertible. Therefore, with
$\mgd{\gamma_T}:=\sigma^2_{\min}(T)>0$,
\begin{equation}\label{eq:norm-equivalence}
        \|\Xi_k\|^2+\|e_{-,k}\|^2
        \ge
        \mgd{\gamma_T} \bigl(\|\Xi_k\|^2+\|\widehat v_k\|^2\bigr)
\end{equation}
for all $(\Xi_k, \widehat v_k)$. Combining the exact identity \eqref{eq:kyp-factor-time}, the dissipation inequality
\eqref{eq:kyp-br-dissipation}, and \eqref{eq:norm-equivalence}, we conclude that
\[
        V(\Xi_{k+1})-V(\Xi_k)+\sigma_k
        \le
        -\delta \mgd{\gamma_T}\bigl(\|\Xi_k\|^2+\|\widehat v_k\|^2\bigr).
\]
Thus \eqref{eq:positive-storage-kyp} holds with
$\varepsilon:=\delta\mgd{\gamma_T}>0$.

\section[Exact stability-boundary for quadratics]{Optimal parameter design for GMM on quadratics}\label{sec:design}

\mgd{We recall that $\mu$ is the strong-convexity constant and $L$ is the Lipschitz constant of the gradient of the objective $f\in\Sml$ and $q=\sqrt{\mu/L}$. As in Section~\ref{sec:quadratic},
$\delta=\alpha\lambda$ denotes the \emph{effective
stepsize} associated with a quadratic
eigendirection of curvature $\lambda$ (see the discussion after \eqref{eq:s-Delta-def}). In this section, the parameters $\nu$ and $\delta_L$ are design variables, not fixed at the \DREAM{} values.}
We determine exactly \sa{the supremum of the rate constants certified by Theorem~\ref{thm:q-gate}}
over all admissible extrapolation
parameters~$\nu$, and we prove that this optimal constant degrades exactly like
$1/\tau$ as the delay grows. %

Theorem~\ref{thm:q-gate} reduces the design of an accelerated delayed method on quadratics to a single scalar question. It shows that whenever the limiting polynomial $\Pinf(\cdot;\delta,\nu)$ is strictly Schur for \mgd{all} $0<\delta\le\delta_L$, the method run with $\alpha=\delta_L/L$ and $\beta=1-2\sqrt{\delta_L}\,q$ attains the accelerated rate
$\rho_{\tau,\kappa}\mathrel{\le}1-cq$ for every $0<c<\sqrt{\delta_L}$ and %
\sa{all sufficiently large $\kappa$}. Thus, 
for \mgd{an admissible interval of effective stepsizes}, the supremum of the rate
constants certified by Theorem~\ref{thm:q-gate} is $\sqrt{\delta_L}$. %
\sa{Since the extrapolation parameter $\nu$ remains free, it is natural to choose $\nu$ so as to maximize the length of this stability interval.} To make this precise, fix an integer $\tau\ge1$ and for $\nu>\tau$, let
\begin{equation}\label{eq:q-deltaLdef}
        \delta_L(\nu):=\inf\bigl\{\delta>0:\Pinf(\cdot;\delta,\nu)\ \text{has a root on the unit circle}\bigr\}
\end{equation}
\mgd{denote the smallest positive
effective stepsize $\delta$ at which
$P_\infty(\cdot;\delta,\nu)$ has a root
on the unit circle.}

By Corollary~\ref{coro:q-smallgain} and Lemma~\ref{lem:q-firstlobe} below,
$\delta_L(\nu)>0$ for every $\nu>\tau$,
and then $\Pinf(\cdot;\delta,\nu)$ is strictly Schur for all
$0<\delta<\delta_L(\nu)$. \sa{Hence, every \mgd{$\widehat\delta_L \in (0,\delta_L(\nu))$} is admissible in Theorem~\ref{thm:q-gate}. Letting $\widehat\delta_L\uparrow\delta_L(\nu)$ shows that the supremum of the certifiable rate constants for a fixed $\nu$ is $\sqrt{\delta_L(\nu)}$.}
The design problem is therefore
\begin{equation}\label{eq:q-designproblem}
        \max_{\nu>\tau}\ \delta_L(\nu),
\end{equation}
whose square root is the supremum of rate constants certifiable by Theorem~\ref{thm:q-gate}. The difficulty in
\eqref{eq:q-designproblem} is that a  root $z= re^{i\theta}$ of $P_\infty$ may reach the unit circle at
any frequency $\theta \in [0,2\pi)$ so $\delta_L(\nu)$ is a priori governed by an infinite family of
competing crossing conditions. We resolve this difficulty completely below: for
every admissible $\nu$, strict stability is first lost at a single %
\sa{explicitly characterized} low-frequency crossing (Lemma~\ref{lem:q-firstlobe}), and consequently
\eqref{eq:q-designproblem} is solved exactly (Corollary~\ref{cor:q-designoptimal}).

\paragraph{The stability boundary.}
At a loss of strict stability, a root of $\Pinf$ lies on the unit circle,
\sa{say} $z=e^{i\theta}$. 
\saa{Since $P_\infty$ has real coefficients, its nonreal roots occur in conjugate pairs. Thus, it suffices to consider unit-circle roots $z=e^{i\theta}$ with $\theta\in(0,\pi)$.} 
\mgd{The real endpoints $z=\pm1$ are treated separately
in Lemma~\ref{lem:q-firstlobe}}. From \eqref{eq:q-Pinf}, a unit-circle root satisfies
\begin{equation}\label{eq:q-boundeq}
        e^{i\tau\theta}(e^{i\theta}-1)^2+\delta\bigl[(1+\nu)e^{i\theta}-\nu\bigr]=0 .
\end{equation}
\saa{This is a complex equation—equivalently, two real equations—in the two real unknowns $(\delta,\nu)$. As $\theta$ ranges over $(0,\pi)$, its solutions characterize parameter pairs $(\delta,\nu)$ for which $P_\infty$ has a nonreal root on the unit circle. The next proposition parametrizes these solutions explicitly, and Lemma~\ref{lem:q-firstlobe} identifies the first loss of strict Schur stability.}

\begin{proposition}[%
\saa{parameterization of unit-circle roots}]\label{prop:q-boundary}
Fix an integer $\tau\ge1$ and write $u:=\theta/2$ and $M=2\tau+1$. For $\theta\in(0,\pi)$, the equation
\eqref{eq:q-boundeq} \mgb{admits a solution $(\delta,\nu) \in (0,\infty)\times \R$ if and only if $\cos(Mu)>0$. When this condition holds, the solution is unique and is given by} %
\begin{equation}\label{eq:q-DN}
        \delta=D_\tau(u):=\frac{4\sin^2u\,\cos(Mu)}{\cos u},
        \qquad
        \nu=N_\tau(u):=\frac{\sin(2\tau u)}{2\sin u\,\cos(Mu)} .
\end{equation}
Both $D_\tau$ and\footnote{\mgd{Here, we slightly overload the notation:
$D_\tau(u)$ gives the value of the effective
stepsize $\delta$ at a unit-circle crossing,
whereas $D_1$ and $D_2$ in
Proposition~\ref{prop:rational} denote factors
in the denominator of $\Phi_{\tau,q}$.
These quantities are unrelated.}} $N_\tau$ are positive on the interval
\begin{equation}\label{eq:q-band}
        \mathcal I_\tau:=\Bigl(0,\tfrac{\pi}{2M}\Bigr),
\end{equation}
and $D_\tau$ vanishes at its two endpoints $u=0$ and $u=\pi/(2M)$.
\end{proposition}

\begin{proof}
Using $(e^{i\theta}-1)^2=-4\sin^2u\,e^{i\theta}$ and $\theta=2u$, equation
\eqref{eq:q-boundeq} is equivalent to
\[
        \delta\bigl[(1+\nu)e^{2iu}-\nu\bigr]=4\sin^2u\,e^{i(2\tau+2)u}.
\]
Factor $e^{iu}$ from each side. On the left,
$(1+\nu)e^{2iu}-\nu=e^{iu}\bigl[\cos u+i(1+2\nu)\sin u\bigr]$; on the right,
$e^{i(2\tau+2)u}=e^{iu}e^{iMu}$ because $2\tau+2=M+1$. Cancelling $e^{iu}$ leaves
\[
        \delta\bigl[\cos u+i(1+2\nu)\sin u\bigr]
        =4\sin^2u\,\bigl[\cos(Mu)+i\sin(Mu)\bigr].
\]
Equating real and imaginary parts, with $\delta$ and $\nu$ real, gives
\begin{equation}\label{eq:q-realimag}
        \delta\cos u=4\sin^2u\,\cos(Mu),
        \qquad
        \delta(1+2\nu)\sin u=4\sin^2u\,\sin(Mu).
\end{equation}
\mgb{Because $u\in(0,\pi/2)$, we have $\cos u>0$ and
$\sin u>0$. Thus any solution with $\delta>0$ must satisfy
$\cos(Mu)>0$. In particular, the denominator defining
$N_\tau(u)$ is nonzero.
The first equation in~\eqref{eq:q-realimag} then uniquely determines
$\delta=D_\tau(u)$.}
Dividing the second by the first and
using the angle-difference identity
$\cos u\,\sin(Mu)-\sin u\,\cos(Mu)=\sin\!\big((M-1)u\big)=\sin(2\tau u)$,
\begin{equation}\label{eq:q-oneplus2nu}
        1+2\nu=\frac{\cos u\,\sin(Mu)}{\sin u\,\cos(Mu)},
        \qquad\text{hence}\qquad
        2\nu=\frac{\sin(2\tau u)}{\sin u\,\cos(Mu)},
\end{equation}
which is $\nu=N_\tau(u)$. Conversely, $(\delta,\nu)=(D_\tau(u),N_\tau(u))$ solves
\eqref{eq:q-realimag} and hence \eqref{eq:q-boundeq}. On $\mathcal I_\tau$ one has
$Mu\in(0,\tfrac\pi2)$, so $\cos(Mu)>0$ and $\cos u>0$, whence $D_\tau(u)>0$; and
$2\tau u<2\tau\cdot\tfrac{\pi}{2M}=\tfrac{\tau\pi}{M}<\tfrac\pi2$, so
$\sin(2\tau u)>0$ and $N_\tau(u)>0$. 
\saa{Finally, since $\sin u\to0$ as $u\to0^+$ and $\cos(Mu)\to0$ as $u\to\pi/(2M)^-$, we have $D_\tau(u)\to0$ at both endpoints of $I_\tau$.}
\end{proof}

\paragraph{Stability loss through unit-circle crossings.}
\saa{The expansion \eqref{eq:z-pm-mod-first} guarantees that
$\Pinf(\cdot;\delta,\nu)$ is strictly Schur stable
for all sufficiently small $\delta>0$.
As $\delta$ increases, strict Schur stability
persists until a root first reaches the unit circle.
Thus,} by \eqref{eq:q-deltaLdef}, $\delta_L(\nu)$
is the smallest positive value of $\delta$ at which
a root reaches the unit circle.

\sa{The crossing equations \eqref{eq:q-realimag} substantially restrict the
possible crossing frequencies. Since $u\in(0,\pi/2)$ and $\delta>0$,
the first equation in \eqref{eq:q-realimag} implies $\cos(Mu)>0$.
Moreover, $\nu>\tau\ge1$ implies $1+2\nu>0$, so the second equation
gives $\sin(Mu)>0$. Hence \mgd{every unit-circle crossing with $\delta>0$} relevant to the
design problem \eqref{eq:q-designproblem} satisfies}
\[
        Mu\in\bigl(2k\pi,\,2k\pi+\tfrac{\pi}{2}\bigr),
        \qquad k\ge0.
\]
The band corresponding to $k=0$ is precisely
$\mathcal I_\tau$ in \eqref{eq:q-band}, the interval of lowest crossing
frequencies. Borrowing the main-lobe terminology for oscillatory kernels
\cite{Harris1978}, we refer to $\mathcal I_\tau$ as the \emph{first lobe}.
It is the band adjacent to the marginal frequency $u=0$, where at \mgd{$\delta=0$} the polynomial $z^\tau(z-1)^2$ has a double root at $z=1$.
The following lemma shows that, among \mgd{all unit-circle crossings with $\delta>0$} for a fixed $\nu>\tau$, \mgd{the first-lobe crossing occurs at the smallest positive value of $\delta$}. 
\saa{Consequently, the infimum defining $\delta_L(\nu)$ in \eqref{eq:q-deltaLdef}
is attained at a crossing frequency $u\in\mathcal I_\tau$.}

\begin{lemma}[all-frequency exclusion and first-lobe optimality]\label{lem:q-firstlobe}
Fix $\tau\ge1$, $M=2\tau+1$, and $\nu>\tau$; write $A:=1+2\nu$, so that $A>M$.
Then the following hold.
\begin{enumerate}[label=\textup{(\alph*)},nosep,leftmargin=2.2em]
\item \emph{(Crossing conditions)} Every unit-circle root of
$\Pinf(\cdot;\delta,\nu)$ with $\delta>0$ and $\theta\in(0,\pi)$ has $u=\theta/2$
satisfying
\begin{equation}\label{eq:q-crossing}
        \tan(Mu)=A\tan u,
        \qquad
        Mu\in\bigl(2k\pi,\,2k\pi+\tfrac\pi2\bigr)\ \text{for some }k\ge0,
        \qquad
        \delta=D_\tau(u).
\end{equation}
\item \emph{(Monotonicity of $D_\tau$ along the crossing set)} Along the crossing set
$\{u:N_\tau(u)=\nu\}$ one has $D_\tau(u)=\Phi_A(\tan u)$, where
\begin{equation}\label{eq:q-PhiA}
        \Phi_A(t):=\frac{4t^2}{\sqrt{(1+t^2)(1+A^2t^2)}}
\end{equation}
is strictly increasing on $(0,\infty)$. 
\sa{Consequently, $D_\tau(u)$ is strictly increasing with $u$ along the crossing set.}
\item \emph{(A unique lowest crossing.)} The map $R(u):=\tan(Mu)/\tan u$ is a
strictly increasing bijection of $\mathcal I_\tau$ onto $(M,\infty)$, with
$1+2N_\tau(u)=R(u)$; hence $N_\tau$ is a strictly increasing bijection
$\mathcal I_\tau\to(\tau,\infty)$. The crossing set therefore contains exactly one
point $u_0(\nu)=N_\tau^{-1}(\nu)$ in $\mathcal I_\tau$, and every other
positive-gain crossing has $u>2\pi/M>\pi/(2M)>u_0(\nu)$.
\item \emph{\mgd{(Endpoint crossings occur no earlier than the first-lobe crossing)}} The endpoint $z=+1$ is never a root \mgd{for $\delta>0$}, and the endpoint $z=-1$ (which can occur only for even $\tau$)
occurs at $\delta=4/A$, which exceeds $\delta_L(\nu)$.
\end{enumerate}
In particular, strict stability is first lost at the first-lobe frequency
$u_0(\nu)$, and
\begin{equation}\label{eq:q-deltaL-closed}
        \delta_L(\nu)=D_\tau\bigl(u_0(\nu)\bigr)=D_\tau\bigl(N_\tau^{-1}(\nu)\bigr).
\end{equation}
\end{lemma}

\begin{proof}
\emph{(a)} \mgd{At a unit-circle root with $\delta>0$}, $u\in(0,\tfrac\pi2)$ and
\eqref{eq:q-realimag} holds with $\delta>0$. Since $\sin u>0$ and $\sin^2u>0$, the
first equation forces $\cos(Mu)>0$ and the second forces $\sin(Mu)>0$, placing
$Mu$ in a first-quadrant band. Dividing the second equation in
\eqref{eq:q-realimag} by the first gives $\tan(Mu)=(1+2\nu)\tan u=A\tan u$, and the
first equation is $\delta=D_\tau(u)$.

\emph{(b)} Fix a crossing $u$ and put $t=\tan u>0$. Since $\cos(Mu)>0$ and
$\tan(Mu)=At$, we have $\cos(Mu)=(1+A^2t^2)^{-1/2}$. Using $\sin^2u=t^2\cos^2u$ and
$\cos u=(1+t^2)^{-1/2}$,
\[
        D_\tau(u)=\frac{4\sin^2u\,\cos(Mu)}{\cos u}
        =4t^2\cos u\,\cos(Mu)
        =\frac{4t^2}{\sqrt{(1+t^2)(1+A^2t^2)}}=\Phi_A(t).
\]
Differentiating $\log\Phi_A$,
\[
        t\,\frac{d}{dt}\log\Phi_A(t)
        =2-\frac{t^2}{1+t^2}-\frac{A^2t^2}{1+A^2t^2}
        >2-1-1=0,
\]
because each subtracted fraction lies in $(0,1)$. Hence $\Phi_A$ is strictly
increasing on $(0,\infty)$, and since $t=\tan u$ is strictly increasing on
$(0,\tfrac\pi2)$, so is $u\mapsto D_\tau(u)$ along the crossing set.

\emph{(c)} By \eqref{eq:q-oneplus2nu}, $1+2N_\tau(u)=R(u)$. On $\mathcal I_\tau$
both $\tan(Mu)>0$ and $\tan u>0$, so $R>0$. \mgd{Moreover,
$\lim_{u\to0^+}R(u)=M$
(since $\tan(Mu)/\tan u\to M$),
and $R(u)\to+\infty$ as
$u\to(\pi/(2M))^-$
(since $\tan(Mu)\to+\infty$).
\footnote{\saa{Here, the superscripts $+$ and $-$ indicate
limits taken from the right and left,
respectively.}}
} %
For monotonicity,
\[
        (\log R)'(u)=\frac{2M}{\sin(2Mu)}-\frac{2}{\sin(2u)},
\]
which is positive iff $M\sin(2u)>\sin(2Mu)$. Setting $v=2u\in(0,\tfrac\pi M)$ and
$h(v):=M\sin v-\sin(Mv)$, one has $h(0)=0$ and
$h'(v)=M\bigl(\cos v-\cos(Mv)\bigr)>0$, because $0<v<Mv<\pi$ and $\cos$ is
strictly decreasing on $(0,\pi)$. Thus $h>0$ on $(0,\tfrac\pi M)$ and $R$ is
strictly increasing on $\mathcal I_\tau$. Consequently
$R:\mathcal I_\tau\to(M,\infty)$ and $N_\tau=(R-1)/2:\mathcal I_\tau\to(\tau,\infty)$
are strictly increasing bijections. Since $A>M$, the equation $R(u)=A$ has a unique
solution $u_0(\nu)=N_\tau^{-1}(\nu)\in\mathcal I_\tau$. \mgd{Any crossing with $\delta>0$}
outside $\mathcal I_\tau$ lies in a band $k\ge1$, hence has $Mu>2\pi$, i.e.\
$u>2\pi/M$, which exceeds $\pi/(2M)$ and therefore $u_0(\nu)$.

\emph{(d)} From \eqref{eq:q-Pinf}, $\Pinf(1;\delta,\nu)=\delta>0$, so $z=+1$ is
\mgd{never a root for $\delta>0$}. Likewise
$\Pinf(-1;\delta,\nu)=4(-1)^\tau-\delta A$, which vanishes only at
$\delta=4(-1)^\tau/A$: negative (hence impossible) for odd $\tau$, and equal to
$4/A$ for even $\tau$. 
\sa{Combining \eqref{eq:q-DN} and \eqref{eq:q-oneplus2nu} at the
first-lobe crossing $u_0(\nu)$ gives
\[
A D_\tau(u_0)
=
D_\tau(u_0)\bigl(1+2N_\tau(u_0)\bigr)
=
4\sin u_0\sin(Mu_0)
<4,
\]
because $u_0\in(0,\pi/2)$ and $Mu_0\in(0,\pi/2)$}.
\mgd{Hence, $D_\tau(u_0)<4/A$. Therefore, when
$\tau$ is even, the crossing at $z=-1$
occurs \mgd{at a strictly larger positive value of $\delta$}
than the first-lobe crossing. When $\tau$
is odd, no crossing at $z=-1$ is possible
\mgd{for $\delta>0$}.}

Finally, assembling the parts: by (a) \mgd{every non-real unit-circle root with $\delta>0$} corresponds to a crossing $u$ with \mgd{$\delta=D_\tau(u)$}, by $(c)$ the smallest such $u$
is $u_0(\nu)$, and by $(b)$ $D_\tau$ is increasing along the crossing set, so \mgd{the smallest positive value of $\delta$ at such a crossing} is $D_\tau(u_0(\nu))$. \mgd{By $(d)$, a crossing at $z=-1$ cannot occur
\mgd{at a smaller positive value of $\delta$ than}
$D_\tau(u_0(\nu))$.} %
Since $P_\infty$ is strictly Schur stable for all \saa{sufficiently
small $\delta>0$ by Corollary~\ref{coro:q-smallgain}, and its roots
depend continuously on $\delta$, strict Schur stability persists
until the first unit-circle root occurs. By parts (a)--(d),
this happens exactly at $\delta=D_\tau(u_0(\nu))$,
proving \eqref{eq:q-deltaL-closed}.}
\end{proof}

\paragraph{\mgd{The optimal stability threshold and extrapolation parameter}.}
Lemma~\ref{lem:q-firstlobe} identifies $\delta_L(\nu)$ with the value of $D_\tau$
at the unique first-lobe crossing $N_\tau^{-1}(\nu)$. Maximizing over $\nu$
therefore amounts to maximizing $D_\tau$ over $\mathcal I_\tau$, which we now carry
out.

\begin{definition}[\mgd{critical stability threshold and extrapolation parameter}]\label{def:q-crit}
Let $u_\tau^\star\in\mathcal I_\tau$ maximize $D_\tau$. Differentiating $\log
D_\tau$, by first-order optimality conditions, $u_\tau^\star$ solves the equation
\begin{equation}\label{eq:q-ustar}
        \frac{D_\tau'(u)}{D_\tau(u)}=2\cot u+\tan u-M\tan(Mu)=0.
\end{equation}
The left-hand side is strictly decreasing on $\mathcal I_\tau$: its derivative is
$-2\csc^2u+\sec^2u-M^2\sec^2(Mu)$, which is negative there because $u<\pi/(2M)\leq \pi/6$ gives
$-2\csc^2u+\sec^2u=-1-2\cot^2u+\tan^2u<0$, while $-M^2\sec^2(Mu)<0$. Since the
left-hand side runs from $+\infty$ as $u\to0^+$ to $-\infty$ as $u\to(\pi/2M)^-$,
the maximizer $u_\tau^\star$ is the unique root of \eqref{eq:q-ustar} in
$\mathcal I_\tau$. We set
\begin{equation}\label{eq:q-critdef}
        \delta_\tau^{\mathrm{crit}}:=D_\tau(u_\tau^\star),
        \qquad
        \nu_\tau^{\mathrm{crit}}:=N_\tau(u_\tau^\star),
        \qquad
        c_\tau^{\mathrm{crit}}:=\sqrt{\delta_\tau^{\mathrm{crit}}} .
\end{equation}
\end{definition}

\begin{corollary}[exact solution of the design problem]\label{cor:q-designoptimal}
For every $\tau\ge1$, the design problem \eqref{eq:q-designproblem} is solved by
$\nu=\nu_\tau^{\mathrm{crit}}$, with optimal value
\begin{equation}\label{eq:q-designvalue}
        \max_{\nu>\tau}\ \delta_L(\nu)
        =\max_{u\in\mathcal I_\tau}D_\tau(u)
        =D_\tau(u_\tau^\star)
        =\delta_\tau^{\mathrm{crit}} .
\end{equation}
Consequently, $c_\tau^{\mathrm{crit}}
=\sqrt{\delta_\tau^{\mathrm{crit}}}$ is the supremum of rate constants
certifiable by Theorem~\ref{thm:q-gate} over the entire admissible family
$\{\nu>\tau\}$. More precisely, for every fixed $\tau$ and every
$0<c<c_\tau^{\mathrm{crit}}$, \mgd{choose an effective stepsize} 
$\bar\delta=\alpha L$ satisfying $c^2<\bar\delta<\delta_\tau^{\mathrm{crit}}$.
Then, with
\[
    \nu=\nu_\tau^{\mathrm{crit}},\qquad
    \alpha=\frac{\bar\delta}{L},\qquad
    \beta=1-2\sqrt{\bar\delta}\,q,
\]
Theorem~\ref{thm:q-gate} yields
$\rho_{\tau,\kappa}\le 1-cq$ for all sufficiently large $\kappa$.
\sa{Moreover, by \eqref{eq:q-designvalue}, no admissible $\nu>\tau$ permits a
larger stability threshold than $\delta_\tau^{\mathrm{crit}}$; hence, no
larger rate constant than $c_\tau^{\mathrm{crit}}$ can be certified by
Theorem~\ref{thm:q-gate}.}
\end{corollary}

\begin{proof}
By Lemma~\ref{lem:q-firstlobe}, $\delta_L(\nu)=D_\tau\bigl(N_\tau^{-1}(\nu)\bigr)$
for every $\nu>\tau$, and $N_\tau^{-1}$ is a bijection of $(\tau,\infty)$ onto
$\mathcal I_\tau$. Hence, as $\nu$ ranges over $(\tau,\infty)$, the argument
$N_\tau^{-1}(\nu)$ ranges over all of $\mathcal I_\tau$, so
\[
        \max_{\nu>\tau}\ \delta_L(\nu)
        =\max_{u\in\mathcal I_\tau}D_\tau(u)
        =D_\tau(u_\tau^\star)
        =\delta_\tau^{\mathrm{crit}},
\]
the maximum being attained at the interior stationary point $u_\tau^\star$ of
Definition~\ref{def:q-crit}, i.e.\ at $\nu=N_\tau(u_\tau^\star)=\nu_\tau^{\mathrm{crit}}$.
\wubtext{Since $\bar\delta<\delta_\tau^{\mathrm{crit}}$, Lemma~\hyperref[lem:q-firstlobe]{\ref*{lem:q-firstlobe}}
gives strict Schur stability for every $0<\delta\le\bar\delta$.} 
\saa{Applying Theorem~\ref{thm:q-gate} with $\delta_L=\bar\delta\in(c^2,\delta_\tau^{\mathrm{crit}})$ and $\nu=\nu_\tau^{\mathrm{crit}}$ yields the claimed rate. Finally, since the square-root function is strictly increasing and $\delta_\tau^{\mathrm{crit}}=\max_{\nu>\tau}\delta_L(\nu)$, the supremum of certifiable rate constants is $c_\tau^{\mathrm{crit}}=\sqrt{\delta_\tau^{\mathrm{crit}}}$.}
\looseness=-1
\end{proof}

Two remarks delimit the scope of Corollary~\ref{cor:q-designoptimal}. First, it is
an optimality statement \emph{within the family and ansatz of
Theorem~\ref{thm:q-gate}}: it identifies the supremum of certifiable rate constants from the
quadratic limiting polynomial $\Pinf$ 
\sa{under the parametrization $\alpha=\bar\delta/L$, $\beta=1-2\sqrt{\bar\delta}q$, with $\nu$ free and $0<\bar\delta<\delta_L(\nu)$.}
It is not a
lower bound over arbitrary delayed first-order methods; that role is played by the
span-model bound of Theorem~\ref{thm:asr-lower}. 
\mgd{
Second, $c_\tau^{\rm crit}$ denotes the supremum of the
quadratic rate constants multiplying $q$ directly in $1-cq$.
In contrast, our convergence guarantee for general
smooth strongly convex objectives given in Theorem \ref{thm:main-nonexplicit-constants} uses the rate
$1-cq/M$, with $c\in(0,2/3)$. 
\saa{Although both rate coefficients scale as $1/\tau$, they arise from different analyses and should not be identified.} 
The proof of Theorem~\ref{thm:main-nonexplicit-constants}
for smooth strongly convex objectives requires
verifying hypotheses \textbf{(H1)} and \textbf{(H2)}.
The restriction $c<2/3$ arises from the low-frequency
analysis of the rate-weighted Zames--Falb certificate
used to establish \textbf{(H2)}. Specifically, it
ensures that the coefficient
$\lambda=(\tau+1)(2-3c)/M$ in
\eqref{eq:constants-coeff} is positive; see also
\eqref{eq:constants-coeff-main}.
In contrast, \textbf{(H1)} is verified using the
conservative quadratic parameter choice of
Theorem~\ref{thm:q-alldelay}, rather than the
optimized parameters studied in this appendix.
}

\mgd{For our convergence analysis on general smooth
strongly convex objectives} we do not need the exact optimizer:
Theorem~\ref{thm:q-alldelay} adopts the slightly conservative closed-form choice
$\nu=M-\tfrac12$ and $\delta_L=M^{-2}$, for which Lemma~\ref{lem:q-nocross}
excludes every unit-circle crossing on $0<\delta\le M^{-2}$ and yields the rate
constant $c<1/M$. As computed in the next subsection, this conservative choice has
the same delay scaling as the optimal one---$\nu_\tau^{\mathrm{crit}}=\Theta(\tau)$
and $c_\tau^{\mathrm{crit}}=\Theta(1/\tau)$, matching $\nu=M-\tfrac12\sim2\tau$ and
$c<1/M$ and $1/M \sim 1/(2\tau)$---so the two agree in order and differ only in constants. \mgd{Corollary~\ref{cor:q-designoptimal} thus
completely characterizes the optimal
rate coefficient certified by
Theorem~\ref{thm:q-gate} within the
quadratic parameter-design framework.}

\subsection[Delay scaling of the quadratic convergence rate]%
{\mgd{The delay scaling law: the best rate constant \texorpdfstring{$c_\tau^{\mathrm{crit}}$}{c_tau^crit}
scales as $1/\tau$}}

Having solved the design problem exactly, we now quantify how the optimal constants
behave as the delay grows. The interval $\mathcal I_\tau=(0,\pi/(2M))$ shrinks like
$1/M$, so the natural variable is the rescaled frequency
$a:=Mu\in(0,\tfrac\pi2)$, on which the maximizer $u_\tau^\star$ concentrates.
\begin{proposition}[delay scaling of the critical constants]\label{prop:q-scaling}
As $\tau\to\infty$,
\begin{equation}\label{eq:q-scaling}
        \delta_\tau^{\mathrm{crit}}=\frac{a_\star^{2}\cos a_\star}{\tau^{2}}\,(1+o(1)),
        \qquad
        c_\tau^{\mathrm{crit}}=\frac{a_\star\sqrt{\cos a_\star}}{\tau}\,(1+o(1)),
        \qquad
        \nu_\tau^{\mathrm{crit}}=\frac{2}{a_\star^{2}}\,\tau\,(1+o(1)),
\end{equation}
where $a_\star\in(0,\tfrac\pi2)$ is the unique solution of
\begin{equation}\label{eq:q-astar}
        a_\star\tan a_\star=2,
\end{equation}
Numerically, $a_\star \approx 1.0768$, $a_\star\sqrt{\cos a_\star}\approx 0.7415$ and
$2/a_\star^{2} \approx 1.7246$.
\end{proposition}

\begin{proof}
Substituting $u=a/M$ into \eqref{eq:q-DN} and using $\sin(a/M)=a/M+\cO(M^{-3})$ and
$\cos(a/M)=1+\cO(M^{-2})$,
\begin{equation}\label{eq:q-Dscaled}
        M^2D_\tau\!\Bigl(\frac aM\Bigr)
        =4M^2\sin^2\!\frac aM\,\frac{\cos a}{\cos(a/M)}
        =4a^2\cos a+\cO(M^{-2}),
\end{equation}
uniformly for $a$ in compact subsets of $(0,\tfrac\pi2)$. For the extrapolation,
$2\tau/M=1-1/M$, so $\sin(2\tau u)=\sin(a-a/M)=\sin a-\tfrac{a\cos a}{M}+\cO(M^{-2})$
and $2\sin(a/M)\cos a=\tfrac{2a\cos a}{M}\bigl(1+\cO(M^{-2})\bigr)$; hence
\begin{equation}\label{eq:q-Nscaled}
        N_\tau\!\Bigl(\frac aM\Bigr)=\frac{\sin(2\tau\,a/M)}{2\sin(a/M)\,\cos a}
        =\frac{M}{2a}\tan a+\cO(1),
\end{equation}
again uniformly on compact subsets of $(0,\tfrac\pi2)$.

We first locate the maximizer. Write $f_M(a):=M^2D_\tau(a/M)$ and
$f_\infty(a):=4a^2\cos a$, so that \eqref{eq:q-Dscaled} states $f_M\to f_\infty$
uniformly on compact subsets of $(0,\tfrac\pi2)$. The limit profile $f_\infty$
vanishes as $a\to0^+$ and as $a\to(\tfrac\pi2)^-$ and is positive between, so it
attains its maximum at an interior point; its derivative $4a(2\cos a-a\sin a)$
vanishes on $(0,\tfrac\pi2)$ exactly where $a\tan a=2$, i.e.\ at $a=a_\star$, which
is therefore the \emph{unique} maximizer of $f_\infty$.

\saa{We claim that the maximizer $a_\tau^\star:=Mu_\tau^\star$ of
$f_M$ converges to $a_\star$. In fact, the convergence
$f_M\to f_\infty$ is uniform on the entire closed interval
$[0,\pi/2]$. To see this, write
\[
f_M(a)
=4a^2\cos a\,
\frac{\bigl(\sin(a/M)/(a/M)\bigr)^2}{\cos(a/M)},
\]
where the ratio is defined by continuity at $a=0$.
Since $0\le a/M\le\pi/(2M)\le\pi/6$, the Taylor expansions
of $\sin t/t$ and $\cos t$ at zero imply
\[
\sup_{a\in[0,\pi/2]}
\left|
\frac{\bigl(\sin(a/M)/(a/M)\bigr)^2}{\cos(a/M)}
-1
\right|
=\cO(M^{-2}).
\]
Consequently,
\[
\sup_{a\in[0,\pi/2]}|f_M(a)-f_\infty(a)|
=\cO(M^{-2})\longrightarrow0.
\]
Both functions extend continuously to $[0,\pi/2]$ and
vanish at its endpoints. Since $f_\infty$ has the unique
maximizer $a_\star\in(0,\pi/2)$, uniform convergence
implies that every accumulation point of the maximizers
$a_\tau^\star$ maximizes $f_\infty$. By uniqueness,
$a_\tau^\star\to a_\star$.} Therefore
\[
        \max_{u\in\mathcal I_\tau}D_\tau(u)
        =\frac{1}{M^2}f_M(a_\tau^\star)
        =\frac{4}{M^2}\bigl(a_\star^2\cos a_\star\bigr)(1+o(1)).
\]
This is $\delta_\tau^{\mathrm{crit}}$; since $M=2\tau+1=2\tau(1+o(1))$,
\[
        \delta_\tau^{\mathrm{crit}}=\frac{4a_\star^2\cos a_\star}{M^2}(1+o(1))
        =\frac{a_\star^2\cos a_\star}{\tau^2}(1+o(1)),
\]
and $c_\tau^{\mathrm{crit}}=\sqrt{\delta_\tau^{\mathrm{crit}}}$ gives the stated
expression. Finally, evaluating \eqref{eq:q-Nscaled} at $a_\tau^\star\to a_\star$
and using $\tan a_\star=2/a_\star$ from \eqref{eq:q-astar},
\[
        \nu_\tau^{\mathrm{crit}}=N_\tau(u_\tau^\star)=\frac{M}{2a_\star}\tan a_\star\,(1+o(1))
        =\frac{M}{a_\star^2}(1+o(1))=\frac{2}{a_\star^2}\,\tau\,(1+o(1)),
\]
again with $M=2\tau(1+o(1))$. The numerical values follow from
$a_\star\approx1.0768739863$.
\end{proof}
\section{The finite-condition-number certificate}\label{sec:finiteq}

\mgd{Throughout this appendix, fix $\tau\ge1$, $M=2\tau+1$, and
$c\in(0,2/3)$, and use the normalized parameters
\eqref{eq:norm}--\eqref{eq:params}. The expressions at $q=0$ are
analytic or continuous extensions of the finite-$q$ quantities, not
algorithm parameters for a finite condition number.}

The limiting positivity certificate for $r_\tau$ established in
Section~\ref{sec:limit}, or equivalently the strict negativity of the
limiting FDI, is obtained in the limit $q\downarrow0$ and is
therefore asymptotic in nature. The purpose of this section is to make this
analysis quantitative by identifying an explicit range of $q$ for which the
strict FDI continues to hold.

We proceed in two stages. First, we study the joint behavior of
$N_{\tau,q}(\theta)$ in $\theta$ and $q$ through its Taylor expansion near
$(\theta,q)=(0,0)$. This yields a positive threshold
$q_{\rm fdi}(\tau,c)$ such that the FDI is strictly negative whenever
$0<q<q_{\rm fdi}(\tau,c)$; see Lemma~\ref{lem:finite-fdi}. We then make this threshold fully explicit by deriving quantitative bounds
on the derivatives and Taylor remainder of $N_{\tau,q}$. This leads to an
explicit upper bound for $N_{\tau,q}$ and, consequently, to a uniform
negative margin for the FDI; see
Proposition~\ref{prop:explicit-fdi-threshold} and
Corollary~\ref{cor:Phi-uniform-margin}. The latter estimate is the input
used in Section~\ref{sec:explicit} to quantify the dependence of the
prefactor $\mgd{C_{\mathrm{rate}}}$ on $q$.

We first fix the standing constants used in the explicit estimates of this and the next two sections.
Define
\begin{equation}\label{eq:closed-basic-constants}
        M:=2\tau+1,\qquad
        \bar q:=\frac{M}{2M+1},\qquad
        R_{\rm rt}:=\mathrm e,\qquad
        \mgd{\mbox{deg}_\tau}:=2\tau+2=M+1 .
\end{equation}
If $0\le q\le \bar q$ and $0<c<2/3$, then
\begin{equation}\label{eq:rho-lower-closed}
        \rho=1-\frac{cq}{M}\ge 1-\frac{2}{3(2M+1)}=\frac{6M+1}{6M+3}\ge \frac{19}{21}.
\end{equation}
\mgd{For every $0\le s\le 3M$, we obtain the
elementary bound
$\rho^{-s}\le (21/19)^s$.
However, when $s$ is proportional to $M$,
this bound grows exponentially with the delay
$\tau$ and may be overly conservative. We therefore derive an alternative sharper bound below.} Indeed, applying the inequality $-\log(1-x)\le x/(1-x)$ with $x=cq/M$, we obtain
$$-\log(\rho) \le \frac{cq/M}{1- cq/M} \leq \frac{c}{2M+1-c}, $$ where we used $cq/M\le c/(2M+1)<1$. Exponentiating both sides, for every $0\le s\le3M$,
\begin{equation}\label{eq:rho-sharp-envelope}
        \rho^{-s}
        \le
        \exp\!\left(\frac{s c}{2M+1-c}\right)
        \le
        \exp(3c/2)
        <\mathrm e=R_{\rm rt}.
\end{equation}
Thus all rate-weight factors that occur below are bounded by an absolute constant rather than by a quantity exponential in $\tau$.

Extending the limiting FDI at $q=0$ to sufficiently small $q>0$ is straightforward away from zero frequency, but requires separate analysis near $\theta=0$. To this end, recall from \eqref{eq:Ndef} that
\begin{equation}\label{def-recall-Phi-supp}
\Phi_{\tau,q}(\theta)
=
\frac{N_{\tau,q}(\theta)}
{M^2|D_1(\theta,q)|^2|D_2(\theta,q)|^2},\end{equation}
where 
\begin{equation}
\quad D_1(\theta,q) = \rho e^{i\theta}-1, 
\quad 
D_2(\theta,q) = \rho e^{i\theta}-\beta, 
\quad 
\rho = 1 - cq/M, 
\quad \mbox{and} \quad \beta = 1 - 2q/M. \label{eq-D1-D2-expansions}
\end{equation}
At $q=0$,
$
D_1(\theta,0)=D_2(\theta,0)=e^{i\theta}-1,
$
so the limiting denominator becomes
$
M^2|e^{i\theta}-1|^4.
$
For every fixed $\delta>0$, this quantity is uniformly bounded away from zero on $\theta\in[\delta,\pi]$. %
Hence, on any interval separated from the origin, the denominator remains nondegenerate and the finite-$q$ FDI converges uniformly to its $q=0$ limit. The strict negativity of the limiting FDI therefore persists there for all sufficiently small $q>0$.

Near $\theta=0$, however, the denominator degenerates. Indeed,
$
e^{i\theta}-1=i\theta+\cO(\theta^2),
$
and consequently
$
M^2|e^{i\theta}-1|^4
=
M^2\theta^4+\cO(\theta^6),
$ as $\theta\to0$.
\saa{Thus, the denominator is not uniformly bounded away from zero on any neighborhood containing the origin, so the preceding argument for uniform convergence does not apply there.}
Since establishing the finite-$q$ FDI requires a single threshold on $q$ that works uniformly over all frequencies, the low-frequency regime near $\theta=0$ must be treated separately.

The singular behavior near $\theta=0$ is determined by the joint
first-order expansions of the denominator factors. As $(q,\theta)\to(0,0)$, from \eqref{eq-D1-D2-expansions}, we see that
\[
D_1(\theta,q)
=
-\frac{c}{M}q+i\theta
+\cO\!\left(q^2+q|\theta|+\theta^2\right),
\quad
\mbox{and}
\quad
D_2(\theta,q)
=
\frac{2-c}{M}q+i\theta
+\cO\!\left(q^2+q|\theta|+\theta^2\right).
\]
\saa{The $q$-dependent and frequency-dependent terms are comparable when $\theta$ is of order $q/M$. We refer to this low-frequency regime as a \emph{boundary layer} and study it by setting}
\begin{equation}\label{eq:inner-scaling}
\theta=\frac{q\varphi}{M},
\qquad \varphi\ge0,
\end{equation}
under which
\begin{equation}
D_1\left(\frac{q\varphi}{M},q\right)
=
\frac{q}{M}(-c+i\varphi)+\cO(q^2),
\qquad
D_2\left(\frac{q\varphi}{M},q\right)
=
\frac{q}{M}(2-c+i\varphi)+\cO(q^2).
\label{eq-D1-D2-limit}
\end{equation}
The following lemma identifies the resulting limit of $\Phi_{\tau,q}$ with the scaling \eqref{eq:inner-scaling} and shows that the
convergence is uniform for $\varphi$ in compact sets.

\begin{lemma}[Scaling limit of $\Phi_{\tau,q}$]\label{lem:layer}
\saa{Fix an integer $\tau\ge1$ and $c\in(0,2/3)$, and set $b_\tau=(\tau+1)/(2M)$. Then, as $q\downarrow0$, uniformly for $\varphi$ in compact subsets}\footnote{We say that $f_q(\varphi)\to f(\varphi)$ uniformly on compact sets as
$q\downarrow0$ if, for every compact set $K\subset\mathbb{R}$,
$
\sup_{\varphi\in K}
\left|f_q(\varphi)-f(\varphi)\right|
\longrightarrow 0$ as $q\downarrow0$.
Equivalently, for every compact set $K\subset\mathbb{R}$ and every
$\varepsilon>0$, there exists $q_0=q_0(K,\varepsilon)>0$ such that
$
0<q<q_0
\quad\Longrightarrow\quad
\left|f_q(\varphi)-f(\varphi)\right|<\varepsilon$ for all $ \varphi\in K.
$} of $[0,\infty)$,
\begin{equation}\label{eq:layer-profile}
        \frac{q}{2b_\tau}\,\Phi_{\tau,q}\left(\frac{q\varphi}{M}\right)
        \longrightarrow
        B_c(\varphi):=
        \frac{-c(1-c)^2+(3c-2)\varphi^2}
        {(c^2+\varphi^2)((2-c)^2+\varphi^2)} \quad 
        \mbox{as} \quad q\to 0.
\end{equation}
In particular $B_c(\varphi)<0$ for all $\varphi\ge0$ when $c<2/3$.
\end{lemma}

\begin{proof}
Let $z=\rho e^{i\theta}$ where $\rho \mathrel{\wubtext{=}} 1 - cq/M$. Under the scaling \eqref{eq:inner-scaling}, using \eqref{eq-D1-D2-limit}, as $q\to 0$,
\begin{equation}
        \zeta:=\frac{M(z-1)}{q} \to -c+i\varphi,\qquad
        \eta:=\frac{M(z-\beta)}{q} \to 2-c+i\varphi. \label{eq-zeta-eta-scaling}
\end{equation}
\saa{Moreover, \(z^{-\tau}\to1\) and \(P=(M+\tfrac12)z-(M-\tfrac12)\to1\), uniformly for \(\varphi\) in compact sets. Therefore,}
\begin{equation*}
        g_{\tau,q}(z)\to g_0(\varphi):=
        -\frac1{(-c+i\varphi)(2-c+i\varphi)},
\end{equation*}
where $g_{\tau,q}(z)$ is defined by \eqref{eq:g}. By Lemma~\ref{lem:telescoped}, for $h=h_{\rho,\tau}$, we have $1-h=(1-\bar z)\tilde{h}(\bar z)$. Since $\bar z \to 1$ and $\rho\to 1$ as $ q\to 0$, we have $\tilde{h}(\bar z) \to(\tau+1)/2$. On the other hand, ${1-\bar z}/{q} \to {(c+ i \varphi)}/{M}$; therefore,
\begin{equation*}
    \frac{1-h}{q} = \frac{(1-\bar z)\tilde{h}(\bar z)}{q}\to b_\tau(c+i\varphi).
\end{equation*}
Substitution into \eqref{eq:scalar-factor} together with $L=q^{-2}$ gives
\begin{equation*}
\frac{q}{2b_\tau}\Phi_{\tau,q}\left(\frac{q\varphi}{M}\right)
\to
\Ree\bigl[(c+i\varphi)(1-\bar g_0)g_0\bigr],
\end{equation*}
\mg{%
The proof will be complete once we simplify the expression on the right. Let
\[
D:=(-c+i\varphi)(2-c+i\varphi),
\qquad
g_0=-D^{-1}.
\]
Then
\[
(1-\overline{g_0})g_0
=
-\frac{\overline D+1}{|D|^2},
\quad
\mbox{where}
\quad
|D|^2
=
(c^2+\varphi^2)\bigl((2-c)^2+\varphi^2\bigr)
\]
and
\[
\overline D+1
=
(1-c)^2-\varphi^2-2i(1-c)\varphi.
\]
Consequently,
\[
\begin{aligned}
\Ree\bigl[(c+i\varphi)(1-\overline{g_0})g_0\bigr]
=
-\frac{
\Ree\bigl[(c+i\varphi)(\overline D+1)\bigr]
}{
|D|^2
}
=
\frac{
-c(1-c)^2+(3c-2)\varphi^2
}{
(c^2+\varphi^2)\bigl((2-c)^2+\varphi^2\bigr)
}.
\end{aligned}
\]
Since \(0<c<2/3\), both terms in the numerator are nonpositive,
with the first one strictly negative, so the above quantity is
strictly negative for every \(\varphi\geq 0\). This proves \eqref{eq:layer-profile}.}

\mg{The convergence is uniform on compact sets of the rescaled variable
$\varphi$. Indeed, for every $R>0$,
\[
e^{-iq\varphi/M}
=
1-\frac{iq\varphi}{M}+\cO_R(q^2),
\qquad |\varphi|\le R.
\]
\mgd{Here, $\mathcal O_R(q^2)$ denotes a remainder
bounded in absolute value by $C_Rq^2$,
uniformly for $|\varphi|\le R$, where
$C_R$ is independent of $q$.}
Since $\tau$ is fixed,
the finite sums defining $h$ inherit this uniform expansion, yielding
\[
\frac{1-h}{q}
=
b_\tau(c+i\varphi)+\cO_R(q).
\]
Using \eqref{eq-zeta-eta-scaling} and the definition of $g_{\tau,q}(z)$ from \eqref{eq:g}, write $g_{\tau,q}(z)=-z^{-\tau}P/(\zeta\eta)$. Its numerator and denominator converge uniformly
for $|\varphi|\le R$. The limiting denominator
$
(-c+i\varphi)(2-c+i\varphi)
$
is bounded away from zero, since
$
(c^2+\varphi^2)\bigl((2-c)^2+\varphi^2\bigr)
\ge c^2(2-c)^2>0.
$
It follows that $g\to g_0$ uniformly on compact $\varphi$-sets, and
hence the same is true for the rescaled quantity
$
\frac{q}{2b_\tau}
\Phi_{\tau,q}\left(\frac{q\varphi}{M}\right).$ This completes the proof.}
\end{proof}
We next study the Taylor expansion of the numerator $N_{\tau,q}$ in \eqref{def-recall-Phi-supp} around $\theta=0$.
\subsection{Taylor expansion of $N_{\tau,q}$ around $\theta=0$}
\mg{We recall from the definition of $N_{\tau,q}(\theta)$ given in \eqref{eq:Ndef} that
\begin{equation}\label{eq:N-def}
N_{\tau,q}(\theta)=2\Ree(\bar D_1S)
\end{equation}
with 
\begin{align}\label{eq:S-def}
S={}&
\bar D_1\widetilde h\,z^{-\tau}P\bar D_2
+M^2\abs{D_1}^2\widetilde h\abs{D_2}^2
+\alpha\widetilde h\abs P^2\rho^{-2\tau}
+\alpha M^2\widetilde h\,\bar z^{-\tau}\bar P D_1D_2,
\end{align}
where $z=\rho\wubtext{e^{i\theta}}$, $D_1 = z-1$, \wubtextb{and} $D_2 = z-\beta$.
For fixed $\tau$, the function $N_{\tau,q}$ is an even real trigonometric
polynomial of degree at most
$\mgd{\mbox{deg}_\tau}=2\tau+2.$
Before extracting the Taylor coefficients, we record the elementary
Fourier fact that will justify all uniform Taylor estimates below.  Put
\wubtextb{$D:=\mbox{deg}_\tau$ and} $w=e^{i\theta}$ and, in \eqref{eq:S-def}, substitute
\[
z=\rho w,
\qquad
\bar z=\rho w^{-1}.
\]
Then $\bar D_1S$ is a Laurent polynomial in $w$.  Thus there exist
coefficients $s_\ell(q)$ such that
\begin{equation}\label{eq:Laurent-Q}
\bar D_1S
=
\sum_{\ell=-D}^{D}s_\ell(q)w^\ell.
\end{equation}
Since all scalar parameters are real, $s_\ell(q)$ is real for real $q$.
Consequently
\begin{equation}\label{eq:fourier-N}
N_{\tau,q}(\theta)
=
\sum_{\ell=-D}^{D}\widehat N_\ell(q)e^{i\ell\theta}
=
\widehat N_0(q)
+2\sum_{\ell=1}^{D}\widehat N_\ell(q)\cos(\ell\theta),
\end{equation}
where
\begin{equation}\label{eq:fourier-coefficients}
\widehat N_\ell(q)
:=s_\ell(q)+s_{-\ell}(q),
\qquad
\widehat N_{-\ell}(q)=\widehat N_\ell(q)\in\mathbb R.
\end{equation}
Equivalently,
\[
\widehat N_\ell(q)
=\frac{1}{2\pi}\int_0^{2\pi}
N_{\tau,q}(\theta)e^{-i\ell\theta}\,d\theta,
\]
These are the Fourier coefficients of $N_{\tau,q}$. %
\mgd{We next show that the Fourier coefficients are
real-analytic near $q=0$, despite the negative
powers of $\rho$ appearing in their expressions.} Indeed, after the above
substitution, every coefficient $s_\ell(q)$ is a finite sum of products of
$\alpha$, $\beta$, and integer powers of $\rho$, including finitely many
negative powers coming from $z^{-\tau}$, $\bar z^{-\tau}$, and
$\rho^{-2\tau}$.  Since
\[
\rho(0)=1,
\]
all such powers are real-analytic for $q$ in a neighbourhood of zero.
Hence every $s_\ell(q)$, and therefore every Fourier coefficient
$\widehat N_\ell(q)$, is real-analytic near $q=0$.  In particular, for any
sufficiently small fixed $q_*>0$,
\begin{equation}\label{eq:fourier-uniform-bound}
\sup_{0\le q\le q_*}\abs{\widehat N_\ell(q)}<\infty,
\qquad |\ell|\le D.
\end{equation}
Because the number of Fourier modes is finite, \eqref{eq:fourier-uniform-bound}
immediately yields a uniform derivative bound:
\begin{equation}\label{eq:fourth-derivative-bound}
\sup_{0\le q\le q_*}\sup_{\theta\in\mathbb R}
\abs{\partial_\theta^4N_{\tau,q}(\theta)}
\le
\sum_{\ell=-D}^{D}|\ell|^4
\sup_{0\le q\le q_*}\abs{\widehat N_\ell(q)}
=:C_4<\infty.
\end{equation}
This is the precise reason that the fourth-order Taylor remainder may be
controlled uniformly in $q$.}

\mg{Write the even Taylor expansion at the origin as
\begin{equation}\label{eq:taylor-N}
N_{\tau,q}(\theta)
=a_0(q)+a_2(q)\theta^2+a_4(q)\theta^4+O(\theta^6).
\end{equation}
From \eqref{eq:fourier-N}, the first Taylor coefficients are
\begin{equation}\label{eq:taylor-from-fourier}
\begin{aligned}
a_0(q)
&=\widehat N_0(q)+2\sum_{\ell=1}^{D}\widehat N_\ell(q),\\
a_2(q)
&=-\sum_{\ell=1}^{D}\ell^2\widehat N_\ell(q),\\
a_4(q)
&=\frac1{12}\sum_{\ell=1}^{D}\ell^4\widehat N_\ell(q).
\end{aligned}
\end{equation}
In particular, $a_0$, $a_2$, and $a_4$ are real-analytic functions of $q$
near zero.} \mgd{The following lemma identifies the leading dependence on $q$ of $a_0$ and $a_2$, and the limiting quartic coefficient $a_4(0)$}.
\begin{lemma}[low-order coefficients]\label{lem:coeffs}
Let $c\in(0,2/3)$ and put
\begin{equation}\label{eq:constants-coeff}
        \mgd{B_{\mathrm{cub}}}=\frac{(\tau+1)c(1-c)^2}{M^3},
        \qquad
        \lambda=\frac{(\tau+1)(2-3c)}{M},
        \qquad
        A_0=\frac{M(\tau+1)}{3}.
\end{equation}
Then, as $q\downarrow0$,
\begin{equation}\label{eq:coeff-asymp}
        a_0(q)=-\mgd{B_{\mathrm{cub}}}q^3+\cO(q^4),
        \qquad
        a_2(q)=-\lambda q+\cO(q^2),
        \qquad
        a_4(0)=-A_0.
\end{equation}
In particular, the leading coefficients $-\mgd{B_{\mathrm{cub}}}$, $-\lambda$, and $-A_0$ in \eqref{eq:coeff-asymp} are strictly negative when $c<2/3$.
\end{lemma}
\begin{proof}
\mg{%
For $\theta=q\varphi/M$ and $\varphi$ in a fixed compact set,
\[
z
=1+\frac{q}{M}(-c+i\varphi)+O(q^2),
\]
uniformly in $\varphi$. Recall from \eqref{eq-D1-D2-limit} that
\[
D_1=\frac{q}{M}(-c+i\varphi)+\cO(q^2),
\qquad
D_2=\frac{q}{M}(2-c+i\varphi)+\cO(q^2),
\]
and therefore
\begin{equation}\label{eq:den-limit}
\frac{M^2\abs{D_1}^2\abs{D_2}^2}{q^4}
\longrightarrow
\frac{(c^2+\varphi^2)((2-c)^2+\varphi^2)}{M^2}
\end{equation}
uniformly on compact $\varphi$-sets.  Combining
\eqref{def-recall-Phi-supp}, \eqref{eq:layer-profile}, and
\eqref{eq:den-limit}, and using $2b_\tau=(\tau+1)/M$, gives
\begin{equation}\label{eq:N-inner-limit}
\frac{N_{\tau,q}(q\varphi/M)}{q^3}
\longrightarrow
\frac{\tau+1}{M^3}
\Bigl[-c(1-c)^2+(3c-2)\varphi^2\Bigr],
\end{equation}
again uniformly on compact sets.}

\mg{We now extract the first two Taylor coefficients.  By evenness,
$\partial_\theta N_{\tau,q}(0)=\partial_\theta^3N_{\tau,q}(0)=0$.
Taylor's theorem, together with the uniform bound
\eqref{eq:fourth-derivative-bound}, therefore gives
\begin{equation}\label{eq:uniform-taylor}
N_{\tau,q}(\theta)
=a_0(q)+a_2(q)\theta^2+R_4(\theta,q),
\qquad
\abs{R_4(\theta,q)}\le \frac{C_4}{24}\abs{\theta}^4,
\end{equation}
for all $0\le q\le q_*$ and all $\theta$.  Substituting
$\theta=q\varphi/M$ and dividing by $q^3$ yields, uniformly for $\varphi$
in any fixed compact set,
\begin{equation}\label{eq:scaled-Taylor}
\frac{N_{\tau,q}(q\varphi/M)}{q^3}
=
\frac{a_0(q)}{q^3}
+
\frac{a_2(q)}{q}\frac{\varphi^2}{M^2}
+\cO(q).
\end{equation}
}

\mg{Setting $\varphi=0$ in \eqref{eq:N-inner-limit} gives
\[
\frac{a_0(q)}{q^3}
\longrightarrow
-\frac{(\tau+1)c(1-c)^2}{M^3}
=-\mgd{B_{\mathrm{cub}}}.
\]
Since $a_0$ is real-analytic at $q=0$, this limit forces its Taylor
expansion to begin at cubic order, with cubic coefficient $-B_{\mathrm{cub}}$; hence
\begin{equation}\label{eq:a0-asymptotic}
a_0(q)=-B_{\mathrm{cub}}q^3+\cO(q^4).
\end{equation}}

\mg{Next fix any $\varphi\ne0$.  Subtracting the identity
\eqref{eq:scaled-Taylor} at $\varphi=0$ from the one at this $\varphi$, and
then using \eqref{eq:N-inner-limit}, gives
\[
\frac{a_2(q)}{q}\frac{\varphi^2}{M^2}+\cO(q)
\longrightarrow
\frac{\tau+1}{M^3}(3c-2)\varphi^2.
\]
Cancelling $\varphi^2$ yields
\[
\frac{a_2(q)}{q}
\longrightarrow
\frac{\tau+1}{M}(3c-2)
=-\lambda.
\]
Again using the real-analyticity of $a_2$ at the origin, we conclude that
$
a_2(q)=-\lambda q+\cO(q^2).$ It remains to determine the quartic coefficient at $q=0$.  There
$\rho=\beta=1$, so $D_1=D_2=e^{i\theta}-1$, and
\eqref{def-recall-Phi-supp} becomes
\[
N_{\tau,0}(\theta)
=M^2\abs{e^{i\theta}-1}^4\Phi_{\tau,0}(\theta).
\]
Since
\[
\frac{|e^{i\theta}-1|^4}{\theta^4}
\longrightarrow 1
\quad\text{as }\theta\to0,
\qquad
\Phi_{\tau,0}(0)=-\frac{\tau+1}{3M}
\]
\saa{by Lemma~\ref{lem:Phi0-endpoint}, and since the preceding
expansions imply $a_0(0)=a_2(0)=0$, the Taylor expansion
of $N_{\tau,0}$ at $\theta=0$ gives}
\[
a_4(0)
=\lim_{\theta\to0}\frac{N_{\tau,0}(\theta)}{\theta^4}
=M^2\Phi_{\tau,0}(0)
=-\frac{M(\tau+1)}{3}
=-A_0.
\]
This proves \eqref{eq:coeff-asymp}.}
\end{proof}

\mg{Building on the Taylor expansion of \(N_{\tau,q}\), which is the numerator in the representation of \(\Phi_{\tau,q}\), the next result shows that the strict FDI \(\Phi_{\tau,q}<0\) holds for all sufficiently small \(q>0\). This conclusion is still asymptotic, in that the admissible range of \(q\) is not yet made explicit; a quantitative, non-asymptotic version is developed later.}

\begin{lemma}[strict finite-$q$ FDI]\label{lem:finite-fdi}
For every fixed $\tau\ge1$ and $c\in(0,2/3)$ there exists
$q_{\rm fdi}(\tau,c)>0$ such that
\begin{equation}\label{eq:finite-fdi}
        \Phi_{\tau,q}(\theta)<0,
        \qquad
        \theta\in[0,2\pi),\quad 0<q\le q_{\rm fdi}(\tau,c).
\end{equation}
\end{lemma}

\begin{proof}
By evenness and the positive denominator in \eqref{eq:Ndef}, it suffices to prove
$N_{\tau,q}(\theta)<0$ for $\theta\in[0,\pi]$. Near the origin, Taylor's theorem gives, for $0\le\theta\le1/2$,
\begin{equation}\label{eq:tay6}
        N_{\tau,q}(\theta)=a_0(q)+a_2(q)\theta^2+a_4(q)\theta^4+R_6(\theta,q),
        \qquad
        \abs{R_6(\theta,q)}\le C_6\theta^6,
\end{equation}
where $C_6$ is finite because $N_{\tau,q}$ has bounded Fourier coefficients for $0\le q\le q_\ast$ (see \eqref{eq:fourier-uniform-bound}). By
Lemma~\ref{lem:coeffs} and the continuity of $a_4(q)$ at $q=0$ (as the Fourier coefficients of
$N_{\tau,q}$ are continuous in $q$), after decreasing $q_\ast$ if necessary,
\begin{equation*}
        a_0(q)\le-\frac{\mgd{B_{\mathrm{cub}}}}{2}q^3,
        \qquad
        a_2(q)\le-\frac{\lambda}{2}q,
        \qquad
        a_4(q)\le-\frac{A_0}{2}.
\end{equation*}
Choose $0<\theta_\star\mathrel{\wubtext{\le}}\wubtext{1/2}$ so that $C_6\theta_\star^2\le A_0/4$.  Then for
$0\le\theta\le\theta_\star$ and $0<q\le q_\ast$,
\begin{equation}\label{eq:near-neg-appC}
N_{\tau,q}(\theta)
\le -\frac{\mgd{B_{\mathrm{cub}}}}{2}q^3-\frac{\lambda}{2}q\theta^2-\frac{A_0}{4}\theta^4<0.
\end{equation}
\saa{
For $\theta\in[\theta_\star,\pi]$, Corollary~\ref{cor:outer}, specifically
\eqref{eq:outerN}, guarantees that $N_{\tau,q}(\theta)<0$ uniformly over this
interval for all sufficiently small $q>0$. Taking $q_{\rm fdi}(\tau,c)>0$ to be the minimum of the admissible thresholds for
the low- and high-frequency regimes establishes \eqref{eq:finite-fdi}.
}
\end{proof}
\subsubsection{\saa{Derivative bounds for the numerator of the FDI}}
\label{sec:proof-LC4}

The remainder of this \mgd{subsection} makes the threshold $q_{\rm fdi}(\tau,c)$ explicit \mgd{using derivative bounds that grow polynomially} in $M$. The numerator $N_{\tau,q}$ in Proposition~\ref{prop:rational} is an even real trigonometric polynomial of degree at most \mgd{$\mbox{deg}_\tau=2\tau+2$}. \mg{We first obtain explicit bounds on the derivatives of $N_{\tau,q}(\theta)$ with respect to $q$ and $\theta$}.
\begin{lemma}[Derivative bounds for the FDI numerator]\label{lem:closed-fdi-envelope}
For $0\le q\le \bar q$, $0\le\theta\le\pi$, and every integer $r$ with $0\le r\le4$, the following order-specific bounds hold:
\begin{align}
        \left|\partial_q^rN_{\tau,q}(\theta)\right| &\le K_0(\tau):=2^{18}M^2,\label{eq:closed-fdi-envelope-0}\\
        \left|\partial_q^r\partial_\theta^2N_{\tau,q}(\theta)\right| &\le K_2(\tau):=2^{21}M^4,\label{eq:closed-fdi-envelope-2}\\
        \left|\partial_q^r\partial_\theta^4N_{\tau,q}(\theta)\right| &\le K_4(\tau):=2^{26}M^6,\label{eq:closed-fdi-envelope-4}\\
        \left|\partial_q^r\partial_\theta^6N_{\tau,q}(\theta)\right| &\le K_6(\tau):=2^{31}M^8.\label{eq:closed-fdi-envelope-6}
\end{align}
In particular, $\left|\partial_q^r\partial_\theta^sN_{\tau,q}(\theta)\right|\le2^{31}M^8$ for all integers $0\le r\le4$ and $0\le s\le6$.
\end{lemma}

\begin{proof}%
We use \eqref{eq:S-correct}, %
and the identity
\[
        W:=1-h_{\rho,\tau}(e^{i\theta})=-\bar D_1\widetilde h(\bar z).
\]
Writing $N_{\tau,q}=2\Ree T_{\tau,q}$, the product $T_{\tau,q}:=\bar D_1S$ can be regrouped as
\begin{align}\label{eq:T-regrouped-mixed}
T_{\tau,q}
={}&-\bar D_1Wz^{-\tau}P\bar D_2
     -M^2|D_1|^2W|D_2|^2
     -\alpha W|P|^2\rho^{-2\tau}
     -\alpha M^2W\bar z^{-\tau}\bar P D_1D_2 .
\end{align}
\mgd{This regrouping combines $\widetilde h$ with $\bar D_1$
to form $W=1-h_{\rho,\tau}$, thereby avoiding a separate
estimate of the triangular polynomial $\widetilde h$.}

\mgd{For a trigonometric polynomial
$F(\theta)=\sum_\ell f_\ell e^{i\ell\theta}$,
we \mgd{define the unweighted Wiener norm and the weighted
Wiener seminorms, respectively, by}
\[
    \|F\|_A:=\sum_\ell |f_\ell|,
    \qquad
    \|F\|_{A_s}:=\sum_\ell |\ell|^s|f_\ell|,
    \quad s\ge1.
\]
\mgd{For every integer $s\geq 1$}, 
$\|\partial_\theta^s F\|_\infty\le\|F\|_{A_s}$.
\mgd{We will estimate $\|T_{\tau,q}\|_A$ and
$\|T_{\tau,q}\|_{A_s}$ using the submultiplicativity
of the unweighted Wiener norm and the corresponding
product estimates for the weighted seminorms,
as described below.}\footnote{\mgd{Recall that
$\bar z=\rho e^{-i\theta}$,
$\bar D_1=\bar z-1$, and
$
\widetilde h(\bar z)
=\frac{1}{\tau}\sum_{k=0}^{\tau-1}
(\tau-k)\bar z^k.
$
The identity
$
W=-\bar D_1\widetilde h(\bar z)
=1-\frac{1}{\tau}\sum_{j=1}^{\tau}\bar z^j
$
therefore gives
$
\|W\|_A
=1+\frac{1}{\tau}\sum_{j=1}^{\tau}\rho^j
\le 2,
$
whereas
$
\|\widetilde h(\bar z)\|_A
=\frac{1}{\tau}\sum_{k=0}^{\tau-1}
(\tau-k)\rho^k
=\frac{\tau+1}{2}$ when $\rho=1$.
Thus, estimating $W$ directly exploits the cancellation
in the product $-\bar D_1\widetilde h(\bar z)$
and avoids the unnecessary factor of order $\tau$
that would arise from estimating the $\|\cdot \|_A$ norm of the two factors
separately.}}}
We shall use the bound
\begin{equation}\label{eq:mixed-young-As}
        \left\|\prod_{j=1}^m F_j\right\|_{A_s}
        \le m^{s-1}\sum_{j=1}^m
        \|F_j\|_{A_s}\prod_{i\ne j}\|F_i\|_{A},
        \qquad s\ge1,
\end{equation}
Indeed, write $
F_j(\theta)=\sum_{k\in\mathbb Z}\widehat F_j(k)e^{ik\theta}$ for $j=1,\ldots,m,
$ where $\widehat F_j(k)$ denotes the $k$-th Fourier coefficient of $F_j$. Since multiplication corresponds to convolution of Fourier
coefficients,
\[
\begin{aligned}
\left\|\prod_{j=1}^m F_j\right\|_{A_s}
&\le
\sum_{k_1,\ldots,k_m}
\left|k_1+\cdots+k_m\right|^s
\prod_{j=1}^m |\widehat F_j(k_j)| \\
&\le
m^{s-1}\sum_{j=1}^m
\sum_{k_1,\ldots,k_m}
|k_j|^s
\prod_{i=1}^m |\widehat F_i(k_i)| 
=
m^{s-1}\sum_{j=1}^m
\|F_j\|_{A_s}
\prod_{i\ne j}\|F_i\|_A,
\end{aligned}
\]
where we used
$
|k_1+\cdots+k_m|^s
\le
m^{s-1}\sum_{j=1}^m |k_j|^s.
$
For $s=0$, we use the convention $\|\cdot\|_{A_0} = \|\cdot\|_A$.

\rev{In the structural Fourier estimates below we separate the $q$-dependent scalar
amplitudes from the $\theta$-dependent Fourier structure. Every factor appearing in
\eqref{eq:T-regrouped-mixed} is a finite sum of terms $a(q)\,G(\theta)$ in which
$a(q)=q^{p}\beta^{b}\rho^{r}$ with integers $p,b\ge0$ and $r\in\mathbb Z$, and $G$ is a
trigonometric polynomial whose coefficients do not depend on $q$ (they may depend on $\tau$
and $M$):}
\begin{revblock}
\begin{align*}
D_1&=\rho\,e^{i\theta}-1, &
D_2&=\rho\,e^{i\theta}-\beta\cdot 1, &
P&=\rho\,(M+\tfrac12)e^{i\theta}-(M-\tfrac12),\\
z^{\pm\tau}&=\rho^{\pm\tau}\,e^{\pm i\tau\theta}, &
W&=1-\sum_{j=1}^{\tau}\rho^{j}\,\frac{e^{-ij\theta}}{\tau}, &
\alpha&=q^{2}\cdot M^{-2},
\end{align*}
\end{revblock}%
\rev{together with the complex conjugates $\bar D_1,\bar D_2,\bar P,\bar z^{\pm\tau}$.
Multiplying out the factors of a term of \eqref{eq:T-regrouped-mixed} produces its
\emph{induced expansion}
\[
F_{\tau,q}(\theta)=\sum_{\nu}a_\nu(q)\,G_\nu(\theta),
\qquad a_\nu(q)=q^{p_\nu}\beta^{b_\nu}\rho^{r_\nu},
\]
in which amplitudes multiply and structures multiply, and no two terms are merged even when
their amplitudes coincide; the induced expansion of a sum is the concatenation of the
expansions of its summands. For such an expansion we define the \emph{structural Wiener
envelopes}
\[
\|F_{\tau,q}\|_{A,\mathrm{str}}:=\sum_\nu\|G_\nu\|_{A},
\qquad
\|F_{\tau,q}\|_{A_s,\mathrm{str}}:=\sum_\nu\|G_\nu\|_{A_s}\qquad(1\le s\le6),
\]
so that $\|F+G\|_{A_s,\mathrm{str}}\le\|F\|_{A_s,\mathrm{str}}+\|G\|_{A_s,\mathrm{str}}$ and
$\|\overline{F}\|_{A_s,\mathrm{str}}=\|F\|_{A_s,\mathrm{str}}$. The scalar amplitudes
$q^{p_\nu}\beta^{b_\nu}\rho^{r_\nu}$ are thus excluded from the envelope; their
$q$-derivatives are estimated separately below. The envelope controls the seminorms
$\|\cdot\|_{A}$, $\|\cdot\|_{A_s}$ defined above only through the amplitudes,
\begin{equation}\label{eq:str-vs-wiener}
\|F_{\tau,q}\|_{A_s}\le\sum_\nu|a_\nu(q)|\,\|G_\nu\|_{A_s}
\le\Bigl(\max_\nu|a_\nu(q)|\Bigr)\|F_{\tau,q}\|_{A_s,\mathrm{str}},
\qquad 0\le s\le 6,
\end{equation}
and the two quantities differ in general: \wubtextb{for $q>0$,} $\|\bar D_1\|_{A}=1+\rho<2=\|\bar D_1\|_{A,\mathrm{str}}$,
whereas $\|z^{-\tau}\|_{A}=\rho^{-\tau}>1=\|z^{-\tau}\|_{A,\mathrm{str}}$. \wubtextb{At $q=0$, both comparisons are equalities.} The second example shows
that the envelope is not an upper bound for the Wiener norm once negative powers of $\rho$
occur, which is why the factors $\rho^{-\tau}$ and $\rho^{-2\tau}$ of
\eqref{eq:T-regrouped-mixed} are carried as scalar amplitudes throughout.}

The elementary structural bounds are, for $1\le s\le6$,
\begin{revblock}
\begin{align*}
        &\|D_i\|_{A,\mathrm{str}},\ \|\bar D_i\|_{A,\mathrm{str}}\le2\quad(i=1,2),
        &&\|D_i\|_{A_s,\mathrm{str}},\ \|\bar D_i\|_{A_s,\mathrm{str}}\le2\quad(i=1,2),\\
        &\|P\|_{A,\mathrm{str}},\ \|\bar P\|_{A,\mathrm{str}}\le2M,
        &&\|P\|_{A_s,\mathrm{str}},\ \|\bar P\|_{A_s,\mathrm{str}}\le2M,\\
        &\|z^{\pm\tau}\|_{A,\mathrm{str}}=\|\bar z^{\pm\tau}\|_{A,\mathrm{str}}=1,
        &&\|z^{\pm\tau}\|_{A_s,\mathrm{str}}=\|\bar z^{\pm\tau}\|_{A_s,\mathrm{str}}=\tau^s\le M^s,\\
        &\|W\|_{A,\mathrm{str}}=2,
        &&\|W\|_{A_s,\mathrm{str}}=\frac1\tau\sum_{j=1}^\tau j^s\le\tau^s\le M^s .
\end{align*}
\end{revblock}%
\rev{The first two lines follow from the expansions above, each of which has two terms with
structures $e^{\pm i\theta}$ or $1$ (for $P$ the structures carry the coefficients $M\pm\frac12$);
the third line uses that $z^{\pm\tau}$ has the single structure $e^{\pm i\tau\theta}$, its
amplitude $\rho^{\pm\tau}$ being excluded; the last line uses the $\tau+1$-term structure ($1$ and
$\tau^{-1}e^{-ij\theta}$ for $1\le j\le\tau$) of $W$.}

\rev{The following lemma shows how \wubtextb{to bound the} structural Wiener envelopes of products of trigonometric polynomials, giving a structural version of \eqref{eq:mixed-young-As}.}
\begin{revblock}
\begin{lemma}[Products of structural envelopes]\label{lem:str-submult}
Let $F_j=\sum_{\nu_j}a^{(j)}_{\nu_j}(q)\,G^{(j)}_{\nu_j}$ be induced expansions of trigonometric polynomials $F_j$ for $j=1,2,\dots,m$, and consider the induced expansion of their product
$\prod_{j=1}^mF_j$. Then
\begin{equation}\label{eq:str-submult}
\Bigl\|\prod_{j=1}^mF_j\Bigr\|_{A,\mathrm{str}}\le\prod_{j=1}^m\|F_j\|_{A,\mathrm{str}},
\qquad
\Bigl\|\prod_{j=1}^mF_j\Bigr\|_{A_s,\mathrm{str}}
\le m^{s-1}\sum_{j=1}^m\|F_j\|_{A_s,\mathrm{str}}\prod_{i\ne j}\|F_i\|_{A,\mathrm{str}},
\qquad 1\le s\le6 .
\end{equation}
\end{lemma}
\begin{proof}
Note that the induced expansion of the product
is indexed by the multi-index $(\nu_1,\dots,\nu_m)$, with amplitudes $\prod_ja^{(j)}_{\nu_j}(q)$
and structures $\prod_jG^{(j)}_{\nu_j}$; the amplitudes do not enter the envelope. For $s=0$,
the Wiener algebra inequality $\|GH\|_A\le\|G\|_A\|H\|_A$ (the Fourier coefficients of $GH$
are the convolution of those of $G$ and $H$; apply Young's $\ell_1$ inequality) gives
\[
\sum_{\nu_1,\dots,\nu_m}\Bigl\|\prod_{j=1}^mG^{(j)}_{\nu_j}\Bigr\|_A
\le\sum_{\nu_1,\dots,\nu_m}\prod_{j=1}^m\|G^{(j)}_{\nu_j}\|_A
=\prod_{j=1}^m\|F_j\|_{A,\mathrm{str}} .
\]
\saa{For $1\leq s\leq 6$, applying \eqref{eq:mixed-young-As}
to each product of structures and summing over the
multi-indices yields the second inequality in
\eqref{eq:str-submult}, since the resulting sums
factor into the corresponding structural Wiener envelopes.}
\end{proof}
\end{revblock}

\rev{Applying Lemma~\ref{lem:str-submult} with $s=0$ to the four terms of
\eqref{eq:T-regrouped-mixed}, where $|D_1|^2=D_1\bar D_1$, $|D_2|^2=D_2\bar D_2$ and
$|P|^2=P\bar P$ count as two factors each, and inserting the elementary bounds, we obtain}
\[
\begin{aligned}
\revc\|\bar D_1Wz^{-\tau}P\bar D_2\|_{A,\mathrm{str}}
&\revc\le \|\bar D_1\|_{A,\mathrm{str}}\|W\|_{A,\mathrm{str}}\|z^{-\tau}\|_{A,\mathrm{str}}
      \|P\|_{A,\mathrm{str}}\|\bar D_2\|_{A,\mathrm{str}}\\
&\revc\le 2\cdot2\cdot1\cdot(2M)\cdot2=16M,\\
\mgc\|M^2|D_1|^2W|D_2|^2\|_{A,\mathrm{str}}
&\mgc\le M^2\cdot 4\cdot2\cdot4=32M^2,\\
\mgc\|\alpha W|P|^2\rho^{-2\tau}\|_{A,\mathrm{str}}
&\mgc\le \frac1{M^2}\cdot2\cdot(2M)^2=8,\\
\mgc\|\alpha M^2W\bar z^{-\tau}\bar P D_1D_2\|_{A,\mathrm{str}}
&\mgc\le 2\cdot1\cdot(2M)\cdot2\cdot2=16M .
\end{aligned}
\]
\rev{Here the scalar amplitudes $q^2$ (from $\alpha=q^2/M^2$), $\rho^{-\tau}$ (from $z^{-\tau}$
and $\bar z^{-\tau}$) and $\rho^{-2\tau}$ do not enter the envelope, whereas the $q$-independent
constants $M^2$ and $M^{-2}$ do. 
\mgd{In the fourth term, $\alpha M^2=q^2$
is treated as a scalar amplitude and therefore
contributes no additional power of $M$
to the structural envelope.}} \mg{Therefore,
\[
\|T_{\tau,q}\|_{A,\mathrm{str}}
\le
16M+32M^2+8+16M
\le 45M^2,
\qquad M\ge3.
\]}

Since $N_{\tau,q}=2\Ree T_{\tau,q}$, the corresponding structural bound for $N_{\tau,q}$ is at most $90M^2$.
For $1\le s\le6$, applying \eqref{eq:mixed-young-As} term by term in \eqref{eq:T-regrouped-mixed} \mgb{using Lemma \ref{lem:str-submult}} gives
\begin{align*}
        \|\bar D_1Wz^{-\tau}P\bar D_2\|_{A_s,{\rm str}}
        &\le 72\,5^{s-1}M^{s+1},\\
        \|M^2|D_1|^2W|D_2|^2\|_{A_s,{\rm str}}
        &\le 80\,5^{s-1}M^{s+2},\\
        \|\alpha W|P|^2\rho^{-2\tau}\|_{A_s,{\rm str}}
        &\le 20\,3^{s-1}M^{s},\\
        \|\alpha M^2W\bar z^{-\tau}\bar P D_1D_2\|_{A_s,{\rm str}}
        &\le 72\,5^{s-1}M^{s+1}.
\end{align*}
\mg{In the third term, after extracting the scalar factor $q^2$, the factor
$M^{-2}$ from $\alpha=q^2/M^2$ remains and is therefore included in the
structural estimate. In the fourth term, this factor is canceled by the
prefactor $M^2$, since $\alpha M^2=q^2$, so no additional power of $M$
remains in the structural envelope.}
\saa{Using $M\geq 3$ and the invariance of the structural
Wiener envelopes under complex conjugation, namely,
\[
\|\overline{T_{\tau,q}}\|_{A_s,\mathrm{str}}
=
\|T_{\tau,q}\|_{A_s,\mathrm{str}},
\qquad 1\leq s\leq 6,
\]
we obtain, from
$N_{\tau,q}=2\operatorname{Re}T_{\tau,q}
=T_{\tau,q}+\overline{T_{\tau,q}}$,}
\begin{equation}\label{eq:mixed-structural-As-main}
        \|N_{\tau,q}\|_{A_s,{\rm str}}
        \le
        \left(256\,5^{s-1}+\frac{40}{9}\,3^{s-1}\right)M^{s+2},
        \qquad 1\le s\le6 .
\end{equation}

It remains to control $q$-derivatives of the scalar factors.  After the regrouping \eqref{eq:T-regrouped-mixed}, every structural monomial is multiplied by a scalar factor
\[
        q^p\beta^b\rho^a,
        \qquad p\in\{0,2\},\qquad 0\le b\le2,
        \qquad |a|\le3M .
\]
\mg{Indeed, the only explicit power of \(q\) arises from
\(\alpha=q^2/M^2\); the \(q\)-dependence carried by \(\rho\) is tracked separately through the factor \(\rho^a\). Hence \(p\in\{0,2\}\). The only source of \(\beta\) is \(D_2=\rho e^{i\theta}-\beta\) and its conjugate, and at most two such factors appear in any term, so \(0\le b\le2\). Finally, \(W\) contributes powers \(\rho^j\) with \(0\le j\le\tau\), while \(z^{-\tau}\) and \(\rho^{-2\tau}\) contribute \(\rho^{-\tau}\) and \(\rho^{-2\tau}\), respectively. The remaining factors \(D_1,D_2,\) and \(P\) each contribute at most one additional power of \(\rho\). Collecting these contributions yields the uniform bound
$
|a|\le 3M.
$
}
\mg{We next record uniform bounds for the first four $q$-derivatives of these
scalar factors:
\begin{align}
        \left|\partial_q^i q^p\right|
        &\le A_i,
        &
        (A_0,A_1,A_2,A_3,A_4)
        &=(1,1,2,0,0), \label{eq:A-bounds}\\
        \left|\partial_q^j\beta^b\right|
        &\le B_j,
        &
        (B_0,B_1,B_2,B_3,B_4)
        &=\left(1,\frac43,\frac89,0,0\right), \label{eq:B-bounds}\\
        \left|\partial_q^k\rho^a\right|
        &\le R_k,
        &
        (R_0,R_1,R_2,R_3,R_4)
        &=(4,8,18,44,116). \label{eq:R-bounds}
\end{align}
The first two estimates are immediate from
\[
p\in\{0,2\},\qquad
b\in\{0,1,2\},\qquad
\beta=1-\frac{2q}{M},
\]
using $q\le1/2$ and $M\ge3$.  Indeed,
$
|\partial_q q^2|=2q\le1$ and
$|\partial_q^2 q^2|=2,
$
while
\[
|\partial_q\beta^b|
\le \frac{2b}{M}\le\frac43,
\qquad
|\partial_q^2\beta^b|
\le \frac{8}{M^2}\le\frac89;
\]
all higher derivatives appearing in \eqref{eq:A-bounds}--\eqref{eq:B-bounds}
vanish.} \mg{For the powers of $\rho$, since
$
\rho=1-\frac{cq}{M},
$
we have, for $0\le k\le4$,
\[
\partial_q^k\rho^a
=
(a)_k\left(-\frac{c}{M}\right)^k\rho^{a-k},
\]
where $(a)_k=a(a-1)\cdots(a-k+1)$ with the convention $(a)_0=1$ for $k=0$.  Since $|a|\le3M$ with $a$ potentially negative,
$
|(a)_k|
\le
(3M)(3M+1)\cdots(3M+k-1).
$
Moreover, the inequalities provided in \eqref{eq:rho-sharp-envelope} and
\eqref{eq:rho-lower-closed} imply
\begin{equation}\label{ineq-rho-negative-powers-ub}
\rho^{a-k}\le4,
\qquad
|a|\le3M,\quad 0\le k\le4.
\end{equation}
because for this range of $a$ and $k$, using %
\saa{$0< \rho \leq 1$,}
$$\rho^{a-k} \leq \rho^{-3M}\rho^{-4} \leq \exp\!\left(\frac{3M c}{2M+1-c}\right) \cdot \left(\frac{6M+3}{6M+1}\right)^4 < \exp\!\left(\frac{6M }{6M+1}\right) \cdot \left(\frac{6M+3}{6M+1}\right)^4, $$
where we used $c<2/3$. For $M\geq 3$, the right-hand side is maximized at $M=3$ where it takes the value 
$\exp({18/19})\cdot (\frac{21}{19})^4 \approx 3.85 < 4$. This justifies the estimate \eqref{ineq-rho-negative-powers-ub}. Then, using $c<2/3$ again, we obtain
\begin{equation}
|\partial_q^k\rho^a|
\le
4(3M)(3M+1)\cdots(3M+k-1)
\left(\frac{2}{3M}\right)^k,
\end{equation}
for $k\geq 1$ which, for $k=1,2,3,4$, yields the bounds in \eqref{eq:R-bounds}. For $k=0$,  \eqref{ineq-rho-negative-powers-ub} provides directly the estimate $|\rho^a| < 4$ %
given in \eqref{eq:R-bounds}.}
By the three-factor Leibniz rule,
\[
\partial_q^r\!\left(q^p\beta^b\rho^a\right)
=
\sum_{i+j+k=r}
\frac{r!}{i!\,j!\,k!}
\bigl(\partial_q^i q^p\bigr)
\bigl(\partial_q^j\beta^b\bigr)
\bigl(\partial_q^k\rho^a\bigr).
\]
Hence, using the derivative envelopes
\[
|\partial_q^i q^p|\le A_i,
\qquad
|\partial_q^j\beta^b|\le B_j,
\qquad
|\partial_q^k\rho^a|\le R_k,
\]
we obtain
\[
\left|
\partial_q^r\!\left(q^p\beta^b\rho^a\right)
\right|
\le
\sum_{i+j+k=r}
\frac{r!}{i!\,j!\,k!}A_iB_jR_k.
\]
For $r=0,1,2,3,4$, the resulting bounds are
\[
\begin{array}{c|ccccc}
r
&0&1&2&3&4\\
\hline
\displaystyle
\sum_{i+j+k=r}
\frac{r!}{i!\,j!\,k!}A_iB_jR_k
&
4
&
\displaystyle\frac{52}{3}
&
\displaystyle\frac{698}{9}
&
346
&
\displaystyle\frac{4532}{3}
\end{array}
\]
and therefore
\[
\left|
\partial_q^r\!\left(q^p\beta^b\rho^a\right)
\right|
\le
\frac{4532}{3}
<1511,
\qquad
0\le r\le4.
\]
These values follow by direct substitution of the constants $A_i, B_j, R_k$ from \eqref{eq:A-bounds}--\eqref{eq:R-bounds}
into the multinomial Leibniz bound above.

\mg{The estimates separate the $q$-dependence from the Fourier dependence in
$\theta$: the former is controlled coefficientwise by the uniform
bound $1511$ on the scalar $q$-derivatives, whereas the latter
is controlled by the $\ell_1$-based structural Wiener envelopes obtained
above.  Thus, for $0\le r\le4$,
\[
\|\partial_q^r N_{\tau,q}\|_{A_s}
\le
1511\,\|N_{\tau,q}\|_{A_s,{\rm str}},
\]
and substituting the corresponding structural bounds for
$s=0,2,4,6$ gives the following estimates:}
\begin{align*}
        \|\partial_q^rN_{\tau,q}\|_{A}
        &\le 1511\cdot90M^2<2^{18}M^2,\\
        \|\partial_q^rN_{\tau,q}\|_{A_2}
        &\le1511\left(256\cdot5+\frac{40}{9}\cdot3\right)M^4<2^{21}M^4,\\
        \|\partial_q^rN_{\tau,q}\|_{A_4}
        &\le1511\left(256\cdot5^3+\frac{40}{9}\cdot3^3\right)M^6<2^{26}M^6,\\
        \|\partial_q^rN_{\tau,q}\|_{A_6}
        &\le1511\left(256\cdot5^5+\frac{40}{9}\cdot3^5\right)M^8<2^{31}M^8 .
\end{align*}
Since $\|\partial_\theta^sF\|_\infty\le\|F\|_{A_s}$, these inequalities prove \eqref{eq:closed-fdi-envelope-0}--\eqref{eq:closed-fdi-envelope-6}.
\end{proof}
Let the leading Taylor constants from Lemma~\ref{lem:coeffs} be, with $\Lambda$ written for the constant $\lambda$ there,
\begin{equation}\label{eq:closed-leading-coeffs}
        \mgd{B_{\mathrm{cub}}}:=\frac{(\tau+1)c(1-c)^2}{M^3},
        \qquad
        \Lambda:=\frac{(\tau+1)(2-3c)}{M},
        \qquad
        A_0:=\frac{M(\tau+1)}{3}.
\end{equation}
All three are positive for $0<c<2/3$.
With $n:=\tau$, define
\begin{align}
        R_{\rm in}(n)&:=\frac{(n+1)(2n+1)(31n^2-9n+18)}{30n^2},\label{eq:Rin-closed}\\
        a_n&:=\frac{18n^2+18n+4}{n},\label{eq:an-closed}\\
        \mgd{g_n^{\mathrm{out}}}&:=\frac{64n^5+672n^4+622n^3+21n^2-93n-18}{6(3n+1)(3n+2)(32n^2-3)},\label{eq:Deltan-closed}\\
        R_{\rm out}(n)&:=2a_n \mgd{g_n^{\mathrm{out}}},
        \qquad
        R_{\mathrm{min}}(n):=\min\{R_{\rm in}(n),R_{\rm out}(n)\}.\label{eq:Rstar-closed}
\end{align}
The proof of Proposition~\ref{prop:rpositive} gives
\begin{equation}\label{eq:r-lower-closed}
        r_\tau(\cos\theta)\ge R_{\mathrm{min}}(\tau),
        \qquad 0\le\theta\le\pi .
\end{equation}
On the other hand, for \(\theta\neq0\), taking $q\to 0$ yields
$$
\Phi_{\tau,0}(\theta)
=
\frac{N_{\tau,0}(\theta)}
{M^2|e^{i\theta}-1|^4}
$$
while Lemma \ref{lem:rpoly} independently gives
$
\Phi_{\tau,0}(\theta)
=
-\frac{r_\tau(\cos\theta)}{4M^2} 
$
for every $\theta$. %
Equating these expressions yields
\[
        -N_{\tau,0}(\theta) = \frac14|e^{i\theta}-1|^4r_\tau(\cos\theta)
        =4\sin^4(\theta/2)r_\tau(\cos\theta)
\]
for $\theta \neq 0$. Since both sides are continuous in \(\theta\), the identity also holds at \(\theta=0\) with $N_{\tau,0}(0)=0$ and we obtain $N_{\tau,0}(\theta)=\cO(\theta^4)$. In fact, using the bound $\sin u\ge2u/\pi$ on $[0,\pi/2]$, we obtain the global quartic lower bound
\begin{equation}\label{eq:limiting-quartic-floor}
        -N_{\tau,0}(\theta)
        \ge \Gamma_\tau\theta^4,
        \qquad
        \Gamma_\tau:=\frac{4R_{\mathrm{min}}(\tau)}{\pi^4},
        \qquad 0\le\theta\le\pi .
\end{equation}

Next, we combine these ingredients into an explicit threshold for $q$ for the corresponding strict FDI to hold.
\begin{definition}[closed-form FDI threshold]\label{def:qfdi-explicit}
Set
\begin{equation}\label{eq:qfdi-explicit}
        q_{\rm fdi}(\tau,c):=
        \min\left\{
        \bar q,\,
        \frac{12\mgd{B_{\mathrm{cub}}}}{K_0(\tau)},\,
        \frac{2\Lambda}{K_2(\tau)},\,
        \frac{12\Gamma_\tau}{K_4(\tau)}
        \right\}.
\end{equation}
Here, $\mgd{B_{\mathrm{cub}}}$ and $\Lambda$ are defined in \eqref{eq:closed-leading-coeffs}, $\Gamma_\tau$ is defined in \eqref{eq:limiting-quartic-floor}, $K_0(\tau), K_2(\tau)$ and $K_4(\tau)$ are given in Lemma \ref{lem:closed-fdi-envelope} and $\bar{q}$ is \mgd{defined in \eqref{eq:closed-basic-constants}}.
\end{definition}

\begin{proposition}[closed-form FDI threshold]\label{prop:explicit-fdi-threshold}
For every $\tau\ge1$ and $c\in(0,2/3)$,
\begin{equation}\label{eq:closed-fdi-result}
        \Phi_{\tau,q}(\theta)<0,
        \qquad
        0<q\le q_{\rm fdi}(\tau,c),
        \quad
        \theta\in[0,2\pi),
\end{equation}
with $q_{\rm fdi}(\tau,c)$ as in 
\eqref{eq:qfdi-explicit}. \mgd{Consequently, hypothesis \textbf{(H2)}
of Lemma~\ref{lemm:ratecert} holds whenever
$0<q\le q_{\rm fdi}(\tau,c)$.} Moreover the numerator has the explicit global margin
\begin{equation}\label{eq:N-margin-closed}
        -N_{\tau,q}(\theta)\ge \frac{\mgd{B_{\mathrm{cub}}}}{2}q^3,
        \qquad 0\le\theta\le\pi .
\end{equation}
where $\mgd{B_{\mathrm{cub}}}$ is defined in \eqref{eq:closed-leading-coeffs}.
\end{proposition}
\begin{proof}
By evenness it suffices to consider $0\le\theta\le\pi$.  Write
\[
        \mgd{\mathcal{E}_q(\theta)}:=N_{\tau,q}(\theta)-N_{\tau,0}(\theta).
\]
\saa{Recall that $N_{\tau,q}$ is even in $\theta$ and
admits the expansion \eqref{eq:tay6}. Moreover,
Lemma~\ref{lem:coeffs} gives
$a_0(0)=a_2(0)=0$.
Thus, the constant and quadratic Taylor coefficients
of $\mathcal{E}_q$ are $a_0(q)$ and $a_2(q)$,
respectively.} Taylor's theorem with integral
remainder therefore gives
\begin{equation}\label{eq:Delta-fourth-order}
        \mgd{\mathcal{E}_q(\theta)}
        =a_0(q)+a_2(q)\theta^2+
        \frac{1}{6}\int_0^\theta(\theta-t)^3
        \partial_t^4\mgd{\mathcal{E}_q(t)}\,\dd t .
\end{equation}
Since
\[
        \partial_t^4\mgd{\mathcal{E}_q(t)} =  \partial_t^4 N_{\tau,q}(t)-  \partial_t^4 N_{\tau,0}(\mgd{t})
        =\int_0^q\partial_s\partial_t^4N_{\tau,s}(t)\,\dd s,
\]
\mg{Lemma~\ref{lem:closed-fdi-envelope}, with \(r=1\), gives
$$
\left|\partial_s\partial_t^4N_{\tau,s}(t)\right|
\le K_4(\tau),
\qquad 0\le s\le q\le\bar q,\quad 0\le t\le\pi.
$$
Therefore,
$$
\left|\partial_t^4\mgd{\mathcal{E}_q(t)}\right|
\le
\int_0^q
\left|\partial_s\partial_t^4N_{\tau,s}(t)\right|\,\dd s
\le qK_4(\tau).
$$
}
Integrating, this yields the bound
\begin{equation}\label{eq:Delta-fourth-remainder}
        \left|\frac{1}{6}\int_0^\theta(\theta-t)^3
        \partial_t^4\mgd{\mathcal{E}_q(t)}\,\dd t\right|
        \le \frac{K_4(\tau)}{24}q\theta^4,
        \qquad 0\le\theta\le\pi .
\end{equation}
\mgd{We control the coefficient perturbations directly by Taylor's theorem
in $q$. By Lemma~\ref{lem:coeffs}, analyticity and the coefficients 
\eqref{eq:coeff-asymp} that appear in \eqref{eq:taylor-N} imply
\[
a_0(0)=a_0'(0)=a_0''(0)=0,\qquad a_0'''(0)=-6\mgd{B_{\mathrm{cub}}},
\]
and $a_2(0)=0$, $a_2'(0)=-\Lambda$. Moreover,
$a_0^{(4)}(q)=\partial_q^4N_{\tau,q}(0)$, whereas
$a_2''(q)=\tfrac12\partial_q^2\partial_\theta^2N_{\tau,q}(0)$.
Lemma~\ref{lem:closed-fdi-envelope} therefore yields
$|a_0^{(4)}(q)|\le K_0(\tau)$ and
$|a_2''(q)|\le K_2(\tau)/2$ for $0\le q\le\bar q$.
Integral-form Taylor remainders now give
\[
\begin{aligned}
a_0(q)&=-\mgd{B_{\mathrm{cub}}}q^3+
\frac{1}{3!}\int_0^q(q-t)^3a_0^{(4)}(t)\,\dd t,\\
a_2(q)&=-\Lambda q+\int_0^q(q-t)a_2''(t)\,\dd t.
\end{aligned}
\]
Consequently,
\begin{equation}\label{eq:coeff-perturb-explicit}
a_0(q)\le-\mgd{B_{\mathrm{cub}}}q^3+\frac{K_0(\tau)}{24}q^4,
\qquad
a_2(q)\le-\Lambda q+\frac{K_2(\tau)}{4}q^2 .
\end{equation}
The second and third entries in \eqref{eq:qfdi-explicit} give
\begin{equation}\label{eq:coeff-negative-explicit}
        a_0(q)\le-\frac{\mgd{B_{\mathrm{cub}}}}{2}q^3,
        \qquad
        a_2(q)\le-\frac{\Lambda}{2}q .
\end{equation}}
Combining \eqref{eq:limiting-quartic-floor}, \eqref{eq:Delta-fourth-order},
\eqref{eq:Delta-fourth-remainder}, and \eqref{eq:coeff-negative-explicit}, we obtain,
for every \(0\le\theta\le\pi\),
\[
\begin{aligned}
        N_{\tau,q}(\theta)
        &\le
        -\Gamma_\tau\theta^4
        -\frac{\mgd{B_{\mathrm{cub}}}}{2}q^3
        -\frac\Lambda2q\theta^2
        +\frac{\saa{K_4(\tau)}q}{24}\theta^4 .
\end{aligned}
\]
The fourth threshold entry \(q\le12\Gamma_\tau/K_4(\tau)\) 
helps to control the last positive term, and implies
\[
        \frac{\saa{K_4(\tau)}q}{24}\theta^4\le\frac{\Gamma_\tau}{2}\theta^4.
\]
Therefore
\begin{equation}\label{eq:global-margin-explicit}
        N_{\tau,q}(\theta)
        \le
        -\frac{\mgd{B_{\mathrm{cub}}}}{2}q^3
        -\frac\Lambda2q\theta^2
        -\frac{\Gamma_\tau}{2}\theta^4
        \le -\frac{\mgd{B_{\mathrm{cub}}}}{2}q^3,
        \qquad 0\le\theta\le\pi .
\end{equation}
This proves \eqref{eq:N-margin-closed}.
\saa{Since $0<q\le q_{\rm fdi}(\tau,c)\le\bar q$,
we have $0<\beta<\rho<1$, and hence
$D_1=\rho e^{i\theta}-1$ and
$D_2=\rho e^{i\theta}-\beta$ are nonzero
for every $\theta$.
Consequently, the denominator in
\eqref{eq:Ndef} is strictly positive.
Together with \eqref{eq:global-margin-explicit}
and the evenness and periodicity of
$N_{\tau,q}$, this establishes
\eqref{eq:closed-fdi-result} and completes
the proof.}
\end{proof}
\mgd{
In the following result, we derive upper bounds
on the denominator $M^2|D_1|^2|D_2|^2$
of $\Phi_{\tau,q}$ in \eqref{eq:Ndef}.
Combining these estimates with the bound
on the numerator $N_{\tau,q}(\theta)$ in
\eqref{eq:global-margin-explicit}, we show that
the strict frequency-domain inequality
$\Phi_{\tau,q}(\theta)<0$ holds with a
uniform negative margin whose magnitude
depends only on $c$:
}
\begin{corollary}[uniform FDI margin]\label{cor:Phi-uniform-margin}
Assume \(0<q\le q_{\rm fdi}(\tau,c)\).  Then
\begin{equation}\label{eq:Phi-uniform-margin}
        -\Phi_{\tau,q}(\theta)\ge \eta_c^\Phi,
        \qquad 0\le\theta\le\pi,
\end{equation}
where
\begin{equation}\label{eq:etaPhi-def}
        \eta_c^\Phi:=
        \min\left\{\frac{c(1-c)^2}{72},\,\frac{1}{54\pi^4}\right\}.
\end{equation}
Consequently, \(\eta_c^\Phi\ge c(1-c)^2/1000\).
\end{corollary}

\begin{proof}
By the first inequality in \eqref{eq:global-margin-explicit},
\begin{equation}\label{eq:N-two-scale-margin}
        -N_{\tau,q}(\theta)\ge \frac{\mgd{B_{\mathrm{cub}}}}{2}q^3+\frac{\Gamma_\tau}{2}\theta^4 .
\end{equation}
We first derive a uniform lower bound for the quartic constant.  From
\eqref{eq:Rin-closed}--\eqref{eq:Rstar-closed}, with \(n=\tau\) and \(M=2n+1\),
\[
        R_{\rm in}(n)-\frac{M^2}{3}
        =\frac{(2n+1)(11n^3+12n^2+9n+18)}{30n^2}>0,
\]
and
\[
        R_{\rm out}(n)-\frac{M^2}{3}
        =\frac{(4n+3)(304n^3+78n^2-45n-12)}{3n(32n^2-3)}>0 .
\]
Thus, \(R_{\mathrm{min}}(\tau)\ge M^2/3\), and hence
\begin{equation}\label{eq:Gamma-uniform-lower}
        \Gamma_\tau=\frac{4\mgd{R_{\mathrm{min}}}(\tau)}{\pi^4}
        \ge \frac{4M^2}{3\pi^4} .
\end{equation}

We now bound the denominator in \eqref{eq:Ndef}. %
We have $D_1 = (\rho-1) + \rho (e^{i\theta}-1)$, $D_2 = (\rho-\beta) + \rho(e^{i\theta}-1)$ and $|e^{i\theta}-1|\leq \theta$ for $\theta\geq 0$. 
Therefore, for \(0\le\theta\le q/M\), 
\[
        |D_1|\le \frac{cq}{M}+\theta\le \frac{2q}{M},
        \qquad
        |D_2|\le \frac{(2-c)q}{M}+\theta\le \frac{3q}{M},
\]
so \(M^2|D_1|^2|D_2|^2\le36q^4/M^2\).  Using the first term in
\eqref{eq:N-two-scale-margin} and \(q\le\bar q=M/(2M+1)\),
\[
        -\Phi_{\tau,q}(\theta)
        \ge \frac{\mgd{B_{\mathrm{cub}}}q^3/2}{36q^4/M^2}
        \ge \frac{\mgd{B_{\mathrm{cub}}}M^2}{72\bar q}
        =\frac{c(1-c)^2(M+1)(2M+1)}{144M^2}
        \ge \frac{c(1-c)^2}{72},
\]
where we used the definition \eqref{eq:closed-leading-coeffs} of $\mgd{B_{\mathrm{cub}}}$. For \(q/M\le\theta\le\pi\),
\[
        |D_1|\le \frac{cq}{M}+\theta\le2\theta,
        \qquad
        |D_2|\le \frac{(2-c)q}{M}+\theta\le3\theta,
\]
and hence \(M^2|D_1|^2|D_2|^2\le36M^2\theta^4\).  Using the second term in
\eqref{eq:N-two-scale-margin} and \eqref{eq:Gamma-uniform-lower},
\[
        -\Phi_{\tau,q}(\theta)
        \ge \frac{\Gamma_\tau\theta^4/2}{36M^2\theta^4}
        \ge \frac{1}{54\pi^4}.
\]
The two regions prove \eqref{eq:Phi-uniform-margin}.  
\saa{
Finally, since
$c(1-c)^2\leq 4/27$ for $0<c<1$ and
$\frac{1}{54\pi^4}
>\frac{4}{27000}
\geq\frac{c(1-c)^2}{1000}$,
both terms in the minimum defining
$\eta_c^\Phi$ are bounded below by
$c(1-c)^2/1000$.}
\end{proof}
\section{Midpoint stability at finite condition number}\label{sec:midpoint}
\mgd{Throughout this appendix, fix an integer $\tau\ge1$ and $c\in(0,2/3)$, and use $M=2\tau+1$, $\beta=1-2q/M$, and $\rho=1-cq/M$ as in \eqref{eq:params}.}

\mgd{Section~\ref{sec:stability} verifies \textbf{(H1)} qualitatively for sufficiently small $q$. We now provide an explicit admissible threshold ensuring that}
every root of the midpoint characteristic polynomial $P_{\tau,q}$ \mgd{lies} in $\{|z|<\rho\}$ where
 \begin{equation}\label{eq:Pmid}
 P_{\tau,q}(z)=z^\tau(z-1)(z-\beta)+\frac{1+q^{-2}}{2}\frac{q^2}{M^2}
 \bigl((M+\tfrac12)z-(M-\tfrac12)\bigr).
 \end{equation}\saa{To derive an explicit threshold, we consider the
coefficientwise limit of $P_{\tau,q}$ as $q\downarrow 0$:}
\begin{equation}\label{eq:Qfloor-poly}
        Q_\tau(z)=z^\tau(z-1)^2+\frac{1}{2M^2}\bigl((M+\tfrac12)z-(M-\tfrac12)\bigr),
\end{equation}
\mgd{where we used $(1+q^{-2})q^2=1+q^2\to 1$ for $q\downarrow 0$. We will} analyze the properties of $Q_\tau$ on the unit circle. We first introduce some elementary inequalities. 
\mgd{
\begin{lemma}[Elementary trigonometric inequalities]
\label{lem:midpoint-elementary-envelopes}
\saa{For all $a,b\in\mathbb{R}$ and $t\in[0,1)$,
the following inequalities hold:}
\textup{(i)} $\cos t\ge 1-t^2/2$ and $\sin t\le t$;
\textup{(ii)} $|\sin a-\sin b|\le |a-b|$ and
$|\cos a-\cos b|\le |a-b|$;
\textup{(iii)} $\arcsin t-t\le t^3/\sqrt{1-t^2}$.
\end{lemma}
}
\begin{proof}
\mgd{The sine bound in (i) and the bounds in (ii) follow by integrating the derivatives and using $|\sin t|,|\cos t|\le1$. The cosine lower bound in (i) follows from $1-\cos t=\int_0^t\sin u\,du\le t^2/2$. For (iii)},
\[
        \arcsin t=\int_0^t\frac{ds}{\sqrt{1-s^2}}
        \le \frac{t}{\sqrt{1-t^2}},
\]
so
\[
        \arcsin t-t
        \le t\left(\frac1{\sqrt{1-t^2}}-1\right)
        =\frac{t(1-\sqrt{1-t^2})}{\sqrt{1-t^2}}
        =\frac{t^3}{\sqrt{1-t^2}(1+\sqrt{1-t^2})}
        \le \frac{t^3}{\sqrt{1-t^2}}.
\]
\end{proof}
\begin{lemma}[Lower bounding the \mgd{modulus} of $Q_\tau$ on the unit circle]\label{lem:midpoint-exact-modulus}
Let $z=e^{2iu}$ with $u\in(0,\pi)$.  Then
\begin{equation}\label{eq:Q-exact-modulus}
        |Q_\tau(e^{2iu})|
        =\left|-R e^{iMu}+P+iT\right|,
        \qquad
        R=4\sin^2u,
        \quad
        P=\frac{\cos u}{2M^2},
        \quad
        T=\frac{\sin u}{M}.
\end{equation}
Consequently
\begin{align}
        |Q_\tau(e^{2iu})|
        &\ge \left|R-\sqrt{P^2+T^2}\right|,\label{eq:midpoint-amplitude-bound}\\
        |Q_\tau(e^{2iu})|
        &\ge |T\cos(Mu)-P\sin(Mu)|
        =\frac1{M^2}\left|M\sin u\cos(Mu)-\frac12\cos u\sin(Mu)\right|.
        \label{eq:midpoint-rotated-bound}
\end{align}
\end{lemma}

\begin{proof}
Since $z=e^{2iu}$,
\[
        z^\tau(z-1)^2
        =e^{2i\tau u}\bigl(2ie^{iu}\sin u\bigr)^2
        =-4\sin^2u\,e^{i(M+1)u}
        =e^{iu}(-R e^{iMu}).
\]
Moreover
\[
        (M+\tfrac12)e^{2iu}-(M-\tfrac12)
        =e^{iu}\bigl(\cos u+2iM\sin u\bigr).
\]
Multiplication by the unimodular factor $e^{iu}$ does not change the modulus, which proves \eqref{eq:Q-exact-modulus}.  
\saa{The amplitude bound \eqref{eq:midpoint-amplitude-bound}
follows by applying the reverse triangle inequality
to the two summands $-Re^{iMu}$ and $P+iT$
in \eqref{eq:Q-exact-modulus}.}
For the bound \eqref{eq:midpoint-rotated-bound}, multiply the expression inside the modulus in \eqref{eq:Q-exact-modulus} by $e^{-iMu}$ and take imaginary parts to find a lower bound on the modulus.
\end{proof}
\begin{proposition}[Lower bound on the unit circle for $Q_\tau$]\label{prop:midpoint-unit-floor}
For every integer $\tau\ge1$,
\begin{equation}\label{eq:midpoint-unit-floor}
        \min_{|z|=1}|Q_\tau(z)|\ge \frac{1}{10M^2}.
\end{equation}
\end{proposition}

\begin{proof}
\saa{At $z=1$, we have
$Q_\tau(1)=1/(2M^2)>1/(10M^2)$.
Since $Q_\tau$ has real coefficients,
$Q_\tau(\bar z)=\overline{Q_\tau(z)}$.
Thus, it suffices to establish
\eqref{eq:midpoint-unit-floor} for
$z=e^{2iu}$ with $u\in(0,\pi/2]$.}
Indeed, if $u'=\pi-u$, then $M=2\tau+1$ is an odd number and the two lower bounds \eqref{eq:midpoint-amplitude-bound}--\eqref{eq:midpoint-rotated-bound} are invariant under $u\mapsto u'$ up to a sign inside the absolute value.
Set
\[
        x:=M\sin u.
\]
We split $x\in(0,M]$ into three regions.

If $x\le3/10$, then $\sin u\le 1/10$, hence $\cos u\ge\sqrt{1-10^{-2}}\ge99/100$.  From \eqref{eq:midpoint-amplitude-bound},
\begin{equation}
        M^2|Q_\tau(e^{2iu})|
        \ge \left|4x^2-\sqrt{\frac{\cos^2u}{4}+x^2}\right|.\label{ineq-lower-bound-Q-tau}
\end{equation}
Here $\sqrt{\cos^2u/4+x^2}\ge (\cos u)/2\ge99/200$ while $4x^2\le36/100$, so the magnitude of the quantity inside the absolute value is at least
\[
        \frac{99}{200}-\frac{36}{100}
        =\frac{27}{200}.
\]
Thus $M^2|Q_\tau(e^{2iu})|\ge27/200>1/10$.

If $x\ge1/2$, then by \eqref{ineq-lower-bound-Q-tau}, %
\saa{
\[
M^2|Q_\tau(e^{2iu})|
\ge 4x^2-\sqrt{\frac14+x^2}
=:h(x),
\]
where we used $\cos^2u\le1$.
For $x\ge1/2$, we have
\[
h'(x)
=8x-\frac{x}{\sqrt{1/4+x^2}}
\ge8x-1>0.
\]
Consequently,}
\[
        M^2|Q_\tau(e^{2iu})|
        \ge h(x)\ge h(1/2)=1-\sqrt{1/2}>\frac{29}{100}>\frac1{10},
\]
where we used $\sqrt{1/2}<71/100$.

It remains to consider $3/10\le x\le1/2$.  Define
\[
        D(x):=x\cos x-\frac12\sin x.
\]
By Lemma~\ref{lem:midpoint-elementary-envelopes}, on this interval,
\[
        D'(x)=\frac12\cos x-x\sin x
        \ge \frac12\left(1-\frac{x^2}{2}\right)-x^2
        \ge \frac{3}{16}>0.
\]
Thus $D$ is increasing, and again by the elementary trigonometric bounds given in Lemma~\ref{lem:midpoint-elementary-envelopes},
\[
        D(x)\ge D(3/10)
        \ge \frac{3}{10} \left(1- \frac{(3/10)^2}{2}\right)-\frac12\frac{3}{10} %
        =\frac{273}{2000}.
\]
Now $u=\arcsin(x/M)$ and $M\ge3$.  Since $x/M\le1/6$, the rational square-root bound
$\sqrt{35}/6>591/600$ and Lemma~\ref{lem:midpoint-elementary-envelopes} imply
\[
        0\le Mu-x
        \le \frac{x (x/M)^2}{\sqrt{1-(x/M)^2}}
       \mathrel{\wubtext{\le}} \frac{1/2 \cdot (1/6)^2}{591/600}
       = \frac{25}{1773}.
\]
Also
\[
        0\le1-\cos u
        =\frac{\sin^2u}{1+\cos u}
        = \frac{\sin^2u}{1+\sqrt{1-\sin^2(u)}}
        \le \frac{(1/6)^2}{1+591/600}
        =\frac{50}{3573},
\]
where we used $\sin(u)\leq 1/6$ and $\sqrt{35}/6>591/600$. %
\saa{Using the Lipschitz bounds for sine and cosine
in Lemma~\ref{lem:midpoint-elementary-envelopes},
together with $x\le 1/2$ and $Mu\ge x$,
we obtain}
\begin{align*}
&\left|M\sin u\cos(Mu)-\frac12\cos u\sin(Mu)-D(x)\right|\\
&\qquad= \left| x\bigl(\cos(Mu)-\cos x\bigr)
-\frac12\bigl((\cos u-1)\sin(Mu)+\bigl(\sin(Mu)-\sin x\bigr)\bigr) \right|\\
&\qquad \leq x (Mu - x) +\frac12\left((1-\cos u)|\sin(Mu)|+ Mu-x\right)\\
&\qquad\le  \frac{1}{2} \frac{25}{1773} + \frac{1}{2} \frac{50}{3573} + \frac{1}{2}\frac{25}{1773}
= \frac{25}{1773}+\frac{25}{3573}<\frac{3}{140}.
\end{align*}
Therefore, in the middle region when $3/10\le x\le1/2$, by \eqref{eq:midpoint-rotated-bound} and the triangle inequality, 
\begin{eqnarray}
        M^2|Q_\tau(e^{2iu})|
        &\ge& \left|D(x)\right| - \left|M\sin u\cos(Mu)-\frac12\cos u\sin(Mu)-D(x)\right| \\
        &\mathrel{\wubtext{\ge}}&
        \frac{273}{2000}-\frac{3}{140}
        =\frac{1611}{14000}>\frac1{10}.
\end{eqnarray}
\saa{These three regions cover $x\in(0,M]$,
corresponding to $u\in(0,\pi/2]$.
Together with the symmetry argument and
the case $z=1$, this establishes
\eqref{eq:midpoint-unit-floor}.}
\end{proof}

\begin{lemma}[Spectral radius bound for the limiting midpoint polynomial $Q_\tau$]\label{lem:midpoint-poly-gap}
Every root $\zeta$ of $Q_\tau$ satisfies
\begin{equation}\label{eq:midpoint-poly-gap}
        |\zeta|\le1-\frac{1}{30M^3}.
\end{equation}
\end{lemma}

\begin{proof}
\mg{We first argue that every root lies in the open unit disk, using the quadratic analysis in the main text \wubtextb{together with Proposition~\ref{prop:midpoint-unit-floor}}. Indeed, \mgd{the midpoint linear recursion is delayed GMM applied to the quadratic of curvature}%
 $$
 m=\frac12(1+q^{-2})\in[1,q^{-2}].
 $$
 Fix \(c\in(0,2/3)\), and choose any \(c'\in(c,1)\). Applying Theorem~\ref{thm:q-alldelay} with rate constant \(c'/M<1/M\), we obtain, for all sufficiently small \(q\),
 $$
 \max\{|z|:P_{\tau,q}(z)=0\}
 \le 1-\frac{c'}{M}q
 <
 1-\frac{c}{M}q
 =\rho,
 $$
 where 
 \begin{equation}\label{eq:Pmid-recalled}
P_{\tau,q}(z)=z^\tau(z-1)(z-\beta)+\frac{1+q^{-2}}{2}\frac{q^2}{M^2}
 \bigl((M+\tfrac12)z-(M-\tfrac12)\bigr).
 \end{equation}
 \saa{Thus, the midpoint-pole hypothesis
\textbf{(H1)} holds for all sufficiently
small $q>0$. Since $P_{\tau,q}$ converges
coefficientwise to $Q_\tau$ as $q\downarrow0$,
continuity of polynomial roots implies
that every root of $Q_\tau$ lies in the
closed unit disk. Proposition~\ref{prop:midpoint-unit-floor}
excludes roots on the unit circle.
Therefore, $Q_\tau$ is Schur stable.}
 }
 
We next quantify the spectral radius of $Q_\tau$. Let $\zeta=re^{i\theta}$ be a root.  On the segment from $\zeta$ to $e^{i\theta}$, which is contained in the closed unit disk,
\[
\left|\frac{d}{dw}\bigl(w^\tau(w-1)^2\bigr)\right|
\le \tau|w-1|^2+2|w-1|
\le4\tau+4.
\]
The derivative of the affine term in $Q_\tau$ is $(M+1/2)/(2M^2)$. Therefore
\[
        \left| \frac{d}{dw} Q_\tau(w) \right|
        \le 4\tau+4+\frac{M+1/2}{2M^2}
        \le 3M.
\]
By Proposition~\ref{prop:midpoint-unit-floor} and the fact that $Q_\tau(\zeta)=0$,
\[
        \frac{1}{10M^2}
        \le |Q_\tau(e^{i\theta})-Q_\tau(\zeta)|
    \le 3M |e^{i\theta} - re^{i\theta}|
        = 3M(1-|\zeta|),
\]
which gives \eqref{eq:midpoint-poly-gap}.
\end{proof}
\begin{definition}[closed-form midpoint threshold]\label{def:qmid-explicit}
Set
\begin{equation}\label{eq:qmid-explicit}
        q_{\rm mid}(\tau):=
        \min\left\{\bar q,\frac{1}{80M^2}\right\},
\end{equation}
where $\bar q$ is as in \eqref{eq:closed-basic-constants}.
\end{definition}
\begin{proposition}[closed-form midpoint threshold]\label{prop:explicit-mid-threshold}
If $0<q\le q_{\rm mid}(\tau)$, then every root of the midpoint polynomial $P_{\tau,q}$  lies in $\{|z|<\rho\}$ 
\saa{with $P_{\tau,q}$ defined in \eqref{eq:Pmid}.}
\mgd{Consequently, hypothesis
\textbf{(H1)} of Lemma~\ref{lemm:ratecert}
holds}.
\end{proposition}

\begin{proof}
First, Lemma~\ref{lem:midpoint-poly-gap} and $q\le1/(80M^2)$ give
\[
        \rho=1-\frac{cq}{M}
        \ge 1-\frac{1}{120M^3}
        > 1-\frac{1}{30M^3},
\]
so all roots of $Q_\tau$ lie in $|z|<\rho$.
On the circle $|z|=\rho$, write $z=\rho e^{i\theta}$.  By Proposition~\ref{prop:midpoint-unit-floor} and the derivative bound used in Lemma~\ref{lem:midpoint-poly-gap},
\beq\label{eq:Q-floor-on-rho}
        |Q_\tau(\rho e^{i\theta})|
        \ge |Q_\tau(e^{i\theta})|-3M(1-\rho)
        \ge \frac{1}{10M^2}-3cq
        \ge \frac{3}{40M^2}.
\eeq
Using $\beta=1-2q/M$, we have the difference
\begin{equation}\label{eq:PminusQ-poly}
        P_{\tau,q}(z)-Q_\tau(z)
        =\frac{2q}{M}z^\tau(z-1)
        +\frac{q^2}{2M^2}\bigl((M+\tfrac12)z-(M-\tfrac12)\bigr),
\end{equation}
which satisfies, on $|z|=\rho\le1$,
\beq\label{eq:PminusQ-bound}
        |P_{\tau,q}(z)-Q_\tau(z)|
        \le \frac{4q}{M}+\frac{q^2}{M}
        \le \frac{5q}{M}
        \le \frac{1}{16M^3}
        <\frac{3}{40M^2}.
\eeq
\saa{
Thus,
\[
|P_{\tau,q}(z)-Q_\tau(z)|<|Q_\tau(z)|
\qquad\text{for all } |z|=\rho.
\]
By Rouch\'e's theorem, $P_{\tau,q}$ and
$Q_\tau$ have the same number of zeros
inside $|z|<\rho$, counted with multiplicity.
Since $Q_\tau$ has all its $\tau+2$ zeros
inside this circle by Lemma~\ref{lem:midpoint-poly-gap}, the same
holds for $P_{\tau,q}$.}\looseness=-1
\end{proof}

The FDI threshold of Section~\ref{sec:finiteq} and the midpoint threshold above combine into a single admissible threshold.
\mgd{\begin{definition}[Combined condition-number threshold]
\label{def:q0-explicit}
For $\tau\ge1$ and $c\in(0,2/3)$, define
\begin{equation}\label{eq:q0-explicit}
q_0(\tau,c):=
\min\{q_{\rm fdi}(\tau,c),q_{\rm mid}(\tau)\},
\qquad
\kappa_0(\tau,c):=q_0(\tau,c)^{-2}.\end{equation}\end{definition}\begin{proposition}[Admissibility of the combined threshold]\label{prop-admissibility-combined-threshold}
For every $\tau\ge1$ and $c\in(0,2/3)$,
the choice $\kappa_0(\tau,c)$ in
\eqref{eq:q0-explicit} is an admissible
condition-number threshold for the
convergence guarantee of
Theorem~\ref{thm:main}.
In particular, whenever
$\kappa\ge\kappa_0(\tau,c)$,
both hypotheses \textbf{(H1)} and
\textbf{(H2)} of the convergence-rate
certificate are satisfied.
\end{proposition}
\begin{proof}
By Propositions~\ref{prop:explicit-fdi-threshold}
and~\ref{prop:explicit-mid-threshold},
hypotheses \textbf{(H2)} and \textbf{(H1)},
respectively, hold whenever
$0<q\le q_0(\tau,c)$, as defined in
\eqref{eq:q0-explicit}.
Since $q=\kappa^{-1/2}$, this condition is
equivalent to $\kappa\ge\kappa_0(\tau,c)$.
\end{proof}
}
\begin{remark}[Scaling of the threshold with respect to $M$]
\saa{Fix $c\in(0,2/3)$. As $M\to\infty$,
the definitions of the derivative-envelope
constants given in \eqref{eq:closed-fdi-envelope-0}--\eqref{eq:closed-fdi-envelope-4} imply that}
\[
        K_0=\Theta(M^2),\qquad
        K_2=\Theta(M^4),\qquad
        K_4=\Theta(M^6),
\]
as $M\to \infty$, while from the definitions \eqref{eq:closed-leading-coeffs} \mgd{and \eqref{eq:limiting-quartic-floor}}, we observe $\mgd{B_{\mathrm{cub}}}=\Theta_c(M^{-2})$, $\Lambda=\Theta_c(1)$, and $\Gamma_\tau=\Theta(M^2)$.  Hence the entries in \eqref{eq:qfdi-explicit} have the order
\[
        \frac{\mgd{B_{\mathrm{cub}}}}{K_0}=\Theta_c(M^{-4}),\qquad
        \frac{\Lambda}{K_2}=\Theta_c(M^{-4}),\qquad
        \frac{\Gamma_\tau}{K_4}=\Theta(M^{-4}).
\]
The displayed formula \eqref{eq:qfdi-explicit} therefore gives the large-delay scaling $q_{\rm fdi}=\Theta_c(M^{-4})$. Also, $\bar q = \Theta_c(1)$ where $\bar q$ is defined in \eqref{eq:closed-basic-constants}. Hence, we see from \eqref{eq:qmid-explicit} that $q_{\rm mid}=\Theta(M^{-2})$.
Hence, the resulting \mgd{admissible $q$-threshold} scales as $q_{0}=\min(q_{\rm fdi}, q_{\rm mid})= \Theta_c(M^{-4})$ and $\kappa_0=\left(q_{0}\right)^{-2}=\Theta_c(M^{8})$. %
\end{remark}
\section{Explicit prefactor and initialization}\label{sec:explicit}\label{supp:prefactor-module}
\mgd{Throughout this appendix, fix an integer $\tau\ge1$, $c\in(0,2/3)$, and
$0<q\le q_0(\tau,c)$. We use the normalization and parameter choice
\eqref{eq:norm}--\eqref{eq:params}. In the \mgd{analysis} below, the
iterate and query coordinates are centered at $x_*$; equivalently,
the minimizer is translated to the origin.}

Sections~\ref{sec:finiteq}--\ref{sec:midpoint} establish the strict FDI and midpoint stability needed for the rate
\[
        \rho=1-\frac{cq}{M},
        \qquad q=\kappa^{-1/2},
        \qquad M=2\tau+1.
\]
It remains to make the transient prefactor explicit.  We do this directly on the iterate coordinate, after normalizing the midpoint input to the same scale as the iterate.  This avoids the unnecessary loss that results from bounding the full augmented state and gives the polynomial dependence
\[
       \mgd{C_{\mathrm{rate}}} =\cO_c(M^{9/2}\kappa^2)\|\bar\Xi_0\|
\]
for an arbitrary finite history and
\[
        \mgd{C_{\mathrm{rate}}}=\cO_c(M^5\kappa^2)\|x_{\rm init}-x_*\|
\]
under constant-history initialization.

\subsection{Normalized midpoint input and the driven recursion}

Let
\begin{equation}\label{eq:omega-normalized-input}
        \omega_q:=\frac{q^{-2}-1}{2}=\frac{\kappa-1}{2},
        \qquad
        \widehat w_k:=\omega_q^{-1}\widehat v_k,
\end{equation}
where \(\widehat v_k=\widehat u_k-m\widehat y_k\) and
\(m=(1+q^{-2})/2\) is the midpoint slope. Since
\(q\le q_0(\tau,c)\le\bar q<1/2\), we have $3/4\leq 1-q^2 \leq 1$ and 
\begin{equation}\label{eq:omega-elementary-bounds}
        \frac{3}{8q^2}\le \omega_q\le\frac{1}{2q^2},
        \qquad
        \frac38\le q^2\omega_q=\frac{1-q^2}{2}\le\frac12.
\end{equation}
\mgd{Under the normalization $\mu=1$ and $L=\kappa$,
the midpoint and half-width of the interval
of admissible gradient slopes are
\[
m=\frac{1+\kappa}{2},
\qquad
\omega_q=\frac{\kappa-1}{2}.
\]}
\mgb{For an arbitrary vector point--gradient pair $u=\nabla f(y)$ with
$\nabla f(0)=0$, the strong convexity and smoothness inequality \eqref{eq:interp} gives
\[
 \|u-my\|^2
 =\|u\|^2-2m\langle u,y\rangle+m^2\|y\|^2
 \le (m^2-\kappa)\|y\|^2
 =\omega_q^2\|y\|^2.
\]
\saa{Consequently, $\|\widehat v_k\|\le
\omega_q\|\widehat y_k\|$.
Thus, our definition
$\widehat w_k:=\omega_q^{-1}\widehat v_k$
ensures that
$\|\widehat w_k\|\le\|\widehat y_k\|$, i.e., the normalized midpoint residual is
controlled by the extrapolated state.}
}
Therefore \(\omega_q=(\kappa-1)/2\) is a natural normalization factor in the sense that $\widehat w_k:= \wubtext{\widehat v}_k/\omega_q$ 
has the same natural scale as the extrapolated state $\wubtext{\widehat y}_k=\wubtext{\widehat x}_k+\nu(\wubtext{\widehat x}_k - \wubtext{\rho^{-1}\widehat x}_{k-1})$.%

With the normalization, \(\widehat v_k=\omega_q\widehat w_k\)
the normalized initial vector is
\begin{equation}\label{eq:barXi0-closed}
        \bar\Xi_0=
        \begin{bmatrix}
        \widehat x_0^{\mathsf T}&\widehat x_{-1}^{\mathsf T}&\cdots&\widehat x_{-\tau-1}^{\mathsf T}&
        \widehat w_{-1}^{\mathsf T}&\cdots&\widehat w_{-\tau}^{\mathsf T}& \mgb{0_{\tau d}^{\mathsf T}}
        \end{bmatrix}^{\mathsf T}\in\mathbb{R}^{(3\tau+2)d},
\end{equation}
where $0_{\tau d}\in \mathbb{R}^{\tau d}$ is a column vector of zeros. \mgd{The final $\tau d$ zero entries initialize an auxiliary memory
that stores the past values of
$\widehat s_k=L\widehat y_k-\widehat u_k$ needed to
evaluate the delayed terms in the integral quadratic
constraint (IQC). This zero initialization is consistent
with the zero-padding convention used in the finite-horizon
IQC argument: after translating the minimizer to the origin,
we extend the gradient-query pairs by $(0,0)$ outside the
considered time interval, so that the corresponding
$\widehat s_k$ values are zero at negative time indices.
This padding is used only in the analysis and does not
alter the prescribed iterate or midpoint-input histories
of the algorithm.  The stored inputs $\widehat w_{-1},\ldots,\widehat w_{-\tau}$
are the normalized residuals generated by the prescribed iterate history;
only the multiplier memories are initialized at zero.} All norms below are Euclidean for finite vectors and the usual \(\ell_2\) or \(\ell_\infty\) norms for sequences \mgd{indexed by $k\geq 0$}.

To see why the midpoint characteristic polynomial appears \mgd{in the transfer from $\widehat v$ to $\widehat x$}, recall from Lemma~\ref{lem:weighted-realization} that the
rate-weighted midpoint dynamics satisfy
\[
\widehat x_{k+1}
=\frac{1+\beta}{\rho}\widehat x_k
-\frac{\beta}{\rho^2}\widehat x_{k-1}
-\frac{\alpha m(1+\nu)}{\rho^{\tau+1}}\widehat x_{k-\tau}
+\frac{\alpha m\nu}{\rho^{\tau+2}}\widehat x_{k-\tau-1}
-\frac{\alpha}{\rho^{\tau+1}}\widehat v_{k-\tau}.
\]

\saa{
We first consider the zero-state input-output
response; the contribution of the prescribed
initial history will be treated separately.}
For the zero-state response, taking the \(z\)-transform and multiplying by
\(z^{\tau+1}\) gives
$$
\begin{aligned}
\Bigg[
z^{\tau+2}
-\frac{1+\beta}{\rho}z^{\tau+1}
+\frac{\beta}{\rho^2}z^\tau
+\frac{\alpha m(1+\nu)}{\rho^{\tau+1}}z
-\frac{\alpha m\nu}{\rho^{\tau+2}}
\Bigg]\widehat X(z)
=
-\frac{\alpha}{\rho^{\tau+1}}z\widehat V(z).
\end{aligned}
$$
Here the factor \(z\) on the right-hand side results from the delayed
input term \(z^{-\tau}\widehat V(z)\), since
\(z^{\tau+1}z^{-\tau}=z\). This is equivalent to 

\begin{equation}\label{eq:PF-poly-denom}
P_F(z)\widehat X(z)
=
-\frac{\alpha}{\rho^{\tau+1}}\,z\,\widehat V(z),
\quad
\mbox{with}
\quad
P_F(z)
=
\rho^{-(\tau+2)}P_{\tau,q}(\rho z).
\end{equation}
Therefore,
\begin{equation}\label{eq:Gx-poly}
\mgb{\widetilde g_x(z)}
=
-\frac{\alpha}{\rho^{\tau+1}}\frac{z}{P_F(z)}
=
\frac{bz}{P_F(z)},
\qquad
b:=-\frac{\alpha}{\rho^{\tau+1}}.
\end{equation}
\mgb{Thus $\widehat X(z)=\widetilde g_x(z)\widehat V(z)$, and the
$d\times d$ transfer matrix is
$\widetilde G_x(z):=\widetilde g_x(z)I_d$}. 

Note that since \(\alpha=q^2/M^2\), \(\tau+1\le3M\), and \eqref{eq:rho-sharp-envelope} gives
\(\rho^{-(\tau+1)}\le R_{\rm rt}=\mathrm e\),
\begin{equation}\label{eq:b-poly-bound}
        |b|\le R_{\rm rt}\frac{q^2}{M^2}.
\end{equation}
We next bound the modulus of $P_F$
on the unit circle and the %
$\mathcal{H}_2$ and $\mathcal{H}_\infty$ norms of its reciprocal $1/P_F(z)$. We start with defining the $\mathcal{H}_2$ and $\mathcal{H}_\infty$ norms.
\begin{definition}
A scalar discrete-time transfer function \(\varphi(z)\) describes the input--output
map of a scalar linear time-invariant discrete-time system in the \(z\)-domain.
Throughout, \(\varphi\) is called \emph{stable} if all of its poles lie strictly
inside the unit disk,
\[
\{z\in\mathbb C:|z|<1\}.
\] 
For a stable proper rational scalar transfer function \(\varphi\), we use the standard
unit-circle norms
\[
\|\varphi\|_{\Hinf}
:=
\max_{|z|=1}|\varphi(z)|
=
\max_{\theta\in[-\pi,\pi]}
\left|\varphi(e^{i\theta})\right|
\]
and
\[
\|\varphi\|_{\Htwo}
:=
\left(
\frac{1}{2\pi}
\int_{-\pi}^{\pi}
\left|\varphi(e^{i\theta})\right|^2
\,d\theta
\right)^{1/2}.
\]
\mgd{
Thus, $\|\varphi\|_{\Hinf}$ is the largest
magnitude of the transfer function on the unit
circle, whereas $\|\varphi\|_{\Htwo}$ is its
root-mean-square magnitude over the unit circle.
Moreover, for the corresponding linear input--output
map with zero initial state, the $\Hinf$ norm equals
the induced $\ell_2$-gain
\[
\sup_{u\in\ell_2,\;u\ne0}
\frac{\|y\|_{\ell_2}}{\|u\|_{\ell_2}},
\]
where $u$ and $y$ are the input and output
sequences, respectively. The $\Htwo$ norm,
in turn, provides an upper bound on the
$\ell_2$-to-$\ell_\infty$ gain; see,
e.g.,~\cite{zhou1996robust}. In particular, by the definitions of the two norms,
\[
\|\varphi\|_{\Htwo}
\le \|\varphi\|_{\Hinf}.
\]
}
\end{definition}
\begin{lemma}[Lower bound for the denominator $P_F(z)$ norm]\label{lem:finiteq-denominator-floor}
For \(0<q\le q_0(\tau,c)\),
\begin{equation}\label{eq:finiteq-denominator-floor}
        \min_{|z|=1}|P_F(z)|\ge\frac{1}{20M^2}.
\end{equation}
Consequently,
\begin{equation}\label{eq:green-function-norms}
        \left\|\frac1{P_F}\right\|_{\Hinf}\le20M^2,
        \qquad
        \left\|\frac1{P_F}\right\|_{\Htwo}\le20M^2.
\end{equation}
\end{lemma}

\begin{proof}
The proof of Proposition~\ref{prop:explicit-mid-threshold} gives, on \(|w|=\rho\),
\[
        |Q_\tau(w)|\ge\frac{3}{40M^2},
        \qquad
        |P_{\tau,q}(w)-Q_\tau(w)|\le\frac{1}{16M^3}.
\]
Since \(M\ge3\),
\[
        |P_{\tau,q}(w)|
        \ge\frac{3}{40M^2}-\frac{1}{48M^2}
        =\frac{13}{240M^2}
        >\frac{1}{20M^2}.
\]
Now set \(w=\rho z\) for $z$ with $|z|=1$.  Because \(P_F(z)=\rho^{-(\tau+2)}P_{\tau,q}(\rho z)\) and \(\rho<1\), the prefactor \(\rho^{-(\tau+2)}\) can only increase the modulus. This proves \eqref{eq:finiteq-denominator-floor}. \mgd{By Proposition~\ref{prop:explicit-mid-threshold}, all zeros of $P_F$ lie
strictly inside the unit disk, so $1/P_F$ is stable and proper.
The norm bounds \eqref{eq:green-function-norms} therefore follow from
\eqref{eq:finiteq-denominator-floor} and $\|1/P_F\|_{\Htwo}\le\|1/P_F\|_{\Hinf}$.} %
\end{proof}

\mgd{We next bound the contributions of the prescribed initial
history and the normalized midpoint input to the iterate
sequence. To derive these estimates, we analyze the linear midpoint
recursion separately from the nonlinear gradient mapping,
temporarily treating $\widehat w\in\ell_2$ as an arbitrary
input while keeping the prescribed initial history fixed.
This separation is only an analytical device: in the actual
algorithm, $\widehat w$ is determined by the nonlinear
gradient mapping. Once its square-summability has been
established, the resulting bounds will be applied to
the actual trajectory.}
\mgd{
\begin{lemma}[Bounds on the iterates and extrapolated points]
\label{lem:normalized-direct-gains} Under the standing assumptions $\tau\ge1$, $c\in(0,2/3)$, and $0<q\le q_0(\tau,c)$, let
$\widehat w\in\ell_2$, and decompose the response as
$\widehat x=\widehat x^0+\widehat x^w$, where
$\widehat x^0$ is generated by the prescribed initial
history $\bar\Xi_0$ with $\widehat w_k=0$ for $k\ge0$,
and $\widehat x^w$ is generated by $\widehat w$ with
zero initial history. Then
\begin{align}
    \|\widehat x^w\|_{\ell_2}
    &\le60\|\widehat w\|_{\ell_2},
    &
    \|\widehat x^w\|_{\ell_\infty}
    &\le60\|\widehat w\|_{\ell_2},
    \label{eq:xw-direct-gain}\\
    \|\widehat x^0\|_{\ell_2}
    &\le240M^{5/2}\|\bar\Xi_0\|,
    &
    \|\widehat x^0\|_{\ell_\infty}
    &\le240M^{5/2}\|\bar\Xi_0\|.
    \label{eq:x0-normalized-gain}
\end{align}
Moreover, if $\widehat x_{-1}=0$, the rate-weighted
extrapolation map $\widehat{x}\to\widehat{y}$ satisfies
\[
    \|\widehat y\|_{\ell_2}
    \le4M\|\widehat x\|_{\ell_2}.
\]
The corresponding extrapolated responses satisfy
\begin{equation}\label{eq:y0-normalized-gain}
    \|\widehat y^0\|_{\ell_2}
    \le960M^{7/2}\|\bar\Xi_0\|,
    \qquad
    \|\widehat y^w\|_{\ell_2}
    \le240M\|\widehat w\|_{\ell_2}.
\end{equation}
\end{lemma}
}
\begin{proof}
    \mgd{The proof is deferred to Appendix~\ref{app:gains}}.
\end{proof}

\subsection{Frequency-domain bound in normalized midpoint coordinates}\label{subsec-fdi-normalized-midpoint}

The strict FDI from Section~\ref{sec:finiteq} is stated in the original gradient
coordinate. We now transfer this margin to the normalized midpoint input
\(\widehat w\). For \(z=\rho e^{i\theta}\), define
\[
P(z)
:=
\left(M+\frac12\right)z-\left(M-\frac12\right)
=
M(z-1)+\frac{z+1}{2}.
\]
By \eqref{eq:g},
\[
g_{\tau,q}(z)
=
-\frac{q^2}{M^2}
\frac{z^{-\tau}P(z)}{(z-1)(z-\beta)}.
\]
Since \(|z|=\rho\le1\),
\[
|P(z)|
\le
M|z-1|+\frac{|z|+1}{2}
\le
M|z-1|+1.
\]
Moreover, \mgd{since $0<\beta<\rho<1$ on the stated parameter range}, the reverse triangle inequality and the definitions of \(\rho\) and
\(\beta\) give
\[
|z-1|
\ge
1-|z|
=
\frac{cq}{M},
\qquad
|z-\beta|
\ge
\bigl||z|-\beta\bigr|
=
\frac{(2-c)q}{M}.
\]
Using also \(|z|^{-\tau}\le R_{\rm rt}\), we obtain
\begin{equation}\label{eq:g-sharp-bound}
|g_{\tau,q}(z)|
\le
R_{\rm rt}
\left(
\frac{q}{2-c}
+
\frac{1}{c(2-c)}
\right)
\le
R_{\rm rt}\frac{1+c}{c(2-c)}.
\end{equation}
Define
\begin{equation}\label{eq:Upsilon-poly}
\Upsilon_c
:=
1+R_{\rm rt}\frac{1+c}{c(2-c)}.
\end{equation}
Since
\[
m=\frac{1+q^{-2}}{2}\le q^{-2},
\qquad
q^{-2}\ge1,
\]
the triangle inequality and \eqref{eq:g-sharp-bound} imply
\begin{equation}\label{eq:midpoint-denom-poly}
\left|1-mg_{\tau,q}(\rho e^{i\theta})\right|
\le
\Upsilon_c q^{-2}.
\end{equation}

Corollary~\ref{cor:Phi-uniform-margin} gives the uniform FDI margin
\begin{equation}\label{eq:fdi-margin-explicit}
-\Phi_{\tau,q}(\theta)
\ge
\eta_c^\Phi,
\qquad
\eta_c^\Phi
\ge
\frac{c(1-c)^2}{1000}.
\end{equation}
We next express this margin in the midpoint coordinate. \mgd{The \emph{supply} $\sigma_k$ in
\eqref{eq:pointwise-supply} is a scalar
quadratic expression whose sum over iterations
captures the constraint imposed by the
gradient mapping in the convergence proof.
It is not an additional algorithmic variable.}

\mgd{Fix any input $\widehat v\in\ell_2$ and consider
the corresponding response of the stable
rate-weighted midpoint-transformed linear system,
initialized at zero.
Let $\widehat Y$, $\widehat U$, and $\widehat V$
denote the Fourier transforms of the
$d$-dimensional sequences $\widehat y$,
$\widehat u$, and $\widehat v$, respectively.
These transforms satisfy}
\[
\widehat Y(e^{i\theta})
=
g_{\tau,q}(\rho e^{i\theta})\widehat U(e^{i\theta}),
\]
while the midpoint change of variables
\[
\widehat v_k
=
\widehat u_k-m\widehat y_k
\]
gives
\[
\widehat V(e^{i\theta})
=
\left(1-mg_{\tau,q}(\rho e^{i\theta})\right)
\widehat U(e^{i\theta}).
\]
\mgd{By Lemma~\ref{lem:fdi-congruence}, the factor $1-mg_{\tau,q}(\rho e^{i\theta})$ is nonzero}. Thus
\[
\left\|\widehat U(e^{i\theta})\right\|^2
=
\frac{\left\|\widehat V(e^{i\theta})\right\|^2}
{\left|1-mg_{\tau,q}(\rho e^{i\theta})\right|^2}.
\]
Equivalently, the congruence identity in
Lemma~\ref{lem:fdi-congruence} and \eqref{eq:fdi-equivalence}
rewrites the same quadratic form in the midpoint input. Therefore, if
\(\sigma_k^v\) denotes the supply evaluated along the zero-state response driven
by \(\widehat v\), Parseval's identity gives
\[
\sum_{k\ge0}\sigma_k^v
=
\frac{1}{2\pi}
\int_{-\pi}^{\pi}
\frac{\Phi_{\tau,q}(\theta)}
{\left|1-mg_{\tau,q}(\rho e^{i\theta})\right|^2}
\left\|\widehat V(e^{i\theta})\right\|^2
\,d\theta.
\]
By \eqref{eq:midpoint-denom-poly},
\[
\frac{1}
{\left|1-mg_{\tau,q}(\rho e^{i\theta})\right|^2}
\ge
\frac{q^4}{\Upsilon_c^2}.
\]
Combining this with \eqref{eq:fdi-margin-explicit}, \mgd{extended to $[-\pi,\pi]$ by evenness}, yields
\[
\frac{\Phi_{\tau,q}(\theta)}
{\left|1-mg_{\tau,q}(\rho e^{i\theta})\right|^2}
\le
-\frac{\eta_c^\Phi q^4}{\Upsilon_c^2}.
\]
Since
\[
\frac{1}{2\pi}
\int_{-\pi}^{\pi}
\left\|\widehat V(e^{i\theta})\right\|^2
\,d\theta
=
\|\widehat v\|_{\ell_2}^2
\]
by Parseval, we obtain
\begin{equation}\label{eq:gamma-v-closed}
\sum_{k\ge0}\sigma_k^v
\le
-\gamma_{\tau,c}(q)\|\widehat v\|_{\ell_2}^2,
\qquad
\gamma_{\tau,c}(q)
:=
\frac{\eta_c^\Phi q^4}{\Upsilon_c^2}.
\end{equation}
The factor \(q^4\) appears because the midpoint coordinate change contributes
\(\left|1-mg_{\tau,q}\right|^{-2}\), while
\eqref{eq:midpoint-denom-poly} bounds \(\left|1-mg_{\tau,q}\right|\) by a
quantity of order \(q^{-2}\).

Since \(\widehat v=\omega_q\widehat w\) and $q^2 \omega_q \geq 3/8$,
\eqref{eq:omega-elementary-bounds} gives
\begin{equation}\label{eq:gamma-normalized-closed}
\gamma_{\tau,c}(q)\omega_q^2
=
\frac{\eta_c^\Phi(q^2\omega_q)^2}{\Upsilon_c^2}
\ge
\bar\gamma_c,
\qquad
\bar\gamma_c
:=
\frac{9\eta_c^\Phi}{64\Upsilon_c^2}.
\end{equation}
Thus the normalized dissipation margin \(\bar\gamma_c\) is independent of both
\(q\) and \(M\).

We next bound the corresponding sector variables. Recall that
\[
r_k=u_k-y_k,
\qquad
s_k=\kappa y_k-u_k.
\]
Using
\[
u_k=my_k+v_k,
\qquad
v_k=\omega_q w_k,
\qquad
m-1=\kappa-m=\omega_q,
\]
and then applying the rate weighting gives
\begin{equation}\label{eq:rs-normalized-relations}
\widehat r_k
=
\omega_q(\widehat y_k+\widehat w_k),
\qquad
\widehat s_k
=
\omega_q(\widehat y_k-\widehat w_k).
\end{equation}
To separate initialization from the future input, write
\[
\widehat y
=
\widehat y^0+\widehat y^w,
\]
where \(\widehat y^0\) is the zero-input response generated by the prescribed
history and \(\widehat y^w\) is the zero-state response driven by
\(\widehat w\). For the zero-input component, \(\widehat w_k^0=0\) for
\(k\ge0\); hence \eqref{eq:rs-normalized-relations} gives
\[
\widehat r_k^0
=
\widehat s_k^0
=
\omega_q\widehat y_k^0.
\]
For the zero-state component,
\[
\widehat r_k^w
=
\omega_q(\widehat y_k^w+\widehat w_k),
\qquad
\widehat s_k^w
=
\omega_q(\widehat y_k^w-\widehat w_k).
\]
Define
\begin{equation}\label{eq:R-normalized}
R_0(\kappa,M)
:=
1000\kappa M^{7/2},
\qquad
R_w(\kappa,M)
:=
250\kappa M.
\end{equation}
By Lemma~\ref{lem:normalized-direct-gains},
\[
\|\widehat y^0\|_{\ell_2}
\le
960M^{7/2}\|\bar\Xi_0\|,
\qquad
\|\widehat y^w\|_{\ell_2}
\le
240M\|\widehat w\|_{\ell_2}.
\]
Since \(2\omega_q\le\kappa\), the zero-input component satisfies
\begin{equation}\label{eq:rs-normalized-gains}
\begin{aligned}
\|\widehat r^0\|_{\ell_2}
+
\|\widehat s^0\|_{\ell_2}=
2\omega_q\|\widehat y^0\|_{\ell_2}
\le
R_0(\kappa,M)\|\bar\Xi_0\|,
\end{aligned}
\end{equation}
while the zero-state component obeys
\begin{equation}\label{eq:rw-normalized-gain}
\begin{aligned}
\|\widehat r^w\|_{\ell_2}
+
\|\widehat s^w\|_{\ell_2}
\le
2\omega_q
\left(
\|\widehat y^w\|_{\ell_2}
+
\|\widehat w\|_{\ell_2}
\right)
\le
R_w(\kappa,M)\|\widehat w\|_{\ell_2}.
\end{aligned}
\end{equation}

\mgd{Finally, for any two square-summable sequences
$r=\{r_k\}_{k\ge0}$ and $s=\{s_k\}_{k\ge0}$,
define the bilinear expression
\begin{equation}\label{eq-sigma-rs}
\begin{aligned}
\Sigma(r,s)
&:=2\sum_{k=0}^{\infty}\langle r_k,s_k\rangle-\frac{2}{\tau}\sum_{j=1}^{\tau}\rho^j
\sum_{k=j}^{\infty}\langle r_k,s_{k-j}\rangle.
\end{aligned}
\end{equation}
When evaluating the delayed terms, we set $s_k=0$
for $k<0$, consistently with the zero-padding
convention used in the finite-horizon IQC argument.
Since $r,s\in\ell_2$ and only finitely many delays
are involved, each sum converges absolutely by the
Cauchy--Schwarz inequality. Thus $\Sigma(r,s)$ is
well-defined and finite.}

\mgd{For the actual optimization trajectory, the
uniform energy estimate \eqref{eq:weighted-energy-bound}
establishes that the rate-weighted state $\Xi$ and
the midpoint input $\widehat v$ belong to $\ell_2$.
Since $\widehat y$ is a fixed linear combination
of components of $\Xi$ and
$\widehat u=m\widehat y+\widehat v$, it follows that
$\widehat y,\widehat u\in\ell_2$. Consequently,
$\widehat r=\widehat u-\widehat y$ and
$\widehat s=L\widehat y-\widehat u$ are also
square-summable, so $\Sigma(\widehat r,\widehat s)$
is finite. Moreover, by \eqref{eq:pointwise-supply},
\[
\Sigma(\widehat r,\widehat s)
=\sum_{k=0}^{\infty}\sigma_k
=\lim_{N\to\infty}\sum_{k=0}^{N}\sigma_k.
\]
The finite-horizon IQC \eqref{eq:hard-finite-use}
guarantees that every partial sum is non-negative.
Therefore,
\[
\Sigma(\widehat r,\widehat s)\ge0.
\]
In optimization terms, this inequality captures
relationships between gradients and their evaluation
points across different iterations, using properties
of the gradients of smooth, strongly convex functions.} With
\[
\widehat r
=
\widehat r^0+\widehat r^w,
\qquad
\widehat s
=
\widehat s^0+\widehat s^w,
\]
bilinearity gives
\[
\Sigma(\widehat r,\widehat s)
=
\Sigma(\widehat r^0,\widehat s^0)
+
\Sigma(\widehat r^w,\widehat s^w)
+
\Sigma_{\rm cross},
\]
where
\[
\Sigma_{\rm cross}
:=
\Sigma(\widehat r^0,\widehat s^w)
+
\Sigma(\widehat r^w,\widehat s^0).
\]
Thus \(\Sigma_{\rm cross}\) contains exactly the interaction between the
initial-history response and the response driven by \(\widehat w\). 
\saa{Lemma~\ref{lem:normalized-supply-splitting}, together with
\eqref{eq:rs-normalized-gains}--\eqref{eq:rw-normalized-gain},
gives the following bounds on the zero-input supply and the cross terms:}
\begin{align}
|\Sigma(\widehat r^0,\widehat s^0)|
&\le
R_0(\kappa,M)^2\|\bar\Xi_0\|^2,
\label{eq:supply-zero-input-normalized}\\
|\Sigma_{\rm cross}|
&\le
4R_0(\kappa,M)R_w(\kappa,M)
\|\bar\Xi_0\|\,\|\widehat w\|_{\ell_2}.
\label{eq:supply-cross-normalized}
\end{align}
 
\subsection{From auxiliary bounds to the explicit prefactor}

Before summing over the infinite horizon, note that summability is already available from the qualitative rate certificate in the main text.  \mgd{Indeed, for $0<q\le q_0(\tau,c)$,
Sections~\ref{sec:finiteq}--\ref{sec:midpoint}
establish both hypotheses of
Lemma~\ref{lemm:ratecert}.
The finite-horizon energy bound
\eqref{eq:weighted-energy-bound} then implies
that $\widehat v\in\ell_2$ and that the
augmented rate-weighted state, including
the multiplier memory, is square-summable.} %

\mgd{Since $\widehat w=\omega_q^{-1}\widehat v$,
all sequence norms and sums of the supply
terms used below are therefore finite.}
\mgd{
The square-summability established above ensures that
the sums defining $\Sigma(\widehat r,\widehat s)$
are well-defined. Specifically, these consist of
inner products between $\widehat r_k$ and
$\widehat s_k$ at the same iteration, as well as
inner products between $\widehat r_k$ and
$\widehat s_{k-j}$ at different iterations,
for $j=1,\ldots,\tau$. Each sum over
$k\ge0$ converges absolutely by the
Cauchy--Schwarz inequality.
We may therefore let $N\to\infty$ in the
finite-horizon IQC inequality
\eqref{eq:hard-finite-use} to obtain
$\Sigma(\widehat r,\widehat s)\ge0$.}
\saa{The zero-state dissipation estimate
\eqref{eq:gamma-v-closed}, together with
the normalized margin
\eqref{eq:gamma-normalized-closed},
then applies to the complete zero-state
response of the linear recursion.}

We can now bound the normalized midpoint input in terms of the prescribed
initial history.  By bilinearity,
\[
\Sigma(\widehat r,\widehat s)
=
\Sigma(\widehat r^0,\widehat s^0)
+
\Sigma(\widehat r^w,\widehat s^w)
+
\Sigma_{\rm cross}.
\]
The hard IQC for the full \mgd{nonlinear} trajectory gives
\[
0\leq \Sigma(\widehat r,\widehat s) = \Sigma(\widehat r^0,\widehat s^0)
+
\Sigma(\widehat r^w,\widehat s^w)
+
\Sigma_{\rm cross},
\]
whereas \eqref{eq:gamma-normalized-closed} gives, for the zero-state component,
\[
\Sigma(\widehat r^w,\widehat s^w)
\le
-\bar\gamma_c\|\widehat w\|_{\ell_2}^2.
\]
Moreover, \eqref{eq:supply-zero-input-normalized} and
\eqref{eq:supply-cross-normalized} imply
\[
\Sigma(\widehat r^0,\widehat s^0)
\le
R_0^2\|\bar\Xi_0\|^2
\]
and
\[
\Sigma_{\rm cross}=\Sigma(\widehat r^0,\widehat s^w)
+
\Sigma(\widehat r^w,\widehat s^0)
\le
4R_0R_w
\|\bar\Xi_0\|\,\saa{\|\widehat w\|_{\ell_2},}
\]
where \(R_0=R_0(\kappa,M)\) and \(R_w=R_w(\kappa,M)\). Combining these bounds yields
\begin{equation}\label{eq:w-quadratic-normalized}
0\le
-\bar\gamma_c\|\widehat w\|_{\ell_2}^2
+
4R_0R_w\|\bar\Xi_0\|\,\|\widehat w\|_{\ell_2}
+
R_0^2\|\bar\Xi_0\|^2,
\end{equation}%
For \mgd{$t\geq 0$}, \(\bar \gamma>0\) and $\bar P, \bar Q \geq 0$, the elementary implication
\[
        \bar\gamma t^2\le \bar P t+ \bar Q
        \quad\Longrightarrow\quad
        t\le\frac{\bar P}{ \bar \gamma}+\sqrt{\frac{\bar Q}{\bar \gamma}},
\]
\mgd{
\saa{with} $t=\|\widehat w\|_{\ell_2}$,
$\bar\gamma=\bar\gamma_c$, $\bar P=4R_0R_w\|\bar\Xi_0\|$, and $\bar Q=R_0^2\|\bar\Xi_0\|^2$}, %
\saa{implies that}
\begin{equation}\label{eq:w-bound-normalized}
        \|\widehat w\|_{\ell_2}
        \le H_{\rm nl}^{\rm norm}(\kappa,M)\|\bar\Xi_0\|,
\end{equation}
where
\begin{equation}\label{eq:Hnl-normalized}
        H_{\rm nl}^{\rm norm}(\kappa,M):=
        \frac{4R_0(\kappa,M)R_w(\kappa,M)}{\bar\gamma_c}
        +\frac{R_0(\kappa,M)}{\sqrt{\bar\gamma_c}}.
\end{equation}

\begin{proposition}[explicit prefactor]
\label{prop:normalized-direct-prefactor}\label{prop:explicit-prefactor}
\saa{Fix an integer $\tau\ge 1$ and $c\in(0,2/3)$,
and suppose that
$0<q=\kappa^{-1/2}\le q_0(\tau,c)$,
where $q_0(\tau,c)$ is defined in
\eqref{eq:q0-explicit}.
Then every trajectory generated by
\DREAM{} with the parameters in \eqref{eq:params},
for $f\in\mathcal S_{1,\kappa}$ and any finite
prescribed iterate history, satisfies}
\begin{equation}\label{eq:normalized-direct-rate}
        \|x_k-x_*\|
        \le C_{\rm cl}^{\rm norm}(\tau,c,\kappa;\bar\Xi_0)\rho^k,
        \qquad k\ge0,
\end{equation}
where
\begin{equation}\label{eq:C-explicit-closed}
        C_{\rm cl}^{\rm norm}(\tau,c,\kappa;\bar\Xi_0)
        :=\left(240M^{5/2}+60H_{\rm nl}^{\rm norm}(\kappa,M)\right)
        \|\bar\Xi_0\|.
\end{equation}
Moreover, for fixed \(c\in(0,2/3)\),
\begin{equation}\label{eq:Cnorm-growth}
        C_{\rm cl}^{\rm norm}(\tau,c,\kappa;\bar\Xi_0)
        \le K_c^{\rm norm}M^{9/2}\kappa^2\|\bar\Xi_0\|,
        \qquad
        K_c^{\rm norm}:=
        \frac{10^{12}\Upsilon_c^2}{c(1-c)^2},
\end{equation}
\mgd{where $\Upsilon_c$ is as in \eqref{eq:Upsilon-poly}. In addition, 
\begin{equation}\label{eq:Cnorm-growth-arbitrary-history-display}
C_{\rm cl}^{\rm norm}
\le
K_c^{\rm norm}M^{9/2}\kappa^2
\sqrt{1+4M^2}\,\Psi_0
=
\mathcal O_c(M^{11/2}\kappa^2)\Psi_0,
\end{equation}}
where $
\Psi_0=
\left(
\sum_{j=-\tau-1}^{0}\|x_j-x_*\|^2
\right)^{1/2}$. Furthermore, under constant-history initialization
\begin{equation}\label{eq:Cnorm-constant-history}
        x_{-\tau-1}=\cdots=x_0=x_{\rm init},
\end{equation}
one has
\begin{equation}\label{eq:barXi0-constant-history}
        \|\bar\Xi_0\|
        \le \sqrt{2} M^{1/2}\|x_{\rm init}-x_*\|,
\end{equation}
and hence
\begin{equation}\label{eq:Cnorm-growth-constant-history}
        C_{\rm cl}^{\rm norm}
        \le2K_c^{\rm norm}M^5\kappa^2\|x_{\rm init}-x_*\|.
\end{equation}
\end{proposition}

\begin{proof}
By Lemma~\ref{lem:normalized-direct-gains}, \eqref{eq:w-bound-normalized}, and the decomposition \(\widehat x=\widehat x^0+\widehat x^w\),
\[
        \|\widehat x\|_{\ell_\infty}
        \le
        \left(240M^{5/2}+60H_{\rm nl}^{\rm norm}(\kappa,M)\right)
        \|\bar\Xi_0\|.
\]
Since \(\widehat x_k=\rho^{-k}(x_k-x_*)\), this proves \eqref{eq:normalized-direct-rate}--\eqref{eq:C-explicit-closed}.

It remains to make the dependence on \(M\) and \(\kappa\) explicit.  From \eqref{eq:fdi-margin-explicit} and \eqref{eq:gamma-normalized-closed},
\[
        \bar\gamma_c^{-1}
        =\frac{64\Upsilon_c^2}{9\eta_c^\Phi}
        \le\frac{8000\Upsilon_c^2}{c(1-c)^2}.
\]
Since \(\Upsilon_c\ge1\) and \(c(1-c)^2<1\), the same right-hand side also bounds \(\bar\gamma_c^{-1/2}\).  Using
\[
        R_0=1000\kappa M^{7/2},
        \qquad
        R_w=250\kappa M,
\]
we obtain, for \(M\ge3\) and \(\kappa\ge1\),
\begin{align*}
        H_{\rm nl}^{\rm norm}(\kappa,M)
        &\le
        4(1000\kappa M^{7/2})(250\kappa M)
        \frac{8000\Upsilon_c^2}{c(1-c)^2}
        +1000\kappa M^{7/2}
        \frac{8000\Upsilon_c^2}{c(1-c)^2}\\
        &\le
        \frac{10^{10}\Upsilon_c^2}{c(1-c)^2}
        M^{9/2}\kappa^2.
\end{align*}
\mgd{Substituting these bounds into
\eqref{eq:C-explicit-closed} and simplifying
the resulting estimate yields
\eqref{eq:Cnorm-growth}. Indeed, $60\cdot10^{10}+240<10^{12}$, while $M\ge3$, $\kappa\ge1$, and $\Upsilon_c^2/[c(1-c)^2]\ge1$.}

\mgd{For an arbitrary prescribed history, let
$e_j:=x_j-x_*$. Smoothness and strong convexity imply
$\|v(y)\|\le\omega_q\|y-x_*\|$ for the midpoint
residual $v(y)=\nabla f(y)-m(y-x_*)$.
Thus, by the definition
$\widehat w_j=\rho^{-j}v(y_j)/\omega_q$,
\[
\|\widehat w_j\|
\le\rho^{-j}\|y_j-x_*\|.
\]
Since $\rho^{-j}\le1$ for $j\le0$,
the definition~\eqref{eq:barXi0-closed} yields
\begin{align*}
\|\bar\Xi_0\|^2
&\le
\Psi_0^2
+
\sum_{j=-\tau}^{-1}
\|(1+\nu)e_j-\nu e_{j-1}\|^2\\
&\le
\bigl[1+(1+2\nu)^2\bigr]\Psi_0^2
=
(1+4M^2)\Psi_0^2.
\end{align*}
Here the second inequality follows from the triangle
inequality for the stacked sequences, and the equality
uses $\nu=M-\tfrac12$.
Combining this estimate with \eqref{eq:Cnorm-growth} gives
\begin{equation}\label{eq:Cnorm-growth-arbitrary-history}
C_{\rm cl}^{\rm norm}
\le
K_c^{\rm norm}M^{9/2}\kappa^2
\sqrt{1+4M^2}\,\Psi_0
=
\mathcal O_c(M^{11/2}\kappa^2)\Psi_0.
\end{equation}
}
Finally, suppose that the prescribed history is constant:
\[
x_{-\tau-1}
=
x_{-\tau}
=
\cdots
=
x_{-1}
=
x_0
=
x_{\rm init},
\]
and define the initial displacement from the minimizer by
$
d_{\rm init}
:=
x_{\rm init}-x_*.
$ 
Since
$
y_j
=
x_j+\nu(x_j-x_{j-1}),
$
all momentum differences vanish on the prescribed history. Hence
$
y_j=x_{\rm init}
$
for every stored \mgd{input index $j=-\tau,\ldots,-1$}.

We next bound the midpoint residual using only strong convexity and
Lipschitz continuity of the gradient. Recall that the objective has been
normalized so that \(f\) is \(1\)-strongly convex and has
\(\kappa\)-Lipschitz gradient. Let
\[
d:=x-x_*,
\qquad
g:=\nabla f(x)-\nabla f(x_*).
\]
Since \(x_*\) is the minimizer of the differentiable function \(f\),
$
\nabla f(x_*)=0,
$
and hence
$
g=\nabla f(x).
$

For a \(1\)-strongly convex function with \(\kappa\)-Lipschitz gradient,
the gradient differences satisfy 
\begin{equation}
(1+\kappa)\langle g,d\rangle
\ge
\|g\|^2+\kappa\|d\|^2,
\label{ineq-interpolation}
\end{equation}
by \eqref{eq:interp}.
We now center the gradient around the midpoint between the two curvature
parameters. Define
\[
m:=\frac{1+\kappa}{2},
\qquad
\omega_q:=\frac{\kappa-1}{2}.
\]
Then
\[
\begin{aligned}
\|g-md\|^2
&=
\|g\|^2
-
2m\langle g,d\rangle
+
m^2\|d\|^2
=
\|g\|^2
-
(1+\kappa)\langle g,d\rangle
+
m^2\|d\|^2.
\end{aligned}
\]
Using \eqref{ineq-interpolation},
we obtain
\[
\begin{aligned}
\|g-md\|^2
\le
\|g\|^2
-
\bigl(\|g\|^2+\kappa\|d\|^2\bigr)
+
m^2\|d\|^2
=
(m^2-\kappa)\|d\|^2.
\end{aligned}
\]
Since
\[
m^2-\kappa
=
\left(\frac{1+\kappa}{2}\right)^2-\kappa
=
\frac{(\kappa-1)^2}{4}
=
\omega_q^2,
\]
it follows that
\[
\|g-md\|
\le
\omega_q\|d\|.
\]
Recalling that \(g=\nabla f(x)\) and \(d=x-x_*\), we therefore have
\[
\left\|
\nabla f(x)
-
\frac{1+\kappa}{2}(x-x_*)
\right\|
\le
\frac{\kappa-1}{2}\|x-x_*\|.
\]
Applying this inequality at \(x=x_{\rm init}\) gives
\begin{equation}
\left\|
\nabla f(x_{\rm init})
-
\frac{1+\kappa}{2}d_{\rm init}
\right\|
\le
\frac{\kappa-1}{2}\|d_{\rm init}\|
=
\omega_q\|d_{\rm init}\|.
\label{ineq-sector}
\end{equation}

We now use this bound to control the normalized midpoint inputs stored in the
initial history.  Recall that the midpoint residual is
\[
v_j
=
\nabla f(y_j)-m(y_j-x_*),
\]
and that
\[
w_j
=
\frac{v_j}{\omega_q}.
\]
Since \(y_j=x_{\rm init}\) on the prescribed history,
$
v_j
=
\nabla f(x_{\rm init})-m d_{\rm init}.
$
The inequality \eqref{ineq-sector} therefore yields
$
\|v_j\|
\le
\omega_q\|d_{\rm init}\|,
$
and hence
\[
\|w_j\|
=
\frac{\|v_j\|}{\omega_q}
\le
\|d_{\rm init}\|.
\]
On the other hand, the rate-weighted midpoint input is
$
\widehat w_j
=
\rho^{-j}w_j.
$
For a stored negative-time index \(j=-\ell\), with \(\ell\ge0\),
$
\widehat w_{-\ell}
=
\rho^\ell w_{-\ell}.
$
Since \(0<\rho<1\),
$
\rho^\ell\le1,
$
and therefore
$
\|\widehat w_{-\ell}\|
\le
\|w_{-\ell}\|
\le
\|d_{\rm init}\|.
$

The stored rate-weighted iterates satisfy the same bound. Indeed,
$
\widehat x_j
=
\rho^{-j}(x_j-x_*),
$
and hence, for \(j=-\ell\),
$
\widehat x_{-\ell}
=
\rho^\ell(x_{\rm init}-x_*)
=
\rho^\ell d_{\rm init}.
$
Thus
\[
\|\widehat x_{-\ell}\|
\le
\|d_{\rm init}\|.
\]

The initial-history vector \(\bar\Xi_0\) contains
\(\tau+2\) stored rate-weighted iterate values and \(\tau\) stored normalized
midpoint inputs. Therefore it contains
$
(\tau+2)+\tau
=
2\tau+2
$
\mgd{potentially nonzero blocks, in addition to the $\tau$ zero multiplier-memory blocks.} Since \mgd{each of the potentially nonzero blocks} has norm at most
\(\|d_{\rm init}\|\),
\[
\begin{aligned}
\|\bar\Xi_0\|^2
&\le
(2\tau+2)\|d_{\rm init}\|^2
=
(M+1)\|d_{\rm init}\|^2,
\end{aligned}
\]
where we used
$
M=2\tau+1.
$
Since \(M\ge1\),
$
M+1\le2M,
$
and therefore
$
\|\bar\Xi_0\|^2
\le
2M\|d_{\rm init}\|^2.
$
Equivalently,
\[
\|\bar\Xi_0\|
\le
\sqrt{2}M^{1/2}\|d_{\rm init}\|.
\]
This proves \eqref{eq:barXi0-constant-history}.
\mgd{Finally, substituting this estimate into the general normalized-history bound
\eqref{eq:Cnorm-growth} gives
\[
C_{\rm cl}^{\rm norm}
\le
K_c^{\rm norm}M^{9/2}\kappa^2\|\bar\Xi_0\|
\le
\sqrt{2}K_c^{\rm norm}M^{9/2}\kappa^2
M^{1/2}\|d_{\rm init}\|.
\]
Hence
\[
C_{\rm cl}^{\rm norm}
\le
2K_c^{\rm norm}M^5\kappa^2
\|x_{\rm init}-x_*\|,
\]
where we used $\sqrt{2}\le2$. This proves \eqref{eq:Cnorm-growth-constant-history}.}
\end{proof}
\begin{remark}[prefactor \saa{bound} without normalization of $\hat v$ sequence]\label{rem:Xi0-compatibility}
Let \(\Xi_0\) denote the same prescribed \mgd{iterate} history as \(\bar\Xi_0\), but with the stored midpoint inputs \(\widehat v_{-1},\ldots,\widehat v_{-\tau}\) in place of \(\widehat w_{-1},\ldots,\widehat w_{-\tau}\), and with the same zero multiplier memories.  Since \(q<1/2\) implies
\(\omega_q=(q^{-2}-1)/2>1\), the relation \(\widehat w=\widehat v/\omega_q\) gives
\[
        \|\bar\Xi_0\|\le \|\Xi_0\|.
\]
Thus Proposition~\ref{prop:normalized-direct-prefactor} also yields the estimate
\[
        C_{\rm cl}^{\rm norm}
        \le K_c^{\rm norm}M^{9/2}\kappa^2\|\Xi_0\|.
\]
\end{remark}

\begin{remark}[dependence on \(M\) and \(\kappa\)]\label{remark-improving-kappa-scaling}
\saa{
The normalization $\widehat w=\widehat v/\omega_q$
yields a positive lower bound, independent of
the condition number, on the dissipation
coefficient associated with $\widehat w$.}
Indeed, before normalization,
\[
\gamma_{\tau,c}(q)
=
\frac{\eta_c^\Phi q^4}{\Upsilon_c^2}
=
\Theta_c(\kappa^{-2}),
\]
whereas
\[
\gamma_{\tau,c}(q)\omega_q^2
=
\frac{\eta_c^\Phi(q^2\omega_q)^2}{\Upsilon_c^2}
=
\Theta_c(1),
\]
because \(\omega_q=\Theta(\kappa)\) and
\(q^2\omega_q=(1-\kappa^{-1})/2=\Theta(1)\).
\saa{
The $\kappa^2$ dependence in the final
prefactor is therefore not caused by
degeneration of the normalized dissipation
coefficient. Rather, it arises from the
product of the two response bounds
$R_0(\kappa,M)$ and $R_w(\kappa,M)$
in \eqref{eq:Hnl-normalized}, while the
remaining term involving $R_0(\kappa,M)$
has only linear dependence on $\kappa$; indeed,}
\[
R_0(\kappa,M)=\cO(\kappa M^{7/2}),
\qquad
R_w(\kappa,M)=\cO(\kappa M),
\]
whose product is \(\cO(\kappa^2M^{9/2})\).
Reading the final estimate directly on \(\widehat x\) then gives
\[
        C_{\rm cl}^{\rm norm}
        =\cO_c(M^{9/2}\kappa^2)\|\bar\Xi_0\|,
\]
and constant history contributes only the additional factor \(\|\bar\Xi_0\|=O(M^{1/2})\|x_{\rm init}-x_*\|\).
\end{remark}

\section{Auxiliary bounds for the explicit prefactor}\label{app:gains}
\mgd{
This appendix establishes two sets of estimates
used in Section~\ref{sec:explicit}: bounds on
the rate-weighted iterates and extrapolated
points, and bounds on the time-summed IQC
supply. These estimates are derived directly
from the linear recursion for the iterates and
the normalized midpoint input $\widehat w$,
without requiring an explicit bound on the
entire augmented state.}
\mgd{We retain the standing assumptions of Appendix~\ref{sec:explicit}:
$\tau\ge1$, $M=2\tau+1$, $c\in(0,2/3)$, and
$0<q\le q_0(\tau,c)$, with the parameters in
\eqref{eq:norm}--\eqref{eq:params} and coordinates centered at $x_*$. Sequence norms are taken over $k\geq0$,
and the arbitrary inputs used in the
linear-recursion estimates belong to
$\ell_2(\mathbb{R}^d)$, the space of
square-summable sequences in $\mathbb{R}^d$.
}
\subsection{Proof of Lemma~\ref{lem:normalized-direct-gains}}
We treat separately the response generated by the normalized midpoint input
and the transient generated by the prescribed initial history.

\paragraph{Zero-state response.}
Recall from \eqref{eq:Gx-poly} that the zero-state transfer function from
\(\widehat v\) to \(\widehat x\) is
\[
\mgb{\widetilde g_x(z)}
=
-\frac{\alpha}{\rho^{\tau+1}}\frac{z}{P_F(z)}.
\]
\saa{On the unit circle, \(|z|=1\); hence,}
Lemma~\ref{lem:finiteq-denominator-floor}, together with
$
\alpha={q^2}/{M^2}$ and
$\rho^{-(\tau+1)}\le R_{\rm rt}$
gives
\[
\|\mgb{\widetilde g_x}\|_{\Hinf}
\le
\frac{R_{\rm rt}q^2/M^2}{1/(20M^2)}
=
20R_{\rm rt}q^2.
\]
Since the normalized \(\Htwo\) norm is the root-mean-square magnitude of the
frequency response, it is bounded by the \(\Hinf\) norm:
\[
\|\mgb{\widetilde g_x}\|_{\Htwo}
\le
\|\mgb{\widetilde g_x}\|_{\Hinf}
\le
20R_{\rm rt}q^2.
\]

Now \(\widehat v=\omega_q\widehat w\). Therefore the transfer function from
\(\widehat w\) to the corresponding zero-state response
\(\widehat x^w\) is \(\omega_q\mgb{\widetilde g_x}\). Using
\(q^2\omega_q\le 1/2\) and \(R_{\rm rt}<3\),
\[
\|\omega_q\mgb{\widetilde g_x}\|_{\Hinf},
\;
\|\omega_q \mgb{\widetilde g_x}\|_{\Htwo}
\le
20R_{\rm rt}q^2\omega_q
\le
10R_{\rm rt}
<
30.
\]
Recall that the \(\Hinf\) norm is the induced \(\ell_2\)-gain of the stable linear
recurrence, whereas the \(\Htwo\) norm bounds its
\(\ell_2\)-to-\(\ell_\infty\) gain. \mgd{These are scalar-transfer norms: the vector transfer matrix is
$(\omega_q\widetilde g_x)I_d$. Parseval's identity applies coordinatewise,
and Cauchy--Schwarz applied to the scalar impulse response gives the
$\ell_2$-to-$\ell_\infty$ bound, with no dimension-dependent factor.} Consequently,
\[
\|\widehat x^w\|_{\ell_2}
\mgd{\leq}
30\|\widehat w\|_{\ell_2},
\qquad
\|\widehat x^w\|_{\ell_\infty}
\mgd{\leq}%
30\|\widehat w\|_{\ell_2},
\]
and hence
\[
\|\widehat x^w\|_{\ell_2}
+
\|\widehat x^w\|_{\ell_\infty}
\le
60\|\widehat w\|_{\ell_2}.
\]

\paragraph{Transient generated by the initial history.}
We next consider the zero-input response, namely, the trajectory generated
solely by the prescribed history when \(\widehat w_k=0\) for \(k\ge0\).
Write the rate-weighted recursion in the form
\begin{equation}\label{eq:x-homogeneous-forcing-form}
\widehat x_{k+1}
=
\sum_{j=0}^{\tau+1}a_j\widehat x_{k-j}
+
b\omega_q\widehat w_{k-\tau},
\qquad
b:=-\frac{\alpha}{\rho^{\tau+1}}.
\end{equation}
The coefficients \(a_j\) contain the two momentum coefficients and the two
midpoint-feedback coefficients. \mgd{When $\tau=1$, the contributions at the common lag $j=1$
are added; the triangle inequality gives the same bounds below.} We first bound their total absolute sum.

Using \(0\le\beta\le1\) and \(\rho\ge19/21\), the momentum contribution
satisfies
\[
\left|\frac{1+\beta}{\rho}\right|
+
\left|\frac{\beta}{\rho^2}\right|
<5.
\]
The two midpoint-feedback coefficients satisfy
\[
\frac{\alpha m(1+\nu)}{\rho^{\tau+1}}
+
\frac{\alpha m\nu}{\rho^{\tau+2}}
\le
R_{\rm rt}\alpha m\bigl((1+\nu)+\nu\bigr).
\]
Since
\[
\alpha m
=
\frac{1+q^2}{2M^2}
\le
\frac{5}{8M^2},
\qquad
(1+\nu)+\nu=2M,
\]
and \(R_{\rm rt}<3\), their total contribution is smaller than \mgd{$15/(4M)\le5/4<3$, since $M\ge3$}. Therefore
$
\sum_{j=0}^{\tau+1}|a_j|
\le 8$. Moreover, by \eqref{eq:b-poly-bound} and \(q^2\omega_q\le1/2\),
\[
|b|\omega_q
\le
R_{\rm rt}\frac{q^2\omega_q}{M^2}
\le
\frac{R_{\rm rt}}{2M^2}
\le1.
\]
We now rewrite the effect of the prescribed history as a finite input to the
same recurrence started from zero history. This allows us to use standard
convolution estimates. For \(k\ge0\), some terms in
\eqref{eq:x-homogeneous-forcing-form} initially involve the prescribed values
\[
\widehat x_{-\tau-1},\ldots,\widehat x_0,
\qquad
\widehat w_{-\tau},\ldots,\widehat w_{-1}.
\]
Define \(e_k\) to be the sum of exactly those terms that involve these
prescribed history values:
\[
e_k
:=
\sum_{\substack{0\le j\le\tau+1\\ k-j\le0}}
a_j\widehat x_{k-j}
+
\begin{cases}
b\omega_q\widehat w_{k-\tau},
& k-\tau<0,\\
0,
& k-\tau\ge0.
\end{cases}
\]
Here the quantities on the right-hand side are the original prescribed
history entries. \rev{Let $\widetilde x=\{\widetilde x_k\}_{k\ge0}$ be the solution of the same
recurrence started from \emph{zero} history and driven by $e$ as an external input,
\begin{equation}\label{eq:zero-history-recurrence}
\widetilde x_{k+1}=\sum_{j=0}^{\tau+1}a_j\widetilde x_{k-j}+e_k\quad(k\ge0),
\qquad
\widetilde x_j=0\quad(j\le0).
\end{equation}
Comparing \eqref{eq:zero-history-recurrence} with \eqref{eq:x-homogeneous-forcing-form}
(in which $\widehat w_k=0$ for $k\ge0$) term by term, and using induction on $k$, shows that
\mgd{\begin{equation}\label{eq:zero-history-vs-zero-input}
\widetilde x_k=\widehat x^0_k\quad\text{for all }k\ge1,
\qquad\text{whereas}\qquad
\widetilde x_0=0,\qquad\widehat x^0_0=\widehat x_0 .
\end{equation}}
Thus the zero-input response $\widehat x^0$ and the auxiliary sequence $\widetilde x$ \mgd{can} differ
only at $k=0$, where $\widehat x^0_0=\widehat x_0$ is a prescribed history value. Below we
bound $\widetilde x$ through $e$ and account for the single term $\widehat x_0$ separately.} The sequence \(e=\{e_k\}_{k\ge0}\) is therefore only a bookkeeping device for
the initial history. It is nonzero only during the first few iterations.
Indeed, if \(k>\tau+1\), then
\[
k-j>0
\qquad
\text{for every }0\le j\le\tau+1,
\]
and also \(k-\tau\ge0\). Hence no prescribed history term remains and
\(e_k=0\). Thus \(e\) is supported on at most
\(\tau+2\le M+1\) time indices.

Every prescribed history entry appearing above is a $d$-dimensional block of
\(\bar\Xi_0\), and hence
\[
||\widehat x_j||
\le
\|\bar\Xi_0\|,
\qquad
||\widehat w_j||
\le
\|\bar\Xi_0\|.
\]
Using the coefficient bounds just established,
\[
\mgb{||e_k||}
\le
\left(
\sum_{j=0}^{\tau+1}|a_j|
+
|b|\omega_q
\right)
\|\bar\Xi_0\|
\le
9\|\bar\Xi_0\|.
\]
Since \(e_k\) is nonzero at no more than \(M+1\) time indices,
\[
\|e\|_{\ell_2}^2
=
\sum_{k\ge0}\|e_k\|^2
\le
(M+1)\bigl(9\|\bar\Xi_0\|\bigr)^2,
\]
and therefore
\begin{equation}\label{eq:effective-history-forcing}
\|e\|_{\ell_2}
\le
9\sqrt{M+1}\,\|\bar\Xi_0\|.
\end{equation}

To identify the transfer function from the effective forcing \(e\) to
\mgd{\(\widetilde x\)}, consider the zero-history recurrence
\[
\mgd{\widetilde x_{k+1}} %
=
\sum_{j=0}^{\tau+1} a_j \mgd{\widetilde x_{k-j}} %
+
e_k.
\]
Taking the \(z\)-transform gives
\[
\left(
z-\sum_{j=0}^{\tau+1}a_j z^{-j}
\right) \mgd{\widetilde X(z)} %
=
E(z).
\]
Multiplying by \(z^{\tau+1}\) clears the negative powers of \(z\), and the
resulting polynomial multiplying \mgd{\(\widetilde X(z)\)} %
is precisely the
characteristic polynomial \(P_F(z)\). Hence
\[
P_F(z) \widetilde X(z)%
=
z^{\tau+1}E(z),
\]
so the zero-history transfer function from \(e\) to \(\widetilde x\) is
\[
\mgb{g_e(z)}=
\frac{z^{\tau+1}}{P_F(z)} \quad \mbox{with} \quad
\widetilde X(z) = g_e(z)E(z)
\]
\mgd{Since $P_F$ is Schur stable and has degree $\tau+2$, $g_e$ is
stable and strictly proper.}\footnote{\mgd{A rational transfer function is \emph{strictly proper}
if its numerator has strictly lower degree than its
denominator, equivalently, if it tends to zero as
$z\to\infty$.}} 
The factor \(z^{\tau+1}\) %
\mgd{has unit modulus on the unit circle} and therefore does not affect the \(\mathcal H_\infty\) or
\(\mathcal H_2\) norm estimates used below. %
Let \(\mgb{\zeta}=\{\mgb{\zeta}_j\}_{j\ge0}\) denote the causal impulse response corresponding to the transfer function \mgb{$g_e$}, i.e. 
\(\mgb{\zeta}\) is the trajectory produced by the zero-history recurrence when the
input is a unit impulse,
\[
e_0=1,
\qquad
e_k=0
\quad\text{for }k\ne0.
\]
The coefficients \mgb{\(\zeta_j\)} are therefore determined by the recurrence
coefficients \(a_j\) \mgd{with $\zeta_0=0$}. For \mgd{the history forcing \(e\) above, extended by zero for negative indices}, linearity and time invariance
give the convolution representation
\[
\widehat x_k^0
=
(\zeta*e)_k
=
\sum_{j\ge0}\zeta_j e_{k-j} \quad 
\mgd{\mbox{for} \quad k\geq 1}.
\]

The factor \(z^{\tau+1}\) in \mgb{\(g_e\)} has unit modulus on the unit circle.
Therefore it does not affect the frequency-response norms, and
\[
\|\mgb{g_e}\|_{\Hinf}
=
\left\|\frac{1}{P_F}\right\|_{\Hinf},
\qquad
\|\mgb{g_e}\|_{\Htwo}
=
\left\|\frac{1}{P_F}\right\|_{\Htwo}.
\]
By Lemma~\ref{lem:finiteq-denominator-floor},
\[
\min_{|z|=1}|P_F(z)|
\ge
\frac{1}{20M^2},
\]
and consequently
\[
\|\mgb{g_e}\|_{\Hinf}
=
\sup_{|z|=1}\frac{1}{|P_F(z)|}
\le
20M^2.
\]
Also,
\[
\|\mgb{g_e}\|_{\Htwo}
\le
\|\mgb{g_e}\|_{\Hinf}
\le
20M^2.
\]
The \(\Hinf\) estimate gives the induced \(\ell_2\) bound
\[
\|\widehat x^0\|_{\ell_2}
\le
\wubtext{\|\widehat x_0\|+{}}\|g_e\|_{\Hinf}\|e\|_{\ell_2}.
\]
For the \(\ell_\infty\) estimate, the convolution formula and
Cauchy--Schwarz give, for every \(k\wubtext{{}\ge1}\),
\[
\|\widehat x_k^0\|
=
\left\|
\sum_{j\ge0}\zeta_j e_{k-j}
\right\|
\le
\left(\sum_{j\ge0}|\zeta_j|^2\right)^{1/2}
\left(\sum_{j\ge0} \mgd{\|e_{k-j}\|^2}\right)^{1/2}
\le
\|\zeta\|_{\ell_2}\|e\|_{\ell_2}.
\]
By Parseval's identity,
$
\|\zeta\|_{\ell_2}
=
\|g_e\|_{\Htwo},
$
and hence
\[
\|\widehat x^0\|_{\ell_\infty}
\le
\wubtext{\|\widehat x_0\|+{}}\|g_e\|_{\Htwo}\|e\|_{\ell_2}.
\]
Combining these estimates with
\eqref{eq:effective-history-forcing} yields
\[
\|\widehat x^0\|_{\ell_2},
\;
\|\widehat x^0\|_{\ell_\infty}
\le
\wubtext{\|\widehat x_0\|+{}}20M^2\,9\sqrt{M+1}\,\|\bar\Xi_0\|
\le
\wubtext{217M^{5/2}\|\bar\Xi_0\|\le{}}240M^{5/2}\|\bar\Xi_0\|.
\] \wubtext{Here we used $\|\widehat x_0\|\le\|\bar\Xi_0\|$ and $\sqrt{M+1}\le(6/5)\sqrt M$ for $M\ge3$.}
{Finally, the extrapolated sequence \(\widehat y\) depends only on the current
and previous values of \(\widehat x\).  Indeed, from
$
y_k=(1+\nu)x_k-\nu x_{k-1}
$
and the rate weighting \(\widehat x_k=\rho^{-k}x_k\), we obtain
\[
\widehat y_k
=
(1+\nu)\widehat x_k
-
\nu\rho^{-1}\widehat x_{k-1}.
\]
Thus \(\widehat y\) is obtained from \(\widehat x\) by taking a fixed linear
combination of two consecutive iterates.  By the triangle inequality,
\[
\|\widehat y\|_{\ell_2}
\le
\bigl(|1+\nu|+\nu\rho^{-1}\bigr)
\|\widehat x\|_{\ell_2}\wubtext{{}+\nu\rho^{-1}\|\widehat x_{-1}\|}.
\]
Since
\[
\nu=M-\frac12,
\qquad
\rho^{-1}\le R_{\rm rt}<3,
\]
we have
\[
|1+\nu|+\nu\rho^{-1}
\le
M+\frac12
+
3\left(M-\frac12\right)
<
4M.
\]
Hence
\[
\|\widehat y\|_{\ell_2}
\mathrel{\wubtext{\le}}
4M\|\widehat x\|_{\ell_2}\wubtext{{}+3M\|\widehat x_{-1}\|}.
\] \wubtext{Using
$\|\widehat x^0\|_{\ell_2}\le217M^{5/2}\|\bar\Xi_0\|$
and $\|\widehat x_{-1}\|\le\|\bar\Xi_0\|$, we obtain
\[
\|\widehat y^0\|_{\ell_2}
\le (4M\cdot217M^{5/2}+3M)\|\bar\Xi_0\|
\le960M^{7/2}\|\bar\Xi_0\|.
\]
For the zero-state response driven by $\widehat w$, the extra term vanishes because
$\widehat x_{-1}^w=0$, giving
$\|\widehat y^w\|_{\ell_2}\le240M\|\widehat w\|_{\ell_2}$. \saa{This completes the proof of
\eqref{eq:y0-normalized-gain}.}}
}
\subsection{\mgd{Bounds for the IQC supply}}
\label{app:gains-supply}
\mgd{The preceding subsection bounds the contributions of the
prescribed initial history and the normalized midpoint input
to the iterates and extrapolated points. We now use these
estimates to control their contributions to the time-summed
IQC supply $\Sigma$. In particular, the following lemma
bounds the terms arising solely from the initial history
and the cross terms between the initial-history and
input-driven responses.}
\begin{lemma}[normalized supply splitting]\label{lem:normalized-supply-splitting} In the context of Section \ref{subsec-fdi-normalized-midpoint}, \mgd{for $\widehat w\in\ell_2$ and zero initial multiplier memories}, with the decomposition \((\widehat r,\widehat s)=(\widehat r^0,\widehat s^0)+(\widehat r^w,\widehat s^w)\), the inequalities \eqref{eq:supply-zero-input-normalized}--\eqref{eq:supply-cross-normalized} hold.
\end{lemma}

\begin{proof}

\mgd{
Recall the bilinear expression $\Sigma(r,s)$
defined in \eqref{eq-sigma-rs} for
$r,s\in\ell_2$. Let $S$ denote the
unit-delay operator on sequences, with
zero extension to negative indices. Thus,
for every integer $j\ge1$,
\[
(S^j s)_k=
\begin{cases}
s_{k-j}, & k\ge j,\\
0, & 0\le k<j.
\end{cases}
\]
Consequently,
\[
\sum_{k\ge0}\langle r_k,(S^j s)_k\rangle
=
\sum_{k\ge j}\langle r_k,s_{k-j}\rangle.
\]
}
\mgd{Each of the finitely many lagged sums is absolutely convergent by
Cauchy--Schwarz. The} total absolute coefficient weight is at most
\begin{equation}
        2+\frac{2}{\tau}\sum_{j=1}^{\tau}\rho^j\le4,
        \label{eq-weight-bound}
\end{equation}
and \(S^j\) is nonexpansive on \(\ell_2\), i.e. $\|S^j s\|_{\ell_2} \leq \| s\|_{\ell_2}$. Therefore, \mgd{taking absolute values} 
in \eqref{eq-sigma-rs}, \mgd{applying Cauchy--Schwarz to each lagged sum}, and using the triangle inequality and the weight bound \eqref{eq-weight-bound},
\[
        |\Sigma(r,s)|
        \le4\|r\|_{\ell_2}\|s\|_{\ell_2}
        \le(\|r\|_{\ell_2}+\|s\|_{\ell_2})^2.
\]
Applying the preceding supply estimate to
\((r,s)=(\widehat r^0,\widehat s^0)\) gives
\[
|\Sigma(\widehat r^0,\widehat s^0)|
\le
4\|\widehat r^0\|_{\ell_2}\|\widehat s^0\|_{\ell_2}.
\]
Using \(4ab\le(a+b)^2\) for \(a,b\ge0\), followed by
\eqref{eq:rs-normalized-gains} again, we obtain
\[
|\Sigma(\widehat r^0,\widehat s^0)|
\le
\left(
\|\widehat r^0\|_{\ell_2}
+
\|\widehat s^0\|_{\ell_2}
\right)^2
\le
R_0(\kappa,M)^2\|\bar\Xi_0\|^2,
\]
which is \eqref{eq:supply-zero-input-normalized}. By bilinearity of $\Sigma$ and the decompositions
$\widehat r=\widehat r^0+\widehat r^w$ and
$\widehat s=\widehat s^0+\widehat s^w$, we have
\[
\begin{aligned}
\Sigma(\widehat r,\widehat s)
&=\Sigma(\widehat r^0,\widehat s^0)
 +\Sigma(\widehat r^w,\widehat s^w)
 +\Sigma_{\rm cross},
\end{aligned}
\]
where
\[
\Sigma_{\rm cross}
:=\Sigma(\widehat r^0,\widehat s^w)
 +\Sigma(\widehat r^w,\widehat s^0).
\]
Recall that the bound
\[
|\Sigma(r,s)|
\le 4\|r\|_{\ell_2}\|s\|_{\ell_2}
\]
holds for any square-summable sequences $r,s$.
Applying this bound to the two cross terms and using
the triangle inequality gives
\[
\begin{aligned}
|\Sigma_{\rm cross}|
&\le
4\|\widehat r^0\|_{\ell_2}
 \|\widehat s^w\|_{\ell_2}
+
4\|\widehat r^w\|_{\ell_2}
 \|\widehat s^0\|_{\ell_2}
\\
&\le
4\bigl(
\|\widehat r^0\|_{\ell_2}
+\|\widehat s^0\|_{\ell_2}
\bigr)
\bigl(
\|\widehat r^w\|_{\ell_2}
+\|\widehat s^w\|_{\ell_2}
\bigr).
\end{aligned}
\]
The second inequality follows by expanding the
product on the right and observing that all
terms are non-negative. Finally, \mgd{using the normalized $\ell_2$-norm estimates} \eqref{eq:rs-normalized-gains} and \eqref{eq:rw-normalized-gain}, we obtain
\[
\begin{aligned}
|\Sigma_{\rm cross}|
&\le
4R_0(\kappa,M)R_w(\kappa,M)
\|\bar\Xi_0\|\,\|\widehat w\|_{\ell_2},
\end{aligned}
\]
which proves \eqref{eq:supply-cross-normalized}.

\end{proof}

\end{document}